\pdfoutput=1
\documentclass[twoside,leqno]{article}
\newif\ifpersonal

\usepackage{amsmath,amsthm,amscd,amssymb,mathrsfs,mathtools,bm,tensor} 
\usepackage{microtype,lmodern,textcomp} 
\usepackage{enumitem,comment,braket,xspace,tikz-cd,adjustbox} 
\usepackage[utf8]{inputenc} 
\usepackage[T1]{fontenc} 
\usepackage[centering,vscale=0.7,hscale=0.7]{geometry} 
\usepackage[hidelinks]{hyperref}
\usepackage[capitalize]{cleveref}
\usepackage{fancyhdr}

\setlist[enumerate]{label=(\arabic*), ref=\arabic*, itemsep=0pt, topsep=3pt}

\usepackage{titling}
\usepackage{arydshln}
\usepackage[all]{xy}

\usepackage{fancyhdr}
\newcommand\blfootnote[1]{%
	\begingroup
	\renewcommand\thefootnote{}\footnote{#1}%
	\addtocounter{footnote}{-1}%
	\endgroup
}

\usepackage[numbers]{natbib}
\numberwithin{equation}{section}

\theoremstyle{plain}
\newtheorem{theorem}[equation]{Theorem}
\newtheorem*{theorem*}{Theorem}
\newtheorem{lemma}[equation]{Lemma}
\newtheorem*{lemma*}{Lemma}
\newtheorem{claim}[equation]{Claim}
\newtheorem*{claim*}{Claim}

\newtheorem{proposition}[equation]{Proposition}
\newtheorem*{proposition*}{Proposition}
\newtheorem{corollary}[equation]{Corollary}

\theoremstyle{definition}
\newtheorem{definition}[equation]{Definition}
\newtheorem{definition-theorem}[equation]{Definition-Theorem}
\newtheorem{definition-lemma}[equation]{Definition-Lemma}

\newtheorem{notation}[equation]{Notation}
\newtheorem{example}[equation]{Example}
\newtheorem*{example*}{Example}
\newtheorem{remark}[equation]{Remark}
\newtheorem{variant}[equation]{Variant}

\ifpersonal
\newcommand{\personal}[1]{\textcolor[rgb]{0,0,1}{(Personal: #1)}}

\else
\newcommand{\personal}[1]{\ignorespaces}
\newcommand{\discussion}[1]{\ignorespaces}

\fi

\providecommand{\abs}[1]{\lvert#1\rvert}

\newcommand{\bbC}{\mathbb C}
\newcommand{\bbD}{\mathbb D}

\newcommand{\bbG}{\mathbb G}

\newcommand{\bbN}{\mathbb N}
\newcommand{\bbP}{\mathbb P}
\newcommand{\bbQ}{\mathbb Q}
\newcommand{\bbR}{\mathbb R}

\newcommand{\bbZ}{\mathbb Z}

\newcommand{\fX}{\mathfrak X}

\newcommand{\fZ}{\mathfrak Z}

\newcommand{\cA}{\mathcal A}

\newcommand{\cE}{\mathcal E}

\newcommand{\cH}{\mathcal H}

\newcommand{\cK}{\mathcal K}

\newcommand{\cM}{\mathcal M}

\newcommand{\cO}{\mathcal O}

\newcommand{\cR}{\mathcal R}

\newcommand{\cT}{\mathcal T}

\newcommand{\hU}{\widehat U}

\newcommand{\tB}{\widetilde{B}}

\newcommand{\tG}{\widetilde G}
\newcommand{\tH}{\widetilde H}

\newcommand{\tQ}{\widetilde Q}

\newcommand{\tU}{\widetilde U}
\newcommand{\tX}{\widehat{X}}

\newcommand{\tbeta}{\widetilde\beta}
\newcommand{\ttau}{\widetilde\tau}

\makeatletter
\let\save@mathaccent\mathaccent
\newcommand*\if@single[3]{%
	\setbox0\hbox{${\mathaccent"0362{#1}}^H$}%
	\setbox2\hbox{${\mathaccent"0362{\kern0pt#1}}^H$}%
	\ifdim\ht0=\ht2 #3\else #2\fi
}
\newcommand*\rel@kern[1]{\kern#1\dimexpr\macc@kerna}
\newcommand*\widebar[1]{\@ifnextchar^{{\wide@bar{#1}{0}}}{\wide@bar{#1}{1}}}
\newcommand*\wide@bar[2]{\if@single{#1}{\wide@bar@{#1}{#2}{1}}{\wide@bar@{#1}{#2}{2}}}
\newcommand*\wide@bar@[3]{%
	\begingroup
	\def\mathaccent##1##2{%
		\let\mathaccent\save@mathaccent
		\if#32 \let\macc@nucleus\first@char \fi
		\setbox\z@\hbox{$\macc@style{\macc@nucleus}_{}$}%
		\setbox\tw@\hbox{$\macc@style{\macc@nucleus}{}_{}$}%
		\dimen@\wd\tw@
		\advance\dimen@-\wd\z@
		\divide\dimen@ 3
		\@tempdima\wd\tw@
		\advance\@tempdima-\scriptspace
		\divide\@tempdima 10
		\advance\dimen@-\@tempdima
		\ifdim\dimen@>\z@ \dimen@0pt\fi
		\rel@kern{0.6}\kern-\dimen@
		\if#31
		\overline{\rel@kern{-0.6}\kern\dimen@\macc@nucleus\rel@kern{0.4}\kern\dimen@}%
		\advance\dimen@0.4\dimexpr\macc@kerna
		\let\final@kern#2%
		\ifdim\dimen@<\z@ \let\final@kern1\fi
		\if\final@kern1 \kern-\dimen@\fi
		\else
		\overline{\rel@kern{-0.6}\kern\dimen@#1}%
		\fi
	}%
	\macc@depth\@ne
	\let\math@bgroup\@empty \let\math@egroup\macc@set@skewchar
	\mathsurround\z@ \frozen@everymath{\mathgroup\macc@group\relax}%
	\macc@set@skewchar\relax
	\let\mathaccentV\macc@nested@a
	\if#31
	\macc@nested@a\relax111{#1}%
	\else
	\def\gobble@till@marker##1\endmarker{}%
	\futurelet\first@char\gobble@till@marker#1\endmarker
	\ifcat\noexpand\first@char A\else
	\def\first@char{}%
	\fi
	\macc@nested@a\relax111{\first@char}%
	\fi
	\endgroup
}
\makeatother
\newcommand{\op}[1]{\operatorname{#1}}

\newcommand{\dbb}[1]{[\![#1]\!]}
\newcommand{\dbp}[1]{(\!(#1)\!)}

\newcommand{\strongemph}[1]{\textbf{#1}}

\newcommand{\loc}{\mathrm{loc}}
\newcommand{\sF}{\mathsf F}
\newcommand{\Atoms}{\mathsf{Atoms}}

\newcommand{\HAtoms}{\mathsf{HAtoms}}

\newcommand{\an}{\mathrm{an}}
\newcommand{\prim}{\mathrm{prim}}
\newcommand{\QDM}{\mathrm{QDM}}
\newcommand{\la}{\mathrm{la}}
\newcommand{\even}{\mathrm{ev}}
\newcommand{\odd}{\mathrm{odd}}
\newcommand{\td}{\tilde d}
\newcommand{\tq}{\tilde q}
\newcommand{\fq}{\mathfrak q}
\newcommand{\fs}{\mathfrak s}
\newcommand{\fy}{\mathbf{y}}

\newcommand{\Se}{S^\mathrm{e}}
\newcommand{\Aff}{\mathsf{Aff}}
\newcommand{\MT}{\mathsf{MT}}
\newcommand{\hooklongrightarrow}{\lhook\joinrel\longrightarrow}

\newcommand{\Gm}{\mathbb G_\mathrm{m}}
\newcommand{\Ga}{\mathbb G_\mathrm{a}}

\newcommand{\gs}{\geqslant}

\newcommand{\tphi}{\widetilde\phi}

\newcommand{\id}{\mathrm{id}}

\newcommand{\ev}{\mathrm{ev}}
\newcommand{\pr}{\mathrm{pr}}
\newcommand{\longto}{\longrightarrow}
\newcommand{\bbk}{\Bbbk}
\newcommand{\obbk}{\overline{\bbk}}
\newcommand{\bkappa}{\boldsymbol{\kappa}}

\newcommand{\nc}{\textsf{nc}}

\newcommand{\Deg}{\mathsf{Deg}}
\newcommand{\Gr}{\mathsf{Gr}}

\newcommand{\GW}{Gromov-Witten\xspace}

\newcommand{\Fbun}{\operatorname{{\sf Fbun}}}

\newcommand{\bw}{\boldsymbol{\mathsf{w}}}
\newcommand{\hodge}{\mathsf{Hod}}
\newcommand{\ztwogr}{\mathsf{S}}
\newcommand{\HHgr}{\mathsf{H}}
\newcommand{\motM}{\mathsf{Mot}}
\newcommand{\bdelta}{\boldsymbol{\delta}}
\newcommand{\Mbar}{\overline{\mathcal{M}}}

\newcommand{\homeq}{{\boldsymbol{\sim}_{\mathsf{hom}}}}

\newcommand{\NS}{\mathsf{NS}}

\newcommand{\maxeigbup}{\mathbb{U}}
\newcommand{\univ}{\mathfrak{U}}
\newcommand{\univop}{\mathfrak{M}}
\newcommand{\folunivop}{\mathfrak{L}}
\newcommand{\ball}{\mathbb{B}}

\newcommand{\CA}[1]{\mathcal{C}^{#1}}
\newcommand{\GA}[1]{\mathsf{G}^{#1}}
\newcommand{\bfh}{\boldsymbol{\mathfrak{h}}}
\newcommand{\bfH}{\mathbf{H}}

\newcommand{\bbNE}{\underline{\mathsf{NE}}}
\newcommand{\Nov}{\mathsf{Nov}}

\DeclareMathOperator{\Cone}{Cone}
\DeclareMathOperator{\Coeff}{Coeff}
\DeclareMathOperator{\im}{im}
\DeclareMathOperator{\Sp}{Sp}
\DeclareMathOperator{\Spec}{Spec}

\DeclareMathOperator{\Rep}{Rep}
\DeclareMathOperator{\Gal}{Gal}
\DeclareMathOperator{\codim}{codim}
\DeclareMathOperator{\End}{End}

\DeclareMathOperator{\Aut}{Aut}

\DeclareMathOperator{\NE}{NE}
\DeclareMathOperator{\GL}{GL}
\DeclareMathOperator{\Lie}{Lie}

\DeclareMathOperator{\Eu}{\mathsf{Eu}}
\DeclareMathOperator{\Spf}{Spf}

\makeatletter
\newcommand{\directint}{\mathop{\mathpalette\direct@int\relax}\!\int}
\newcommand{\direct@int}[2]{%
  \sbox\z@{%
    \m@th
    \ooalign{%
      \hidewidth$\demote@style{#1}{\boxplus}$\hidewidth\cr
      $#1\phantom{\int}$\cr
    }%
  }%
  \wd\z@=\z@\box\z@
}
\newcommand{\demote@style}[2]{%
  \ifx#1\displaystyle\scriptstyle#2\else
    \raise.5\fontdimen22\textfont2\hbox{$\scriptscriptstyle#2$}%
  \fi
}
\makeatother

\DeclareMathAlphabet{\mathscrbf}{OMS}{mdugm}{b}{n}
\newcommand{\mycal}[1]{\mathscr{#1}}
\newcommand{\mycalb}[1]{\mathscrbf{#1}}

\DeclareFontFamily{U}{MnSymbolC}{}
\DeclareSymbolFont{MnSyC}{U}{MnSymbolC}{m}{n}
\DeclareFontShape{U}{MnSymbolC}{m}{n}{
    <-6>  MnSymbolC5
   <6-7>  MnSymbolC6
   <7-8>  MnSymbolC7
   <8-9>  MnSymbolC8
   <9-10> MnSymbolC9
  <10-12> MnSymbolC10
  <12->   MnSymbolC12}{}
\DeclareMathSymbol{\iprod}{\mathbin}{MnSyC}{'270}

\DeclareMathAlphabet{\mathbbold}{U}{bbold}{m}{n}

\newcommand{\bplus}{
  \mathop{
    \vphantom{\bigoplus} 
    \mathchoice
      {\vcenter{\hbox{\resizebox{\widthof{$\displaystyle\bigoplus$}}{!}{$\boxplus$}}}}
      {\vcenter{\hbox{\resizebox{\widthof{$\bigoplus$}}{!}{$\boxplus$}}}}
      {\vcenter{\hbox{\resizebox{\widthof{$\scriptstyle\oplus$}}{!}{$\boxplus$}}}}
      {\vcenter{\hbox{\resizebox{\widthof{$\scriptscriptstyle\oplus$}}{!}{$\boxplus$}}}}
  }\displaylimits 
}

\DeclarePairedDelimiter\floor{\lfloor}{\rfloor}

\newcommand{\kay}{\cK}
\newcommand{\kaybar}{\overline{\kay}}
\newcommand{\kbar}{\overline{k}}

\newcommand{\qup}{\star}
\newcommand{\forB}[1]{B_{#1}^{\mathbf{f}}}
\newcommand{\forcH}{\cH^{\mathbf{f}}}
\newcommand{\fornabla}{\nabla^{\mathbf{f}}}
\newcommand{\affa}[1]{{\mathfrak{a}}_{#1}}
\newcommand{\fp}{\circ}
\newcommand{\CH}{\mathsf{CH}} 
\newcommand{\CHhom}{\mathsf{CH}^{\mathsf{hom}}}
\newcommand{\omB}{\mycalb{B}}
\newcommand{\omH}{\mycalb{H}}
\newcommand{\bnabla}{\boldsymbol{\nabla}}
\newcommand{\cf}{\mathfrak{cf}}
\newcommand{\CF}{\mathsf{CF}}
\newcommand{\germG}{\mathcal{G}}
\newcommand{\balpha}{\boldsymbol{\alpha}}
\newcommand{\hpl}{\boldsymbol{P}}
\newcommand{\smallbase}{\mathfrak{b}}
\newcommand{\smallfiber}{\mathfrak{h}}
\newcommand{\imageH}{\boldsymbol{R}}
\newcommand{\dd}{\boldsymbol{d}}
\newcommand{\rr}{\color{red}{1}}
\newcommand{\str}{{\color{red}{*}}}
\newcommand{\net}{\boldsymbol{A}}
\newcommand{\actleft}{\rotatebox[origin=c]{-90}{$\boldsymbol{\circlearrowright}$}}
\newcommand{\exch}{\boldsymbol{\mathsf{n}}}
\newcommand{\bpsi}{\boldsymbol{\psi}}
\newcommand{\bPsi}{\boldsymbol{\Psi}}
\newcommand{\triv}{\mathbbold{1}}

\newcommand{\bluebf}[1]{\textcolor{black}{\mathbf{#1}}}

\title{\large\bfseries \uppercase{Birational Invariants from Hodge Structures  \\ 
and Quantum Multiplication}}

\author{Ludmil Katzarkov, Maxim Kontsevich, Tony Pantev, Tony Yue Yu}

\date{}

\begin{document}

\maketitle

\blfootnote{\emph{Date}: August 7, 2025 (Revised March 5, 2026).}


\begin{quotation}
\small
\textbf{Abstract. }
We introduce new invariants of smooth complex projective varieties, called Hodge atoms. Their construction combines rational Gromov-Witten invariants with classical Hodge theory and relies on the notion of an F-bundle, which is a non-archimedean version of a non-commutative Hodge structure. The Hodge atoms arise from the spectral decomposition of the F-bundle under the Euler vector field action, and behave additively under blowups, in accordance with Iritani's blowup theorem. We compute several examples and demonstrate applications to birational geometry. In particular, we prove that a very general cubic fourfold is not rational. 
We also obtain a new proof of the equality of Hodge numbers of birational Calabi-Yau manifolds in any dimension.
Furthermore, we show that the framework naturally extends to representations of other motivic Galois groups. This enables the theory of atoms to produce new obstructions to rationality over non-algebraically closed fields of characteristic zero as well.
\end{quotation}

\tableofcontents

\section{Introduction}

We propose new birational invariants of algebraic
varieties that arise from a combination of classical Hodge theory and
Gromov-Witten theory. Even though we were guided and inspired by ideas
from Homological Mirror Symmetry (HMS), our arguments do not use the
mirror duality directly. Instead of HMS we utilize motivic information
encoded in the algebraic and symplectic geometry on the same
projective manifold.  This approach leads to new non-rationality
results, and resolves a number of classical problems studied in the
foundational works of Beauville, Donagi, Debarre, Hassett, Huybrechts, Kulikov,
Kuznetsov, Pirutka, Thomas, Tschinkel, Voisin, and many others
\cite{Beauville-quadrics,Beauville-Donagi,DK-GM,DK-GMIJ,debarre-GM,
Hassett-cubic4fold,Hassett-survey,HPT-qb,HPT-ds,HPT-threeq,
Huybrechts-survey,Hyubrechts-K3cubic,
kulikov-cubic4fold,Kuznetsov-cubic4fold,Kuznetsov-derived,KP-Mukai,
KP-Fano3nc,Pirutka-0cycles,ABP-qb,AT,
Voisin-unirational,Voisin-decomposition,Voisin-stable,Voisin-integral}.
In particular, we prove the non-rationality of very general
cubic fourfolds over the complex numbers.
We also study the non-rationality of varieties over non-algebraically closed fields, as well as applications and examples in higher dimensions.

In this work, we only scratch the surface of the realm of potential applications of quantum cohomology to birational geometry.
Given a smooth complex projective variety $X$, our main invariants, called \emph{Hodge atoms}, are pieces in an intrinsic decomposition of the so-called $\mathsf{A}$-model F-bundle associated to $X$, which is a non-archimedean version of variation of $\mathsf{A}$-model non-commutative Hodge structures  (or \nc-Hodge structures for short, see \cite{Katzarkov_Hodge_theoretic_aspects}).
The $\mathsf{A}$-model F-bundle consists of enumerative data (coming from Gromov-Witten theory) on top of the $\bbZ/2$-weighted version\footnote{Having $\bbZ/2$-weights here  means that we have folded the $\bbZ$-grading of cohomology to a $\bbZ/2$-grading, or equivalently,  that we consider Hodge structures modulo Tate twists.
The loss of the $\bbZ$-grading is natural from the point of view of non-commutative geometry, see \cite{Katzarkov_Hodge_theoretic_aspects}.} of the classical polarized $\bbQ$-Hodge structure $H$ of $X$.
We show that at a general point in the locus of Hodge classes inside the base of the F-bundle (i.e.\ the Frobenius manifold), the generalized eigenspaces for the quantum multiplication by the Euler vector field are compatible with  the Hodge structure on $H$. Concretely, we show that each generalized eigenspace  decomposes canonically into $\widebar{\bbQ}$-linear representations of the Mumford-Tate group $\hodge$ of $\bbZ/2$-graded polarizable pure Hodge structures, which we call \emph{Hodge atoms}.
We prove that the \emph{chemical formula} or  \emph{atomic composition} of $X$, i.e.\ the collection of Hodge atoms appearing in the decomposition of $H$, behaves additively under blowups and can be used as an obstruction to birational equivalence.

The setup readily generalizes to other realizations of motives leading to variants of the notion of Hodge atoms.
The formalism of motivic atoms we develop here is just the beginning of a more  intricate story.
Atoms are fairly coarse building blocks of \nc-Hodge structures.
For instance there are no natural morphisms of atoms, so they do not form a category and thus are not well-suited for studying specific birational maps.
Atoms work well as first level obstructions to birational equivalence but are far from being complete birational invariants.
We will explore other more enhanced theories in future works.

\medskip

Here is a more detailed overview of our work. Suppose for concreteness $X$ is a complex smooth projective variety. 
Let $H\coloneqq H^{\bullet}_{B}(X,\bbQ )$  be the rational Betti cohomology of $X$ taken with its \linebreak $\bbZ/2$-grading, let 
$\{T_i\}$ be a homogeneous basis of $H$, and let $t=(t_i)$ be the dual coordinates on $H$.
Write $\NE(X,\bbZ) \subset \mathsf{N}_{1}(X,\bbZ) = \op{im}\big(\CH_{1}(X) \to H_{2}(X,\bbQ)\big)$ for the monoid of numerically effective curve classes  on $X$, and let
$\bbQ \dbb{q}$ denote the completion of the monoid algebra \linebreak $\bbQ [\NE (X,\bbZ )] = \bbQ [q^{\beta}\, \vert\, \beta\in \NE (X,\bbZ )]$ with respect to the maximal ideal $(q^{\beta} , \,\beta\neq 0)$.
The quantum product for $X$
\[
\qup \colon \, H\otimes H \longrightarrow H\dbb{q,t}.
\]
is a deformation of the cup product on $H$ over the big Novikov ring $\Nov_{X} \coloneqq \bbQ\dbb{q,t}$ of $X$. This deformation is  
given by Gromov-Witten invariants (see \cite{Kontsevich_Gromov-Witten_classes}, which we will review in \cref{sssec:Amodelformal}).
It gives rise to a structure of formal Frobenius manifold on a quotient of $\Spf \Nov_{X}$ by a foliation given by the divisor axiom (see \cite{Dubrovin_Geometry_of_2D,Manin_Frobenius_manifolds}).

\medskip

To extract  birational invariants from symplectic data, we will work with a variant of the notion of Frobenius manifold, which we call a F-bundle (see \cref{sec:definition_of_F-bundles}).
It is a non-archimedean version of several existing notions in the literature: TE-structure \cite{hertling_tt*_geometry,hertling-frobenius}, quantum $D$-module \cite{Iritani_Quantum_D-modules}, and a variation of \nc-Hodge structures \cite{Katzarkov_Hodge_theoretic_aspects}.
We have to pass to non-archimedean geometry for two reasons. First, the complex analytic convergence of the series defining the
quantum product is not known in general. 
Second, and more crucially, the spectral decomposition theorem does not hold over the complex numbers due to the Stokes phenomenon \cite{Katzarkov_Hodge_theoretic_aspects,HYZZ_Decomposition}.

Abstractly a $\bbk$-analytic  F-bundle is specified by data $(\cH,\nabla)/B$, where $B$ is a $\bbk$-analytic supermanifold, $\cH$ is a $\bbk$-analytic vector bundle on $B\times \bbD$, where $\bbD$ denotes the germ at $0$ in an analytic disk with coordinate $u$, and $\nabla$ is a meromorphic connection on $\cH$ with  poles at most at $u =0$, and such that 
$\nabla_{u^{2}\partial_{u}}$ and $\nabla_{u\xi}$ have no poles for all $\bbk$-analytic vector 
fields $\xi \in T_{B}$. Sections $\xi \in T_{B}$ act on $\cH|_{u=0}$ by the residues of $\nabla_{\xi}$ at $u=0$, and we say that the F-bundle is \strongemph{overmaximal} if locally on B we can find a cyclic vector for the action, and further say it is  \strongemph{maximal} if the evaluation on the cyclic vector gives an isomorphism between $T_{B}$ and $\cH|_{u=0}$.

The $\mathsf{A}$-model F-bundle $(\omH,\bnabla)/\omB_{X}$ associated to $X$ consists of a base $\omB_{X}$ which is roughly a non-archimedean $\bbk$-analytic space obtained from $\Spf\bbQ\dbb{q,t}$, a trivial vector bundle $\omH$ with fiber $H$ over $\omB\times \bbD$, and a meromorphic flat connection $\nabla$ on $\cH$ called the quantum connection:
\[
\left| \ 
\begin{aligned}
    \nabla_{u\partial_u} &= u\partial_u - u^{-1}\Eu\qup(-) + \Gr,\\
    \nabla_{\partial_{t_i}} &= \partial_{t_i} + u^{-1} T_i \qup(-),\\
    \nabla_{\xi q\partial_q} &= \xi q\partial_q + u^{-1}\xi\qup(-) \quad\text{ for }\xi\in H^2(X,\bbk),
\end{aligned}
\right.
\]
where $\Eu \in \Gamma(\omB,\omH)$ is the Euler element, $\Gr$ is the grading operator, and $\xi q\partial_q$ is the derivation of $\bbk\dbb{q}$ given by $\xi q\partial_q(q^\beta)=(\beta\cdot\xi)q^\beta$ (see \cref{subsec:Amodelexample} for the details).

The $\mathsf{A}$-model F-bundle $(\omH,\bnabla)/\omB_{X}$ is overmaximal and can be made maximal by removing redundant degree $2$ variables on the base $\omB_{X}$ (implied by the divisor axiom of Gromov-Witten invariants). Write $(\cH,\nabla)/B_{X}$ for the resulting maximal F-bundle. For a rigid point $b \in B_{X}$, the generalized eigenspace decomposition of the $\Eu_{b}\qup(-)$ action on $\cH_{b,0}$ extends to a subbundle decomposition of $\cH|_{u=0}$ locally near $b$. The spectral decomposition theorem from \cite{HYZZ_Decomposition} implies  the base $B_{X}$ also locally splits into a product $\prod B_i$, and the F-bundle $(\cH,\nabla)/B$ splits into an external direct sum of maximal F-bundles $(\cH_i,\nabla_i)/B_i$, extending the above generalized eigenspace decomposition.

\medskip
Let us look at an example of the spectral decomposition in the case of blowups.
Consider the blowup $\pi\colon\tX\to X$ along a smooth closed subvariety $Z \subset X$ of codimension $r \geq 2$.
We have a classical decomposition of cohomology
\[ H^{\bullet}(\tX,\bbQ) \simeq H^{\bullet}(X,\bbQ) \oplus \bigoplus_{i=1}^{r-1} H^{\bullet}(Z,\bbQ)[-2i].\]
This decomposition is in fact the eigenspace decomposition of the Euler vector field action at the limiting point $b_\infty$ of $B$ corresponding to the pullback $\pi^*$ of an ample class on $X$.
The eigenvalues are $0$ and $(r-1)e^{\frac{\pi \sqrt{-1}(2j-1)}{r-1}}$ for $j=1,\dots,r-1$.
The $0$-eigenspace is isomorphic to $H^{\bullet}(X,\bbQ)$, and the others all isomorphic to $H^{\bullet}(Z,\bbQ)$.
As we move away from the limiting point, each eigenvalue above may further split, as shown in Figure~\ref{fig:eigenvalues}.
In Section~\ref{sec:blowup}, we identify the factors resulting from the spectral decomposition with the $\mathsf{A}$-model F-bundle of $X$ and $(r-1)$ copies of the $\mathsf{A}$-model F-bundles of $Z$, using Iritani's blowup formula \cite{Iritani_blowup}.

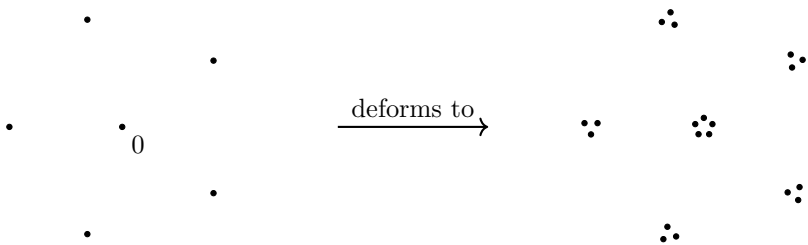
\begin{figure}[!ht] \label{fig:eigenvalues}
\centering
\begin{tikzpicture}
  \begin{scope}[xshift=-110pt]
    \filldraw[black] (0,0) circle (1pt) node[below right] {0};
    \def\radius{1.5}
    \foreach \j in {1,...,5} {
      \pgfmathsetmacro\angle{180*(2*\j - 1)/5} 
      \filldraw[black] ({\radius*cos(\angle)},{\radius*sin(\angle)}) circle (1pt);
    }
  \end{scope}

  \draw[->, thick] (-1,0) -- (1,0) node[midway, above] {deforms to};
  
  \begin{scope}[xshift=110pt]
    \def\clusterradius{0.1}
    \foreach \k in {0,1,2,3,4} {
      \pgfmathsetmacro\clusterangle{90 + 72*\k}
      \filldraw[black] ({\clusterradius*1.2*cos(\clusterangle)},{\clusterradius*1.2*sin(\clusterangle)}) circle (1pt);
    }
    
    \def\radius{1.5}
    \foreach \j in {1,...,5} {
      \pgfmathsetmacro\angle{180*(2*\j - 1)/5} 
      \pgfmathsetmacro\centerx{\radius*cos(\angle)}
      \pgfmathsetmacro\centery{\radius*sin(\angle)}
      
      \foreach \k in {0,1,2} {
        \pgfmathsetmacro\clusterangle{\angle + 90 + 120*\k} 
        \pgfmathsetmacro\pointx{\centerx + \clusterradius*cos(\clusterangle)}
        \pgfmathsetmacro\pointy{\centery + \clusterradius*sin(\clusterangle)}
        \filldraw[black] (\pointx,\pointy) circle (1pt);
      }
    }
  \end{scope}
\end{tikzpicture}
\caption{Eigenvalues of the $\Eu$-action at $b_\infty$}
\end{figure}

\medskip

Next we explain the construction of Hodge atoms.
Let again $X$ be a smooth complex projective variety. Consider the subspace of its even
cohomology spanned by the Hodge classes
\[
H(X)^{\hodge} \coloneqq \bigoplus_i H^{i,i}(X) \cap H^{2i}(X, \mathbb{Q}) .
\]
It gives a closed $\bbk$-analytic subspace $B^\hodge_{X} \subset B_{X}$ of the base of the maximal $\mathsf{A}$-model F-bundle of $X$.
Let $\tB_{X} \to B_{X}$ be the ramified spectral  covering given by the eigenvalues of the $\Eu\qup (-)$-action, let be $U_X \subset B^\hodge$ the locus where the number of eigenvalues of the $\Eu\qup (-)$-action is maximal, and $\tU_X \coloneqq \tB_{X,\op{red}} \times_B U_X$ be the unramified part of the reduced spectral cover.
The \strongemph{set of local Hodge atoms} of $X$ is the set 
$\pi_{0}(\tU_{X})$ of connected components of this unramified spectral cover.
The spectral decomposition theorem (see \cite{HYZZ_Decomposition} and Theorem~\ref{thm:K-decomposition}), blowup theorem (see \cite{Iritani_blowup} and Theorem~\ref{thm:blowup}), and the projective bundle  decomposition theorem (see \cite{Iritani_Koto,HYZZ_Decomposition} and Theorem~\ref{thm:projective_bundle}) imply that 
there are natural correspondences between the local Hodge atoms of $X$ and the local Hodge atoms of a blowup of $X$ or the local Hodge atoms of a projective bundles over $X$. Using these correspondences we define the \strongemph{set of Hodge atoms of smooth complex projective varieties} as the quotient set
\[
\HAtoms \, \coloneqq \, \Bigg( \bigsqcup_{[X]} \, \pi_{0}(\tU_{X})/\op{Aut}(X)\Bigg)\Bigg/\sim, 
\] 
where the union is taken over isomorphism classes of complex smooth projective varieties, and the equivalence relation $\sim$ is generated by the elementary equivalences corresponding to 
disjoint unions, blowups with smooth centers, and projective bundles.

The \strongemph{atomic composition} or \strongemph{chemical formula} of a particular variety $X$ is then the multiset encoded in the finite support function $\CF(X) = \sum_{\alpha\in\pi_0(\tU_{X})/\Aut(X)} \, \mathsf{mult}_{\alpha} \delta_{\balpha} \colon \HAtoms \to \bbZ_{\geq 0}$, where
$\mathsf{mult}_{\alpha}$ is the degree of the cover $\tU_{X,\alpha} \to U_{X}$ corresponding to the connected component $\alpha$, and $\delta_{\balpha}\colon \HAtoms \to \bbZ_{\geq 0}$ is the delta function supported at the atom $\balpha \in \HAtoms$ which is represented by the local atom $\alpha$.

\begin{example*}
If the canonical class $K_X$ of $X$ is nef, then $\CF(X)$ is a singleton, because the $\Eu$-action has a single eigenvalue in this case (see Lemma~\ref{lem:nefK}).
\end{example*}

The above discussion on blowups implies that $\CF(\tX) = \CF(X) + (r-1) \CF(Z)$.
Since birational equivalences between smooth projective varieties are generated by blowups with smooth centers of codimension $\geq 2$ we arrive at the following.

\paragraph*{\bfseries Non-rationality criterion:}
    \emph{Let $X$ be a $d$-dimensional smooth complex project variety (with $d\ge 2$).
    If there is an Hodge atom in the Hodge atomic composition of $X$ that does not appear in the Hodge atomic composition of any variety of dimension less than or equal to $d-2$, then $X$ is not rational.}

\medskip

Next we apply this criterion to show the non-rationality of certain $4$-dimensional varieties.
To that end, we have to understand the Hodge atoms coming from varieties of dimension less than or equal to $2$, i.e.\ from points, curves and surfaces.
Moreover, for each birational class of surfaces it is sufficient to consider the atomic composition of one representative in the class.

For every Hodge atom $\balpha$ represented by a local atom $\alpha \in \pi_{0}(\tU_{X})/\op{Aut}(X)$, we get a well defined F-bundle  $(\cH^{\alpha},\nabla)/U_{X}$,  which is a generalized eigenspace for the action of $\Eu\qup (-)$ on the maximal $\mathsf{A}$-model F-bundle $(\cH,\nabla)/B_{X}$ of $X$. The atoms formalism assures that this generalized eigenspace F-bundle comes equipped with a natural action of the universal Mumford-Tate group $\hodge$ controlling the $\bbZ/2$-weighted, pure, rational Hodge structures. The isomorphism class of the finite dimensional $\hodge$-representation $\cH^{\alpha}|_{b,u=0}$ is well-defined independent of the choice the representative local atom $\alpha$ or of the subsequent choice of a point $b$.
In particular, we consider the following invariants of $\balpha$:
\begin{enumerate}[wide]
\item \strongemph{dimension $\rho_{\alpha}$ of the space of Hodge classes}: the dimension of the $\hodge$-fixed subspace of the fiber $\cH^{\balpha}|_{b,u=0}$.
\item \strongemph{Hodge polynomial} $P_{\balpha} \in \mathbb{Z}[t,t^{-1}]$: the coefficient at $t^k$ is equal to the dimension of the subspace of the fiber $\cH^{\balpha}|_{b,u=0}$ with $(p-q)$-degree equal to $k$.
\end{enumerate}

\medskip

Let us calculate these invariants for atoms coming from points, curves and surfaces.
\begin{enumerate}[wide]
    \item For any Hodge atom $\balpha$ coming from points or curves, we have $\Coeff_{t^2}(P_{\balpha}) = 0$.
    \item For any Hodge atom $\balpha$ coming from a surface $X$ with $h^{2,0}=0$, we have $\Coeff_{t^2}(P_{\balpha}) = 0$.
    \item By the classification of surfaces, we know that a minimal surface with a geometric genus $\neq 0$ is always birational to a smooth $S$ for which  $K_S \ge 0$.
    Hence by the above example, a surface with geometric genus $\neq 0$  will contribute  a single Hodge atom $\balpha$ with $\rho_{\balpha} \ge 3$ (containing the classes in $H^0$, $H^4$ and the first Chern class of an ample line bundle in $H^{2}$).
\end{enumerate}

\medskip

Let $X$ be a very general cubic 4-fold.
Its Hodge diamond splits into the primitive part and an irreducible transcendental part (see \cite{Voisin-cubicTorelli}).

\newcommand{\matrixtemplate}[7]{%
  \vcenter{\xymatrix@=1.8pt@M=1.8pt{
    & & & & & *=0{} \ar@{-}[dddddlllll] \ar@{-}[dddddrrrrr] & & & & & \\
    & & & & & #1 & & & & & \\
    & & & & 0 & & 0 & & & & \\
    & & & 0 & & #2 & & 0 & & & \\
    & & 0 & & 0 & & 0 & & 0 & & \\
    *=0{} & 0 & & #3 & & #4 & & #5 & & 0 & *=0{} \\
    & & 0 & & 0 & & 0 & & 0 & & \\
    & & & 0 & & #6 & & 0 & & & \\
    & & & & 0 & & 0 & & & & \\
    & & & & & #7 & & & & & \\
    & & & & & *=0{} \ar@{-}[uuuuulllll] \ar@{-}[uuuuurrrrr] & & & & &
  }}%
}

\[
  \matrixtemplate{\bluebf{1}}{\bluebf{1}}{\bluebf{1}}{\bluebf{21}}{\bluebf{1}}{\bluebf{1}}{\bluebf{1}}
  \;=\;
  \matrixtemplate{\bluebf{1}}{\bluebf{1}}{0}{\bluebf{1}}{0}{\bluebf{1}}{\bluebf{1}}
  \;\bigoplus\;
  \matrixtemplate{0}{0}{\bluebf{1}}{\bluebf{20}}{\bluebf{1}}{0}{0}
\]

Let $(\cH,\nabla)/B_{X}$ be the maximal $\mathsf{A}$-model F-bundle associated to $X$.
By Givental's calculation \cite{Givental_Homological_geometry_and_mirror_symmetry} of the quantum product at a particular point $b\in B$, the $\Eu$-action has eigenvalues $\{0, 9, 9 e^{2\pi i/3}, 9 e^{4\pi i/3}\}$.
The generalized eigenspace $V$ corresponding to the eigenvalue 0 has dimension 24, containing the whole transcendental part.
We have $\dim V^{\hodge} = 2$, $\dim V^{p-q=-2}=\dim V^{p-q=2}=1$ and $\dim V^{p-q=0}=22$.
The other three eigenspaces are 1-dimensional.
As we move from $b$ to generic points in $U_X$, $V$ may further split, contributing more atoms,  but there will always be one atom $\balpha$ with $\rho_{\balpha} \le 2$ and $\Coeff_{t^2}(P_{\balpha}) = 1$, which \emph{can not} come from varieties of dimensions less than or equal to 2.
We conclude that $X$ can not be rational. 

\begin{theorem*}[see Theorem~\ref{thm:cubic4}]
A very general 4-dimensional cubic hypersurface in $\bbC\bbP^5$ is not rational.
\end{theorem*}

In Section~\ref{ssec:birCY}, we give another application of Hodge atoms.
We obtain a new proof of Batyrev's theorem \cite[Corollary~6.29]{Batyrev-StringyHodge} asserting that birationally equivalent Calabi-Yau manifolds have equal Hodge numbers.
The traditional proofs of this result \cite{Batyrev-StringyHodge,Wang-Kequivalence,Ito} use either motivic integration, or $p$-adic Hodge theory and C\v{e}botarev density. The proof via atoms is based on completely different technology.

\begin{theorem*}[see Theorem~\ref{thm:two.norms}]
Birationally equivalent Calabi-Yau manifolds have equal Hodge numbers.   
\end{theorem*}

By design, the theory of Hodge atoms is motivic in nature. It is a special case of a theory of $G$-atoms governed by some proreductive algebraic group $G$ which controls a particular realization of Andr\'{e} motives.  In the case of Hodge atoms the relevant group is the  Galois group $\hodge$ of the Tannakian category of polarizable pure $\bbZ/2$-weighted $\bbQ$-Hodge structures. 
We develop the general theory of $G$-atoms for a general motivic Galois group $G$ in Section~\ref{sec:Gatoms}.
One variant of particular interest is the theory of \strongemph{motivated atoms}, given by $G=\motM^{\kay}$, the Galois group of the category of a $\bbZ/2$-graded version of André's motives \cite{Andre_Pour_une_theorie_inconditionnelle_des_motifs} equipped with the fiber functor of Betti cohomology.

Motivated atoms are useful for detecting non-rationality over a subfield $\kay\subset \bbC$ which is not algebraically closed.
To illustrate the idea, consider a smooth hypersurface $X$ of multi-degree $(1, 1, 1, 1)$ in 
$\bbP^1 \times \bbP^1 \times \bbP^1 \times \bbP^1$, defined over the algebraic closure $\kaybar$ of $\kay$. It is isomorphic to the blowup of $\bbP^1 \times \bbP^1 \times \bbP^1$ at an elliptic curve $E$, and hence the motivated atomic decomposition of $X$ contains 8 atoms associated to points and one more complicated atom $\alpha_E$ associated to $E$.

Now we consider a model $\fX$ of $X$ defined over the non-closed field $\kay$ such that the Galois group $\Gal(\kaybar/\kay)$ acts by permuting transitively the 4 factors in the product $\bbP^1 \times \bbP^1 \times \bbP^1 \times \bbP^1$.
Let $(\cH,\nabla)/B_{\fX}$ be the $\mathsf{A}$-model F-bundle associated to $\fX$. At a particular point $b\in B_{\fX}$, the $\Eu$-action has 3 eigenvalues with multiplicities 1, 4 and 7 respectively.
The last generalized eigenspace has Hodge polynomial $5 + t + t^{-1}$ and only $2$ algebraic classes defined over $\kay$.
Then this representation of $\motM^{\kay}$ cannot split further into representations coming from atoms associated to $0$ and $1$-dimensional varieties over $\kay$.
Hence we conclude the non-rationality of $\fX/\kay$.

Many similar examples can be worked out in an analogous manner but we will leave these to the future work \cite{CGKK}.
In addition to the Hodge atoms and motivated atoms, we will consider other enhanced 
versions of atoms, e.g.\ atoms which are furnished with the image of the topological $K$-theory in the \nc-Hodge structure, or atoms enhanced with Mukai pairings or Serre automorphisms.
The theory of enhanced atoms will be developed in the forthcoming work \cite{KKPY-Gamma}, where many more applications are worked out.
Here we concentrate on the foundations of the theory of atoms and their most basic applications as obstructions to rationality.

\medskip

This paper is organized as follows.
In Section~\ref{sec:Bmodel}, we motivate the theory of F-bundles and atoms from a special case of the mirror symmetry $\mathsf{B}$-model, in the  elementary context of  equivariant isolated singularities.
The constructions in this section have no direct relation to the $\mathsf{A}$-model invariants or rationality we discuss in the rest of the paper.
In Section~\ref{sec:definition_of_F-bundles}, we introduce and study F-bundles, which is a formal/non-archimedean version of non-commutative Hodge structure.
In particular in Section~\ref{sssec:AmodelF} we introduce the non-archimedean $\mathsf{A}$-model F-bundle of a complex smooth projective variety, the main object underpinning the construction of atoms. Section~\ref{sec:decompose} discusses the spectral decomposition of F-bundles as well as the decomposition of F-bundles associated to blowups and projective bundles.
In Section~\ref{sec:Gatoms}, we develop the general theory of $G$-atoms for any motivic Galois group $G$.
Finally in Section~\ref{sec:obstructions}, we apply the theory of atoms to obtain new obstructions to rationality in birational geometry.
After recasting Givental's classical quantum cohomology computation in the language of F-bundles in Section~\ref{ssec:Cubics},  we use Hodge atoms to show that the very general cubic fourfold is not rational.
In Section~\ref{ssec:birCY} we use Hodge atoms to give a new proof of the equality of Hodge numbers of birational Calabi-Yau manifolds in any dimension.
Finally in Section~\ref{ssec:enhanced} we discuss briefly enhancements of atoms and illustrate how enhanced atoms lead to further non-rationality results.

\subsection*{Acknowledgements}

We would like to thank Denis Auroux, Hiroshi Iritani and Yuri Tschinkel for their interest and for their invaluable help at various stages of this work.
We are grateful to Leonardo Cavenaghi, Sheel Ganatra, Mark Gross, Thorgal Hinault, Y.P.\ Lee, Ernesto Lupercio, James McKernan, Todor Milanov, Yong-Geun Oh, Yukinobu Toda, Chenyang Xu, Song Yu, Chi Zhang and Shaowu Zhang for inspiring discussions around the subject. 

During the preparation of this paper Ludmil Katzarkov was supported by the National Science Fund of Bulgaria, the National Scientific Program ``VIHREN'' (Project no.\ KP-06-DV-7), the Institute of Mathematics and Informatics of the Bulgarian Academy of Sciences, the Simons Foundation (grant SFI-MPS-T-Institutes-00007697), the Ministry of Education and Science of the Republic of Bulgaria (contract DO1-239/10.12.2024), the Basic Research Program of the National Research University Higher School of Economics, the Simons Investigators Award (no.\ 003136), the Simons Collaboration on Homological Mirror Symmetry (award no.\ 003093), and by the NSF FRG grant DMS-2245099. Tony Pantev was supported by the Simons Collaboration on Homological mirror symmetry (award no.\ 347070) and the NSF FRG grants DMS-2244978 and DMS-2245099.
Tony Yue Yu was partially supported by the NSF grants DMS-2302095 and DMS-2245099.

In addition, Ludmil Katzarkov and Maxim Kontsevich are grateful to Jeffrey Fuqua for his continuous financial support of mathematical research throughout the years, and for his support and leadership in establishing the excellent creative environment at IMSA, University of Miami, where much of this work was done.

\subsection*{Notation and conventions} 

\begin{description}[style=multiline, leftmargin=7.5em, labelwidth=7em, align=left, itemsep=0pt]
\item[$k$] a field of characteristic zero, coefficient field of a Weil cohomology theory.
\item[$\bbk \supset k$] an algebraically closed non-archimedean field whose absolute value is trivial on $k$.
\item[$\kay$] a field of characteristic zero, base field of smooth projective varieties, typically chosen to be a subfield of $\bbC$.
\item[$\bbD$] the germ at $0$ of a $\bbk$-analytic affine line with coordinate $u$.
\item[$(\cH, \nabla)/B$] an F-bundle over a $\bbk$-analytic space $B$.
\item[$\bkappa = \nabla_{u^{2}\partial_{u}}|_{\cH_{u=0}}$] the residual endomorphism on $\cH_{u=0}$.
\item[$\fp$] the Frobenius product on $T_{B}$ associated with a maximal F-bundle.
\item[$\Eu$] the Euler vector field on $B$ associated with a maximal F-bundle.
\item[$\cH_{o}$] the restriction $\cH|_{u = 0}$.
\item[$\braket{\gamma_1\cdots \gamma_n}_{g,n,\beta}$] \GW invariants.
We omit $g$ and $n$ from the notation when $g=0$ and $n$ is evident.
\item[$\qup$] the big quantum product on the cohomology $H^\bullet(X)$ of a smooth projective variety $X$.
\item[$\Deg$] the degree operator on $H^{\bullet}(X)$.
\item[$\Gr$] the grading operator on $H^{\bullet}(X)$.
\item[$\NE(X,\bbZ)$] the monoid of numerical equivalence classes of effective curves.
\item[$\NS(X,\bbZ)$] the N\'{e}ron-Severi group of $X$.
\item[$\mathsf{N}^{1}(X,\bbZ)$] the group of numerical equivalence classes of divisors in $X$, also equal to 
 the maximal torsion-free quotient $\NS(X,\bbZ)_{\op{tf}}$ of the N\'{e}ron-Severi group of $X$. 
\item[$\mathsf{N}_{1}(X,\bbZ)$] the group of numerical equivalence classes of curves in $X$.
\item[$\CA{\kay}$] the $\bbQ$-linear Tannakian category of Andr\'{e} motives of smooth projective $\kay$-varieties.
\item[$\CA{\kay}/\mathsf{Tate}$] the $\bbQ$-linear category of \nc-Andr\'{e} motives of smooth projective $\kay$-varieties.
\item[$\GA{\kay}$] the pro-reductive motivic Galois group for Andr\'{e} motives of smooth projective $\kay$-varieties.
\item[$(G,\epsilon_{G})$] a proreductive algebraic group over $k$ with a marked central element $\epsilon_{G}$ of order $2$.
\item[$(\HHgr,\epsilon_{\HHgr})$]  the symmetry group of Hochschild graded Dolbeault cohomology.
\item[$(\hodge,\epsilon_{\hodge})$]  the Mumford-Tate group of $\bbZ/2$-weighted pure Hodge structures. 
\item[$(\motM,\epsilon_{\motM})$] the Galois group of the $\bbZ/2$-weighted version of Andr\'{e}'s motivated motives.
\item[$\tU_{\fX}$] the unramified part of the reduced spectral cover associated with the operator $\bkappa$ attached to the $\mathsf{A}$-model $\bbk$-analytic F-bundle of $\fX$.
\item[$\Atoms_{G}^{\kay}$]  the set of $G$-atoms of smooth projective $\kay$-varieties.
\item[$\CF_{G}(\mathfrak{X})$] the chemical formula or $G$-atomic composition of $\mathfrak{X}$, i.e.\ the multiset of $G$-atoms of $\fX$ together with their multiplicities.
\item[$\Atoms_{G}^{\kay,\mathsf{F}}$] the set of geometric $G$-atomic F-bundles for smooth projective $\kay$-varieties.
\item[$\CF_{G}^{\mathsf{F}}(\mathfrak{X})$]  the multiset of geometric $G$-atomic F-bundles arising from $\fX$ together with their multiplicities.
\item[$\HAtoms$] the set of Hodge atoms of complex smooth projective varieties.
\end{description}


\addtocontents{toc}{\protect\setcounter{tocdepth}{2}}

\section{Atoms of equivariant singularities} \label{sec:Bmodel}

To motivate the general theory we start with a particular $\mathsf{B}$-model example dealing with  simplified  \nc-Hodge theoretic invariants of equivariant isolated singularities of functions. We  regard these simplified invariants as ``baby atoms''. On its own this example is not related to any birational invariants but we include it since it provides valuable conceptual insights into the general theory we will develop later.

\subsection{Isolated singularities of functions} \label{ssec:fx0}

Recall that if $f \colon (X,x_{0}) \to \bbC$ is an analytic germ of a function on a $d$-dimensional manifold $X$, with an isolated critical point $x_{0}$, then the \strongemph{Milnor number} $\mu_{f}$ of $f$ is a basic numerical invariant of the germ, which is given by 
\[
\mu_{f} \, \coloneqq \, \dim_{\bbC} \, \big( \cO_{X,x_{0}}/J(f)\big),
\]
where $J(f)=\op{image}\Big[\cT_{X, x_0} \xrightarrow{(-)\iprod \, df} \cO_{X,x_0} \Big]$ is the Jacobian ideal of $f$.
We use $\cT_X$ to denote the tangent sheaf, and $T_X$ to denote the tangent bundle.
Furthermore, since the singularity of $f$ at $x_{0}$ is isolated, the Milnor number is finite and the function germ $(X,x_{0},f)$ has a \strongemph{universal unfolding} which is a family of function germs
\[
F \colon \big(X\times \univ_{f},(x_{0},u_{0})\big) \longto \bbC
\]
such that $F|_{(X,x_{0})\times\{u_{0}\}} = f$, the family is parametrized by a germ $(\univ_{f},u_{0})$ of a $\mu_{f}$-dimensional mani\-fold, and the natural map 
$T_{\univ_{f},u_{0}} \to \cO_{X,x_{0}}/J(f)$ given by $v \mapsto [(\Lie_v F)|_{X\times \{u_{0}\}}]$ is an isomorphism of vector spaces.

Explicitly, a model for the universal unfolding is constructed as follows. Let
$\mathsf{Milnor}_{f,x_{0}} \coloneqq \cO_{X,x_{0}}/J(f)$ be the  \strongemph{Milnor ring}  of $f$ at $x_{0}$, and let 
$g_{1}, \ldots, g_{\mu_{f}} \, \in \cO_{X,x_{0}}$ be a collection of function germs whose images form a basis of $\mathsf{Milnor}_{f,x_{0}}$ as a complex vector space. Then we get a family of holomorphic functions 
\begin{equation} \label{eq:miniversal}
F \colon \big(X\times \bbC^{\mu_{f}}, (x_{0},0)\big) \longrightarrow \bbC, \quad (x,z) \longmapsto 
f(x) + \sum_{i=1}^{\mu_{f}} \, z_{i}g_{i}(x).
\end{equation}
We can then take $(\univ_{f},u_{0})$ to be the germ of $(\bbC^{\mu_{f}},0)$, and $F$ to be the germ of the function \eqref{eq:miniversal}.  The universal unfolding is a well-defined family of germs and the analytic automorphisms of the singularity $(X,x_{0},f)$ act naturally on the  model universal unfolding $\big(X\times \univ_{f}, (x_{0},u_{0}),F\big)$ defined above. 

Intrinsically one can construct and work with the universal unfolding as follows. Let $X$ be a $d$-dimensional complex manifold, $x_{0} \in X$, and let $f \colon X \to \bbC$ be the germ at $x_{0}$ of a holomorphic function for which $x_{0}$ is an isolated critical point. Let $x_{0} \in V \subset X$ be a Stein neighborhood on which $f$ is defined, and such that $x_{0}$ is the only critical point of $f$ in $V$. Let $x_{0} \in \ball \, \subset \, V$, be a smaller Stein neighborhood, which is an open Euclidean ball and such that the closure $\overline{\ball}$ of $\ball$ in $X$, is contained in $V$. 

Consider the space
\[
\univop_{\ball} \, \coloneqq \, \left\{ \, g \in \cO(\overline{\ball}) \, \left| \ dg_{x} \neq 0 \, \in \, T_{X,x}^{\vee}, \ \text{for all} \ x \in \partial \overline{\ball} \, \right.\right\}.
\]
Note that the algebra $\cO(\overline{\ball})$ of holomorphic functions on the Stein compact $\overline{\ball}$ is naturally a nuclear Fr\'{e}chet $\bbC$-vector space with respect to the topology of compact convergence \cite[Chapter V.6.1]{GrauertRemmert_Stein}, and so the open subspace 
$\univop_{\ball} \, \subset \, \cO(\overline{\ball})$ is naturally an infinite dimensional complex  
Fr\'{e}chet manifold. This manifold represents the moduli functor which to every Stein set $S$ assigns the set of all holomorphic functions on $S\times \overline{\ball}$ whose relative differential over $S$ does not vanish anywhere on $S\times \partial\overline{\ball}$. 

At each point $g \in \univop$ the tangent space $T_{\univop_{\ball},g}$ is naturally identified with $\cO(\overline{\ball})$, and it contains a natural closed subspace of finite codimension, namely
the image
\[
\folunivop_{g}  \, \coloneqq \, \mathsf{image} \Big[\xymatrix@1@C+1.2pc@M+0.2pc{\cT_{X}(\overline{\ball}) \ar[r]^-{\mathsf{Lie}_{(-)} \, g} & \cO(\overline{\ball})}\Big] 
\]
of the ``evaluation on $g$'' map $\xi \to \mathsf{Lie}_{\xi} \, g$ from holomorphic vector fields to holomorphic functions on $\ball$. 
The codimension of $\folunivop_{g}$ in $T_{\univop_{\ball},g}$ is equal to the
degree of the natural map 
\begin{equation} \label{eq:deg}
\xymatrix@R-1.2pc@C+1.2pc@M+0.2pc{
m_{g} \, \colon \, \hspace{-5pc} & \partial\overline{\ball} \ar@{=}[d] \ar[r] & (\bbC^{d}-\{0\})/\bbR_{>0} \ar@{=}[d]\\
& S^{2d-1} & S^{2d-1}
}, \qquad m_{g}(x) \coloneqq \frac{dg_{x}}{||dg_{x}||},
\end{equation}
and so $\op{codim}(\folunivop_{g}\subset T_{\univop_{\ball},g})$ is equal to the 
sum $\sum_{x \in \mathsf{crit}(g)} \, \mu_{g,x}$ of the Milnor numbers of $g$ at all critical points $\mathsf{crit}(g) = \mathsf{zero}(dg)$ of $g$ in $\ball$. 

In particular $\folunivop \subset T_{\univop_{\ball}}$ is a holomorphic foliation of finite codimension on $\univop_{\ball}$. Given a point $g \in \univop_{\ball}$ we will write  $[g]$ for the leaf of this foliation through $g$. The universal unfolding $\univ_{f}$ of the singularity germ $(X,x_{0},f)$ is canonically the germ of the leaf space $\left[\univop_{\ball}/\folunivop\right]$ of this foliation at the point $[f]$. The explicit model we considered above is a holomorphic slice through $f$ to the leaves of $\folunivop$. Note that this approach allows us to define a canonical universal unfolding $\univ_{g}$ for any $g \in \univop_{\ball}$ by simply taking $\univ_{g}$ to be the germ at $[g]$ of the leaf space $\left[\univop_{\ball}/\folunivop\right]$.

For $[g] \in \univ_{f}$, we still have the identification $T_{\univ_{f},[g]} \cong \mathsf{Milnor}_{g} = \cO(\ball)/J(g)$, where
$J(g) = \op{image}\Big[T(\ball) \xrightarrow{(-)\iprod \, dg} \cO(\ball)\Big]$
is the global Jacobian ideal of $g$ on $\ball$. This identification induces an analytic 
commutative associative product  
\[
\fp \colon \ T_{\univ_{f}}\otimes_{\cO_{\univ_{f}}} T_{\univ_{f}} \longrightarrow T_{\univ_{f}},
\]
which is just K.~Saito's Frobenius manifold structure \cite{KyojiSaito-residue,KyojiSaito-primitive,MorihikoSaito-Brieskorn,Dubrovin_Geometry_of_2D,Sabbah-Isomonodromic} on $\univ_{f}$.  
Being open in a Fr\'{e}chet vector space, the infinite dimensional manifold $\univop$ will have a natural holomorphic Euler vector field corresponding to the scaling action of $\bbC$ on $\cO(\overline{\ball})$. The projection of this vector field in $T_{\univop_{\ball}}/\folunivop$ gives a holomorphic vector field $\Eu$ on the universal unfolding $\univ_{f}$.   
      
The operator $\bkappa \coloneqq \, \Eu\, \fp \, (-) \colon \ T_{\univ_{f}} \ \to \ T_{\univ_{f}}$ 
is an analytic vector bundle endomorphism, and for any point $[g] \in \univ_{f}$ we have that 
$\bkappa_{[g]} \colon \, T_{\univ_{f},[g]} \ \to \ T_{\univ_{f},[g]}$ preserves the natural  decomposition $T_{\univ_{f},[g]} = \mathsf{Milnor}_{g} =  \oplus_{x \in \mathsf{crit}(g)} \, \mathsf{Milnor}_{g,x}$, where $\mathsf{Milnor}_{g,x} = \cO_{X,x}/J(g)$. Furthermore, for each $x \in \mathsf{crit}(g)$, the subspace 
$\mathsf{Milnor}_{g,x} \subset \mathsf{Milnor}_{g}$ is the generalized eigenspace of $\bkappa_{[g]}$ of eigenvalue $g(x)$.

\medskip

Saito's theory of Frobenius structures of singularities \cite{KyojiSaito-residue}, \cite[Chapter~VI]{Sabbah-Isomonodromic} implies that we have the following decomposition

\begin{theorem} \label{theo:decompsing} 
Let $(X,x_{0},f)$ be the germ of an isolated singularity of a holomorphic function on a $d$-dimensional manifold. Let $g \in \univop_{\ball}$ be a point close to $f \in \univop_{\ball}$. Then the germ $\univ_{f}([g])$ of the universal unfolding of $f$ at the point $[g] \in \univ_{f}$ decomposes canonically into a product 
\[
\univ_{f}([g]) \ \cong \ \prod_{x \in \mathsf{crit}(g)} \, \univ_{{\tensor[^x]{g}{}}},
\]
where $(X,x,{\tensor[^x]{g}{}})$ is the germ of $g$ at $x$. Moreover this decomposition is compatible with Saito's Frobenius structures and with Euler vector fields.
\end{theorem}

Let $\mathsf{Aut}(f)$ denote the group of biholomorphic automorphisms of the germ $(X,x_{0},f)$. It is a Fr\'{e}chet analytic Lie group which naturally acts on the germ $\univop_{f}$ of the universal unfolding.  Its connected component $\mathsf{Aut}^{0}(f)$ acts 
trivially and $\mathsf{Aut}(f)$ acts trough its quotient $\mathsf{Aut}(f) \to \pi_{0}\left( 
\mathsf{Aut}(f)\right)$. By the approximation theorem, the group $\pi_{0}\left( 
\mathsf{Aut}(f)\right)$ will act with finite stabilizers and so 
the moduli $\mycal{M}_{f} \coloneqq \left[\univop_{f}/\pi_{0}\left( 
\mathsf{Aut}(f)\right)\right]$ of deformations of $f$ will be a germ of a complex analytic orbifold.

The intrinsic construction of $\univ_{f}$ works equally well to produce a ``global'' universal unfolding parametrizing functions  on $d$-dimensional manifolds with isolated singularities and a given Milnor number $\mu$. More precisely, we can consider the subset  $\univop_{\ball,\mu} \subset \univop_{\ball}$  comprised of all $g \in \univop_{\ball}$ such that $\deg \, m_{g} = \mu$. Then $\univop_{\ball,\mu}$ is a union of connected components of 
$\univop_{\ball}$. By the discussion above the leaf space $\left[\univop_{\ball,\mu}/\folunivop\right]$ is a smooth complex analytic orbifold $\mycal{U}_{d,\mu}$ with finite stabilizers (possibly non-separated), which we will think of as the  universal unfolding of $d$-dimensional germs functions with isolated singularities with Milnor number $\mu$. Again the tangent bundle $T_{\mycal{U}_{d,\mu}}$ is equipped with a Frobenius product, we have an Euler vector field $\Eu$ on $\mycal{U}_{d,\mu}$, and the germ of $\mycal{U}_{d,\mu}$ at any point $[g]$ decomposes into a product of pieces compatible with Frobenius products and Euler fields, so that each piece corresponds to the generalized eigenspace of the 
operator $\Eu_{[g]}\, \fp \, (-)$.

\medskip

The moduli $\mycal{M}_{d,\mu}$ of $d$-dimensional  germs of functions  with a single isolated singularity of Milnor number $\mu$ is then a closed, possibly singular, analytic substack in 
$\mycal{U}_{d,\mu}$. Note that $\mycal{M}_{d,\mu}$ is in fact an algebraic Deligne-Mumford stack defined over $\overline{\bbQ}$. The formal completion $\mycal{U}_{d,\mu}^{\wedge}$  of 
$\mycal{U}_{d,\mu}$ along $\mycal{U}_{d,\mu}$ is a formal stack over $\overline{\bbQ}$ (in general non-algebraizable) which is a smooth thickening of $\mycal{M}_{d,\mu}$.

\begin{remark} \label{rem:witt.move} 
By the theory of Frobenius manifolds \cite{Manin_Frobenius_manifolds} we know that the Frobenius powers $\Eu^{k} \coloneqq \Eu^{\fp k}$ of the Euler vector field satisfy the relations of the Witt Lie algebra of polynomial vector fields on the affine line, i.e.\ 
\begin{equation} \label{eq:witt.relations}
\left[\, \Eu^{k}, \, \Eu^{l}\,\right] \ = \ (l-k) \Eu^{k+l-1}, \quad \text{for all} \ k, l \geq 0,
\end{equation}
where we view $\Eu^{\fp k}$ as a holomorphic vector field on 
$\mycal{U}_{d,\mu}$, and the bracket in \eqref{eq:witt.relations} is the Lie bracket on vector fields on $\mycal{U}_{d,\mu}$. 

In an open dense domain of $\mycal{U}_{d,\mu}$ the function has $\mu$ isolated Morse singularities, and the local coordinates on $\mycal{U}_{d,\mu}$ are given by the critical values $x_1,\dots,x_\mu$. The Euler field and its powers are given by
$$ \Eu=\sum_i x_i\partial_{x_i},\quad \Eu^k=\sum_i x_i^k\partial_{x_i}\,.$$

Suppose $[f] \in \mycal{U}_{d,\mu}$ is a germ of a function with isolated singularities with   total Milnor number $\mu$. One can show that there is a well defined germ of a complex analytic submanifold
$\mycal{W} \subset \mycal{U}_{d,\mu}$ at point $[f]$ such that $\mathsf{span}_{\bbC}\left(\left\{\Eu^{k}_{[g]}\right\}_{k \geq 0}\right) \, = \, T_{\mycal{W},[g]}$ for all $[g]\in \mycal{W}$ close to $[f]$. 
Moreover, the  critical points of $[g]$ are in canonical bijection with the critical points of 
$[f]$ and that bijection preserves the multiplicities. 
\end{remark}

The picture becomes even more interesting when we consider equivariant isolated singularities of functions. 

\subsection{Equivariant isolated singularities of functions} \label{sec:equiv.singf}

For ease of reference let us introduce the following 

\begin{definition}\label{def:G-equivariant isolated singularity} Let $G$ be a finite group. 
A $d$-dimensional complex analytic \strongemph{$G$-equivariant function germ with isolated singularities}  is the data $\left(G \, \actleft \, X,f\right)$, where
\begin{enumerate}[wide]
    \item $Y$ is a possibly disconnected complex $d$-dimensional manifold .
    \item $Y$ is equipped with a holomorphic $G$-action $G\, \actleft \, Y$.
    \item $\bw \colon  Y \to \bbC$ is a $G$-invariant holomorphic function, with isolated critical points, such that the $G$-action on $\mathsf{crit}(\bw)$ is transitive.
    \item $(X,f)$ is the $G$-equivariant germ of $(Y,\bw)$ at the subset $\mathsf{crit}(\bw)$.
\end{enumerate}
\end{definition}

Note that our condition requiring that $\bw$ is invariant and that the $G$-action on $\mathsf{crit}(\bw)$ is transitive implies that a $G$-equivariant germ always has a single critical value.

\begin{example}
Consider the action of $\bbZ/N$ on $\bbC$  given  by  $k \mapsto (z \mapsto \zeta^{k}\cdot z)$ for some primitive $N$-th root of unity $\zeta$. The function
$\mathbb{C} \to \mathbb{C}$, $z \mapsto z^N$
is $\bbZ/N$-invariant, and its derivative $Nz^{N-1}$ vanishes only at $z=0$. 
Therefore if we take 
$X$ to be the germ of  $\bbC$ at $0 \in \bbC$ and $f$ to be the  germ of $z \mapsto z^N$ at $0$, then the data $\left({\bbZ/N}\, \actleft \, X,f \right)$  is a $\bbZ/N$-equivariant function germ with isolated singularities.
\end{example}

\begin{remark} \label{rem:induce}
\begin{enumerate}[wide]
\item If $G = \{\mathbf{e}\}$ is the trivial group, then a $d$-dimensional $G$-equivariant function germ with isolated singularities is just  a germ $(X,f)$ of a holomorphic function on a $d$-dimensional manifold at an isolated critical point of the function.

\item Suppose $\left(G\, \actleft\, X,f \right)$  is a $G$-equivariant function germ with isolated singularities and let $p \in X$ be a critical point of $f$. Let $H \subset G$ be the stabilizer of $p \in X$, let $X(p)$ be the connected component of $X$ containing $p$, and let $\tensor[^p]{f}{} = f|_{X(p)}$. Equivalently $p \in \mathsf{crit}(\bw) \subset Y$,  $X(p)$ is the germ of $Y$ at $p$, and $\tensor[^p]{f}{}$ is the germ of $\bw$ at $p \in Y$. In this situation $H$ acts on $X(p)$ by analytic isomorphisms and 
$\left(H\, \actleft\, X(p),\tensor[^p]{f}{} \right)$ is an $H$-equivariant 
function germ with a single isolated singularity.

\item Conversely, suppose $\left(H\, \actleft\, X,f \right)$ is a $d$-dimensional $H$-equivariant 
function germ with a single isolated singularity. Suppose $H \subset G$ of $H$ is embedded as a subgroup in some finite group $G$. Then we can construct an induced $d$-dimensional $G$-equivariant germ   $\left(G\, \actleft\, \mathsf{ind}_{H}^{G}(X),\mathsf{ind}_{H}^{G}(f)\right)$. Here 
\begin{enumerate}[wide]
\item  $\mathsf{ind}_{H}^{G}(X) \coloneqq G\times_{H} X = (G\times X)/H$, where $h \in H$ acts on $(a,x) \in G\times X$ by $h\cdot (a,x) = (a h^{-1},h(x))$.
\item $\mathsf{ind}_{H}^{G}(X) = G\times_{H} X$ is equipped with the left translation $G$-action, i.e.\ $a \in G$ acts on $[(b,x)]$ by $a\cdot[(b,x)] \coloneqq [(ab,x)]$. 
\item $\mathsf{ind}_{H}^{G}(f) \colon \mathsf{ind}_{H}^{G}(X) \ \longrightarrow \ \bbC$ is the holomorphic function given by  $\mathsf{ind}_{H}^{G}(f)([(a,x)]) = f(x)$. This function is well defined since $f$ is $H$-invariant.
\end{enumerate}

The group $H$ acts freely on $G\times X$ and so $\mathsf{ind}_{H}^{G}(X)$ is a complex manifold germ. It maps holomorphically onto the zero dimensional complex manifold $G/H$ and in fact   $\mathsf{ind}_{H}^{G}(X)$ is (non-canonically) identified with a disjoint union of $\# \, (G/H)$-copies of $X$.
\end{enumerate}
\end{remark}

\begin{definition} Let $(G\, \actleft\, X,f)$ be a $G$-equivariant function germ with isolated singularities.
We define the \strongemph{$G$-equivariant Milnor number} $\mu_f^{G}$ of the germ  to be the dimension of the $G$-invariants in the Milnor ring $\mathsf{Milnor}_{f}(X)$ of $f$.
\end{definition}

By definition $\mathsf{Milnor}_{f}(X)$ is the ring of global analytic function on the closed analytic subgerm in 
$X$ defined by the sheaf of Jacobian ideals $\mathcal{J}(f) \coloneqq \op{image}\Big[\cT_{X} \xrightarrow{(-)\iprod \, df} \cO_X \Big]$
of $f$. Since by assumption $f$ has isolated critical points, we
have that 
\[
\mathsf{Milnor}_{f}(X) = \bigoplus_{x \in \mathsf{crit}(f)} \, \mathsf{Milnor}_{f,x},
\]
and so
\[
\mu_f^G = \dim_{\bbC} \left(\mathsf{Milnor}_{f}(X)\right)^{G} \, = \, 
\dim_\mathbb{C} \left(\bigoplus_{x \in \mathsf{crit}(f)} \, \mathsf{Milnor}_{f,x}\right)^G.
\]
Equivalently, $\mu_{f}^{G}$ is the dimension of the $G$-invariants $T_{\univ_{f},u_{0}}^{G}$ in the 
tangent space of the base of the universal unfolding at the base point $u_{0} \in \univ_{f}$.
Note that by definition we have $0 \, \leq \, \mu_{f}^{G} \, \leq \mu_{f}$.

\begin{example} \ Let as before $((\bbZ/N)\,  \actleft \, X,f)$ denote the 
germ of $(\bbC,0,z^{N})$ viewed as $\bbZ/N$-equivariant germ for the action given by a primitive $N$-th root of unity.   In this case the Milnor ring is $\bbC\{z\}/(z^{N-1}) \cong \bbC^{N-1}$ and
$(\mathbb{C}\{z\}/(z))^{\bbZ/N} \cong \bbC$, 
and so $\mu_{f} = N-1$ and $\mu_{f}^{\bbZ/N} =  1$.
\end{example}
 
In the $G$-equivariant context, if the finite group $G$ acts on the germ $X$ preserving $f$, then, as explained in the previous section, the group $G$ will act on the universal unfolding $(\univ_{f},u_{0})$. In particular, the linear action of $G$ on the Milnor ring  $\mathsf{Milnor}_{f}(X)$ is compatible with the linear action of $G$  on the tangent space $T_{\univ_{f},u_{0}}$. 

With this picture in mind we now have the following

\begin{definition}
We will say that a  $G$-equivariant germ  $(G \, \actleft \, X,f)$ of a function with isolated singularities is  \strongemph{strongly atomic} if, for a general point $g \in \univ_f^{G}$, the $G$-invariant function $g \colon X \to \bbC$ has a unique critical point, and we will say that $(G\, \actleft \, X,f)$ is  \strongemph{atomic} if, for a  general point $g \in \univ_f^{G}$, the function $g$  has a single $G$-orbit of critical points.
\end{definition}

Note that strongly atomic equivariant germs of functions are automatically atomic.
In the other direction, if a  $G$-equivariant germ  $(G\, \actleft\, X,f)$ is atomic, then the single singular fiber of $f$ will be a $G$-stable hypersurface with finitely many isolated singularities which are transitively permuted by the action of $G$. In other words, $f$ will have finitely many isolated critical points which form a single $G$-orbit with stabilizer $H \subset G$ (defined up to conjugation), and such that 
the germ of $(X,f)$ at each critical point $s \in \mathsf{crit}(f) \subset X$ is an $H$-equivariant germ of a function with isolated singularities. In other words an atomic $G$-equivariant germ has a generic $G$-equivariant unfolding which splits into a $G$-orbit of $H$-equivariant strongly atomic singularity germs. This constrained behavior of the $G$-equivariant unfolding of atomic singularities reflects the minimality of the deformation and is tightly related to the structure of the universal deformation.

Thus, unfolding spaces — particularly  the base of the universal deformation — provide the natural geometric framework for understanding the infinitesimal deformations of singularities, while the \(G\)-action further refines this picture by isolating those deformation directions that are compatible with the symmetry. The notion of atomicity for a \(G\)-equivariant isolated hypersurface singularity relies on the universal deformation structure since it ensures that, in a generic \(G\)-symmetric unfolding, the critical locus remains as simple as possible.

\medskip
The first general observation we have is the following.

\begin{proposition}
\begin{enumerate}[wide,label=(\roman*)]
    \item The moduli problem for atomic $G$-equivariant isolated hypersurface singularities of fixed dimension $d$, fixed Milnor number $\mu$, and fixed equivariant Milnor number $\mu^{G}$, is representable by an algebraic Deligne--Mumford stack of finite type 
    $\mycal{M}_{d,\mu,\mu^{G}}$, defined over $\overline{\bbQ}$.
    \item The moduli stack  $\mycal{M}_{d,\mu,\mu^{G}}$ can be thickened canonically $\mycal{M}_{d,\mu,\mu^{G}}\subset \mycal{U}_{d,\mu,\mu^{G}}$ to an analytic smooth orbifold  $\mycal{U}_{d,\mu,\mu^{G}}$ of dimension $\mu$, equipped with a $G$-action, and such that 
    \[
    \mycal{M}_{d,\mu,\mu^{G}} \ = \ \left(\mycal{U}_{d,\mu,\mu^{G}}\right)^{G}.
    \]
    \item The formal completion 
    \[
    \mycal{U}_{d,\mu,\mu^{G}}^{\wedge}
    \]
    of the thickened moduli stack $\mycal{U}_{d,\mu,\mu^{G}}$
    along $\mycal{M}_{d,\mu,\mu^{G}}$ is a smooth generally non-algebraizable formal Deligne--Mumford stack of finite type over $\overline{\bbQ}$.
\end{enumerate}
\end{proposition}
\begin{proof}
(i) \ The moduli stack \( \mycal{M}_{d,\mu,\mu^G} \) is constructed as the quotient of the \( G \)-invariant universal  unfolding base space by the automorphism group of the germ. By Definition \ref{def:G-equivariant isolated singularity}, the automorphisms preserving the \( G \)-action form a proper groupoid with finite stabilizers, ensuring the stack is Deligne--Mumford. Finite type over $\overline{\bbQ}$ follows from the fact that the ring of \( G \)-invariant jets is Noetherian.

(ii) \ The thickened moduli $\mycal{U}_{d,\mu,\mu^G}$ is the quotient of the germ of the universal unfolding base space along its $G$-invariant part. The \( G \)-action on the universal unfolding induces the structure of a \( G \)-equivariant analytic germ. 

(iii) \ Smoothness follows from the unobstructedness of \( G \)-equivariant deformations. Non-algebraiza\-bi\-lity arises as the thickening captures formal but non-convergent deformations. The stack structure persists via formal GAGA for Artin stacks.
\end{proof}

Next we look at the special structures one has on these moduli.

\begin{proposition}
\begin{enumerate}[label=(\roman*),wide]
    \item The moduli space $\mycal{U}_{d,\mu,\mu^{G}}$ carries a natural structure of a smooth  F-stack in the sense of Hertling-Manin. In particular 
    its tangent bundle $T$ is equipped with a commutative, associative 
    product $\fp \, \colon \, T\times T  \ \to \ T$ and an Euler vector field 
    $\Eu \in \, \Gamma(\mycal{U}_{d,\mu,\mu^{G}},T)$.
    \item The F-stack structure on $\mycal{U}_{d,\mu,\mu^{G}}$
is $G$-equivariant and hence the tangent bundle of the moduli stack 
\[
\mycal{M}_{d,\mu,\mu^{G}} \ = \ \left(\mycal{U}_{d,\mu,\mu^{G}}\right)^{G}
\]
inherits a  well-defined multiplication and an Euler field $\mathbf{eu} \coloneqq \Eu|_{\mycal{M}_{d,\mu,\mu^{G}}}$, such that the reduced spectrum of the operator $\mathbf{eu}\fp\, (-) \ \colon T_{\mycal{M}_{d,\mu,\mu^{G}}} \ \to \  T_{\mycal{M}_{d,\mu,\mu^{G}}}$ consists of a single point.
\end{enumerate}
\end{proposition}
\begin{proof}
\ (i) \ The tangent space at a point $\big(G\, \actleft\, X,f\big)$ of \( \mycal{U}_{d,\mu,\mu^G}\) is the  Milnor algebra $\cO(X)/J(f)$ of $f$ and the value of $\Eu$ at this point is given by 
$f \in \cO(X)/J(f)$.  Saito's Frobenius structure on the  base of the universal unfolding provides the multiplication, and everything is compatible with the \( G \)-action. The Euler field $\Eu$ scales deformations homogeneously and is preserved under \( G \).

(ii) \ On the \( G \)-fixed locus, the multiplication reduces to the tensor product of invariant subspaces. The restriction \( \mathbf{eu}\fp (-) \) of the multiplication by Euler vector field to the subbundle of $G$-invariants acts with a single eigenvalue as deformations are graded by \( G \)-invariant monomials, collapsing the spectrum.
\end{proof}

We will write $\mycal{{}^{strong}M}_{d,\mu,\mu^{G}}^{G}$ for the union of all components of $\mycal{M}_{d,\mu,\mu^{G}}
$ that parametrize strongly atomic $G$-equivariant germs of functions. 

With these structures in place we can now define \strongemph{the set of atoms of $G$-equivariant isolated hypersurface singularities} by setting  
\[
\mathsf{Atoms}_{G}^{\op{sing}} = \pi_{0} \Bigl(\bigsqcup_{H \subset G}\bigsqcup_{d,\mu,\mu^{H}} \mycal{{}^{strong}M}_{d,\mu,\mu^{H}}^{H}\Bigr)\Big/G.
\]
where $G$ acts on the set of its subgroups by conjugation.

The set  $\mathsf{Atoms}_{G}^{\op{sing}}$ is countable and to every $G$-equivariant isolated hypersurface  singularity $(G\, \actleft\, X,f)$ , one can associate its atomic composition which consists of all atoms represented by atomic singularities appearing in the $G$-invariant unfolding of $f$. Thus we get a chemical formula of $(G\, \actleft\, X,f)$ which is a finite combination of atoms with non-negative integer coefficients. 

\medskip

Given two $G$-equivariant germs  $(G\, \actleft\, X,f)$ and $(G\, \actleft\,Y,y_{0},g)$ of isolated hypersurface singularities, we will denote their external sum by
$(X\times Y, (x_{0},y_{0}), f\boxplus g)$, where $(f\boxplus g)(x,y) \coloneqq f(x) + g(y)$ for all $(x,y) \in X\times Y$.  The external sum is again a germ of an isolated hypersurface singularity, which is $G$-equivariant  for the diagonal action of $G$ on $X\times Y$.
The Thom-Sebastiani formula for the external sum of potentials 
implies that we have an isomorphism of complex vector spaces
\[
\mathsf{Milnor}_{f \boxplus g} = \mathsf{Milnor}_{f} \otimes_{\bbC} \mathsf{Milnor}_{g},
\]
which is compatible with $G$-actions. 

\medskip

Splitting the external sum of $G$-equivariant singularity germs into its atomic composition gives a well defined map 
\begin{equation} \label{eq:split}
 \mathsf{Atoms}_{G}^{\op{sing}} \, \times \,  \mathsf{Atoms}_{G}^{\op{sing}} \ \longrightarrow \  \op{Maps}_{\op{fsupp}}\left( \mathsf{Atoms}_{G}^{\op{sing}}, \, \bbZ_{\geq 0} 
\right)
\end{equation}
where $\op{Maps}_{\op{fsupp}}\left( \mathsf{Atoms}_{G}^{\op{sing}}, \, \bbZ_{\geq 0}\right)$
denotes all maps from $\mathsf{Atoms}_{G}^{\op{sing}}$ to $\bbZ_{\geq 0}$ with finite support.

Explicitly the map \eqref{eq:split} is defined as follows. Suppose $H_{1}, H_{2} \subset G$ are two subgroups and let $\balpha_{1} \in \mycal{{}^{strong}M}_{d_{1},\mu_{1},\mu_{1}^{H_{1}}}$ and $\balpha_{2} \in \mycal{{}^{strong}M}_{d_{2},\mu_{2},\mu_{2}^{H_{2}}}$ be two $G$-atoms of singularities, that are represented by equivariant germs 
$(H_{1}\, \actleft\, X_{1},f_{1})$ and  $(H_{2}\, \actleft\, X_{2},f_{2})$ respectively. Then the map \eqref{eq:split} sends the pair $(\balpha_{1},\balpha_{2})$ to the formal sum of atoms with multiplicities appearing in the atomic composition of the Thom-Sebastiani sum
\[
\left( G\, \actleft\, \left(\mathsf{ind}_{H_{1}}^{G}(X_{1})\times \mathsf{ind}_{H_{2}}^{G}(X_{2})\right), \mathsf{ind}_{H_{1}}^{G}(f_{1})\boxplus \mathsf{ind}_{H_{1}}^{G}(f_{1})\right).
\]

\medskip

Extending the map \eqref{eq:split} by multiplicativity then induces a  product structure
on \linebreak 
$\op{Maps}_{\op{fsupp}}\left( \mathsf{Atoms}_{G}^{\op{sing}}, \, \bbZ_{\geq 0} 
\right)$  whose structure constants are integers. This equips the additive semigroup 
$\op{Maps}_{\op{fsupp}}\left( \mathsf{Atoms}_{G}^{\op{sing}}, \, \bbZ_{\geq 0} 
\right)$ with the structure of a unital commutative semiring and the abelian group 
$\op{Maps}_{\op{fsupp}}\left( \mathsf{Atoms}_{G}^{\op{sing}}, \, \bbZ 
\right)$ with the structure of a unital commutative ring.

\begin{example} \label{ex:TS} Fix a positive integer $N > 1$.
Let $X_{1} = \bbC$ with coordinate $x_{1}$, $X_{2} = \bbC$ with coordinate $x_{2}$, and
let $f_{1}(x_{1}) = x_{1}^{N}$ and $f_{2}(x_{2}) = x_{2}^{N}$. The group $G$ of $N$-th roots of unity acts on $X_{1}$ and $X_{2}$ by 
\[
\zeta \mapsto (x_{1}\to \zeta x_{1}) \quad \text{and}  \quad k \mapsto (x_{2}\to \zeta^{-1}x_{2}).
\]
The equivariant germs $(G\, \actleft\, X_{1},f_{1})$ and $(G\, \actleft\, X_{2},f_{2})$ are atomic and their external sum is 
$(G\, \actleft\, X_{1}\times X_{2},x_{1}^{N} + x_{2}^{N})$.
Consider the following $G$-equivariant (nonuniversal) unfolding of $(G\, \actleft\, X_{1}\times X_{2},x_{1}^{N} + x_{2}^{N})$
\[
F \, \colon \, X_{1}\times X_{2} \times \bbC^2 \ \longrightarrow \ \bbC, \qquad F(x,z) = x_{1}^{N} + x_{2}^{N} - z_{1}x_{1}x_{2} + z_{2}.
\]
Fix a general point $z\in \bbC^2$, s.t. $z_{1} \neq 0$. With this fixed $z$ we can view $F(x,z)$ as a function of $x$ whose Jacobian ideal is generated by 
\[
\partial_{x_{1}}F(x,z) = x_{1}^{N-1} - (z_{1}/N) x_{2} \quad  \text{and} 
\quad 
\partial_{x_{2}}F(x,z) = x_{2}^{N-1} - (z_{1}/N) x_{1}.
\]
Thus 
$F(x,z)$ has $(N-1)^2$ Morse critical points which map to $N-1$ critical values which are equal to $z_{2}$, $z_{2} - (N-2)c_{1}$, \ldots, $z_{2} - (N-2)c_{N-2}$ where $c_{1}, \ldots, c_{N-2}$ label the distinct $(N-2)$nd roots of $(z_{1}/N)^{N}$. 
Over the critical value $z_{2}$ there is a single critical point $(x_{1},x_{2}) = (0,0)$, and over each of the other critical values there is a set of $N$ critical points on which $G$ acts simply transitively. 
More precisely, over the critical value $z_{2} - (N-2)c_{i}$ we have $N$ Morse critical points with coordinates 
$(x_{1},x_{2}) = (\beta,(z_{1}/N)\beta^{-1})$, where $\beta$ runs over the $N$th roots of $c_{i}$.
Thus in the atomic composition of $(G\, \actleft\, X_{1}\times X_{2},x_{1}^{N} + x_{2}^{N})$ we have two atoms. One is isomorphic to the $G$-equivariant $2$-dimensional Morse singularity $(\bbC^2,x_{1}x_{2})$ equipped with the $G$-action $\zeta \colon (x_{1},x_{2}) \mapsto (\zeta x_{1},\zeta^{-1} x_{2})$, and the other  is isomorphic to the $2$-dimensional Morse singularity $(\bbC^2,x_{1}x_{2})$ viewed as an $H$-equivariant germ of function with isolated singularities, where $H = \{1\} \subset G$ is the trivial subgroup. In the chemical formula of $(G\, \actleft\, X_{1}\times X_{2},x_{1}^{N} + x_{2}^{N})$, the first atom enters with multiplicity $1$, while the second atom enters with multiplicity $N-2$.
\end{example}

\begin{remark}
\begin{enumerate}[wide]
\item The structure theory of atoms of equivariant singularities  can be upgraded to the setting of F-bundles discussed in the next section.
\item In the case of the trivial group $G$, we recover the discussion from the previous section: we have a unique atom corresponding to the Morse quadratic singularity $x_{1}^2 + \cdots + x_{d}^2 = 0$ since Morsification shows that each isolated singularity $(X,x_{0},f)$ unfolds to a disjoint union of $\mu_{f}$ quadratic Morse singularities.
\item One can stabilize the above definition of atoms of $G$ equivariant isolated singularities by adding Morse singularities with trivial $G$-action.
\item The analogue of Remark~\ref{rem:witt.move} also holds for $G$-equivariant germs of functions with isolated singularities. More precisely, if $[f] \in \mycal{M}_{d,\mu,\mu^{G}} = (\mycal{U}_{d,\mu,\mu^{G}})^{G}$ and we know that the Frobenius powers of $\Eu$ span $T_{\mycal{M}_{d,\mu,\mu^{G}},[f]}$, then the atomic composition of $[f]$ can not split further for germs $[g]$ in a neighborhood of $[f]$ in $\mycal{M}_{d,\mu,\mu^{G}}$.
\end{enumerate}
\end{remark}

\section{Generalities on \texorpdfstring{$F$}{F}-bundles} \label{sec:Fbundles}

In this section, we introduce the notion of an F-bundle, which captures variations of  linearizations of non-commutative geometry, given by the formal part of non-commutative Hodge 
structures (see \cite{Katzarkov_Hodge_theoretic_aspects}).

\subsection{\texorpdfstring{$F$}{F}-bundles} \label{sec:definition_of_F-bundles}

Our main objects of interest will be defined in the non-archimedean analytic setting.  

\subsubsection{Definition of \texorpdfstring{$F$}{F}-bundles}
\label{sssec-rigidF}

We fix a non-archimedean normed field $\bbk$ of characteristic $0$. Let $\bbD$ denote the germ at $0$ in a $\bbk$-analytic unit disk with coordinate $u$. 

\begin{definition} \label{def:F-bundle} 
A \strongemph{non-archimedean $\bbk$-analytic F-bundle}, or simply an  \strongemph{F-bundle} is a triple  $(\cH,\nabla)/B$, in which  
\begin{enumerate}[wide]
\item $B$ is a smooth $\bbk$-analytic super variety or a germ of smooth $\bbk$-analytic super variety along an even closed smooth $\bbk$-analytic subvariety.
\item $\cH$ is a $\bbk$-analytic super vector bundle  over 
  $B \times \bbD$.
\item $\nabla$ is a meromorphic flat connection on $\cH$ with poles at most at $B\times \{0\}$, such that for any locally defined vector field $\xi$ on $B$ the operators 
$\nabla_{u^2\partial_u}$ and $\nabla_{u\xi}$ 
 are regular, i.e.\ have no poles.
\end{enumerate}
\end{definition}

\begin{remark} \label{rem:F-bundle}
\begin{enumerate}[wide]
\item[(i)] The letter $F$ in the name F-bundle comes from \emph{Frobenius}, since F-bundles are variants of the Frobenius manifolds introduced by Dubrovin \cite{Dubrovin_Geometry_of_2D} for the study of WDVV equations. An F-bundle is also a non-archimedean version of the notion of TE-structures introduced by Hertling \cite{hertling_tt*_geometry}. The complex analytic TE-structures are not suitable for our $\mathsf{A}$-model considerations without the convergence assumptions on the Gromov-Witten potential. Furthermore, the spectral decomposition theorem (Theorem~\ref{thm:K-decomposition}) does not hold in the complex analytic setting due to the Stokes phenomenon.
By working in the non-archimedean setting, we avoid these issues, but retain most of the intuition coming  from the complex picture.

\item[(ii)] From the perspective of non-commutative geometry, an F-bundle is a non-archimedean analytic version of the de Rham part of a variation of non-commutative Hodge structures (see \cite[\S2.1.5,\S2.3.1]{Katzarkov_Hodge_theoretic_aspects}).
Examples of F-bundles arise naturally from deformations of
$A_\infty$ categories, from Gromov-Witten invariants, and from deformations of singularities.

\item[(iii)] In examples one sometimes  encounters  F-bundles which have additional tame poles along divisors in $B$. So, in practice, it is sometimes necessary to work with \strongemph{logarithmic F-bundles}, i.e.\  with triples  $(\cH,\nabla)/B$ in which $B$ and $\cH$ are as before but $\nabla$ is a meromorphic flat connection on $\cH$ with poles at  most at $(B\times \{0\})\cup (\mathfrak{D}\times \bbD)$, for some even strict normal crossings divisor $\mathfrak{D} \subset B$. The logarithmic  F-bundle condition then is that for all locally defined vector fields $\xi$ on $B$ which are tangent to $\mathfrak{D}$ the operators  $\nabla_{u^2\partial_u}$ and $\nabla_{u\xi}$ should have no poles. Of course, if $B$ is a non-archimedean super manifold (rather than a germ) and $(\cH,\nabla)/B$ is a logarithmic F-bundle with log poles along $\mathfrak{D}$ we can eliminate the log singularity by simply puncturing $B$ along $\mathfrak{D}$, i.e.\  by restricting  $(\cH, \nabla)$ to $(B - \mathfrak{D})\times \bbD$. This is another advantage of the non-archimedean analytic setting.
\end{enumerate}
\end{remark}

\subsubsection{Variants: formal and complex analytic \texorpdfstring{$F$}{F}-bundles} \label{sssec:formal.and.Canalytic}

While historically prevalent in the mirror symmetry and Gromov-Witten literature, the complex analytic or formal analogues of an F-bundle will not be useful for constructing birational invariants directly.  For disambiguation and ease of reference we will name these variants differently. 
 
\medskip

Suppose $B$ is a complex analytic  super manifold or a germ of a complex analytic super manifold. Let  $\mathbf{D}$ denote a small analytic disk in $\bbC$ centered at zero, or the germ of $\bbC$ at zero, and let $u$ denote the coordinate on $\mathbf{D}$. A  \strongemph{complex analytic F-bundle} consists of $(\cH, \nabla)/B$ where $\cH$ is a complex analytic super vector bundle on $B\times \mathbf{D}$ and $\nabla$ is a meromorphic flat connection on $\cH$ with poles at most at $u=0$, such that for all locally defined vector fields on $B$, the operators $\nabla_{u^2\partial_u}$ and $\nabla_{u\xi}$ have no poles.
 We will say that $(\cH, \nabla)/B$ is a \strongemph{logarithmic complex analytic F-bundle} if $B$ and $\cH$ are as before, and $\nabla$ is a meromorphic flat connection on $\cH$ with poles at most at $(B\times \{0\})\cup(\mathfrak{D}\times \mathbf{D})$ for some even strict normal crossings divisor $\mathfrak{D} \subset B$, such that  for all locally defined vector fields $\xi$ on $B$ which are tangent to $D$, we have that the operators $\nabla_{u^2\partial_u}$ and $\nabla_{u\xi}$ have no poles.

To be consistent with the common mirror symmetry terminology from \cite{Katzarkov_Hodge_theoretic_aspects} or the singularity theory terminology from \cite{hertling_tt*_geometry} we will also call such a triple a \strongemph{(logarithmic) de Rham part of a \nc \ variation of Hodge structures} or  a \strongemph{(logarithmic) TE structure}. 
 
\medskip

Suppose $B$ is a smooth formal super scheme over a field $k$, $\op{char}(k) = 0$. Then a  \strongemph{formal F-bundle defined over $k$} consists of $(\cH, \nabla)/B$, where $\cH$ is a super vector bundle on $B\times \Spf \, k\dbb{u}$ and $\nabla$ is a meromorphic flat connection on $\cH$ with poles at most at $u=0$, such that the operators $\nabla_{u^2\partial_u}$ and $\nabla_{u\xi}$ have no poles. Similarly, a 
\strongemph{formal logarithmic F-bundle defined over $k$} consists of $(\cH, \nabla)/B$ with $B$ and $\cH$ as before, and 
$\nabla$ a meromorphic flat connection on $\cH$ with poles at most 
along $(B\times \{0\})\cup (\mathfrak{D}\times \Spf \, k\dbb{u})$ for an even strict normal crossing divisor $\mathfrak{D} \subset B$, and such that  for all locally defined vector fields $\xi$ on $B$ which are tangent to $D$, we have that the operators $\nabla_{u^2\partial_u}$ and $\nabla_{u\xi}$ have no poles.

\subsection{Operations on \texorpdfstring{$F$}{F}-bundles}

For a given base $B$, F-bundles over $B$ form a category
$\Fbun(B)$ which is a $\bbk$-linear, additive, Karoubi closed tensor category with duals. The unit object in the category $\Fbun(B)$ is the trivialized rank $1$ bundle on $B$ with the trivial connection.
Note that the category $\Fbun(B)$ is not abelian, it is only quasi-abelian (hence an \strongemph{exact category}).
An example of a morphism in $\Fbun(\Sp\bbk)$ without an adequate kernel is the operator of multiplication by $u$:
\[
  \big( \cO, \nabla_{\partial_u}=\partial_u \big) \xrightarrow{u\cdot(-)} \bigg({\cO},\nabla_{\partial_u}=\partial_u-\frac{1}{u}\bigg),\qquad
  u\cdot \partial_u(f)=\bigg(\partial_u-\frac{1}{u}\bigg)(uf)
  \quad \forall f\in {\bbk}\dbb{u}
\]
For any map $\phi\colon B'\to B$, and any F-bundle $(\cH, \nabla)/B$ over $B$ as above, we define the pullback to be the F-bundle over $B'$ given by $(B',\tphi^*\cH,\tphi^*(\nabla))$ where $\widetilde\varphi$ is $\varphi\times{\rm id}_{\bbD}$.
The pullback is an exact tensor functor from $\Fbun(B)$ to $\Fbun(B')$.

\begin{definition}
For two F-bundles $(B_{1},\cH_1,\nabla_1)$ and
$(B_{2},\cH_2,\nabla_2)$, we define their \strongemph{external sum} to be the F-bundle over $B_1\times B_2$ given by 
\[
(B_{1},\cH_1,\nabla_1)\boxplus(B_{2},\cH_2,\nabla_2) \coloneqq \big(B_{1}\times B_{2},\pr_{13}^*\cH_1 \oplus \pr_{23}^*\cH_2,\ pr_{13}^*\nabla_1 \oplus \pr_{23}^*\nabla_2\big),
\]
where $\pr_{13}$ and $\pr_{23}$ denote the projections from $B_1\times B_2\times \bbD$ to $B_1\times \bbD$ and $B_2\times \bbD$ respectively.

Similarly, we define their \strongemph{external product} to be the F-bundle over $B_1 \times B_2$ as
\[
(B_{1},\cH_1,\nabla_1) \boxtimes (B_{2},\cH_2,\nabla_2) \coloneqq \big(B_{1}\times B_{2}, p_{13}^{*} \cH_1\otimes p_{23}^{*}\cH_2,\ p_{13}^{*}\nabla_1\otimes \id +
\id\otimes p_{23}^{*}\nabla_2\big) .\]
\end{definition}

\begin{remark}
One can define external sums and products of any finite collections of F-bundles.
For the external sums and products of finitely many copies of the same F-bundle, one can replace the finite indexing set labeling the copies by a reduced finite scheme (or finite analytic space)  over $\bbk$, i.e.\ by the finite set of its
$\obbk$-points endowed with the action of the Galois group $\Gal(\obbk/\bbk)$.
\end{remark}

\subsection{Natural symmetries of \texorpdfstring{$F$}{F}-bundles}

Write $\Aff(1) = \Gm\ltimes \Ga$ for the group of affine linear automorphisms of the analytic affine line $\mathbb{A}^{1,\op{an}}_{\bbk}$.
The group $\Aff(1)\times\Ga$ acts on the category $\Fbun(B)$ of F-bundles over a base $B$ as follows.
We denote elements in the group as triples
$(\lambda,c_{-2},c_{-1}) \in (\Gm \ltimes \Ga)\times \Ga = \Aff(1)\times \Ga$.
\begin{description}[wide]
  \item[\emph{Dilation}:] The component $\lambda\in \Gm\subset \Aff(1)$ acts on $\Fbun(B)$ by pullback via $\id_B \times d_\lambda \colon B\times \bbD \to B\times \bbD$, where $d_\lambda\colon u \mapsto \lambda u$ denotes the dilation of the coordinate $u$.
  We denote the action by $D_\lambda$.
  \item[\emph{Exponential shift}:] The component $c_{-2} \in \Ga \subset \Aff(1)$ acts on $\Fbun(B)$ by tensoring with the rank-one F-bundle
  \begin{equation} \label{eq:exponential_shifts}
  \left(B;\mathcal{O}_{B\times \bbD},\ \nabla = d + \frac{c_{-2}}{u^2} du\right).
  \end{equation}
  We denote the action by $\Se_{c_{-2}}$.
  \item[\emph{Power shift}:] The component $c_{-1}$ in the second copy of $\Ga$ acts on $\Fbun(B)$ by tensoring with the rank-one F-bundle
  \begin{equation} \label{eq:power_shifts}
    \left(B;\mathcal{O}_{B\times \bbD},\ \nabla = d + \frac{c_{-1}}{u} du\right).
  \end{equation}
  We denote the action by $S_{c_{-1}}$.
\end{description}

\begin{remark} \label{rem:action}
  The exponential shift corresponds to multiplying the flat sections in an F-bundle  with $e^\frac{c_{-2}}{u}$, and the power shift corresponds to multiplying  with $u^{c_{-1}}$, whence their names.
\end{remark}

\subsection{Maximal and overmaximal \texorpdfstring{$F$}{F}-bundles} \label{sec:maximal}

Let $(\cH,\nabla)/B$ be an F-bundle (resp.\ a logarithmic F-bundle with $\nabla$ having log poles along an even strict normal crossings divisor $\mathfrak{D} \subset B$).
For any vector field $\xi$ on $B$ (resp.\ a vector field on $B$ tangent to $\mathfrak{D}$ along $\mathfrak{D}$), the restriction of the operators $\nabla_{u\xi}$ to $\cH_{|u=0}$ gives a  map
\[ 
\mu \colon T_{B} \longto \End(\cH_{|u=0}),
\qquad \text{respectively} \ \mu \colon T_{B}(-\log \mathfrak{D}) \longto \End(\cH_{|u=0}). 
\]
For any geometric point $b$ in $B$, we have fiberwise maps
\[
\mu_b \colon T_{B,b} \to \End(\cH_{(b,0)}), 
\qquad \text{respectively} \ \mu_{b} \colon T_{B,b}(-\log \mathfrak{D}) \longto \End(\cH_{(b,0)}). 
\]
The flatness of $\nabla$ implies that the image of $\mu_b$ consists of commuting operators. 

\begin{definition} \label{def:overmaximal}
An F-bundle $(\mathcal H,\nabla)/B$ over $B$ is \strongemph{maximal} (resp.\ \strongemph{
over-maximal}) at a geometric point $b\in B$ if there exists a cyclic vector for the action $\mu_{b}$, i.e.\ a vector $h \in \cH_{(b,0)}$ such that the map $\ev_h \circ \mu_b\colon T_{B,b} \to \cH_{(b,0)}$ is an isomorphism (resp.\ epimorphism). 

Similarly a logarithmic F-bundle $(\mathcal H,\nabla)/B$ over $B$ is \strongemph{maximal} (resp.\ \strongemph{
over-maximal}) at a geometric point $b\in B$ if there exists a cyclic vector for the action $\mu_{b}$, i.e.\ a vector $h \in \cH_{(b,0)}$ such that the map $\ev_h \circ \mu_b\colon T_{B,b}(-\log \mathfrak{D})  \to \cH_{(b,0)}$ is an isomorphism (resp.\ epimorphism).
  
A (logarithmic)  F-bundle is maximal (resp.\ over-maximal) if it is maximal (resp.\ over-maximal) everywhere.
\end{definition}

Note that being maximal (resp.\ over-maximal) is an open condition. The significance of being over-maximal or maximal is captured by the following

\begin{lemma} \label{lem:algebra}
For a  F-bundle $(\cH,\nabla)/B$ which is over-maximal at $b\in B$, the image of $\mu_b$ is a unital commutative associative super algebra over $\bbk$, and the fiber $\cH_{(b,0)}$ is a free rank one module over 
$\im \mu_b \simeq T_{B,b}/\ker \mu_b$.

Similarly, for a logarithmic F-bundle $(\cH,\nabla)/B$ which is over-maximal at $b\in B$, the image of $\mu_b$ is a unital commutative associative super algebra over $\bbk$, and the fiber $\cH_{(b,0)}$ is a free rank one module over 
$\im \mu_b \simeq T_{B,b}(-\log \mathfrak{D})/\ker \mu_b$.
\end{lemma}
\begin{proof}
The claim in the lemma boils down to  the following linear algebra statement:  Let $T$ and $H$ be $\bbk$-vector spaces, and $\mu \colon T \to \End(H)$ a linear map satisfying the following conditions:
  \begin{enumerate}
    \item[(1)] The image $\mu(T)$ consists of commuting operators.
    \item[(2)] There is a vector $h \in H$ such that the map $\ev_h \circ \mu \colon T \to H$ is surjective.
  \end{enumerate}
  Then $\mu(T) \subset \End(H)$ is a unital commutative associative subalgebra and $H$ is a free rank one module over $\mu(T)$.

To see that this statement holds we need to check that if  $A \in \End(H)$ is an operator that commutes with all elements in $\mu(T)$, then  $A \in \mu(T)$. By Condition (2), there exists $v \in T$ such that $A h = \mu(v) h$.   Then for any $w \in T$, by Condition (1) we have
  \[ A \mu(w) h = \mu(w) A h = \mu(w) \mu(v) h = \mu(v) \mu(w) h. \]
  By Condition (2) again, we deduce that $A = \mu(v)$, in particular, $A \in \mu(T)$.
  
  Note that for any $v, w \in T$, the product $\mu(v)\mu(w)$ commutes with all elements in $\mu(T)$, so it lies in $\mu(T)$ by the above reasoning. For the same reason, $\id \in \mu(T)$.
  Hence $\mu(T) \subset \End(H)$ is a unital commutative associative subalgebra.

  By Condition (2), $h \in H$ is a generator of $H$ as $\mu(T)$-module.
  For any $A \in \mu(T)$, by Condition (1), if $A h = 0$, then for any $v \in T$, $A \mu(v) h = \mu(v) A h = 0$,
  so $A = 0$ by Condition (2).
  Therefore, $H$ is a free rank one module over $\mu(T)$ generated by $h$.
  This completes the proof of the lemma.
\end{proof}

For an overmaximal F-bundle $(\cH, \nabla)/B$ the subbundle $\ker\mu$ of the tangent bundle $T_{B}$ is an integrable distribution on $B$, i.e.~it defines a foliation $\cR$ on $B$. We will call $\cR = \ker \mu$ the \strongemph{redundancy foliation} of the overmaximal F-bundle $(\cH, \nabla)/B$.

\begin{lemma} \label{lem:foliation}
Suppose $(\cH, \nabla)/B$ is an overmaximal F-bundle. Then for any rigid point $b \in B$, locally near $b$ the F-bundle $(\cH, \nabla)/B$ is pullback of a maximal F-bundle on a non-archimedean analytic super manifold which is a transversal slice of the redundancy foliation $\cR$ through $b$.  Equivalently, globally $(\cH, \nabla)/B$  is a pullback of a maximal formal F-bundle on the formal stack $\cM \, \coloneqq \, [B/\cR]$ of leaves of $\cR$.
\end{lemma}

\begin{proof}
  A vector field $\xi$ on $B$ lies in the sub bundle $\ker\mu$ if and only if the operator $\nabla_\xi$ is regular (i.e.\ has no poles) along $u=0$.
  If $\xi_1$, $\xi_2$ are two such vector fields, then so is $[\xi_1,\xi_2]$, because
  \[
  \nabla_{[\xi_1,\xi_2]}=[\nabla_{\xi_1},\nabla_{\xi_2}]
  \]
  by the flatness of $\nabla$.

  By construction, along the leaves of the foliation $\cR$, the connection $\nabla$ is regular at $u=0$.
  So it trivializes the bundle $\cH$ along the leaves.
  This implies the second statement.
\end{proof}

Let $(\cH,\nabla)/B$ be a maximal F-bundle (resp.\ maximal logarithmic F-bundle) over $B$.
It follows from \cref{lem:algebra} that the tangent bundle $T_{B}$ (resp.\ the logarithmic tangent bundle $T_{B}(-\log \mathfrak{D})$) inherits 
the structure of a unital commutative associative superalgebra from its embedding 
in $\op{End}\left(\cH_{u=0}\right)$. 

\begin{definition} \label{def:maxFfp}
We denote the induced product by~$\fp$ and call it \strongemph{the Frobenius product associated with the maximal F-bundle}.
\end{definition}

\begin{definition} \label{def-euler} 
Let $(\cH, \nabla)/B$ be a maximal F-bundle (resp.\ maximal logarithmic F-bundle) on $B$. The
\strongemph{Euler vector field} of $(\cH, \nabla)/B$ is the unique
even vector field $\Eu$ on $B$ which under the action $\mu$ maps to
the residual endomorphism $\bkappa \coloneqq \nabla_{u^2\partial_{u}}|_{\cH_{|u=0}}$ of $\cH_{|u=0}$ associated with $(\cH, \nabla)/B$, that is $\mu(\Eu) = \bkappa$. In the logarithmic case  $\Eu$ is automatically a section in $T_{B}(-\log \mathfrak{D}) \subset T_{B}$. 
\end{definition}

The definition makes sense because operator $\bkappa$ commutes with the all operators from $\mu(T_B)$, as follows from the flatness of $\nabla$.

\begin{remark} \label{rem-overmaxFqupEu} 
{\bfseries  (a)} \ If  $(\cH,\nabla)/B$ is not maximal but just overmaximal, then we still get a sheaf $\mu(T_{B}) = T_{B}/\cR$ of unital, associative, supercommutative algebras on $B$, where again the product structure on $\mu(T_{B})$ is inherited from the product of endomorphisms, through the embedding 
$\mu(T_{B}) \, \subset \op{End}\left(\cH_{u=0}\right)$. Furthermore,
$\cH|{u=0}$ is a locally free $\mu(T_{B})$-module of rank one.

\smallskip

\noindent
{\bfseries (b)} \ Suppose now  $(\cH,\nabla)/B$ is an overmaximal F-bundle which has a global cyclic vector\footnote{For instance,  the $\mathsf{A}$-model F-bundle we consider in the next section is of this type.} $h \in \Gamma(B,\cH|_{u=0})$, then $\mathsf{ev}_{h}\colon \mu(T_{B}) \cong \cH|_{u = 0}$  is an isomorphism of coherent sheaves,  which is compatible with the $\mu(T_{B})$-actions on source and target. 

\quad Transporting the product on $\mu(T_{B})$ via $\mathsf{ev}_{h}$ gives an associative supercommutative product on $\cH|_{u=0}$ which we will again denote by $\fp$.
The image of the operator $\bkappa \in \Gamma(B,\mu(T_{B}))$ via this isomorphism is an element
in $\mathsf{ev}_{h}(\bkappa)  \in \Gamma(B,\cH|_{u=0})$, which we will  again denote by $\Eu$ and will call the Euler element (or section) associated with
$(\cH,\nabla)/B$. By construction we have $\bkappa = \Eu\fp(-)$, and when
$(\cH,\nabla)/B$, this is precisely the Euler vector field defined above.
\end{remark}

The Euler vector field interacts in an interesting way with the
Frobenius product on $T_{B}$. Indeed, using $\fp$ we can define an
infinite sequence of vector fields on $B$ by taking powers of $\Eu$
with respect to $\fp$:
\[
\begin{aligned}
  \Eu^0 & = \mathbf{1}_{B} \coloneqq \text{ unit for } \fp, \\
  \Eu^1 & = \Eu \\
  \Eu^{k} & = \underbrace{\Eu\fp \Eu\fp \cdots \fp\Eu}_{k \ \text{times}}
\end{aligned}
\]
The reasoning of \cite[Proposition~3.6.2,Theorem~5.6]{Manin_Frobenius_manifolds}
shows that the vector fields $\{ \Eu^{k} \}_{k\geq 0}$ give an action of the Witt
Lie algebra $\bbk[t]\partial_t$ 
on $B$:
\[
\forall k,l \gs 0 \colon \  [\Eu^k, \Eu^l]= (l-k)\Eu^{k+l-1} \qquad
(\text{cf.\ }[t^k\partial_t,t^l\partial_t]=(l-k) t^{k+l-1}\partial_t ).
\]

\noindent
The analogue of Remark~\ref{rem:witt.move} holds in this setting as well. 

\begin{remark} \label{rem:Fbundle.witt.move}
let $(\cH,\nabla)/B$ be a maximal F-bundle, and let $b \in B^{\even} \subset B$ be a rigid even point  of $B$. Then there exists a unique germ $W \subset B^{\even}$ at $b$ of a purely even closed analytic submanifold in $B$ s.t. 
$\mathsf{span}_{\bbk}\left(\{\Eu^{k}_{w}\}_{k\geq 0}\right) = T_{W,w}$ for all $w\in W$ close to $b$. 
The number of distinct eigenvalues of $\bkappa|_{W} \, \colon \, T_{W} \ \to \ T_{W}$ and the list of their multiplicities are constant on a neighborhood of $b \in W$. Moreover, more fine results can be proved, like, e.g., if $(\cH,\nabla)/B$ is of exponential type (see \cite{Katzarkov_Hodge_theoretic_aspects} for the definition of exponential type) at some point of $W$, then it is of exponential type at \textit{all}  points of $W$. This implies that the isomorphism classes of formal meromorphic connections with regular singularities over $\op{Spf}\bbk\dbb{u}$ corresponding to the eigenvalues of $\bkappa|_{W}$ via Hukuhara-Levelt-Turrittin decomposition are constant along $W$.
\end{remark}

\subsection{The \texorpdfstring{$\mathsf{A}$}{A}-model \texorpdfstring{$F$}{F}-bundle} \label{subsec:Amodelexample}

In this section we will describe the main example of the paper: the F-bundle underlying the $\mathsf{A}$-model variation of \nc-Hodge structures.  We start with the formal picture, which we subsequently modify to get the non-archimedean analytic version which we will use in the applications.

\subsubsection{Formal logarithmic  \texorpdfstring{$\mathsf{A}$}{A}-model 
\texorpdfstring{$F$}{F}-bundles} \label{sssec:Amodelformal}

Fix an algebraically closed  field $\kay = \kaybar$ of characteristic zero and let $X$ be a smooth projective variety over $\kay$. Gromov-Witten classes are cohomology classes of algebraic cycles, so to talk about them we have to fix a cohomology theory. For convenience, in this section we will work with de Rham cohomology since it does not depend on any external choices. Later on we will study the dependence on the choice of cohomology theory more carefully.

As in the complex analytic setting, for every $n \in \bbZ_{\geq 0}$, and every 
\[
\beta \in \CHhom_{1}(X) = \CH_{1}(X)/\homeq = 
\op{im}\left[\CH_{1}(X) \to H^{2\dim X - 2}_{dR}(X/\kay)\right]
\]
we will write
\[
\bdelta(n,\beta) = n + (\dim X - 3) + \int_{\beta}c_{1}(T_{X}) 
\]
for the expected dimension of the moduli stack $\Mbar_{0,n}(X,\beta)$ parametrizing stable genus zero $n$-marked  maps to $X$. Concretely  $\Mbar_{0,n}(X,\beta)$ is the moduli of maps
$\varphi \colon  (C,p_{1},\ldots, p_{n}) \to X$, where:
\begin{enumerate}[wide]
\item $C$ is a connected  nodal, genus $0$ curve, defined over $\kay$.
\item $p_{1},\ldots,p_{n}$ are distinct smooth $\kay$-points of $C$.
\item $\varphi$ is a morphism satisfying $\varphi_{*}[C] = \beta$
and the stability condition that if $\varphi$ contracts a component of $C$ to a point, 
then the number of marked points and nodes on this component is $\geq 3$.
\end{enumerate}
The stack $\Mbar_{0,n}(X,\beta)$ is a proper Deligne-Mumford stack/$\kay$. 
It is equipped (see \cite{Kontsevich_Gromov-Witten_classes,
Behrend_Gromov-Witten_invariants,Behrend_Intrinsic_normal_cone}) with a 
\strongemph{virtual fundamental cycle class}
\[
\left[\Mbar_{0,n}(X,\beta)\right]_{\op{vir}}  \in \CH_{\bdelta(n,\beta)}\left(\Mbar_{0,n}(X,\beta)\right)
\]
of dimension $\bdelta(n,\beta)$, which further maps to 
a \strongemph{virtual fundamental class} modulo homological equivalence 
\[
\left[\Mbar_{0,n}(X,\beta)\right]_{\op{vir}}^{\mathsf{hom}}  \in \CHhom_{\bdelta(n,\beta)}\left(\Mbar_{0,n}(X,\beta)\right)
\]
We consider the \strongemph{Gromov-Witten cycle class}
$I_{n,\beta}(X)$ defined as the image
\[
\begin{aligned}
I_{n,\beta}(X)  = \op{ev}_{*}\left[\Mbar_{0,n}(X,\beta)\right]_{\op{vir}}^{\mathsf{hom}} 
& \in \CHhom_{\bdelta(n,\beta)}\left(X^{\times n} \right)\otimes \bbQ\\
& \quad \subset \ H^{2n\dim X - 2\bdelta(n,\beta)}_{dR}(X^{\times n}/\kay)
\end{aligned}
\]
under the natural evaluation map 
\[
\op{ev} \colon \Mbar_{0,n}(X,\beta) \to X^{\times n}, \qquad
(C,p_{1},\ldots,p_{n},\varphi) \mapsto (\varphi(p_{1}), \ldots, \varphi(p_{n})).
\]
Similarly, if $\mathsf{J}$ is any finite set of cardinality $n$, we can consider the moduli stack of genus zero stable maps with marked points labeled by the elements in $\mathsf{J}$, and define in the same way Gromov-Witten classes 
$I_{\mathsf{J},\beta}(X) \in 
\CHhom_{\bdelta(n,\beta)}\left(X^{\mathsf{J}}\right)\otimes \bbQ$.

\medskip
  
\noindent{\bfseries Properties of the Gromov-Witten classes \cite{Kontsevich_Gromov-Witten_classes}:} \

\medskip

\noindent
{\bfseries (i)} \ $I_{n,\beta}(X) \neq 0$ implies that $\beta$ is the class of an effective one cycle.
For future reference we will write c.

\medskip

\noindent
{\bfseries (ii)} \ When $\beta = 0$, then $I_{n,0}(X) \neq 0$ implies that $n = 3$, and  $I_{3,0}(X) = 
(X \hookrightarrow X^{\times 3})_{*}[X]$, where $X \hookrightarrow X^{\times 3}$ is the diagonal embedding.

\medskip

\noindent
{\bfseries (iii)} \ $I_{n,\beta}(X)$ is invariant under the permutation action of the symmetric group $\mathsf{S}_{n}$ on $X^{\times n}$. 
Thus, for a finite set $\mathsf{J}$ of cardinality $n$,  the class $I_{\mathsf{J},\beta}(X)$ can be identified unambiguously  with  $I_{n,\beta}(X)$. Indeed,  any ordering of the elements in $\mathsf{J}$ identifies $I_{\mathsf{J},\beta}(X)$ with $I_{n,\beta}(X)$ and the $\mathsf{S}_{n}$-invariance guarantees that this identification is independent of the choice of order.

\medskip

\noindent
{\bfseries (iv)} \ $I_{n,\beta}(X)$ satisfy the \strongemph{unit and divisor axioms}. To state these tersely we will introduce some notation. Suppose $n, m \geq 0$ are integers and let $\xi \in \CHhom_{\bullet}(X^{\times(n+m)})$ and $\eta \in \CH^{\mathsf{hom}, \, \bullet}(X^{\times m})$. We define a new class $\int_{\xi} \eta \in \CHhom_{\bullet}(X^{\times n})$ by setting
\[
\int_{\xi} \eta \  = \ \left(X^{\times (n+m)} \to X^{\times n}\right)_{*} \left(\xi \frown \left(X^{\times (n+m)} \to X^{\times m}\right)^{*} \eta\right) 
\]
where $X^{\times (n+m)} \to X^{\times n}$ is the projection on the first $n$ factors and 
$X^{\times (n+m)} \to X^{\times m}$ is the projection the last $m$ factors. With this notation we can now state that $I_{\beta,n}(X)$ satisfy the following two axioms:
\begin{description}
\item[(Unit axiom)] For all $n \geq 0$, and $\beta \neq 0$ we have $\int_{I_{n+1,\beta}(X)} \mathbf{1}_{X} = 0 \in \CHhom_{\bullet}(X^{\times n})$, where $\mathbf{1}_{X} \in \CH^{\mathsf{hom}, \, 
\bullet}(X)$ denotes  the unit. 
\item[(Divisor axiom)] For any line bundle $\mathcal{L}$ on $X$ we have
\[
\int_{I_{n+1,\beta}(X)} c_{1}(\mathcal{L}) = \left( \int_{\beta} c_{1}(\mathcal{L})\right)\cdot I_{n,\beta}(X) \in \CHhom_{\bullet}(X^{\times n}).
\]
\end{description}

\medskip

\noindent
{\bfseries (v)} \ For all $n \geq 4$ and all $\beta$ the Gromov-Witten classes obey the \strongemph{Witten-Dijkgraaf-Verlinde-Verlinde (WDVV) equation}:
\[
\begin{aligned}
\sum_{\substack{\beta = \beta_{1}+ \beta_{2} \\ \{5,6,\ldots, n\} = \mathsf{J_{1}}\sqcup \mathsf{J}_{2}}} &
\int_{I_{\{1,2\}\sqcup \mathsf{J}_{1}\sqcup \{\mathbf{left}\},\beta_{1}}(X)\times I_{\{3,4\}\sqcup \mathsf{J}_{2}\sqcup \{\mathbf{right}\},\beta_{2}}(X)}\displaylimits \left[X \stackrel{\mathsf{diag}}{\longrightarrow} X\times X\right] \quad = \\ & \
- \sum_{\substack{\beta = \beta_{1}+ \beta_{2} \\ \{5,6,\ldots, n\} = \mathsf{J_{1}}\sqcup \mathsf{J}_{2}}} 
\int_{I_{\{1,3\}\sqcup \mathsf{J}_{1}\sqcup \{\mathbf{left}\},\beta_{1}}(X)\times I_{\{2,4\}\sqcup \mathsf{J}_{2}\sqcup \{\mathbf{right}\},\beta_{2}}(X)}\displaylimits \left[X \stackrel{\mathsf{diag}}{\longrightarrow} X\times X\right]
\end{aligned}
\]
Here, we are capping with the class $[X \stackrel{\mathsf{diag}}{\to} X\times X] \in 
\CH^{\dim X}(X\times X)$ of the diagonal in the Cartesian square $X\times X$ which parametrizes the pairs of points with labels 
$(\mathbf{left},\mathbf{right})$.

\begin{remark} \label{rem:nonclosedL} If $\mathfrak{X}$ is a smooth projective variety  defined over an arbitrary field $\kay$ of characteristic zero, then by the above procedure, we can construct Gromov-Witten classes 
\[
I_{n,\beta}(\mathfrak{X}\otimes \kaybar) \in \CHhom_{\bdelta(n,\beta)}(\mathfrak{X}^{\times n}\otimes\kaybar)\otimes \bbQ.
\]
The construction  is manifestly invariant under the action of the Galois group $\Gal(\kaybar/\kay)$, i.e.\ for every $g \in \Gal(\kaybar/\kay)$
we have $g_{*}I_{n,\beta}(\mathfrak{X}\otimes \kaybar) = I_{n,g_{*}\beta}(\mathfrak{X}\otimes \kaybar)$. 
\end{remark}

\medskip

One largely unexplored aspect of Gromov-Witten theory is its algebraic flexibility.  The algebraic nature of Gromov-Witten invariants makes them a powerful tool for encoding and analyzing cycle-theoretic and motivic information. Above we summarized the constructions and properties of Gromov-Witten invariants in the classical setting where we were dealing with a variety over a field of characteristic zero, and the homological equivalence was defined with respect to the algebraic de Rham cohomology. For the applications to birational geometry, we would like to  take full advantage of this algebraic flexibility,  and so we will need the  following variant of Gromov-Witten invariants which makes sense for arbitrary fields of definition and for homological equivalence extracted from a cycle map to a reasonably general Weil cohomology theory.

\begin{variant} \label{var-generalWeil} 
Fix two fields $\kay$ and $k$, both of characteristic zero. Let $H^{\bullet}$ be a classical Weil cohomology theory  defined on non-empty smooth projective  $\kay$-varieties, with values in $k$-vector spaces  (see e.g.\ \cite[Definition~1.2.13]{murre} and \cite[Section~45.7,tag~0FGS]{stacksproject} for the definition).

\begin{example} \label{ex:standardLkH}
For us standard choices will be:

\smallskip 

\noindent
{\bfseries (a) (de Rham setting):} \ $\kay$ any field, $\mathsf{char} \, \kay = 0$, $k = \kay$, and $H^{\bullet}$ is de Rham cohomology over $k = \kay$. That is for any non-empty smooth projective variety $\mathfrak{X}$ over $\kay$ we have $H^{\bullet}(\mathfrak{X}) \coloneqq H^{\bullet}_{dR}(\mathfrak{X}/k)$.

\smallskip 

\noindent
{\bfseries (b) (Dolbeault setting):} \ $\kay$ any field, $\mathsf{char} \, \kay = 0$, $k = \kay$, and $H^{\bullet}(-)$ is the Dolbeault  cohomology over $k = \kay$. That is, for any $\kay$-variety $\mathfrak{X}$ we have $H^{\bullet}(\mathfrak{X}) \coloneqq H^{\bullet}_{Dol}(\mathfrak{X}/k) = \oplus_{p,q} \, H^{q}(\mathfrak{X},\Omega^{p}_{\mathfrak{X}/k})$.

\smallskip

\noindent
{\bfseries (c) (Betti setting):} \ $\kay\subset \bbC$, $k = \bbQ$, and $H^{\bullet}$ is the Betti cohomology of the associated complex manifold. Indeed, for every $\kay$-variety $\mathfrak{X}$ we set   
$H^{\bullet}(\mathfrak{X}) = H^{\bullet}_{B}(\mathfrak{X}(\bbC)^{\op{an}},\bbQ)$, where $\mathfrak{X}(\bbC)^{\op{an}}$ is the complex analytic space of $\bbC$-valued points of the base change $\mathfrak{X}\times_{\kaybar} \bbC$, taken with the complex topology. 

\smallskip

\noindent
{\bfseries (d) ($\ell$-adic setting):} \ $\kay$ any field of $\mathsf{char}\, \kay = 0$, $k = \bbQ_{\ell}$, and $H^{\bullet}(-)$ is the $\ell$-adic cohomology. That is, we fix an algebraic closure $\kay \subset \kaybar$, and for any $\kay$-variety $\mathfrak{X}$ we set $H^{\bullet}(\mathfrak{X}) = H^{\bullet}_{\text{\'{e}t}}(\mathfrak{X}\times_{\kay} \kaybar,\bbQ_{\ell}) = {\displaystyle \lim_{\substack{\longleftarrow}} \, H^{\bullet}_{\text{\'{e}t}}(\mathfrak{X}\times_{\kay} \kaybar,\, \bbZ/\ell^{k}\bbZ)\otimes_{\bbZ_{\ell}} \bbQ_{\ell}}$.
\end{example}

\medskip

To be specific, from now on we will focus on the Betti setting. 
Let $\mathfrak{X}/\kay$ be a smooth projective variety. Let, as above, $\kay \subset \kaybar \subset \bbC$ be the natural algebraic closure of $\kay$ in $\bbC$ and let $X = \mathfrak{X}\times_{\kay} \kaybar \coloneqq \mathfrak{X}\times_{\Spec\, \kay} \, \Spec\, \kaybar$ denote the base change of $\mathfrak{X}$ to $\kaybar$.
The Chow group $GH_{1}(X)$ of $1$-dimensional algebraic cycles on $X$ maps by the cycle class map to the second singular homology $H_{2}(X(\bbC)^{an},\bbQ)$ of the associated complex manifold, and as above we set 
\[
\CHhom_{1}(X)  \, \coloneqq  \, \CH_{1}(X)/\homeq \, = \, \op{im}\left[ \CH_{1}(X) \to H_{2}(X(\bbC)^{\op{an}},\bbQ)\right].
\]
The group $\CHhom_{1}(X)$ is a subgroup of finite index in the group of all degree two rational homology classes on $X$ that pair integrally with  the N\'{e}ron-Severi lattice $\NS(X,\bbZ) \subset H^{2}_{B}(X(\bbC)^{\op{an}},\bbQ)$. In particular $\CHhom_{1}(X) \cong \bbZ^{\rho}$ is a free abelian group of rank $\rho$ equal to the N\'{e}ron-Severi rank of $X$. 

Now, for any $\beta \in \CHhom_{1}(X)$ we can follow the above procedure and define Gromov-Witten cohomology classes 
\[
I_{n,\beta}(X) \in \CHhom_{\delta(n,\beta)}(X^{\times n})\otimes \bbQ \ \subset \ 
H^{2n - 2\delta(n,\beta)}_{B}(X^{\times n},\bbQ).
\]
The reasoning of \cite{Kontsevich_Gromov-Witten_classes} applies here and implies that these classes satisfy the properties {\bfseries (i)}-{\bfseries (v)}.  

\medskip

Note also that since Betti cohomology is a Weil cohomology theory, the Chow group integration defined in the description of property {\bfseries (iv)} transfers verbatim to the cup product 
of the associated cycle classes. Concretely, suppose $Y$ is a smooth projective $\kaybar$-variety, and let \linebreak $\mathsf{cl} \colon \CH_{\bullet}(Y) = \CH^{\dim Y-\bullet}(Y) \to H^{2\dim Y - 2\bullet}_{B}(Y,\bbQ)$ denote the cycle class map to Betti cohomology. Now if 
$\xi \in \CH_{\bullet}(X^{\times (n+m)})$ and $\eta \in \CH^{\bullet}(X^{\times n})$ we will have 
\[
\mathsf{cl}\left(\int_{\xi} \eta\right) = \int_{\mathsf{cl}(\xi)} \mathsf{cl}(\eta) \ \coloneqq \
\  (\op{pr}^{n+m}_{n})_{*}\left(\mathsf{cl}(\xi) \smile 
(\op{pr}^{n+m}_{m})^{*}\mathsf{cl}(\eta)\right),
\]
where  $\op{pr}^{n+m}_{n} \colon X^{\times (n+m)} \to X^{\times n}$ and 
$\op{pr}^{n+m}_{m} \colon X^{\times (n+m)} \to X^{\times m}$ are the projections on the first $n$ and the last $m$ factors respectively, $(\op{pr}^{n+m}_{m})^{*}$ is the pullback map on $H^{\bullet}_{B}(-,\bbQ)$, and $(\op{pr}^{n+m}_{n})_{*}$ is the Gysin map on $H^{\bullet}_{B}(-,\bbQ)$ which is the adjoint  to 
$(\op{pr}^{n+m}_{n})^{*}$ with respect to the Poincar\'{e} pairing on 
$H^{\bullet}_{B}(-,\bbQ)$.
\end{variant}

With all of this in place we can now describe the requisite formal logarithmic F-bundles.  To streamline the exposition we will focus on the case $\kay = \kaybar \subset \bbC$, $k$ of characteristic zero, and $H^{\bullet} = H^{\bullet}_{B}(-,k)$. The case of varieties over a subfield $\kay \subset \bbC$ which is not algebraically closed is also interesting and we will discuss it in the next section.  For now, assume $\kay = \kaybar \subset \bbC$ and consider a smooth projective $\kay$-variety $X$. Fix a non-empty collection  of ample classes $\omega_{1}, \ldots, \omega_{m}$ in $\CH^{1,\mathsf{hom}}(X) \subset H^2(X)$ on $X$ which are linearly independent over $\bbQ$, and a collection 
$\{\Delta_{j}\} \subset H^{\bullet}(X)$ of pure degree cohomology classes, which completes 
the collection $\omega_{1}, \ldots, \omega_{m}$ to a graded basis of the $k$-vector space $H^{\bullet}(X)$.  Let $\{q_{i}\}_{i = 1}^{m}\cup \{t_{j}\}$ be the dual linear coordinates on $H^{\bullet}(X)$. Focusing on the $\bbZ/2$-grading of $H^{\bullet}(X)$, i.e.\ remembering only the parity of the cohomological degree,  we can consider 
the associated super affine space $\affa{X}  \coloneqq \Spec \left(k[ \op{Sym}(H^{\bullet}(X)^{\vee}\right)])$. Next we consider  the formal neighborhood of $0 \in \affa{X}$. This is the smooth formal super scheme given by
\[
\forB{X} \coloneqq \Spf k\left[\left[ (q_{i}), (t_{j})\right]\right],
\]
where all $q_{i}$ are even, and the $t_{j}$'s are even or odd depending on their cohomological degree.  
\medskip

This geometric setting gives rise to an 
$\mathcal{O}_{\forB{X}}$-linear \strongemph{quantum cup product} 
\[
\qup \colon \left(\forcH_{u=0}\right)^{\otimes 2} \to \forcH_{u=0}
\]
on the trivial vector bundle 
 $\forcH_{u=0} = \mathcal{O}_{\forB{X}}\otimes_{k} H^{\bullet}(X)$, which in the current context is defined by
the formula
\begin{equation} \label{eq:*omegas}
\gamma_{1}\qup \gamma_{2} \coloneqq \gamma_{1}\smile \gamma_{2} + \sum_{\substack{\beta \neq 0}} \prod_{i=1}^{m} q_{i}^{\int_{\beta} \omega_{i}} \cdot \sum_{\substack{n \geq 0}} \frac{1}{n!} \int_{I_{n+3,\beta}(X)} \gamma_{1}\boxtimes \gamma_{2} \boxtimes 
\left( \sum_{j}\displaylimits  t_{j}\Delta_{j} \right)^{\boxtimes n},
\end{equation}
where $\gamma_{1}\smile \gamma_{2}$ is the classical cup product on $H^{\bullet}(X)$, and the integrand is viewed as a class in $X\times X \times X^{\times n} = X^{\times (n+2)}$. 
The properties of the Gromov-Witten classes listed above imply that 
$\qup$ is an associative unital product which can be viewed  as a deformation along $\forB{X}$ of the classical cup product $\smile$  on the closed fiber $\forcH_{u=0}|_{q_{j}=0,t_{i}=0} = H^{\bullet}(X)$. 

We will also consider the trivial  bundle $\forcH \coloneqq H^{\bullet}(X)\otimes \cO_{\forB{X}\times \Spf \, k\dbb{u}}$. By definition we have $\forcH|_{u = 0} = \forcH_{u=0}$
and also $\forcH = \mathsf{pr}_{\forB{X}}^{*} \, \forcH_{u=0}$. In particular the quantum product pulls back via $\mathsf{pr}_{\forB{X}}$ to an $\cO_{\forB{X}\times \Spf \, k\dbb{u}}$-linear product on $\forcH$ which by abuse of notation we will also denote by $\qup \colon (\forcH)^{\otimes 2} \to \forcH$.

Note next that our choice of coordinates on 
$\affa{X}$ gives an even, strict normal crossing divisor 
$\mathfrak{d} = \sum_{i =1}^{m} \mathfrak{d}_{i}$ where $\mathfrak{d}_{i} \subset \affa{X}$ is given by the equation $q_{i} = 0$. This defines a natural log structure on $\affa{X}$ and hence on $\forB{X}$ where the corresponding logarithmic tangent bundles $T_{\affa{X}}(-\log \mathfrak{d})$ and $T_{\forB{X}}(-\log \mathfrak{d})$ are spanned by the vector fields $q_{i}\partial_{q_{i}}$, and $\partial_{t_{j}}$.

\medskip

The quantum product now gives rise to a meromorphic connection on
$\forcH \coloneqq H^{\bullet}(X)\otimes \cO_{\forB{X}\times \Spf \, k\dbb{u}}$ defined by
the formulas:
\begin{equation} \label{eq:formalbignabla}
\left| \
\begin{array}{lcl}
\fornabla_{\partial_{u}} & = & \partial_{u} - u^{-2} \left(\Eu \, \qup\, (-)\right) + u^{-1} \frac{\Deg - \dim X\cdot \op{id}}{2}, \\[+0.5pc]
\fornabla_{q_{i}\partial_{q_{i}}} & = & q_{i}\partial_{q_{i}} + u^{-1} \left( \omega_{i} \, \qup \, (-)\right), \\[+0.5pc]
\fornabla_{\partial_{t_{j}}} & = & \partial_{t_{j}} + u^{-1} \left( \Delta_{j} \, \qup \, (-)\right),
\end{array}
\right.
\end{equation}
where $\Deg \colon \forcH \to \forcH$ is the degree operator $\Deg = \oplus_{a = 0}^{2\dim X} a\cdot \op{id}_{H^{a}(X)}$ and $\Eu$ denotes the Euler vector field on $\forB{X}$, defined by 
\begin{equation} \label{eq:Eulog}
\Eu_{\gamma} = c_{1}(T_{X}) + \frac{\Deg - 2\op{id}}{2}(\gamma)
\in \ (\forcH_{u=0})_{\gamma} = \forcH_{\gamma\times 0} \cong T_{\forB{X},\gamma}(-\log \mathfrak{d}) \subset T_{\forB{X},\gamma}.
\end{equation}

\begin{remark} \label{rem:compare.with.KKP}
 Here we have adopted a different convention than the one used
 in \cite{Katzarkov_Hodge_theoretic_aspects} to describe the quantum
 connection. Specifically, the parameter $u$ used here and the
 parameter (also called $u$) used
 in \cite[Section~3.1]{Katzarkov_Hodge_theoretic_aspects} differ by a
 minus sign. If we set $q = q_{1}$ and if in the
 formula \eqref{eq:formalbignabla} we make the change of variable
 $u \mapsto -u$, then the restriction of $\fornabla$ to the purely even
 formal submanifold $\mathsf{Spf}\dbb{q} \subset \forB{X}$ reproduces
 the formula used to define the de Rham part of the one parameter
 $\mathsf{A}$-model variation of \nc-Hodge structures
 in \cite[Section~3.1]{Katzarkov_Hodge_theoretic_aspects}. In
 addition, note that there is a typo in the description of the
 connection in \cite[Section~3.1]{Katzarkov_Hodge_theoretic_aspects}:
 the class $\bkappa_{X}$ (used there to denote the Euler vector field) is the first Chern class of the tangent bundle of $X$, and not the first Chern class of the cotangent bundle, as  written in \cite[Section~3.1]{Katzarkov_Hodge_theoretic_aspects}. The
 formula \eqref{eq:formalbignabla} corrects this misprint  and agrees with the common usage. In particular  it agrees with
 the formulas used to compute the $\Gamma$-class correction in \cite[Section~3.1]{Katzarkov_Hodge_theoretic_aspects}, with the  classical Dubrovin definition \cite{Dubrovin_Geometry_of_2D} of the quantum connection, and with the standard 
 description of the quantum connection used recently in \cite{Iritani_blowup,PS-ET,Chen-ET,HYZZ_Decomposition}.
\end{remark}

\medskip

By definition now  
$(\forcH,\fornabla)/\forB{X}$ is a formal logarithmic F-bundle for the given logarithmic structure on $\forB{X}$. Furthermore, the properties of the Gromov-Witten classes listed above, imply that $(\forcH,\fornabla)/\forB{X}$ is a 
maximal formal logarithmic F-bundle, i.e.\ $\fornabla$ induces a unital associative Frobenius product $\fp$ on $T_{\forB{X}}(-\log \mathfrak{d})$
together with an Euler vector field $\Eu \in \Gamma(\forB{X}, T_{\forB{X}}(-\log \mathfrak{d}))$ so that under the 
identification $\forcH \cong 
T_{\forB{X}}(-\log \mathfrak{d})$ given by the frame $(q_{i}\partial_{q_{i}}, \partial_{t_{j}})$ the quantum product $\qup$ is identified with the Frobenius product $\fp$, and the Euler vector field $\Eu$ is given by the formula \eqref{eq:Eulog}.

For ease of reference we give the following

\begin{definition} \label{def:formalFomegas} Fix 
$\kay  = \kaybar \subset \bbC$, $k$ - a field of $\mathsf{char} \, k = 0$, $H^{\bullet} = H^{\bullet}_{B}(-,k)$. Let  $X$ be a smooth projective variety, defined over $\kay$, and let $\{\omega_{i}\}_{i=1}^{m}$, and $\{\Delta_{j}\}$ as above. Then
the \strongemph{formal $\mathsf{A}$-model F-bundle} associated
with $(X,\omega_{1},\ldots,\omega_{m})$ is the formal logarithmic F-bundle
$(\forcH,\fornabla)/\forB{X}$, where $\forB{X}$ is the formal
neighborhood of $0$ in the superscheme $\affa{X} = \Spec \,
k\left[\op{Sym}\left(H^{\bullet}(X)^{\vee}\right)\right]$,
$\forcH = \mathcal{O}_{\forB{X}}\otimes H^{\bullet}(X)$, and $\nabla$
is given by
\eqref{eq:formalbignabla}.
\end{definition}

\begin{remark} \label{rem:ample.choices} Traditionally, the quantum product is defined by the  series \eqref{eq:*omegas}  in which the sum is taken  over \emph{all} monomials  $q^{\beta}$ in the monoid algebra $k[\NE(X,\bbZ)]$. In the above construction we restricted our attention only to  monomials corresponding to a \emph{finite simplicial cone} of ample classes, so that we can work with formal power series rather than  general elements in a Novikov ring, and so that we can control the logarithmic structures explicitly. We will come back to the issue of general versus restricted quantum products when we discuss the $\mathsf{A}$-model analytic F-bundles in Section~\ref{sssec:AmodelF}  and the analytic F-bundle version of the blowup formula in Section~\ref{sec:blowup}. 
\end{remark}

\begin{remark} \label{rem:ambiguities}
At a first glance this definition will not provide a very good invariant of a  projective $\kaybar$-variety $X$,  as it seem to depend on several additional choices. However, this dependency on choices is very mild and can be disregarded. Indeed
\begin{enumerate}[wide]
\item[(a)] The first issue is that even when we fix the ample classes $\omega_{i}$  the definition of $\fornabla$ seem to depend on the choice of the $\Delta_{j}$'s that complete the ample $\omega_{i}$'s to a graded basis of $H^{\bullet}(X)$.  This is a mild problem since a different choice $\widetilde{\Delta}_{j}$ of $\Delta_{j}$'s will be related to the original choice by a graded linear automorphism of $H^{\bullet}$ and this linear automorphism will induce an automorphism of $\affa{X}$ and a posteriori of the pair $(\forB{X},\forcH)$. The induced  automorphism of $\affa{X}$ acts trivially on the divisor $\mathfrak{d} \subset \affa{X}$ and thus preserves the logarithmic structure of $\forB{X}$. Also, by construction, the automorphism of $(\forB{X},\forcH)$  pulls back the logarithmic connection $\nabla$ defined in terms of the $\Delta_{j}$ to the logarithmic connection $\widetilde{\nabla}^{\mathbf{f}}$ defined via the  $\widetilde{\Delta}_{j}$. Thus the two $\mathsf{A}$-model F-bundles for 
$(X,\omega_{1}, \ldots, \omega_{m})$ that arise from the choices of $\{\Delta_{j}\}$ and 
$\{\widetilde{\Delta}_{j}\}$ are naturally isomorphic. 
\item[(b)] The second issue with Definition~\ref{def:formalFomegas} is the explicit dependence on the choice of ample classes 
$\omega_{1}, \ldots, \omega_{m}$. One way to get around this is to work with a maximal collection of such $\omega_{i}$, i.e.\ choose a basis of the rational N\'{e}ron-Severi group $\CH^{1,\mathsf{hom}}(X)\otimes \bbQ$ \ that consists of ample classes. Indeed, any two  maximal collections of ample classes will be related by a change of basis automorphism of $H^{\bullet}(X)$. Therefore, the contribution of a monomial in $q$-coordinates for one picture coincides with the contribution of the corresponding monomial in another picture. Therefore, there is no loss of information for going between one basis to another.  
\item[(c)] The independence of the choice of ample  classes can be refined further by noticing that 
the formal $\mathsf{A}$-model F-bundle corresponding to $(X,\omega_{1},\ldots,\omega_{m})$ and the formal $\mathsf{A}$-model F-bundle corresponding to $(X,\omega_{1})$ contain the same information, i.e.\ can be reconstructed from one another, see \cite[Lemma~2.24]{HYZZ_Decomposition}.
\end{enumerate}
\end{remark}

The formal logarithmic $\mathsf{A}$-model F-bundle $(\forcH,\fornabla)/\forB{X}$ we just described, captures the motivic and Gromov-Witten information  about $X$ that we need. However the logarithmic structures and the formal variables make these F-bundles awkward to manipulate. We will circumvent this by repackaging the data  encoded in $(\forcH,\fornabla)/\forB{X}$ into convergent non-archimedean analytic data, i.e.\ into a non-archimedean, analytic F-bundle.

\subsubsection{Non-archimedean \texorpdfstring{$\mathsf{A}$}{A}-model \texorpdfstring{$F$}{F}-bundles} \label{sssec:AmodelF}

We will use the Betti setup from the previous subsection. 
Fix two fields $\kay \subset \bbC$ and $k$, with $\op{char} \, k =0$, and the
Weil cohomology theory $H^{\bullet} = H^{\bullet}_{B}(-,k)$. We will also fix an algebraically closed non-archimedean field $\bbk \supset k$, such that the valuation $\mathfrak{v}  \colon \bbk \to \bbR\cup \{\infty\}$ is trivial on $k$. A convenient choice is $\bbk=\kbar\dbp{\fy^{\bbQ}}$.

\paragraph{The case of varieties over $\kaybar$.}  \label{par:casekaybar}

Let 
$X/\kaybar$ be a smooth projective variety, and let 
$H^{\bullet}(X) = H^{\bullet}_{B}(X(\bbC)^{\op{an}},k)$. Let 
 $\mathsf{N}_{1}(X,\bbZ) = 
\CHhom_{1}(X)$ denote the group of classes of curves in $X$ modulo Betti homological equivalence, and let 
$\NE(X,\bbZ) \subset \mathsf{N}_{1}(X,\bbZ)$ be the \strongemph{the monoid of effective curve classes in $X$}.  

Consider the monoid algebra $k[\NE(X,\bbZ)]$, i.e.\ the commutative unital $k$-algebra, whose underlying 
$k$-vector space has a basis labeled by the elements in 
$\NE(X,\bbZ)$, with a product given by the additive semigroup operation. Concretely, we will write $q^{\beta}$ for the basis vector labeled by $\beta \in k[\NE(X,\bbZ)]$ and will call $q^{\beta}$  the \strongemph{standard monomial} corresponding to $\beta$. Thus we have
\[
k[\NE(X,\bbZ)] \  =  \ \bigoplus_{\beta \in \NE(X,\bbZ)} \ k\cdot \, q^{\beta},
\]
as a $k$-vector space, with a unit $\mathbbold{1} = q^0$, and multiplication given by $q^{\beta_{1}}q^{\beta_{2}} \coloneqq  q^{\beta_{1} + \beta_{2}}$. Write $k\dbb{q} = k\dbb{\NE(X,\bbZ)}$ for  the completion of $k[\NE(X,\bbZ)]$ along its maximal monomial ideal, i.e.\ the ideal generated by $q^\beta$, with $\beta \neq 0 \in\NE(X,\bbZ)$.

\medskip

Let $\{T_i\}_{i=0,\dots, r}$ be a homogeneous basis of the $k$-vector space $H^{\bullet}(X)$, such that  $T_0=\mathbf{1} \in H^0(X)$ is the unit in the cohomology algebra $H^{\bullet}(X)$, and such that the degree two basis vectors spanning $H^2(X)$ are chosen by first chosing separately bases of the subspace $H^2(X)_{\op{alg}} \coloneqq \CH^{1,\mathsf{hom}}(X) \otimes k \ \subset \ H^2(X)$ and of the subspace $H^2(X)_{\op{trans}}$ of all transcendental classes in $H^2(X)$, i.e.\ all  classes orthogonal with respect to the intersection pairing to cycle classes of algebraic curves.  Let $(t_{i})$ be the corresponding system of $k$-linear coordinates on $H^{\bullet}(X)$. 

\medskip
Given $n \geq 0$, $\beta \in \NE(X,\bbZ)$, and classes $\gamma_{1}, \ldots, \gamma_{n} \in H^{\bullet}(X)$, we will write
\[
\left\langle \gamma_{1} \cdots \gamma_{n} \right\rangle_{\beta} \ \coloneqq 
\ \int_{I_{n,\beta}(X)} \, \gamma_{1}\boxtimes \cdots \boxtimes \gamma_{n} \quad \in \ k \, = \,  H^0(X^{\times n}),
\]
for the corresponding correlator. Equivalently, by the K\"{u}nneth and Poincar\'{e} duality properties of $H^{\bullet}$ we have  
\[
\left\langle \gamma_{1} \cdots \gamma_{n} \right\rangle_{\beta} 
\ \coloneqq  
\ \left( I_{n,\beta}(X) \, , \, 
\op{pr}_{1}^{*}\gamma_{1}\otimes \cdots  \otimes \op{pr}_{n}^{*}\gamma_{n} \right),
\]
where we still write $(-,-)$ for the Poincar\'{e} pairing on 
$X^{\times n}$.

\begin{definition} \label{defi-GWPhi}
The $H^{\bullet}$-valued \strongemph{genus $0$ Gromov-Witten potential} of $X$ is the formal power series with $k\dbb{q}$ coefficients in the set of super variables $t_{0}, \ldots, t_{r}$:
\begin{equation} \label{eq:GW_potential}
  \Phi(q;t)  \coloneqq \sum_{n\ge 0,\beta  \in \NE(X,\bbZ)} \frac{q^\beta}{n!} \sum_{i_1,\dots,i_n} \braket{T_{i_1}\cdots T_{i_n}}_\beta t_{i_1}\cdots t_{i_n} \ \in \ k\dbb{q}\dbb{t_{0},\ldots,t_{r}}.
\end{equation}
\end{definition}

\begin{notation} \label{not:novikovX} Here 
$k\dbb{q}\dbb{t_{0},\ldots,t_{r}} \cong k\dbb{\NE(X,\bbZ)}\, \widehat{\otimes}\, \widehat{\mathsf{Sym}}\, H^{\bullet}(X)^{\vee}$ is a coordinated version of the \strongemph{big Novikov ring for $X$} and we will sometimes denote it by $\Nov_{X}$.
\end{notation}

Note that by the unit axiom, the variable $t_0$ appears only in the terms with $\beta=0$ and $n \le 3$. In particular $\Phi$ is polynomial in $t_{0}$. 

Recall that a formal power series in a set of super variables is a usual formal series in even variables with coefficients in the exterior algebra of the odd variables. For the moment we  view $\Phi$ as an $R$-valued function on the formal super scheme/$k$ which is the formal neighborhood of some (closed, even) base point  of the algebraic super scheme 
$\op{Spec} \, k[\{t_{i}\}]$ whose $k$-points give the super vector space $H^{\bullet}(X)$.  
The third derivative $\partial^3\Phi/\partial t_{a}\partial t_{b}\partial t_{c}$ of $\Phi$ in the coordinates $t_{0}, \ldots, t_{r}$
defines a linear map $H^{\bullet}(X)^{\otimes 3} \to \Nov_{X}$ and, by the standard Gromov-Witten lore \cite{Kontsevich_Gromov-Witten_classes},
after contraction with the Poincar\'{e} pairing gives a quantum product 
\[
\qup \colon \ H^{\bullet}(X)^{\otimes 2} \to 
H^{\bullet}(X)\otimes_{k} \Nov_{X},
\]
which is the general (unrestricted by a choice of $\{\omega_{i}\}_{i = 1}^{m}$) version of the  product \eqref{eq:*omegas}. 

\medskip

Consider next the torsion free part $\NS(X,\bbZ)_{\op{tf}} \coloneqq 
\op{im}\left[\CH^{1}(X) \to H^2(X)\right]$ of the N\'{e}ron-Severi group of $X$. It is a free abelian group of finite rank $\rho$. Let
\[
\mathcal{T}_{\bbk}(X) \coloneqq \NS(X,\bbZ)_{\op{tf}}\otimes_{\bbZ} \mathbb{G}_{m,\bbk}
\]
denote the
split affine torus over $\bbk$ with cocharacter lattice $\NS(X,\bbZ)_{\op{tf}}$ and let 
$\mathcal{T}_{\bbk}(X)^{\op{an}}$ denote the associated non-archimedean $\bbk$-analytic torus. We can also consider the $\bbk$-vector space 
$\NS(X,\bbk) \coloneqq \NS(X,\bbZ)_{\op{tf}}\otimes_{\bbZ} \bbk$. By definition, the tangent space to $\mathcal{T}_{\bbk}(X)^{\op{an}}$ at each rigid point can be identified canonically with $\NS(X,\bbk)$.

\medskip

With this notation we now define the following analytic $\bbk$-spaces:

\begin{description}
\item $B_{X,q} \subset \mathcal{T}_{\bbk}(X)^{\op{an}}$ is the preimage of the ample cone in $\NS(X,\bbR)$ under the valuation map  \linebreak $\mathcal{T}_{\bbk}(X)^{\op{an}} \ \to \ \NS(X,\bbR)$.   Note that by construction $B_{X,q}$ is an admissible  open in  $\mathcal{T}_{\bbk}(X)^{\op{an}}$.
\item $\omB^{\even}_{X,t}$  is the product of an analytic affine line with an open  unit polydisk, where the affine line has coordinate $t_0$, and the unit polydisk has coordinates $t_i$ corresponding to the basis vectors $T_{i}$ with $\deg T_{i} \in \{2, 4, 6, \ldots\}$. 
\item $B^{\even}_{X,t}$ is the product of the analytic affine line with coordinate $t_{0}$ and a unit polydisk with coordinates $t_{i}$ corresponding to basis vectors $T_{i}$ with  $\deg T_i \in \{4, 6, \ldots\}$, and to basis vectors $T_{i} \in H^2(X)_{\op{transc}}$ whenever $\deg T_{i} = 2$. 
\item $B^{\odd}_{X}$ is the $\bbk$-super analytic variety corresponding to the purely odd $\bbk$-vector space $H^{\odd}(X)\otimes_{k} \bbk$, i.e.\ the super analytic variety whose underlying even analytic variety is a point and whose algebra of functions is the  exterior algebra in the coordinates $t_i$ for $\deg T_i \in \{1, 3, 5, \dots\}$.
\end{description}

Set 
\begin{equation} \label{eq:non-archimedean_base}
\omB_{X} \coloneqq B_{X,q} \times \omB^{\even}_{X,t}\times B^{\odd}_{X} \quad \text{and} \quad 
B_{X} \coloneqq B_{X,q} \times B^{\even}_{X,t}\times B^{\odd}_{X}.
\end{equation}
By definition $\omB_{X}$ is open in the $\bbk$-super analytic variety $\mathcal{T}_{\bbk}(X)^{\op{an}}\times (H^{\bullet}(X)\otimes_{k} \bbk)^{\op{an}}$ and $B_{X}$ is open in the $\bbk$-super analytic variety $\mathcal{T}_{\bbk}(X)^{\op{an}}\times ((H^2(X)_{\op{transc}}\oplus H^{*\neq 2}(X))\otimes_{k} \bbk)^{\op{an}}$. Note that the vector space inclusion map 
$H^2(X)_{\op{transc}}\otimes_{k} \bbk \ \hookrightarrow  \ H^2(X)\otimes_{k} \bbk$ induces a closed analytic embedding  $B_{X} \to \omB_{X}$ of non-archimedean super analytic varieties.

With this notation, we now have the following important observation

\begin{lemma} \label{lem:GW-potential_analytic}
The genus zero Gromov-Witten potential $\Phi(q,t)$ defines a $\bbk$-valued  analytic function on $\omB_{X}$.
\end{lemma}
\begin{proof} Recall, that for any smooth projective $\kaybar$-variety $X$  we can find a finite collection of ample line bundles on $X$, so that their first Chern classes form a basis of the  $\bbQ$-N\'{e}ron-Severi space $\NS(X,\bbQ) \coloneqq \CH^{1,\mathsf{hom}}(X)\otimes \bbQ \subset H^2(X)$.  
Let $L_{1}, \ldots, L_{m} \in \op{Pic}(X)$ be such ample line bundles, and write $\omega_{i} \coloneqq \mathsf{cl}(c_{1}(L_{i}))$. 

Consider the open simplicial cone 
\[
\sigma \coloneqq \left\{\left. \, \sum_{i = 1}^{m} a_{i}\omega_{i} \ \right| \ a_{i} \in \bbR_{> 0}\, \right\} \quad \subset \ \bbR\omega_{1}\oplus \cdots \oplus \bbR\omega_{m}.
\]
Let $B_{\sigma,q} \subset \NS(X,\bbk)$ be the preimage of $\sigma$ by the valuation map, and let  
$\omB_{\sigma} \coloneqq B_{\sigma,q} \times \omB_{X,t}^{\even} \times B^{\odd}_{X}$. Note that by definition $\omB_{\sigma} \subset \omB_{X}$
is an open analytic subvariety.

Now note, that the formula for  $\Phi$ defines a $\bbk$-valued analytic function on $\omB_{\sigma}$. Indeed, let  
$q_{1}, \ldots, q_{m}$ be the coordinates on $\mathcal{T}_{\bbk}(X)$  that are dual to the basis 
$\omega_{1}, \ldots, \omega_{m}$. Using $q_{j}$ and $t_{i}$ as coordinates on the super analytic space $\mathcal{T}_{\bbk}(X)^{\op{an}}\times (H^{\bullet}(X)\otimes_{k} \bbk)^{\op{an}}$ we see that the restriction of the formal germ 
\eqref{eq:GW_potential} to the open $\omB_{\sigma}$ is given by the power series
  \[ 
\begin{aligned}
\Phi|_{\omB_{\sigma}} = \sum_{n\ge 0,\beta} \frac{1}{n!} q_1^{(\beta,\omega_1)}\cdots q_m^{(\beta,\omega_m)} \sum_{i_1,\dots,i_n} & \braket{T_{i_1}\cdots T_{i_n}}_\beta t_{i_1}\cdots t_{i_n} \\ 
& \qquad \in k\dbb{\{q_j\},\{t_i\}}  \, \subset 
\bbk\dbb{\{q_j\},\{t_i\}}.
\end{aligned}
\]
This series is polynomial in $t_0$, and hence an analytic $\bbk$-valued function on  the non-archimedean analytic super domain $\omB_{\sigma}$. Since the union of all such $\sigma$ covers the ample cone of $X$, this implies that $\Phi$ can be viewed as an analytic function on $\omB_{X}$  as claimed.
\end{proof}

Let as before $\bbD$ denote the germ at $0$ of an analytic disk with coordinate $u$. Write $\cO_{\omB_{X}\times \bbD}$ for the sheaf of analytic $\bbk$-valued functions on the non-archimedean  analytic supermanifold $\omB_{X}\times \bbD$ and let $\omH = H^{\bullet}(X)\otimes_{k} \cO_{\omB_{X}\times \bbD}$ be the trivial analytic vector bundle on $\omB_{X}\times \bbD$ with fiber $H^{\bullet}(X)\otimes_{k} \bbk$. Note that our chosen basis of $H^{\bullet}(X)$ gives a global frame of $\cH$, i.e.\ we have a canonical identification 
$\omH = \oplus_{i = 1}^{r} \cO_{\omB_{X}\times\bbD}\cdot T_{i}$.
The inverse $(g^{ij})$ of the Gram matrix of the Poincar\'{e} pairing gives a section $\mathbf{PD}^{-1} \in H^0(\omB_{X}\times \bbD,\omH\otimes \omH)$. The section is constant in the frame $\{ T_{i}\}$ and therefore analytic.

\medskip

The third derivatives of $\Phi$ with respect to the coordinates 
$\{t_{i}\}$ define  a map of analytic super vector bundles
\[
\omH^{\otimes 3} \to \cO, \qquad T_{a}\otimes T_{b}\otimes T_{c} \mapsto \frac{\partial^3\Phi}{\partial t_{a} \partial t_{b} \partial t_{c}}
\]
and the contraction of this map with $\mathbf{PD}^{-1} \in \omH^{\otimes 2}$ gives an analytic quantum product
\[
\xymatrix@1{\qup \colon \hspace{-1.5pc} &  \omH\otimes \omH \ar[r] & \omH},
\]
which is independent of $u$ and tautologically satisfies the properties {\bfseries (i)-(v)}.

\medskip

Furthermore, the quantum product  $\qup$ gives rise to a \strongemph{non-archimedean analytic quantum connection}   $\bnabla \colon \omH \to \omH \otimes \Omega^1_{\omB_{X}\times \bbD}[u,u^{-1}]$
which over an open chart $(\omB_{\sigma}, (\{q_{j}\},\{t_{i}\})$ 
is defined by
\begin{equation} \label{eq:anbignabla}
\left| \ 
\begin{aligned}
\bnabla_{\partial_{u}} & = \partial_{u} - u^{-2} \left(\Eu \, \qup\, (-)\right) + u^{-1} \frac{\Deg - \dim X\cdot \op{id}}{2}, \\
\bnabla_{\partial_{q_{j}}} & = \partial_{q_{i}} + u^{-1}q_{j}^{-1} \left( \omega_{j} \, \qup \, (-)\right), \\
\bnabla_{\partial_{t_{i}}} & = \partial_{t_{i}} + u^{-1} \left( T_{i} \, \qup \, (-)\right),
\end{aligned}
\right.
\end{equation}
Here again  $\Deg \colon \omH \to \omH$ is the degree operator $\Deg = \oplus_{a = 0}^{2\dim X} a\cdot \op{id}_{H^{a}(X)}$ and $\Eu$ denotes the analytic Euler vector field, defined for every point 
$\gamma \in \omB_{\sigma} \subset (H^{\bullet}(X)\otimes_{k}  \bbk)^{\op{an}}$ by 
\begin{equation} \label{eq:Euan}
\Eu_{\gamma} = c_{1}(T_{X}) + \frac{\Deg - 2\cdot \op{id}}{2}(\gamma)
\in  \ T_{\omB_{\sigma},\gamma} \ \subset \ \omH_{\gamma}.
\end{equation}

By construction $(\omH,\bnabla)/\omB_{X}$ is an over maximal F-bundle over $\omB_{X}$ and 
$B_{X} \subset \omB_{X}$ is a closed analytic subvariety which meets transversally every leaf through a rigid point of the redundancy foliation of $(\omH,\bnabla)/\omB_{X}$. Thus  the restriction of $(\omH,\bnabla)|_{B_{X}\times \bbD}$ is the maximal F-bundle $(\cH,\nabla)/B_{X}$ over $B_{X}$.

\begin{definition} \label{def:nonarchF}
The F-bundle $(\omH,\nabla)/\omB_{X}$ will be called the \emph{\bfseries
non-archimedean overmaximal $\mathsf{A}$-model F-bundle associated
to $X$}, $\kaybar \subset \bbC$, $\bbk \supset k$, and $H^{\bullet}$.  The F-bundle $(\cH,\nabla)/B_{X}$
will be called the \strongemph{non-archimedean maximal
$\mathsf{A}$-model F-bundle associated to $X$}, $\kaybar \subset \bbC$, 
$\bbk \supset k$, and $H^{\bullet}$.
\end{definition}

\begin{remark} \label{rem:formal_in_u}
The analyticity of $\Phi$ implies that F-bundles $(\omH,\bnabla)/\omB_{X}$ and $(\cH,\nabla)/B_{X}$ are non-archimedean analytic in the sense of Definition~\ref{def:F-bundle}.
\end{remark}

\paragraph{The case of varieties over $\kay$.}  \label{par:casekay}
Fix two fields $k$ and $\kay$, both of characteristic zero.

\subparagraph{Preliminaries on motives.} \label{subpar:motivess}
Let $\CA{\kay}$ be the category of pure Andr\'{e} motives (aka
motivated motives)
\cite{Andre_Pour_une_theorie_inconditionnelle_des_motifs} of smooth
projective $\kay$-varieties. This is a $\bbQ$-linear, rigid,
semisimple tensor category which is defined unconditionally.  For a
$\kay$-variety $\mathfrak{X}$ we will write 
$h(\mathfrak{X})$ for the associated Andr\'{e} motive.

Fix an embedding $\kay \subset  \bbC$. Then we
have a natural algebraic closure $\kay \subset \kaybar \subset \bbC$
of $\kay$ inside $\bbC$, and a base change functor
\[
(-)\times_{\kay} \kaybar \, \colon \, \CA{\kay} \, \longrightarrow \, \CA{\kaybar},
\quad h(\mathfrak{X}) \, \mapsto \, h(\mathfrak{X}\times_{\kay} \kaybar),
\]
which by definition is a monoidal functor of $\bbQ$-linear tensor
categories.

Taking Betti cohomology of the associated complex analytic space gives
a Weil cohomology theory on $\kaybar$-varieties \cite{stacksproject}
\begin{equation} \label{eq:kaybarB}
\xymatrix@1@M+0.5pc{
  H^{\bullet}_{B}(-) \,\colon \hspace{-2.4pc} &
  \CA{\kaybar}  \ar[r] &  \mathsf{Vect}_{\bbQ}}, \quad
\xymatrix@1@M+0.5pc{h(X) \ar@{|->}[r] & 
H^{\bullet}_{B}((X\times_{\kaybar}\bbC)^{\op{an}},\bbQ)},
\end{equation}
and by precomposing with base change gives also a compatible Weil
cohomology theory on $\kay$-varieties
\begin{equation} \label{eq:kayB}
  \xymatrix@1@M+0.5pc{
    H^{\bullet}_{B}(-) \,\colon \hspace{-2.4pc} &
    \CA{\kay}  \ar[r] & \mathsf{Vect}_{\bbQ}}, \quad
  \xymatrix@R-3pc@M+0.5pc@C+0.2pc{
    h(\mathfrak{X}) \ar@{|->}[r] &  h(X) \ar@{|->}[r] &
    H^{\bullet}_{B}((X\times_{\kaybar}\bbC)^{\op{an}},\bbQ) \\
    & || & || \\
    & h(\mathfrak{X}\times_{\kay} \kaybar) &
    H^{\bullet}_{B}((\mathfrak{X}\times_{\kay}\bbC)^{\op{an}},\bbQ).
    }
\end{equation}
Abusing notation we denote both of these tensor functors by
$H^{\bullet}_{B}(-)$ since it will be clear from the context which
type of varieties serve as input of Betti cohomology. 

These data are conveniently encoded in  certain proalgebraic groups.  

\medskip

\noindent
{\bfseries (i)} \ The automorphism groups $\GA{\kaybar}$ and
$\GA{\kay}$ of the tensor functors \eqref{eq:kaybarB} and
\eqref{eq:kayB} are the \emph{\bfseries motivic Galois groups of
  Andr\'{e} motives} of smooth projective $\kaybar$-varieties and
$\kay$-varieties respectively.  $\GA{\kaybar}$ and $\GA{\kay}$ are
proreductive algebraic group schemes over $\bbQ$. Tannaka duality
identifies $\CA{\kaybar}$ and $\CA{\kay}$ with the categories
$\mathsf{Rep}_{\bbQ}(\GA{\kaybar})$ and
$\mathsf{Rep}_{\bbQ}(\GA{\kay})$ of representations of $\GA{\kaybar}$
and $\GA{\kay}$ on finite dimensional $\bbQ$-vector spaces.

\medskip

\noindent
{\bfseries (ii)} \ The base change functor $(-)\times_{\kay} \kaybar \colon
\CA{\kay} \to \CA{\kaybar}$ induces a normal embedding
$\GA{\kaybar} \to \GA{\kay}$.

\medskip

\noindent
{\bfseries (iii)} \ The quotient group scheme $\GA{\kay}/\GA{\kaybar}$ is the automorphism group of the restriction
of the functor $H^{\bullet}_{B}(-)$ to the full tensor subcategory
$\mathsf{Art}^{\kay} \subset \CA{\kay}$ of Artin motives, i.e.\ motives
of zero dimensional smooth projective varieties over $\kay$. In other
words, $\GA{\kay}/\GA{\kaybar}$ is  the
automorphism group of the induced fiber functor
\[
H^{\bullet}_{B}(-)|_{\mathsf{Art}^{\kay}}  = H^{0}_{B}(-) \, \colon \,
\mathsf{Art}^{\kay} \, \longrightarrow \, \mathsf{Vect}_{\bbQ}.
\]
This group is naturally identified with the Galois group
$\Gal(\kaybar/\kay)$, where $\Gal(\kaybar/\kay)$ is viewed as a
profinite group scheme over $\bbQ$ for the trivial scheme structure.

In summary, we have a short exact sequence
\[
\xymatrix@1@M+0.5pc@C+1pc{ 1 \ar[r] & \GA{\kaybar} \ar[r] & \GA{\kay}
  \ar[r] & \Gal(\kaybar/\kay) \ar[r] & 1 }
\]
of proreductive group schemes\footnote{It is expected that $\GA{\kaybar}$ is the connected component of the identity of $\GA{\kay}$ or equivalently that $\Gal(\kaybar/\kay) = \pi_{0}(\GA{\kay})$ is the component group of $\GA{\kay}$ but we will not need this fact.} over $\bbQ$.

\medskip

\noindent
{\bfseries (iv)} \ 
The group $\GA{\kay}$ has a natural Lefschetz character $\GA{\kay} \to
\mathbb{G}_{m}$ given by its action on
$H^{2}_{B}(\bbP^{1}_{\kay},\bbQ)$, and we denote the kernel of this
character by $\motM^{\kay}$. The restriction of the Lefschetz
character to $\GA{\kaybar}$ is given by its action on
$H^{2}_{B}(\bbP^{1}_{\kaybar},\bbQ)$ and we denote its kernel by
$\motM^{\kaybar}$. 

\medskip

Again $\motM^{\kay}$ and $\motM^{\kaybar}$ are proreductive algebraic
groups over $\bbQ$ which we view as the \strongemph{motivic Galois
groups of the category of {\sf nc} \ Andr\'{e} motives}.  They can be
interpreted as the Galois groups of the semisimple Tannakian
$\bbQ$-linear categories $\CA{\kay}/\mathsf{Tate}$ and
$\CA{\kaybar}/\mathsf{Tate}$ obtained as the orbit categories of
$\CA{\kay}$ and $\CA{\kaybar}$ for the action of the Tate twist.

Thus
$\CA{\kay}/\mathsf{Tate} = \mathsf{Rep}_{\bbQ}(\motM^{\kay})$ and
$\CA{\kaybar}/\mathsf{Tate} = \mathsf{Rep}_{\bbQ}(\motM^{\kaybar})$
and the quotient functors
\[
\xymatrix@R-2.5pc@M+0.2pc@C+0.5pc{
\CA{\kay} \ar[r] & 
\CA{\kay}/\mathsf{Tate} \\
|| & || \\
\mathsf{Rep}_{\bbQ}(\GA{\kay}) & \mathsf{Rep}_{\bbQ}(\motM^{\kay})
}
\quad \text{and} \quad
\xymatrix@R-2.5pc@M+0.2pc@C+0.5pc{
\CA{\kaybar} \ar[r] & 
\CA{\kaybar}/\mathsf{Tate} \\
|| & || \\
\mathsf{Rep}_{\bbQ}(\GA{\kaybar}) & \mathsf{Rep}_{\bbQ}(\motM^{\kaybar})
}
\]
correspond to the pullback of representations via the inclusions
$\motM^{\kay} \subset \GA{\kay}$ and $\motM^{\kaybar} \subset
\GA{\kaybar}$.

\medskip

\noindent
{\bfseries (v)} \ The induced functor $\mathsf{Art}^{\kay}  \subset \CA{\kay}
\to \CA{\kay}/\mathsf{Tate}$ from (Artin motives) to (Andr\'{e}
motives modulo Tate twists), is still fully faithful and so the
composed homomorphism of group schemes over $\bbQ$
\[
\motM^{\kay} \subset \GA{\kay} \to \Gal(\kaybar/\kay)
\]
is surjective. Thus $\Gal(\kaybar/\kay)$ can also be
identified with the quotient group scheme
$\motM^{\kay}/\motM^{\kaybar}$,  i.e.\ we also  
have a short exact sequence
\[
\xymatrix@1@M+0.5pc@C+1pc{ 1 \ar[r] & \motM^{\kaybar} \ar[r] & \motM^{\kay}
  \ar[r] & \Gal(\kaybar/\kay) \ar[r] & 1}.
\]

Suppose that $\mathfrak{X}$ is a smooth projective $\kay$-variety and
let $X = \mathfrak{X}\times_{\kay} \kaybar$ be its base change to
$\kaybar$. By definition the Betti cohomologies of $\mathfrak{X}$ and
$X$ have the same underlying rational vector space. To distinguish
notationally between the two Betti cohomology theories, our convention
will be to writ e $H^{\bullet}_{B}(X)$ for the Betti cohomology of
$X$, viewed as a $\motM^{\kaybar}$-module, and will write
$H^{\bullet}_{B}(\mathfrak{X})$ for the Betti cohomology of
$\mathfrak{X}$, viewed as a $\motM^{\kay}$-module. 

\medskip

The Galois group $\Gal(\kaybar/\kay)$ acts on
$X$ by automorphisms, and so by pullback acts on the maximal torsion
free quotient  
\[
\mathsf{NS}(X,\bbZ)_{\op{tf}} \, \coloneqq \, \mathsf{NS}(X,\bbZ)/(\text{torsion}) =
\mathsf{im}\left[\mathsf{CH}_{1}(X) \to
  H^{2}_{B}((X\times_{\kaybar}\bbC)^{\op{an}},\bbQ)\right]
\]
of the N\'{e}ron-Severi group of $X$. But
$\mathsf{NS}(X,\bbZ)_{\op{tf}}$ is a free abelian group of
finite rank $\rho$ and hence we get a homomorphism
\[
\Gal(\kaybar/\kay) \to
\GL\left(\mathsf{NS}(X,\bbZ)_{\op{tf}}\right) \, \cong
\, \GL_{\rho}(\bbZ).
\]
The image of this homomorphism is a finite subgroup of
$\GL\left(\mathsf{NS}(X,\bbZ)_{\op{tf}}\right) \cong
\GL_{\rho}(\bbZ)$ which we will denote by $\Gamma_{\mathfrak{X}}$.

\medskip

Consider next the rational vector space 
\[
H^{2}_{B}(X) =
H^{2}_{B}((X\times_{\kaybar} \bbC)^{\op{an}},\bbQ).
\]
It decomposes into a direct sum  of natural subspaces 
\begin{equation} \label{eq:decompH2X}
H^{2}_{B}(X) =  H^{2}_{B}(X)_{\op{alg}}\oplus H^{2}_{B}(X)_{\op{transc}},
\end{equation}
where
\begin{itemize}
\item $H^{2}_{B}(X)_{\op{alg}} = \mathsf{NS}(X,\bbQ)$ is the N\'{e}ron-Severi
  subspace  spanned by
  $\mathsf{NS}(X,\bbZ)_{\op{tf}}$,
\item $H^{2}_{B}(X)_{\op{trasc}} \subset H^{2}_{B}(X)$ is the subspace
  of classes orthogonal to algebraic curves in $X$.
\end{itemize} 
If we view $H^{2}_{B}(X)$ as a representation of $\motM^{\kaybar}$,
then $H^{2}_{B}(X)_{\op{alg}} = H^{2}_{B}(X)^{\motM^{\kaybar}}$ is
just the space of invariants, and $H^{2}_{B}(X)_{\op{transc}}$ is just
the complementary $\motM^{\kaybar}$ subrepresentation.

\medskip

More importantly, the $\motM^{\kay}$-action on
$H^{2}_{B}(\mathfrak{X}) = H^{2}_{B}(X)$ also respects the
decomposition \eqref {eq:decompH2X}.

\begin{proposition}
  {\bfseries (a)} \ 
  The natural action of $\motM^{\kay}$ on $H^{2}_{B}(\mathfrak{X})$
respects the
decomposition
\begin{equation} \label{eq:Mot-decompose}
H^{2}_{B}(\mathfrak{X}) = H^{2}_{B}(X)_{\op{alg}} \oplus
H^{2}_{B}(X)_{\op{transc}}
\end{equation}

\smallskip

\noindent
{\bfseries (b)} \
The induced action of $\motM^{\kay}$ on $H^{2}_{B}(X)_{\op{alg}}$
factors through the homomorphism
\[
\motM^{\kay} \to \Gal(\kaybar/\kay) \to \Gamma_{\mathfrak{X}}
\subset \GL(\mathsf{NS}(X,\bbZ)_{\op{tf}} \subset \GL(H^{2}_{B}(X)_{\op{alg}}).
\]
\end{proposition}
\begin{proof}  The Galois group $\Gal(\kaybar/\kay)$ acts on the $\kaybar$-points of $X$ and thus acts on the 
cohomology space $H^{2}_{B}(X)$. 
If $L$ is a line bundle on $X$, then the image  $\ell$ of $c_{1}(L) \in \CH^{1}(X)$ in $H^{2}_{B}(X)$ is a motivated cycle, and so is fixed by $\motM^{\kaybar}$. Therefore an element $g \in \motM^{\kay}$ acts on $\ell  \in H^{2}_{B}(X)$ through its image $\underline{g} \in \Gal(\kaybar/\kay)$, i.e.\ $g(\ell) = \underline{g}(\ell)$, where on the right hand side we use the natural action of  $\Gal(\kaybar/\kay)$ on $H^{2}_{B}(X)$. 
Since the action of any element of $\Gal(\kaybar/\kay)$  on $X$ sends a divisor to a divisor, it follows that 
the action of $\Gal(\kaybar/\kay)$  on $H^{2}_{B}(X)$ will preserve the subspace $\NS(X,\bbQ) = H^{2}_{B}(X)_{\op{alg}}$. Hence the action of $\motM^{\kay}$ on $H^{2}_{B}(X)$ will preserve $H^{2}_{B}(X)_{\op{alg}}$. This proves that $H^{2}_{B}(X)_{\op{alg}}$ is stable under $\motM^{\kay}$ and, since $\motM^{\kaybar}$ acts trivially on $H^{2}_{B}(X)_{\op{alg}}$, this proves part  {\bfseries (b)} of the Proposition. 

To finish the proof of part {\bfseries (a)} we need to check that $H^{2}_{B}(X)_{\op{transc}}$ is also stable under $\motM^{\kay}$.  Let $\cO_{\mathfrak{X}}(1)$ be a very ample line bundle on the
$\kay$-variety $\mathfrak{X}$ and let 
$\hpl \in
H^{2}_{B}(\mathfrak{X}) = H^{2}_{B}(X)$ be the image of $c_{1}(\cO_{\mathfrak{X}}(1))$ in Betti
cohomology.
The choice of $\cO_{\mathfrak{X}}(1)$  gives rise to a natural $\bbQ$-valued symmetric bilinear
form
\[
\langle \, - \, ,\,  - \, \rangle \colon H^{2}_{B}(X)\otimes H^{2}_{B}(X) \
\longrightarrow \ \bbQ, \qquad 
\langle \alpha,\beta \rangle = \int_{X} \alpha \smile \beta \smile \hpl^{d-2},
\]
where $d$ is the dimension of $X$. 
Since $\hpl$ is the class of line bundle defined over $\kay$, this form is
invariant under the action of $\motM^{\kay}$. But for any $\ell \in H^{2}_{B}(X)_{\op{alg}}$
$\ell\cdot \hpl^{d-2}$ is the class of an algebraic curve in $X$, and so $H^{2}_{B}(X)_{\op{transc}} \subset H^{2}_{B}(X)$ is contained in the $\langle \, - \, ,\,  - \, \rangle$-orthogonal complement of $H^{2}_{B}(X)_{\op{alg}}^{\perp}$ of  $H^{2}_{B}(X)_{\op{alg}} \subset H^{2}_{B}(X)$.

On the other hand, if we
write $\mathsf{Prim}^{2}_{\op{alg}}$ for the $\langle\,-\,,\,
-\,\rangle$-orthogonal complement of $\hpl$ inside
$H^{2}_{B}(X)_{\op{alg}}$, then we know that $\langle \hpl, \hpl \rangle > 0$ by
the ampleness of $\cO_{\mathfrak{X}}(1)$, and we know that the restriction of
$\langle\,-\,,\, -\,\rangle$ to $\mathsf{Prim}^{2}_{\op{alg}}$ is
negative definite by the Hodge index theorem for $X\times_{\kaybar}\bbC$.
In particular, the $\bbQ$-valued form $\langle\,-\,,\, -\,\rangle$  is non-degenerate on the subspace $H^{2}_{B}(X)_{\op{alg}} \subset H^{2}_{B}(X)$, and so 
$H^{2}_{B}(X) = H^{2}_{B}(X)_{\op{alg}}\oplus H^{2}_{B}(X)_{\op{alg}}^{\perp}$. Thus $H^{2}_{B}(X)_{\op{transc}} = H^{2}_{B}(X)_{\op{alg}}^{\perp}$, and since the form  $\langle \, - \, ,\,  - \, \rangle$ is $\motM^{\kay}$-invariant, we conclude that $H^{2}_{B}(X)_{\op{transc}}$ is $\motM^{\kay}$-stable. This concludes the proof of the Proposition. 
\end{proof}

\medskip

\begin{notation}
    For disambiguation we will write $H^{2}_{B}(\mathfrak{X})_{\kaybar-\op{alg}}$ and
$H^{2}_{B}(\mathfrak{X})_{\kaybar-\op{transc}}$ for  the rational vector spaces
$H^{2}_{B}(X)_{\op{alg}}$ and $H^{2}_{B}(X)_{\op{transc}}$ together 
    with the $\motM^{\kay}$-actions given by part {\bfseries (a)} of
    the previous proposition. Note that these
$\motM^{\kay}$-representations are different from the representations
$H^{2}_{B}(\mathfrak{X})_{\op{alg}}$ and
$H^{2}_{B}(\mathfrak{X})_{\op{transc}}$ consisting of $\kay$-algebraic
and $\kay$-transcendental classes respectively.
\end{notation}

\subparagraph{The quantum cohomology algebra for varieties over $\kay$} \label{subpar:qcoh}

Let again $\mathfrak{X}$ be a smooth projective variety over $\kay$,
and let $X = \mathfrak{X}\times_{\kay} \kaybar$ be its base change to
$\mathsf{Spec} \, \kaybar$. Consider the rational vector space
$H^{\bullet}_{B}(X)$ together with its structure of a
$\motM^{\kaybar}$-module.

Let
\[
\mathsf{NE}(X,\bbZ) \subset \mathsf{CH}^{\mathsf{hom}}_{1}(X)  =
\op{im}\left[ \mathsf{CH}_{1}(X) \to
  H_{2}((X\times_{\kaybar} \bbC)^{\op{an}},\bbQ) \right]
\]
denote the monoid of numerically effective curve classes in $X$, and 
let, as before,  $\mathsf{Nov}_{X}$ denote the big
Novikov ring for $X$, i.e.\ the ring
\[
\mathsf{Nov}_{X} \coloneqq \bbQ\dbb{\mathsf{NE}(X,\bbZ)}\ \widehat{\otimes}_{\bbQ} \
\widehat{\mathsf{Sym}}\, H^{\bullet}_{B}(X)^{\vee},
\]
where $\bbQ\dbb{\mathsf{NE}(X,\bbZ)}$ is the completion of the monoid
algebra $\bbQ[\mathsf{NE}(X,\bbZ)]$ in its maximal ideal $\langle
q^{\beta} | \beta \in \mathsf{NE}(X,\bbZ), \beta \neq 0 \rangle$, and
similarly $\widehat{\mathsf{Sym}}\, H^{\bullet}_{B}(X)^{\vee}$ is the
completion of the superalgebra of polynomials on $H^{\bullet}_{B}(X)$
along the maximal ideal of the origin.

The $\motM^{\kaybar}$-action on $H^{\bullet}_{B}(X)$
induces a canonical action on $\mathsf{Nov}_{X}$, preserving the
$\bbQ$-algebra structure of $\mathsf{Nov}_{X}$. Here $\motM^{\kaybar}$
acts trivially on the factor $\bbQ\dbb{\mathsf{NE}(X,\bbZ)}$ (since it
acts trivially on the classes of effective curves defined over
$\kaybar$) and tautologically (i.e.\ by pullback) on the power series
algebra $\widehat{\mathsf{Sym}}\, H^{\bullet}_{B}(X)^{\vee}$.

As explained before, the $H^{\bullet}_{B}(-)$-valued
Gromov-Witten invariants of $X$
equip the rational vector space $H^{\bullet}_{B}(X)$ with an associative,
super commutative quantum
product
\[
\qup \colon H^{\bullet}_{B}(X)\otimes H^{\bullet}_{B}(X)
\ \longrightarrow
\ 
H^{\bullet}_{B}(X)\otimes \mathsf{Nov}_{X}.
\]
Furthermore, since the Gromov-Witten classes are cycle classes of
$\kaybar$-rational
algebraic cycles in Cartesian powers of $X$, it follows that the
quantum product $\qup$ is $\motM^{\kaybar}$-equivariant for the
natural actions of $\motM^{\kaybar}$ on $H^{\bullet}_{B}(X)$ and
$\mathsf{Nov}_{X}$. 

Since $X$ arises as a base change of a variety $\mathfrak{X}$ over
$\kay$, the rational vector space underlying $H^{\bullet}_{B}(X)$ is
the same as $H^{\bullet}_{B}(\mathfrak{X})$ and is therefore also
equipped with a natural $\motM^{\kay}$-action. Again this induces a
tautological action (pulling back formal germs of functions at the
origin) of $\motM^{\kay}$ on $\widehat{\mathsf{Sym}}\,
H^{\bullet}_{B}(X)^{\vee}$. As a $\motM^{\kay}$-module this completed
algebra is simply the completion $\widehat{\mathsf{Sym}}\,
H^{\bullet}_{B}(\mathfrak{X})^{\vee}$. Also $\motM^{\kay}$ acts
naturally on $\mathsf{NE}(X,\bbZ)$ and the action factors through the
homomorphism $\motM^{\kay} \twoheadrightarrow \Gamma_{\mathfrak{X}}$,
where the action of $\Gamma_{\mathfrak{X}}$ on $\mathsf{NE}(X,\bbZ)$
is dual to the action of $\Gamma_{\mathfrak{X}}$ on
$\mathsf{NS}(X,\bbZ)_{\op{tf}}$. This, in turn, induces an action on
the completed monoid algebra $\bbQ\dbb{\mathsf{NE}(X,\bbZ)}$ and hence
an action on the Novikov ring. We will write
$\mathsf{Nov}_{\mathfrak{X}}$ for this $\motM^{\kay}$-equivariant
algebra. Note that again $\mathsf{Nov}_{\mathfrak{X}}$ and
$\mathsf{Nov}_{X}$ are the same as $\bbQ$-algebras (or seen as
$\bbQ$-algebras equipped with a $\motM^{\kaybar}$-action), but that the
action of $\motM^{\kay}$ on $\mathsf{Nov}_{\mathfrak{X}}$ depends on
the $\kay$-model $\mathfrak{X}$.

\begin{proposition}The quantum product viewed as map
\begin{equation} \label{eq:qupkay}
\qup \colon H^{\bullet}_{B}(\mathfrak{X})\otimes
H^{\bullet}_{B}(\mathfrak{X}) \ \longrightarrow
\ H^{\bullet}_{B}(\mathfrak{X})\otimes \mathsf{Nov}_{\mathfrak{X}}
\end{equation}
is $\motM^{\kay}$-equivariant.
\end{proposition}
\begin{proof} \ The Galois group $\Gal(\kaybar/\kay)$ acts on
$\mathsf{NE}(X,\bbZ)$ through its finite quotient $\Gal(\kaybar/\kay)
\twoheadrightarrow \Gamma_{\mathfrak{X}}$. Therefore the colimit of
finite unions of $\Gamma_{\mathfrak{X}}$-orbits in
$\mathsf{NE}(X,\bbZ)$ will be an ind-finite ind-affine smooth semigroup
$\kay$-scheme $\bbNE(X,\bbZ)$ whose set of $\kaybar$-points is the
semigroup of effective curve classes in $X$.

For a $\kay$-variety $\mathfrak{X}$ let us write $h(\mathfrak{X}) \,
\in \, \mathsf{ob}\, \CA{\kay}$ for the corresponding Andr\'{e}
motive and similarly $h(X) \in \CA{\kaybar}$ for the Andr\'{e} motive of the base changed variety. By taking filtered colimits, we can then define
$h(\bbNE(X,\bbZ))$ as a commutative algebra in the ind completed
tensor category $\mathsf{Ind}\, \CA{\kay}$ of ind objects in
$\CA{\kay}$.

For any $\Gamma_{\mathfrak{X}}$-orbit $\mathfrak{o} \subset
\mathsf{NE}(X,\bbZ)$ of effective curve classes on $X$, and any $n$, 
the collection of Gromov-Witten classes
$\left(I_{n,\beta}(X)\right)_{\beta \in \mathfrak{o}}$ can be viewed (see Remark~\ref{rem:nonclosedL}) as
a motivated cycle on the $\kay$-scheme
$\bbNE(X,\bbZ)\times \mathfrak{X}^{\times n}$, i.e.\ as a map from the trivial motive
\[
I_{n,\mathfrak{o}} \ \colon \ \triv \
\longrightarrow \ h\left(\bbNE(X,\bbZ)\times \mathfrak{X}^{\times n}\right) \, = \,
h\left(\bbNE(X,\bbZ)\right)\, \widehat{\otimes}  \, h(\mathfrak{X})^{\otimes n}
\]
in the category $\mathsf{Ind}\,\CA{\kay}$. In particular, for every
$n = m+3$, after contracting with the
Poincar\'{e} pairing on $h(\mathfrak{X})$, we get
a map in $\mathsf{Ind}\, \CA{\kay}$:
\[
\qup_{m,\mathfrak{o}} \ \colon \ \triv
\longrightarrow \  h\left(\bbNE(X,\bbZ)\right) \, \widehat{\otimes} \,
h(\mathfrak{X})^{\vee \otimes m} \, \otimes \, h(\mathfrak{X})^{\vee \otimes 2}\, \otimes\,
h(X)
\]
which by construction is invariant under the action of
$\mathsf{S}_{m} \times \mathsf{S}_{2}$. This is the $\mathfrak{o}$-contribution
to the $m$-th homogeneous piece of
the quantum product, realized as a map of Andr\'{e} motives. Summing up
over all $m$ and $\mathfrak{o}$ we get the quantum product map in
$\mathsf{Ind}\, \CA{\kay}$:
  \[
  \qup \coloneqq \sum_{m\geq 0}\, \sum_{\mathfrak{o} \in
    \mathsf{NE}(X,\bbZ)/\Gamma_{\mathfrak{X}}} \ \frac{1}{m!} \,
  \qup_{m,\mathfrak{o}} \ \colon
  \ h(\mathfrak{X})^{\otimes 2} \ \longrightarrow \ h(\mathfrak{X})\otimes
  h\left(\bbNE(X,\bbZ)\right) \, \widehat{\otimes} \,
  \widehat{\mathsf{Sym}} \, h(\mathfrak{X})^{\vee}. 
  \]
Now applying the fiber functor $H^{\bullet}_{B}(-) \colon \CA{\kay} \ \to
\ \mathsf{Vect}_{\bbQ}$ and passing to the quotient by Tate twists yields
the statement of the Proposition.  
\end{proof}

Next suppose $G$ is a proreductive group over $k$ and $\epsilon_{G}
\in G$ is a fixed central element of order two.  The category of
representations $\mathsf{Rep}_{k}(G)$ of $G$ on finite dimensional
$G$-vector spaces is a semisimple $k$-linear tensor category which is
naturally $\bbZ/2$-graded by the eigenvalues of the action of
$\epsilon_{G}$.

Suppose we are given a
    {\bfseries $(G,\epsilon_{G})$-symmetric Weil cohomology theory on
      $\kay$-varieties}, i.e.\ a tensor functor
\begin{equation} \label{eq:GWeil}
  \xymatrix@R-2.5pc@M+0.2pc@C+0.5pc{
 \CA{\kay}/\mathsf{Tate} \ar[r]^-{\bfH^{\bullet}}  &
  \mathsf{Rep}_{k}(G),
}
\end{equation}
of $\bbZ/2$-graded categories, which commutes with duality.

Through the identification $\CA{\kay}/\mathsf{Tate} \cong
\mathsf{Rep}_{\bbQ}(\motM^{\kay})$ given by the fiber functor \linebreak 
$[\mathfrak{X}] \mapsto H^{\bullet}_{B}(\mathfrak{X})$ we can view
$\bfH^{\bullet}$ as a functor $\bfh\colon\mathsf{Rep}_{\bbQ}(\motM^{\kay}) \to \mathsf{Rep}_{k}(G)$ which
respects $\bbZ/2$-gradings commutes with duality. With this definition
we have a strictly commutative diagram of functors
\[
 \xymatrix@R-0.5pc@M+0.2pc@C+0.5pc{
 \CA{\kay}/\mathsf{Tate} \ar[r]^-{\bfH^{\bullet}}  \ar@{=}[d]_-{H^{\bullet}_{B}(-)} &
 \mathsf{Rep}_{k}(G).
 \\
 \mathsf{Rep}_{\bbQ}(\motM^{\kay}) \ar[ru]_-{\bfh} &}
\]
Applying $\bfH^{\bullet}$ to the Andr\'{e} motive version of the quantum
cohomology algebra we get a super commutative,  associative quantum product
\begin{equation} \label{eq:qupGk}
 \qup \, \colon \, \bfH^{\bullet}(\mathfrak{X})\otimes \bfH^{\bullet}(\mathfrak{X})
 \ \longrightarrow \ \bfH^{\bullet}(\mathfrak{X}) \otimes_{k}
 \mathsf{Nov}_{\bfH^{\bullet}(\mathfrak{X})}
\end{equation}
where to simplify notation we write $\bfH^{\bullet}(\mathfrak{X})$ instead of
$\bfH^{\bullet}([\mathfrak{X}])$. By definition we have
\begin{itemize}
\item $\bfH^{\bullet}(\mathfrak{X}) =
  \bfh(H^{\bullet}_{B}(\mathfrak{X}))$ is a $\bbZ/2$-graded
  finite dimensional $k$-vector
space equipped with an action of $G$;
\item $\mathsf{Nov}_{\bfH^{\bullet}(\mathfrak{X})}$ is
  a complete super commutative $k$-algebra
  equipped with an action of $G$;
\item $\qup$ is the image of the quantum product \eqref{eq:qupkay} under the
  functor $\bfh$.
\end{itemize}

\begin{remark}
The $G$-equivariant variant
      $\mathsf{Nov}_{\bfH^{\bullet}(\mathfrak{X})}$ of the big Novikov
      algebra is defined in the obvious way. If we write $\mathbb{M} =
      \langle q^{\beta} \, | \, \beta \in \mathsf{NE}(X,\bbZ), \,
      \beta \neq 0 \rangle$ for the maximal ideal of the monoid
      algebra $\bbQ[\mathsf{NE}(X,\bbZ)]$ and $\mathfrak{m} \subset
      \mathsf{Sym} \, H^{\bullet}_{B}(\mathfrak{X})^{\vee}$ for the
      maximal ideal of the origin in the $\bbQ$ vector space
      $H^{\bullet}_{B}(\mathfrak{X})$, then we have
      \[
      \mathsf{Nov}_{\mathfrak{X}} = \lim_{\substack{\longrightarrow \\ a,b}} \,
      \left(\bbQ[\mathsf{NE}(X,\bbZ)]/\mathbb{M}^{a}\right) \,
      \bigotimes_{\bbQ} \, \left( \left(\mathsf{Sym} \,
      H^{\bullet}_{B}(\mathfrak{X})^{\vee}\right)/\mathfrak{m}^{b}\right),
      \]
 as a limit of finite dimensional $\motM^{\kay}$-equivariant
 $\bbQ$-algebras.
 Applying the functor $\bfh$ yields
 \[
 \bfh\left( \left(\mathsf{Sym} \,
 H^{\bullet}_{B}(\mathfrak{X})^{\vee}\right)/\mathfrak{m}^{b}\right) =
 \left(\mathsf{Sym}_{k} \,
 \bfH^{\bullet}(\mathfrak{X})^{\vee}\right)/\mathfrak{m}_{k}^{b},
 \]
 where $\mathfrak{m}_{k}$ is the maximal ideal in the polynomial
 algebra $\mathsf{Sym} \, \bfH^{\bullet}(\mathfrak{X})^{\vee}$ and the
 right hand side is equipped with the action of $G$ induced from the
 $G$ action on $\bfH^{\bullet}(\mathfrak{X})$.

 Thus we set
 \[
 \begin{aligned}
 \mathsf{Nov}_{\bfH^{\bullet}(\mathfrak{X})} & \coloneqq\lim_{\substack{\longrightarrow \\ a,b}}
 \bfh \left(\bbQ[\mathsf{NE}(X,\bbZ)]/\mathbb{M}^{a}\right) \,
 \bigotimes_{k} \,
 \left(\left(\mathsf{Sym}_{k} \,
 \bfH^{\bullet}(\mathfrak{X})^{\vee}\right)/\mathfrak{m}_{k}^{b}\right)  \\
 & = \, \left( \lim_{\substack{\longrightarrow \\ a  }}
 \bfh \left(\bbQ[\mathsf{NE}(X,\bbZ)]/\mathbb{M}^{a}\right)\right)
 \widehat{\bigotimes}_{k} \, \left(\widehat{\mathsf{Sym}}_{k} \,
 \bfH^{\bullet}(\mathfrak{X})^{\vee}\right)
 \end{aligned}
 \]

To understand the $G$-equivariant $k$-algebra
${\displaystyle \lim_{\substack{\longrightarrow \\ a}} \bfh
  \left(\bbQ[\mathsf{NE}(X,\bbZ)]/\mathbb{M}^{a}\right)}$ better, we
will we will need to look at the way $\bfh$ transports the rational
$\motM^{\kay}$-representations
$\bbQ[\mathsf{NE}(X,\bbZ)]/\mathbb{M}^{a}$ to $k$-linear
$G$-representations. The transport mechanism depends on a relationship
between the two Galois groups $\Gal(\kaybar/\kay)$ and $\Gal(\kbar,k)$
which we describe next.
\end{remark}

\subparagraph{Galois twists} \label{subpar:galois}

Composing the functor $\bfh \colon \mathsf{Rep}_{\bbQ}(\motM^{\kay}) \to
 \mathsf{Rep}_{k}(G)$ with the inclusion
\[
\xymatrix@R-3.5pc@M+0.7pc@C+1.5pc{
  \mathsf{Rep}_{\bbQ}(\Gal(\kaybar/\kay))  \ar@{^{(}->}[r]
  & \mathsf{Rep}_{\bbQ}(\motM^{\kay})  \\
|| & ||  \\
\mathsf{Art}^{\kay} \ar@{^{(}->}[r] & \CA{\kay}/\mathsf{Tate}  
}
\]
we get a tensor functor
\begin{equation} \label{eq:GtoGal}
  \mathsf{Rep}_{\bbQ}(\Gal(\kaybar/\kay)) \longrightarrow \mathsf{Rep}_{k}(G).
\end{equation}
Hence, for any smooth projective $\kay$-variety $\mathfrak{X}$ we can
compose the pullback functor 
\[
\mathsf{Rep}_{\bbQ}(\Gamma_{\mathfrak{X}}) \longrightarrow
\mathsf{Rep}_{\bbQ}(\Gal(\kaybar/\kay))
\]
with the functor \eqref{eq:GtoGal}
to get a tensor functor 
 \begin{equation} \label{eq:GtoGamma}
   \bfh_{\mathfrak{X}} \colon
   \mathsf{Rep}_{\bbQ}(\Gamma_{\mathfrak{X}}) \longrightarrow \mathsf{Rep}_{k}(G).
 \end{equation}
Such a tensor functor can be viewed as homomorphism from $G$ to a
$\Gal(\kbar/k)$-twisted form of $\Gamma_{\mathfrak{X}}$. Indeed, let
us fix an algebraic closure $k \subset \kbar$.  Suppose we are given a
homomorphism of profinite groups
\[
\varphi \colon \Gal(\kbar/k) \to \Gamma.
\]
Such a $\varphi$ gives rise to
\begin{itemize}[wide]
\item[(a)] a tensor functor
\[
F_{\varphi} \colon \mathsf{Rep}_{\bbQ}(\Gamma) \longrightarrow \mathsf{Vect}_{k},
\]
which to every $\bbQ$-representation $\lambda \colon \Gamma \to \GL(V)$ assigns
the $k$-vector space 
\[
\left(V\otimes_{\bbQ} \kbar\right)^{\Gal(\kbar/k)},
\]
where $\Gal(\kbar/k)$ acts on $V$ via the homomorphism $\lambda\circ
\varphi \colon \Gal(\kbar/k) \to \GL(V)$, and acts on $\kbar$
  tautologically.
\item[(b)] a $\varphi$-twisted (inner) form of $\Gamma$, i.e.\ a
  finite group scheme $\Gamma^{\varphi}$ over $\op{Spec}\, k$,  defined by
\[
\Gamma^{\varphi} \coloneqq \op{Spec}\, \kbar \times_{\op{Ad}\circ \varphi} \Gamma. 
\]
\item[(c)] A canonical functorial action
  \[
  \mathsf{can}_{\varphi}(\lambda) \colon \Gamma^{\varphi} \to \GL(F_{\varphi}(\lambda))
  \]
  of $\Gamma^{\varphi}$ on
  $F_{\varphi}(\lambda)$ for all $\bbQ$-representations $\lambda \colon
  \Gamma \to \GL(V)$ of $\Gamma$. The action $\mathsf{can}_{\varphi}(\lambda)$
  is the
  $\Gal(\kbar/k)$-descent of the extension $\lambda\otimes {\kbar} \colon \Gamma
  \to \GL(V\otimes \kbar)$ of $\lambda$ by $\kbar$-linearity.
\end{itemize}

\medskip
With this notation we have the following

\begin{lemma} \label{lem:nuVect}
The assignment $\varphi \mapsto
F_{\varphi}$ gives an equivalence of groupoids:
\[
\left(\text{\begin{minipage}[c]{1.6in}
group homomorphisms $\Gal(\kbar/k) \to \Gamma$, with maps given 
by conjugation by elements in $\Gamma$
\end{minipage}}\right)
\
\stackrel{\cong}{\longrightarrow}
\left(\text{\begin{minipage}[c]{2in}tensor
  functors $\mathsf{Rep}_{\bbQ}(\Gamma) \to \mathsf{Vect}_{k}$
  with maps given by isomorphisms of functors\end{minipage}}\right).
\]
\end{lemma}
\begin{proof} We can construct the inverse functor as follows.
 Let $\nu \colon \mathsf{Rep}_{\bbQ}(\Gamma) \to
\mathsf{Vect}_{k}$ be a fiber functor and let $\omega \colon
\mathsf{Rep}_{\bbQ}(\Gamma) \to \mathsf{Vect}_{\bbQ}$ denote the
forgetful functor.  By the classification
\cite[Theorem~II.8.1]{milne-tannakian} of fiber functors in Tannaka
duality, we know that the functor $A \mapsto
\op{Isom}^{\otimes}(\omega\otimes A,\nu)$ from commutative
$k$-algebras to sets is represented by an affine scheme
$\mathbf{I}(\omega\otimes k,\nu)$ which is a right fppf torsor over
$\Gamma$ viewed as a group scheme over $k$. Such torsors are
classified by $H^{1}(\Gal(\kbar/k),\Gamma)$, i.e.\ the groupoid of such
torsors is the groupoid with objects given by $\Gamma$ valued
$1$-cocycles of $\Gal(\kbar/k)$ and morphisms given by
coboundaries. But the $1$-cocyles are simply homomorphisms $\varphi \colon
\Gal(\kbar/k) \to \Gamma$ and the maps between cocycles are given by
conjugations by elements in $\Gamma$. This proves the lemma.  
\end{proof}

\begin{remark} \label{rem:leftact} 
Let $\nu \colon
\mathsf{Rep}_{\bbQ}(\Gamma) \to \mathsf{Vect}_{k}$. The group scheme
$\mathbf{I}(\omega\otimes k,\nu)\times_{\mathsf{Ad}} \Gamma$ of
automorphisms of the $\Gamma$-torsor $\mathbf{I}(\omega\otimes k,\nu)$
is an inner twisted form of $\Gamma$ as a group scheme over
$\mathsf{Spec}\, k$ and $\mathbf{I}(\omega\otimes k,\nu)$ is a left
fppf torsor over this group scheme. If $\nu$ corresponds to a
homomorphism $\varphi \colon \Gal(\kbar/k) \to \Gamma$, then this group
scheme is canonically identified with $\Gamma^{\varphi}$.
\end{remark}

The previous lemma has an obvious refinement. Indeed, suppose that in
addition to the  homomorphism $\varphi$
we are given also a  homomorphism $\psi \colon G \to
\Gamma^{\varphi}$ of group schemes over $k$.
Proposition with $\psi$ gives a canonical action
\[
\mathsf{can}_{\varphi}(\lambda)\circ \psi \colon G \to \GL(F_{\varphi}(\lambda))
\]
of $G$ on the $k$-vector space $F_{\varphi}(\lambda)$.
This gives rise to a tensor functor
\[
F_{\varphi,\psi} \colon \mathsf{Rep}_{\bbQ}(\Gamma) \longrightarrow
\mathsf{Rep}_{k}(G), \qquad \lambda \mapsto
\mathsf{can}_{\varphi}(\lambda)\circ \psi. 
\]
Also, if $\gamma \in \Gamma$, and we consider the $\gamma$-conjugate homomorphism
\[
\varphi^{\gamma} \coloneqq \mathsf{Ad}_{\gamma}\circ \varphi \, \colon \, \Gal(\kbar/k) \ \to \ \Gamma,
\]
then the inner automorphism 
\[
\mathsf{Ad}_{\gamma} \colon \Gamma \to \Gamma
\]
intertwines the $\Gal(\kbar/k)$ action $\mathsf{Ad}\circ \varphi$ on the source copy of $\Gamma$ with the $\Gal(\kbar/k)$ action $\mathsf{Ad}\circ \varphi^{\gamma}$ on the source copy of $\Gamma$.

In particular $\mathsf{Ad}_{\gamma} \, \colon \, \Gamma \ \to \ \Gamma$, descends to a homomorphism
\[
\mathbf{a}_{\gamma} \, \colon \,  \Gamma^{\varphi} \ \to \ \Gamma^{\varphi^{\gamma}}
\]
of finite group schemes over $\Spec\, k$, and so we get an action of the group $\Gamma$ on the set of pairs of homomorphisms $(\varphi,\psi)$ given by 
\begin{equation} \label{eq:Gammaacts}
(\varphi,\psi) \ \to \ \left(\varphi^{\gamma},\mathbf{a}_{\gamma}\circ\psi\right).
\end{equation}

With this notations we now have the following

\begin{lemma} \label{lem:nuRep}
the assignment $(\varphi,\psi) \mapsto
F_{\varphi,\psi}$ gives an equivalence of groupoids:
\[
\left(\text{\begin{minipage}[c]{2in} pairs
  of group homomorphisms $(\varphi,\psi)$
  with maps given by \eqref{eq:Gammaacts}\end{minipage}}\right) \
\stackrel{\cong}{\longrightarrow} \  \left(\text{\begin{minipage}[c]{2.3in} tensor
functors $\mathsf{Rep}_{\bbQ}(\Gamma) \longrightarrow
\mathsf{Rep}_{k}(G)$, with maps given by isomorphisms of functors.
\end{minipage}}\right).
\]
\end{lemma}
\begin{proof}
As in the proof of Lemma~\ref{lem:nuVect} we can construct the
 inverse functor explicitly. If \linebreak 
 $\nu \colon \mathsf{Rep}_{\bbQ}(\Gamma) \to
 \mathsf{Rep}_{k}(G)$ is a fiber functor, we can compose it with the
 forgetful functor \linebreak $\mathsf{for} \colon \mathsf{Rep}_{k}(G) \to
 \mathsf{Vect}_{k}$ to get a fiber functor $\mathsf{for}\circ\nu \colon
 \mathsf{Rep}_{\bbQ}(\Gamma) \to \mathsf{Vect}_{k}$. As in the proof
 of Lemma~2.1 we consider the $k$-scheme $\mathbf{I}(\omega\otimes
 k,\mathsf{for}\circ \nu)$. It is a right torsor over $\Gamma$ viewed
 as a group scheme over $\mathsf{Spec}\, k$.  By construction, the
 group scheme $G$ acts on the left on $\mathbf{I}(\omega\otimes
 k,\mathsf{for}\circ \nu)$ by torsor automorphisms. In other words,
 we have a homomorphism from $G$ to the group scheme of torsor
 automorphisms of $\mathbf{I}(\omega\otimes k,\mathsf{for}\circ
 \nu)$. As explained in Remark~\ref{rem:leftact} this group scheme
 $\mathbf{I}(\omega\otimes k,\mathsf{for}\circ \nu)\times_{\op{Ad}}
 \Gamma$ is an inner form of $\Gamma$ as a group scheme over
 $\mathsf{Spec}\, k$, and when $\mathsf{for}\circ \nu = F_{\varphi}$,
 it is equal to the twist $\Gamma^{\varphi}$. Thus we get a
 homomorphism $\psi \colon G \to \Gamma^{\varphi}$ and an identification
 $\nu = F_{\varphi,\psi}$. This proves the lemma.
\end{proof}

Now observe that the $\motM^{\kay}$ action on
$\bbQ[\mathsf{NE}(X,\bbZ)]/\mathbb{M}^{a}$ is induced from the
$\motM^{\kay}$-action on $\mathsf{NE}(X,\bbZ)$ which, as we explained
above, factors through the homormorphism $\motM^{\kay} \to
\Gal(\kaybar/\kay) \to \Gamma_{\mathfrak{X}}$ and the natural action
of $\Gamma_{\mathfrak{X}}$ on the lattice $\mathsf{NE}(X,\bbZ) \cong
\bbZ^{\rho}$. By the previous discussion, the tensor functor
\[
\bfh_{\mathfrak{X}} \colon \mathsf{Rep}_{\bbQ}(\Gamma_{\mathfrak{X}}) \to \mathsf{Rep}_{k}(G)
\]
corresponds to a pair of homomorphisms $\varphi \colon \Gal(\kbar/k) \to
\Gamma_{\mathfrak{X}}$, and $\psi \colon G \to
  \Gamma_{\mathfrak{X}}^{\varphi}$. In particular we see that as a
  $k$-vector space
we have
  \[
  \lim_{\substack{\longrightarrow \\ a}} \,
  \bfh\left(\left(\bbQ[\mathsf{NE}(X,\bbZ)\right)/\mathbb{M}^{a}\right)
    = \kbar\dbb{\mathsf{NE}(X,\bbZ)}^{\Gal(\kbar/k)},
    \]
    where the $\Gal(\kbar/k)$ acts tautologically on $\kbar$ and acts
    on $\mathsf{NE}(X,\bbZ)$ via $\varphi \colon \Gal(\kbar/k) \to
    \Gamma_{\mathfrak{X}}$. The action of $G$ on
    $\kbar\dbb{\mathsf{NE}(X,\bbZ)}^{\Gal(\kbar/k)}$ is given by the
    composition of $\psi \colon G \to \Gamma_{\mathfrak{X}}^{\varphi}$ and
    the canonical action of $\Gamma_{\mathfrak{X}}^{\varphi}$ on
    $\kbar\dbb{\mathsf{NE}(X,\bbZ)}^{\Gal(\kbar/k)}$.

\subparagraph{Weil cohomology realization of the \texorpdfstring{$\mathsf{A}$}{A}-model F-bundle} \label{subpar:twistF}

Suppose we are in the above setup. That is we have two fields $\kay$
and $k$, both of characteristic zero, and we have fixed an inclusion
$\kay \subset \bbC$ and an algebraic closure $k \subset \kbar$.  Let
$G$ be a given proreductive group over $k$ with a fixed central
involution $\epsilon_{G}$ and suppose we are given a
$(G,\epsilon_{G})$-symmetric Weil cohomology theory on
$\kay$-varieties, i.e.\ a tensor functor
\[
\bfh \colon \mathsf{Rep}_{\bbQ}(\motM^{\kay}) \,  \longrightarrow \,
\mathsf{Rep}_{k}(G).
\]
With these data we can associate a natural $\bbZ/2$-graded $k$-vector
space equipped with a $G$-action, namely $\bfH^{\bullet}(\mathfrak{X})
\coloneqq \bfh(H^{\bullet}_{B}(\mathfrak{X}))$. This vector space will serve
as a fiber of our F-bundle. As explained in section~2,
$\bfH^{\bullet}(\mathfrak{X})$ comes equipped with a super
commutative, associative, $G$-equivariant quantum product
\eqref{eq:qupGk}.

Recall that the Betti cohomology 
$H^{\bullet}_{B}(\mathfrak{X})$  is a $\bbQ$-representation
of $\motM^{\kay}$ which decomposes naturally into a direct sum of
rational $\motM^{\kay}$-subrepresentations
\[
H^{\bullet}_{B}(\mathfrak{X}) = H^{2}_{B}(\mathfrak{X})_{\kaybar-\op{alg}} \oplus
H^{2}_{B}(\mathfrak{X})_{\kaybar-\op{transc}} \oplus H^{* \neq
  2}_{B}(\mathfrak{X}).
\]
so that the action of $\motM^{\kay}$ on
$H^{2}_{B}(\mathfrak{X})_{\kaybar-\op{alg}}$ factors through the
finite quotient $\motM^{\kay} \twoheadrightarrow
\Gamma_{\mathfrak{X}}$.

As explained in Section~2, the choice of the fiber functor $\bfh$ gives
rise to a tensor functor
\[
\bfh_{\mathfrak{X}} \colon \mathsf{Rep}_{\bbQ}(\Gamma_{\mathfrak{X}})
\longrightarrow \mathsf{Rep}_{k}(G),
\]
which is equivalent to specifying a pair of homomorphisms $\varphi \colon
\Gal(\kbar/k) \to \Gamma_{\mathfrak{X}}$ and \linebreak $\psi \colon G
\to \Gamma_{\mathfrak{X}}^{\varphi}$. 
Using this we define the following
$k$-representations of $G$:
\begin{itemize}
\item $\bfH^{2}_{\op{alg}}(\mathfrak{X}) \coloneqq
  \bfh_{\mathfrak{X}}(H^{2}_{B}(\mathfrak{X})_{\kaybar-\op{alg}})$.
  Explicitly, this $k$-linear
  $G$-representation is given by
  the $k$-vector space $(\mathsf{NS}(X,\bbQ)\otimes
  \kbar)^{\Gal(\kbar/k)}$, where the $\Gal(\kbar/k)$-action on
  $\mathsf{NS}(X,\bbQ)$ is given by $\varphi$ while the
  $\Gal(\kbar/k)$-action on $\kbar$ is tautological. The $G$-action on
  this vector space is given by the homomorphism $\psi \colon G \to
  \Gamma_{\mathfrak{X}}^{\varphi}$ and the natural action of
  $\Gamma_{\mathfrak{X}}^{\varphi}$ on the vector space
  $(\mathsf{NS}(X,\bbQ)\otimes \kbar)^{\Gal(\kbar/k)}$.
\item $\bfH^{2}_{\op{transc}}(\mathfrak{X}) \coloneqq
  \bfh\left(H^{2}_{B}(\mathfrak{X})_{\op{transc}}\right) \subset \bfH^{2}$
\item 
  $\bfH^{* \neq 2}(\mathfrak{X}) \coloneqq \bfh\left(H^{* \neq
  2}_{B}(\mathfrak{X})\right)$.
\end{itemize}

Altogether we get a decomposition of $\bfH^{\bullet}(\mathfrak{X})$ as a
direct sum of $G$-representations:
\begin{equation} \label{eq:WeilGdecomp}
\bfH^{\bullet}(\mathfrak{X}) \coloneqq
\bfH^{2}_{\op{alg}}(\mathfrak{X})\oplus
\bfH^{2}_{\op{transc}}(\mathfrak{X})\oplus \bfH^{* \neq
  2}(\mathfrak{X}).
\end{equation}
Passing to analytifications and choosing suitable admissible domains in
the pieces of this decomposition will allow us to describe the base of
the $\mathsf{A}$-model F-bundle for $\mathfrak{X}$.

\medskip

Concretely, let
\[
\mathsf{N}_{1}(X,\bbZ) = \mathsf{CH}^{\mathsf{hom}}_{1}(X) =
\mathsf{im} \left[ \mathsf{CH}_{1} \to
  H_{2}((X\times_{\kaybar} \bbC)^{\op{an}},\bbQ)
  \right]
\]
denote the group of cycle classes of algebraic curves in $X$. This is
again a free abelian group of rank $\rho$ equipped with a
$\Gamma_{\mathfrak{X}}$-action.  The cap product between singular
homology and cohomology on $(X\times_{\kaybar}\bbC)^{\op{an}}$ identifies the
group $\mathsf{N}_{1}(X,\bbZ)$ with the dual of the N\'{e}ron-Severi lattice
$\mathsf{NS}(X,\bbZ)_{\op{tf}}$ and this identification is compatible
  with the natural $\Gamma_{\mathfrak{X}}$-actions.

Additionally, $\mathsf{N}_{1}(X,\bbZ)$ contains the
$\Gamma_{\mathfrak{X}}$-stable sub semigroup
\[
\mathsf{NE}(X,\bbZ) \subset \mathsf{N}_{1}(X,\bbZ) 
\]
of numerically effective curve classes in $X$.

The homomorphisms $\varphi$ and $\psi$ define a non-split affine torus
$\mathcal{T}(\mathfrak{X})$ over $k$ on which $G$ acts by automorphisms.  Here
$\mathcal{T}(\mathfrak{X}) \coloneqq (\NS(X,\bbZ)_{\op{tf}}\otimes
\mathbb{G}_{m,\kbar})^{\Gal(\kbar/k)}$, where $\mathbb{G}_{m,\kbar}$
denotes the pullback of the multiplicative group scheme to $\kbar$,
$\Gal(\kbar/k)$ acts tautologically on $\mathbb{G}_{m,\kbar}$, and
acts on $\NS(X,\bbZ)_{\op{tf}}$ via the homomorphism $\varphi \colon \Gal(\kbar/k)
\to \Gamma_{\mathfrak{X}}$. As a group scheme over $k$ this torus is
equipped with an action of the finite group scheme
$\Gamma^{\varphi}_{\mathfrak{X}}$, and the $G$ action on $\mathcal{T}(\mathfrak{X})$
is given by $\psi \colon G \to \Gamma^{\varphi}_{\mathfrak{X}}$. 

\bigskip

Now consider the non-archimedean field $k\dbp{\fy}$ of
Laurent series in an auxiliary variable $\fy$. Let
$\mathcal{T}_{k\dbp{\fy}}(\mathfrak{X})$ denote the base change of
$\mathcal{T}(\mathfrak{X})$ to an affine (still non-split) torus over
$k\dbp{\fy}$, and let
$\mathcal{T}_{k\dbp{\fy}}(\mathfrak{X})^{\op{an}}$ be the non-archimedean
analytification of the torus
$\mathcal{T}_{k\dbp{\fy}}(\mathfrak{X})$. Following the same steps as in the
construction of the analytic $\mathsf{A}$-model F-bundle for a variety
over $\kaybar$, we define the following rigid analytic spaces over
$k\dbp{\fy}$:
\begin{description}
\item $B_{\mathfrak{X},q} \subset
  \mathcal{T}_{k\dbp{\fy}}(\mathfrak{X})$ is the preimage of the ample cone
  in $\mathsf{NS}(X,\mathbb{R})$ under the valuation map
  $\mathcal{T}_{k\dbp{\fy}}^{\op{an}}(\mathfrak{X}) \to
  \mathsf{NS}(X,\mathbb{R})$.
\item $\omB_{\mathfrak{X,t}}^{\even}$ is the product of the
  analytic affine line $(\bfH^{0}(\mathfrak{X})\otimes
  k\dbp{\fy})^{\op{an}}$
  with an open unit polydisk in the
  analytification of the $k\dbp{\fy}$-affine space
$\bfH^{2\bullet \, \geq 2}(\mathfrak{X})
  \otimes k\dbp{\fy}$
\item $B_{\mathfrak{X,t}}^{\even}$ is the product of the analytic
  affine line $(\bfH^{0}(\mathfrak{X})\otimes
  k\dbp{\fy})^{\op{an}}$ with an open unit polydisk in the
  analytification of the $k\dbp{\fy}$-affine space
  $\left(\bfH^{2}_{\op{transc}}(\mathfrak{X})\oplus
  \bfH^{2\bullet\, >2}(\mathfrak{X})\right) \otimes_{k} k\dbp{\fy}$.
\item $B_{\mathfrak{X}}^{\odd}$ is the super analytic variety
  corresponding to the purely odd vector space \linebreak
  $\bfH^{\odd}(\mathfrak{X})\otimes_{k} k\dbp{\fy}$.
\end{description}
Multiplying these together we get two $k\dbp{\fy}$-analytic spaces 
\[
\omB_{\mathfrak{X}} \coloneqq B_{\mathfrak{X},q} \, \times \,
\omB_{\mathfrak{X,t}}^{\even} \, \times \, B_{\mathfrak{X}}^{\odd}
\quad \text{and} \quad B_{\mathfrak{X}} \coloneqq B_{\mathfrak{X},q} \,
\times \, B_{\mathfrak{X,t}}^{\even} \, \times \,
B_{\mathfrak{X}}^{\odd},
\]
which are $G$-stable admissible opens in the $k\dbp{\fy}$-analytic varieties
$\mathcal{T}_{k\dbp{\fy}}(\mathfrak{X})^{\op{an}}\times
(\bfH^{\bullet}(\mathfrak{X})\otimes_{k} k\dbp{\fy})^{\op{an}}$ and
$\mathcal{T}_{k\dbp{\fy}}(\mathfrak{X})^{\op{an}}\times ((\bfH^{2}(\mathfrak{X})_{\op{transc}}\oplus \bfH^{\bullet\neq 2}(\mathfrak{X}))\otimes_{k} k\dbp{\fy})^{\op{an}}$
respectively.  In particular, we can view $\omB_{\mathfrak{X}}$ as a
non-archimedean replacement of the formal scheme $\op{Spf}\,
\mathsf{Nov}_{\mathfrak{X}}$.

\medskip

Let $\boldsymbol{D}$ denote the germ at $0$ in a
$k\dbp{\fy}$-analytic disk with coordinate $u$.  $\omH$
denote the trivial analytic super vector bundle on
$\omB_{\mathfrak{X}}\times \boldsymbol{D}$ with fiber
$\bfH^{\bullet}(\mathfrak{X})\otimes k\dbp{\fy}$. Then, the
same argument that we used to prove
Lemma~\ref{lem:GW-potential_analytic} implies that the formal
$G$-equivariant quantum product \eqref{eq:qupGk} defines an analytic
quantum product
\[
\qup \, \colon \, \omH\otimes \omH \ \longrightarrow \ \omH
\]
which is moreover independent of $u$.  Using this product we again
define a quantum connection, i.e.\ a meromorphic flat connection
$\bnabla$ on $\omH$ given by the usual formulas
\[
\left| \ 
\begin{aligned}
    \bnabla_{u\partial_u} &= u\partial_u - u^{-1}\Eu\qup(-) + \Gr,\\
    \bnabla_{\partial_{t_{i}}} &= \partial_{t_{i}} + u^{-1}T_{i}\qup(-), \qquad\text{ for }\tau\in \bfH^{\bullet}(\mathfrak{X})\otimes k\dbp{\fy},\\
    \bnabla_{\xi q\partial_q} &= \xi q\partial_q + u^{-1}\xi\qup(-),  \qquad\text{ for }\xi\in \bfH^{2}(\mathfrak{X})\otimes k\dbp{\fy},
\end{aligned}
\right.
\]
where $\Eu$ is the Euler vector field, $\Gr$ is the grading operator,
and $\xi q\partial_q$ is the derivation of
$\mathcal{O}_{\omB_{\mathfrak{X}}}$ given by $\xi
q\partial_q(q^\beta)=(\beta\cdot\xi)q^\beta$.

\medskip

Again $(\omH,\bnabla)/\omB_{\mathfrak{X}}$ is an over maximal F-bundle
over $\omB_{\mathfrak{X}}$ and $B_{\mathfrak{X}} \subset
\omB_{\mathfrak{X}}$ is a closed analytic subvariety which meets
transversally every leaf through a rigid point of the redundancy
foliation of $(\omH,\bnabla)/\omB_{\mathfrak{X}}$. Thus the
restriction of $(\omH,\bnabla)|_{B_{\mathfrak{X}}\times
  \boldsymbol{D}}$ is the associated maximal F-bundle
$(\mathcal{H},\nabla)/B_{\mathfrak{X}}$ over $B_{\mathfrak{X}}$.

\begin{definition} \label{def:nonarchFkay}
The F-bundle $(\omH,\nabla)/\omB_{\mathfrak{X}}$ will be called the
\emph{\bfseries non-archimedean overmaximal $\mathsf{A}$-model
F-bundle associated to $\mathfrak{X}$}, $\kay$, $k$, and
$\bfH^{\bullet}$.  The F-bundle $(\mathcal{H},\nabla)/B_{\mathfrak{X}}$ will
be called the \strongemph{non-archimedean maximal $\mathsf{A}$-model F-bundle
  associated to $\mathfrak{X}$}, $\kay$, $k$, and $\bfH^{\bullet}$.
\end{definition}

\begin{remark} \label{rem:basechange} For the applications to atoms we will
frequently base change the spaces $\omB_{\mathfrak{X}}$ and
$B_{\mathfrak{X}}$, and the corresponding F-bundles, to analytic
spaces, and analytic F-bundles to an algebraically closed field $\bbk$
containing $k\dbp{\fy}$, e.g.\  $\bbk =
\kbar\dbp{\fy^{\bbQ}}$. We do this to ensure that the
quantum multiplication by the Euler vector field has well defined
eigenvalues. If however the eigenvalues of $\Eu\qup (-)$ exist over
$k\dbp{\fy}$ one can work with the analytic F-bundles over
the smaller field $k\dbp{\fy}$ and get finer obstructions to
rationality. Note that the process of base changing to $\bbk$
definitely loses information, e.g.\ the base changed torus
$\mathcal{T}_{\bbk}(\mathfrak{X})$ is automatically split.
\end{remark}

\section{Decomposition theorems} \label{sec:decompose}

In this section we discuss the basic decomposition theorems for F-bundles.
First we recall the spectral decomposition theorem, which works for arbitrary F-bundles. After that we discuss the blowup and projective bundle decomposition theorems that are specific to $\mathsf{A}$-model F-bundles associated to smooth projective varieties.

\subsection{Spectral decomposition theorem} \label{ssec:specd}

 Suppose $\bbk$ is a non-archimedean field of characteristic zero.

\begin{theorem} \label{thm:K-decomposition}
    Let $B$ be an admissible open neighborhood of a rigid point $b$ in a
    $\bbk$-analytic affine space, and let $(\cH ,\nabla)/B$ be an
    F-bundle over $B$ maximal at $b$.  Assume that we have a
    decomposition $\cH_{b,0} \simeq \bigoplus_{i\in I} H_i$ stable
    under $\bkappa_b = \Eu_{b}\fp (-)$, and that for any $i\neq j\in
    I$, the spectra of $\bkappa_b\vert_{H_i}$ and $\bkappa_b\vert_{H_j}$
    are disjoint.  Then there exists an admissible open neighborhood
    $U$ of $b$ such that the restricted F-bundle $(\cH,\nabla)/U$ decomposes
    into an external sum of maximal F-bundles $(\cH_i ,\nabla_i ) /U_i$
    extending the decomposition of $\cH_{b,0}$.
\end{theorem}
\begin{proof}
    See \cite[Theorem 1.2]{HYZZ_Decomposition}.
\end{proof}

\begin{remark} \label{rem:Can.and.formal.decomp}
The spectral decomposition theorem has obvious versions that hold for 
formal or \linebreak $\bbC$-analytic F-bundles. Indeed, let $(\cH, \nabla)/B$ be a formal maximal F-bundle over an algebraically closed  field $k$ of characteristic $0$, let $b \in B$ be a closed point, and let 
$\{\lambda_{i}(b)\}_{i=1}^{s} \in k$ be the distinct eigenvalues of $\bkappa_{b}$. Then there exists a formal neighborhood $U$ of $b$ in $B$ and maps $\lambda_{i} \colon U \to \Spf \, k\dbb{t-\lambda_{i}(b)}$, sending $b$ to $\lambda_{i}(b)$, so that $\lambda_{i}$ are the distinct eigenvalues of $\bkappa_{|U}$, and $(\cH,\nabla)/U$ decomposes canonically into an exterior sum of formal $F$ bundles corresponding to the generalized eigenbundles of $\bkappa_{|U}$. Similarly, if $(\cH, \nabla)/B$ is a de Rham part of a \nc \ variation of Hodge structures, i.e.\ a complex analytic F-bundle, then for any point $b \in B$ there exist disjoint small disk $\mathbf{D}_{i} \subset \bbC$, centered at the distinct eigenvalues of $\bkappa_{b}$, and $\bbC$-analytic functions $\lambda_{i} \colon U \to \mathbf{D}_{i} \subset \bbC$ from a small neighborhood $U$ of $b$ in $B$, so that $\lambda_{i}$ are the distinct  eigenvalues of $\bkappa_{U}$ and a decomposition of $(\cH,\nabla)/U$ into an external sum of F-bundles $\left(\cH_{i},\nabla_{i})\right)/U_{i}$ where $U_{i}$ is a complex manifold, $(\cH_{i},\nabla_{i})$ is a bundle with connection on $U_{i}\times \Spf \, \bbC\dbb{u}$, and the decomposition isomorphism is $\bbC$-analytic in the $U$ direction and only formal in the $u$-direction.
\end{remark}

\subsection{The blowup formula} \label{sec:blowup}

Let $X$ be a smooth complex projective variety, and $Z \subset X$ a smooth subvariety of pure codimension $r \ge 2$. Let $\tX$ be the blowup of $X$ along $Z$, and let $X'$ be the disjoint union of $X$ with $(r-1)$ copies of $Z$.
We have $\mathsf{A}$-model maximal F-bundles $(\widehat{\cH}, \widehat{\nabla})/B_{\tX}$, and $(\cH', \nabla')/B_{X'}$. The \emph{blowup decomposition} theorem identifies  the germs of these F-bundles at appropriately chosen base  points. It is extracted from Iritani's blowup formula \cite{Iritani_blowup} which gives the same identification for the formal $\mathsf{A}$-model logarithmic F-bundles of $\tX$ and $X'$.

\medskip

Before we give a precise formulation and proof of the blowup decomposition, let us look more closely at its shape and structure. The blowup of $Z$ in $X$ gives rise to a standard fibered square
\begin{equation}\label{eq:blowupdiagram}
\xymatrix{
D \ar[r]^-{\jmath} \ar[d]_-{\pi_{D}} & \widehat{X} \ar[d]^-{\pi} \\
Z \ar[r]_-{\imath} & X
}
\end{equation}
where $\imath \colon Z \hookrightarrow X$ is the inclusion, $\pi \colon \widehat{X} \to X$ is the blowup map, $\jmath \colon D \hookrightarrow \widehat{X}$ is the exceptional divisor,
and $\pi_{D} \colon D = \bbP(N_{Z/X}) \to Z$ is the natural projection. Write \, $\exch \colon H^{\bullet}(D,\bbZ) \, \to \, H^{\bullet}(D,\bbZ)$  for the operator of (classical) multiplication by $c_{1}(\cO_{D}(1)) = -c_{1}(N_{D/\widehat{X}})$. Then, there is a natural identification of $\bbZ/2$-graded cohomology spaces
\begin{equation} \label{eq:split.cohomology}
H^{\bullet}(X',\bbQ) \ = \ H^{\bullet}(X,\bbQ) \bigoplus \underbrace{H^{\bullet}(Z,\bbQ) \bigoplus \cdots \bigoplus H^{\bullet}(Z,\bbQ)}_{(r-1)-\text{copies}} \ \widetilde{\to} \ H^{\bullet}(\widehat{X},\bbQ)
\end{equation}
given by 
\[
(\alpha,\beta_1,\dots,\beta_{r-1})\mapsto \pi^*(\alpha)-\jmath_*\pi_{D}^{*}(\beta_1)-\jmath_* \exch\pi_{D}^{*}(\beta_2)-\dots - \jmath_* \exch^{r-2}\pi_{D}^{*}(\beta_{r-1}).
\]

\medskip

To simplify the discussion suppose $\kay = \bbC$, $k = \bbQ$, $\bbk = \overline{\bbQ}\dbp{\fy^{\bbQ}}$, $H^{\bullet}(Y) = H^{\bullet}_{B}(Y(\bbC)^{\op{an}},\bbQ)$  for any complex projective manifold $Y$. Let $\mathsf{Amp}(Y) = \op{int}(\mathsf{Nef}(Y))  \subset H^2(Y,\bbR)$ denote the  ample cone of $Y$. Then for any point $\theta \in \mathsf{Amp}(Y)\cap H^2(X,\bbQ)$ we get a rigid point in $B_{Y,q}$ which is given by the homomorphism $\NE(X,\bbZ) \to \bbk^{\times}$, $\beta \to \fy^{(\beta,\theta)} \in \bbk^{\times}$. We will denote this rigid point by 
$\fy^{\theta} \in B_{Y,q}$

Let $\omega \in H^2(X,\bbZ)$ be an ample class on $X$, and let $\widehat{\omega} \in H^2(\widehat{X},\bbQ)$ be a class of the form $\widehat{\omega} \coloneqq \epsilon^{-1}\varphi^{*}\omega -  K_X$ for some $\epsilon \in (0,1)\cap  \bbQ$. For a small enough $\epsilon > 0$ the class $\widehat{\omega}$ is  ample on $\tX$. Then for every positive $\epsilon \ll 1$ we can consider the point $\widehat{Q} = \fy^{\widehat{\omega}} \in B_{\tX,q}$. This is a special point  
in $B_{\widehat{X},q} \subset B_{\widehat{X}}$ which is close to the boundary of the tube domain $B_{\widehat{X},q}$ and whose valuation is close to $\infty$.

\medskip

In the limit $\epsilon \to 0$ only stable maps that are vertical 
for the map $\pi \colon \tX \to X$ contribute to the  quantum product on $\tX$. 
Furthermore, in this limit, only points and vertical lines will contribute to the quantum product and the contribution of lines will be $\fy^{r-1}$.   
Consider the operator $\widehat{\bkappa}$ which is the quantum multiplication in $H^{\bullet}(\widehat{X},\bbk)$ with the Euler vector field $\Eu_{\widehat{X},\widehat{Q}}$, computed at the special point $\widehat{Q}$.  A direct calculation shows that in the limit $\epsilon \to 0$ the operator $\widehat{k}$ has the following  block form  with respect to the decomposition \eqref{eq:split.cohomology}.

\[\tiny
\text{\normalsize $\widehat{\bkappa} \ = \ $} \left(\begin{array}{c|c:c:c:c:c:c}
K_X&&&& &&(r-1)\imath_*\\
\hline -(r-1)\imath^*& K_Z-c_1(N)&&&&& (r-1)(\fy^{r-1}+c_{r-1}(N))\\
\hdashline & -(r-1) &K_Z-c_1(N)&&& &(r-1)c_{r-2}(N)\\
\hdashline && -(r-1) &K_Z-c_1(N)&&& (r-1)c_{r-3}(N)\\
\hdashline &&&\dots  &\dots& &\dots\\
\hdashline &&&&-(r-1)& K_Z-c_1(N) &(r-1) c_2(N)\\
\hdashline &&&& &-(r-1)&K_Z+(r-2)c_1(N)\end{array}\right)
\]
\noindent where $N = N_{Z/X}$ is the rank $r$  normal bundle to $Z$ in $X$, $\imath\colon Z \hookrightarrow X$ is the natural inclusion, and each coefficient acts by classical multiplication on the cohomology of $\widehat{X}$.

\medskip

This operator preserves the filtration $H^{\geq \bullet}(\widehat{X},\bbk)$ by degree of cohomology and the associated graded with respect to this filtration  gives the operator

{\tiny
\[
\text{\normalsize $\mathsf{gr} \, \widehat{\bkappa} \ = \ $} 
\left(\begin{array}{c|c:c:c:c:c:c}
0&&&&&& 0\\
\hline
-(r-1)\imath^{*} & 0 &&&&& (r-1)\widehat{Q}\\
 \hdashline
 & -(r-1) & 0 &&& & 0 \\
\hdashline && -(r-1) & 0 &&& 0 \\
\hdashline &&&\dots  &\dots& &\dots\\
\hdashline &&&&-(r-1)& 0 & 0\\
\hdashline &&&& &-(r-1)& 0\end{array}\right)
\]
}

This implies that the eigenvalues of $\widehat{\bkappa}$ split into $r$ clusters and each cluster is contained  in a small analytic disk in $\bbk$ so that:

\begin{enumerate}[wide]
\item The disks are disjoint, one of them is centered at $0$,
while the rest are centered at \linebreak 
$(r-1)e^{\frac{\pi \sqrt{-1}(2j-1)}{r-1}}$ for $j=1,\dots,r-1$.
\item The sum of the generalized eigenspaces corresponding to
the eigenvalues that are close to $0$ give a super-vector space
isomorphic to  $H^{\bullet}(X,\bbk)$.
\item The sum of the generalized eigenspaces corresponding to the
eigenvalues  close to a rescaled $(m -
    1)$-st root of $1$ give a super-vector space isomorphic
to $H^{\bullet}(Z,\bbk)$.
\end{enumerate}

By analytic continuation, this picture extends to an appropriate open domain in $B_{X}$ and yields the desired blowup decomposition. To show rigorously that this reasoning identifies the $\mathsf{A}$-model F-bundles of $\widehat{X}$ and $X'$ we will have to identify the special point $\omega'$ in $B_{X'}$ which corresponds to $\widehat{\omega}$. This is quite involved and requires a tricky inductive argument which we will not carry out here. Instead we will extract the blowup decomposition directly by using Iritani's formal blowup decomposition \cite{Iritani_blowup} and identifying big open subsets in $B_{\widehat{X}}$  and $B_{X'}$ where Iritani's identification of the F-bundles 
for $\widehat{X}$ and $X'$ is actually analytic. It is also instructive to note that in the formal setting the above clustering behavior of the eigenvalues is proven rigorously in Iritani's theorem \cite[Theorem~5.8(6)]{Iritani_blowup}.

\medskip

More precisely we have the following

\begin{theorem} \label{thm:blowup}
  There is a canonical isomorphism of maximal F-bundles between $(\widehat{\cH}, \widehat{\nabla})$ and $(\cH', \nabla')$, over an analytic domain $\widehat{U}$ in $B_{\tX}$ and an analytic domain $U'$ in $B_{X'}$.
  The union of different choices of $\widehat{U}$ is connected and nonempty, same for $U'$.
\end{theorem}

\begin{proof}
We will deduce the explicit description of the domains and the isomorphism in the statement of Theorem~\ref{thm:blowup} from Iritani's blowup formula \cite{Iritani_blowup}. 
First we have the following

\begin{lemma} \label{lem:convergence}
  Let $k$ be a field of characteristic zero.
  Let
  \[ \varphi \colon k\dbb{x_1, \dots, x_l} \longto k[p, p^{-1}]\dbb{y_1, \dots, y_m} \]
  be a map of $k$-algebras.
  Write
  \[ \varphi(x_j) = \sum_{\lambda \in \bbZ} \ \sum_{i_1, \dots, i_m \in \bbN} a^j_{i_1, \dots, i_m} y_1^{i_1} \cdots y_m^{i_m} p^\lambda .\]
  Assume there exist constants $\alpha, \beta, \gamma, \delta \in \bbZ$ with $\alpha < 0 < \beta$ such that for every term $y_1^{i_1} \cdots y_m^{i_m} p^\lambda$ with nonzero coefficient, we have
  \[ \alpha d + \gamma < \lambda < \beta d + \delta, \]
  where $d = i_1 + \dots + i_m$ is the total degree.
  Let $\bbk$ be an algebraically closed non-archimedean field containing $k$, such that the absolute value restricted to $k$ is trivial.
  Then after base change from $k$ to $\bbk$, for any $\epsilon \in (0, 1)$, the map $\varphi$ is convergent over the admissible domain given by $\abs{y_i} \in [0, \epsilon)$ for $i = 1, \dots, m$ and $\abs{p} \in (\epsilon^{-\frac{1}{2\alpha}}, \epsilon^{-\frac{1}{2\beta}})$.
\end{lemma}
\begin{proof}
  Since the coefficients $a^j_{i_1, \dots, i_m}$ are in $k$, and the absolute value of $\bbk$ restricted to $k$ is trivial, it suffices to show that for terms $y_1^{i_1} \cdots y_m^{i_m} p^\lambda$ with nonzero coefficients, $\epsilon^d \cdot \abs{p}^\lambda \to 0$ as the powers tend to infinity.

  As $\lambda \to +\infty$, we have
  \[ \epsilon^d \cdot \abs{p}^\lambda < \epsilon^d \cdot \epsilon^{-\frac{\lambda}{2\beta}} = \epsilon^{d - \frac{\lambda}{2\beta}}, \]
  which converges to $0$ since $d - \frac{\lambda}{2\beta} \to +\infty$.

  As $\lambda \to -\infty$, we have
  \[ \epsilon^d \cdot \abs{p}^\lambda < \epsilon^d \cdot \epsilon^{-\frac{\lambda}{2\alpha}} = \epsilon^{d - \frac{\lambda}{2\alpha}}, \]
  which converges to $0$ since $d - \frac{\lambda}{2\alpha} \to +\infty$.
  This completes the proof.
\end{proof}

To use Iritani's formula from \cite{Iritani_blowup} to study the convergence of maps between analytic $\mathsf{A}$-model F-bundles, we will embed the requisite Novikov rings into power series rings.

Let $\imath \colon Z\to X$ denote the inclusion, $\varphi\colon \tX \to X$ the projection, $D \subset \tX$ the exceptional divisor, $\rho_Z = c_1(N_{Z/X})$, and $\lambda$ the class of a line in the fiber of $E \to Z$.
Recall that $r$ is the codimension of $Z \subset X$.
Following \cite[(5.11)]{Iritani_blowup}, we set
\[\fs = \begin{cases}
    r-1 & \text{if $r$ is even,} \\
    2(r-1) & \text{if $r$ is odd.}
\end{cases}\]
As usual we will write $\mathsf{N}^1(X,\bbZ) =  \CH^{1,\mathsf{num}}(X) = \CH^{1,\mathsf{hom}}(X)$ for the torsion free part of the N\'{e}ron-Severi group of a smooth projective variety $X/\kaybar$.

Choose a basis $\omega_1, \dots, \omega_l$ of the free abelian group {$\mathsf{N}^1(X, \bbZ)$}, consisting of ample classes, such that the pullbacks $\imath^*\omega_1, \dots, \imath^*\omega_l \in \mathsf{N}^1(Z, \bbZ)$ are contained in a simplicial cone spanned by a basis $\omega'_1, \dots, \omega'_m$ of $\mathsf{N}^1(Z, \bbZ)$ consisting of ample classes.
This is possible by choosing $\omega_1, \dots, \omega_l$ sufficient close to a fixed ray in the ample cone of $X$.

Then we obtain a dual basis $\beta'_1,\dots,\beta'_m \in \mathsf{N}_1(Z, \bbZ)$ such that $\NE(Z, \bbZ)$ is contained in $\Cone(\beta'_1,\allowbreak \dots,\allowbreak \beta'_m)$,
as well as a dual basis $\beta_1,\dots,\beta_l \in \mathsf{N}_1(X, \bbZ)$ such that $\Cone(\beta_1,\dots,\beta_l)$ contains $\NE(X, \bbZ)$ and $\imath_* \beta'_i$ for $i=1,\dots,m$.

\begin{lemma}
  For $i = 1, \dots, l$, there exists a constant $c_i \in \bbZ_{>0}$ such that for $\tbeta_i \coloneqq \beta_i - c_i \lambda$, we have an embedding
  \[ \NE(\tX, \bbZ) \subset \Cone(\tbeta_1,\dots,\tbeta_l,\lambda). \]
\end{lemma}

This induces embeddings
\begin{align*}
  k\dbb{\NE(X, \bbZ)} & \hooklongrightarrow k\dbb{q_1,\dots,q_l}, \\
  k\dbb{\NE(Z, \bbZ)} & \hooklongrightarrow k\dbb{q'_1,\dots,q'_m}, \\
  k\dbb{\NE(\tX, \bbZ)} & \hooklongrightarrow k\dbb{\tq_1,\dots,\tq_l,\fq}.
\end{align*}

Then the inclusion of the Novikov ring of $\tX$
\begin{align*}
  k\dbb{\NE(\tX, \bbZ)} &\hooklongrightarrow k[\fq^{1/\fs}, \fq^{-1/\fs}]\dbb{\NE(X,\bbZ)} \\
  \tQ^{\td} &\longmapsto Q^{\varphi_* \td} \fq^{-[D] \cdot \td}
\end{align*}
induces a map
\begin{align} \label{eq:Novikov_tX}
  k\dbb{\tq_1,\dots,\tq_l,\fq} &\longto k[\fq^{1/\fs}, \fq^{-1/\fs}]\dbb{q_1,\dots,q_l} \\
  \tq_i &\longmapsto q_i \fq^{-c_i} \nonumber \\
  \fq &\longmapsto \fq . \nonumber
\end{align}
Similarly, the map from the Novikov ring of $Z$
\begin{align*}
  k\dbb{\NE(Z, \bbZ)} &\longto k[\fq^{1/\fs}, \fq^{-1/\fs}]\dbb{\NE(X,\bbZ)} \\
  Q_Z^d &\longmapsto Q^{\imath_* d} \fq^{-\rho_Z \cdot d / (r-1)} .
\end{align*}
induces a map
\begin{align} \label{eq:Novikov_Z}
  k\dbb{q'_1,\dots,q'_l} &\longto k[\fq^{1/\fs}, \fq^{-1/\fs}]\dbb{\tq_1,\dots,\tq_l} \\
  q'_i &\longmapsto F_i(\tq_1, \dots, \tq_l) \fq^{c'_i}, \nonumber
\end{align}
where $F_i$ are monomials in $\tq_1, \dots, \tq_l$ with non-negative exponents, and $c'_i \in \frac{1}{\fs}\bbZ$. Now we can apply \cref{lem:convergence} to the maps \eqref{eq:Novikov_tX} and \eqref{eq:Novikov_Z} to obtain domains of convergence after passing to the non-archimedean field $\bbk$.
The union of such domains of convergence from different choices is connected.

It is shown in \cite[Theorem 5.18]{Iritani_blowup} that there exists a formal invertible change of variables
\begin{align*}
H^{\bullet}(\tX) &\longto H^{\bullet}(X) \oplus H^{\bullet}(Z)^{\oplus(r-1)}, \\
\ttau &\longmapsto \big(\tau(\ttau), \{ \varsigma_j(\ttau)\}_{0 \le j \le r-2}\big)
\end{align*}
defined over $k[\fq^{1/\fs}, \fq^{-1/\fs}]\dbb{\NE(X,\bbZ)}$ and an isomorphism
\[ \Psi \colon \QDM(\tX)^\la \longmapsto \tau^*\QDM(X)^\la \oplus \bigoplus_{j=0}^{r-2} \varsigma_j^*\QDM(Z)^\la \]
that commutes with the quantum connection.
Here, for a smooth projective variety $X$,  $\QDM(X)^\la$ denotes the formal logarithmic $\mathsf{A}$-model F-bundle of $X$ (called the quantum $D$-module of $X$ in \cite{Iritani_blowup}), after a  base change to $k[u, \fq^{1/\fs}, \fq^{-1/\fs}]\dbb{\NE(X,\bbZ)}$.
Furthermore, the maps $\tau$, $\varsigma_j$ and $\Psi$ above also satisfy the assumption of \cref{lem:convergence} by  homogeneity, so we obtain domains of convergence after passing to the non-archimedean field $\bbk$.

Finally, we deduce \cref{thm:blowup} by locally taking the quotient by the canonical foliations on the base of the over-maximal F-bundles.
The canonicity of the isomorphism follows from \cite[\S 5.4]{HYZZ_Decomposition}.
\end{proof}

\begin{remark} \label{rem:blowup_uniqueness}
The isomorphism in \cref{thm:blowup} is unique given algebraic initial conditions, see \cite[\S 5.4]{HYZZ_Decomposition} and \cite[\S 5.8]{Iritani_blowup} for details.
\end{remark}

\subsection{The Leray-Hirsch formula} \label{ssec:LH}

Let $X$ be a smooth projective variety over $\kaybar$, $E \to X$ a vector bundle of rank $r$, and $\bbP(E) \to X$ the associated projective bundle.
Let $X'$ be the disjoint union of $r$ copies of $X$.
We have $\mathsf{A}$-model maximal F-bundles $(\cH, \nabla)/B_{\bbP(E)}$, and $(\cH', \nabla')/B_{X'}$ for $\bbP(E)$ and $X'$ respectively.

\begin{theorem} \label{thm:projective_bundle}
  There is a canonical isomorphism of the maximal F-bundles between $(\cH, \nabla)$ and $(\cH', \nabla')$, over an analytic domain $U$ in $B_{\bbP(E)}$ and an analytic domain $U'$ in $B_{X'}$.
  The union of the different choices of $U$ is connected and nonempty, as is the union of the different choices of $U'$.
\end{theorem}

\begin{proof}
  As in the proof of \cref{thm:blowup}, \cref{lem:convergence} reduces the non-archimedean statement to the formal version, which is shown by Iritani-Koto in \cite[Theorem 5.1]{Iritani_Koto}.
\end{proof}

\section{\texorpdfstring{$G$}{G}-atoms} \label{sec:Gatoms}

Atoms are elementary F-bundle building blocks which arise from the interplay of Gromov-Witten and motivic information. As it is usually the case, one can build the corresponding theory by using some universal framework of motivic cohomology, such as the framework of Nori motives  \cite{nori-tifr,hms-periods.nori}, or the framework of Andr\'{e} motivated cycles \cite{Andre_Pour_une_theorie_inconditionnelle_des_motifs,Andre-book} (for varieties defined over subfields of $\bbC$). Alternatively, for a more concrete and constructive description, we can work with some realization of the motives,  given by a choice of a suitable Weil cohomology theory.

We will use the conventional notion of a Weil cohomology theory, i.e.\ a cohomology theory on smooth projective varieties which is functorial for pullbacks, which is equipped with a Tate twist, with trace and cycle class maps, and which satisfies the K\"{u}nneth formula and Poincar\'{e} duality. We will not spell out the details here but the precise definition and axioms can be found in standard sources, e.g \cite[Definition~1.2.13]{murre} and \cite[Chapter~45,tag~0FFG]{stacksproject}.  To avoid lengthy qualifications we choose to work with Andr\'{e} motives and only consider Weil cohomology theories that factor through Andr\'{e} motives. In practical terms this means that all Weil cohomology theories we work with, arise as tensor functors from the Karoubi closed symmetric monoidal category $\CA{\kay}$ of pure Andr\'{e} motives over $\kay$ to the monoidal category of graded $k$-vector spaces or to the monoidal category of graded $k$-linear $G$-representations of a reductive group $G$.

To access the full power and flexibility of the atomic compositions that we will introduce, we would like  to be able to treat the $\mathsf{A}$-model F-bundles associated with general Weil cohomology theories in a unified manner. To that end, we will need to keep track of the slightly complex
interaction between the symmetries of the field of definition  of our varieties and the symmetries of the coefficient field of the particular Weil cohomology theory we are  utilizing.

\subsection{Mumford-Tate groups in the \texorpdfstring{\nc}{nc} setting} \label{ssec:MTnc} \ 
Let again $\kay \subset \bbC$ and let $\CA{\kay}$ denote the category of pure Andr\'{e} motives of smooth projective varieties over $\kay$. Let $k$ be a field of characteristic zero, and 
let $G$ be a proreductive group over $k$. We  also fix a central element $\epsilon_{G} \in G$ of order $2$.

\begin{definition}  \label{def:MTsymmetry} 
A  $(G,\epsilon_{G})$-\strongemph{symmetric}  Weil cohomology theory defined on smooth projective $\kay$-varieties is a tensor functor $H^{\bullet}(-) \, \colon \, \CA{\kay} \ \to \ \Rep_{k}(G)$
to $\bbZ$-graded representations of $G$, which commutes with duality and satisfies the $\bbZ/2$ folding 
conditions
\begin{itemize}
    \item The cohomology group $H^2(\bbP^1_{\kay})$ is a trivial rank one  $G$-module.
    \item For any smooth projective $\kay$-variety $\mathfrak{X}$ and any $i \in \bbZ$, the element $\epsilon_{G}$ acts on $H^i(\mathfrak{X})$ by $(-1)^i$.
\end{itemize}
\end{definition}

\begin{remark} \label{rem:symmetryG}
\ {\bfseries (a)} \ As discussed in section~\ref{subpar:motivess}, a $(G,\epsilon_{G})$-symmetric Weil cohomology theory can be described equivalently as  a tensor functor $H^{\bullet}(-) \, \colon \, \CA{\kay}/\mathsf{Tate} \ \to \ \ \mathsf{Rep}_{k}(G)$ from the category of Tate periodized Andr\'{e} motives to the category of representations of $G$.

{\bfseries (b)} \ 
The notion of $(G,\epsilon_{G})$-symmetry of a Weil cohomology theory is designed so that it captures a realization of pure motives over $\kay$ in which the cohomological grading is folded to a $\bbZ/2$-grading, and the cohomologies are considered modulo Tate twists. In particular $(G,\epsilon_{G})$-symmetric Weil cohomology theories are candidates for cohomology theories that could potentially be extended to produce realizations of $\nc$-motives. The following are the standard examples we will be working with.
\end{remark}

\begin{example} \label{ex:Z/2graded}
    Let $\kay \subset \bbC$, and let $k$ be a field with $\op{char}\, k = 0$. Let 
    $\ztwogr \coloneqq \bbZ/2 = \{0,1\}$ and $\epsilon_{\ztwogr} \coloneqq 1 \in \ztwogr$. The category $\op{Rep}_{k}(\ztwogr)$ is the category of $\bbZ/2$-graded vector spaces. Let $H^{\bullet} \colon \CA{\kay} \to \mathsf{Vect}_{k}$ be any Weil cohomology theory with values in $k$-vector spaces. We can refine $H^{\bullet}$ to a functor 
    $H^{\bullet} \colon \CA{\kay} \to  \op{Rep}_{k}(\ztwogr)$ by letting the generator $\epsilon_{\ztwogr} \in \ztwogr$ act by multiplication by $(-1)^i$ on $H^{i}$. We will refer to the pair $\left(\ztwogr,\epsilon_{\ztwogr}\right)$ as the symmetry of \strongemph{$\bbZ/2$-graded Weil cohomology theories}.
\end{example}

\begin{example} \label{ex:Hochschildgraded}
    Let $k = \kay \subset \bbC$, $\HHgr \coloneqq \Gm$, and $\epsilon_{\HHgr} = -1 \in \HHgr = \Gm$. The category $\mathsf{Rep}_{k}(\HHgr)$ of  finite dimensional $k$-linear representations of $\HHgr$ is the category of $\bbZ$-graded $k$-vector spaces, with the action of $\epsilon_{\HHgr}$ recording the parity of grading. 
    For any smooth projective variety $\mathfrak{X}/\kay$, let $H^i_{\op{Dol}}(\mathfrak{X}) = \bigoplus_{p+q=i} H^{p,q}$, where $H^{p,q} \coloneqq H^q(X, \Omega^p_{X})$ and $X = \mathfrak{X} \times_{\kay} \bbC$. Let $\HHgr = \Gm$ act on $H^{p,q}$ with weight $p-q$. We will refer to the pair $\left(\HHgr,\epsilon_{\HHgr}\right)$ as the symmetry of \strongemph{Hochschild graded Dolbeault cohomology}.
\end{example}

\begin{example} \label{ex:Z/2Hodge}
    Let $\kay \subset \bbC$  and $k = \bbQ$. Let $\MT$ denote \strongemph{the universal classical Mumford-Tate group} \cite{DeligneMilne,milne-MT}, i.e.\ the group associated with the neutral Tannakian category of polarizable pure $\bbQ$-Hodge structures with the forgetful fiber functor. $\MT$ is a proreductive algebraic group over $\bbQ$,  whose algebraic quotients are precisely the Mumford-Tate groups of symmetries of polarizable pure $\bbQ$-Hodge structures. 
    We have a Lefschetz character $\MT \to \Gm$ associated to the action of $\MT$ on $H^2_{B}\big((\bbP^1_{\bbC})^{\op{an}},\bbQ\big)$.
    Let $\hodge = \ker(\MT \to \Gm)$ be the kernel of the Lefschetz character and let $\epsilon_{\hodge} \in \hodge \subset \MT$ be the central element of order two, which acts as $\pm 1$ on $H^{\even/\odd}$ for each $\bbZ/2$-graded Hodge structure. We will refer to the pair  $(\hodge,\epsilon_{\hodge})$ as the symmetry of \strongemph{polarizable pure rational $\bbZ/2$-weighted Hodge structures}. The de Rham and  Hodge theorems imply that  the $\bbZ/2$-graded Betti  cohomology  gives a fiber functor $H^{\bullet}_{B}(-) \, \colon \, \CA{\kay}/\mathsf{Tate} \ \to \ \mathsf{Rep}_{\bbQ}(\hodge)$, i.e.\ $H^{\bullet}_{B}(-)$ is a 
    $(\hodge,\epsilon_{\hodge})$-symmetric Weil cohomology theory. 
\end{example}

\begin{example} \label{ex:MotM}
    Let $\kay \subset \bbC$ and $k = \bbQ$.
    Let $\CA{\kay}$ be the semisimple neutral Tannakian category of pure motivated motives (see \cite{Andre_Pour_une_theorie_inconditionnelle_des_motifs,Andre-book}) with fiber functor the forgetful functor to rational vector spaces. Consider the associated group $\GA{\kay}$ of tensor automorphisms of the fiber functor is the proreductive group scheme over $\bbQ$, whose finite dimensional representations
    Again we have homomorphisms $\iota \colon \Gm \to \GA{\kay}$ and $\GA{\kay} \to \Gm$ controlling the weights of the motives and the action on  $H^{2}_{B}(\bbP^1_{\kay})$. We set $\motM \coloneqq \ker(\GA{\kay} \to \Gm)$, and $\epsilon_{\motM} = \iota(-1)$. Then $\left(\motM,\epsilon_{\motM}\right)$ is the symmetry pair of  \strongemph{Andr\'{e}'s $\bbZ/2$-graded pure motivated motives}. Again the $\bbZ/2$-graded Betti cohomology $H^{\bullet}(-)$  is     $\left(\motM,\epsilon_{\motM}\right)$-symmetric.
\end{example}

\begin{example} \label{ex:compareG} 
Suppose $\kay \subset \bbC$ and $k = \bbQ$, then by definition we have embeddings of proreductive $\bbQ$-groups  
\[
\ztwogr \ \subset \ \hodge \ \subset \ \motM 
\]
so that $\epsilon_{\ztwogr} = \epsilon_{\hodge} = \epsilon_{\motM}$. Furthermore if we consider 
$\kay = \bbC$ and base-change these groups to $\Spec \, \bbC$  we get embeddings of the corresponding  complex proreductive groups 
which can be further augmented to a chain of embeddings 
\[
\ztwogr_{\bbC} \ \subset \ \HHgr_{\bbC} \ \subset \ \hodge_{\bbC} \ \subset \ \motM_{\bbC},
\]
again with matching central elements of order two. Here the embedding 
$\HHgr_{\bbC}  \to \hodge_{\bbC}$ arises as follows. We have a natural central embedding 
$\HHgr_{\bbC} \subset \MT_{\bbC}$ in which $\lambda \in \HHgr_{\bbC} = \bbC^{\times}$ acts as multiplication by $\lambda^{p-q}$ on $H^{p,q}$ for each $\bbZ$-graded $\bbQ$ pure Hodge structure $H$. The Lefschetz character is thus trivial on the subgroup $\HHgr_{\bbC} \subset \MT_{\bbC}$ and so $\HHgr_{\bbC}  \to \hodge_{\bbC}$.
\end{example}

The basic idea for constructing our invariants is the observation that the motivic information captured by a pair $(G,\epsilon_{G})$ naturally amalgamates the analytic F-bundle data arising from \nc-Hodge theory. Since the passage from formal to analytic data will in general break the action of symmetries we will have to exercise  some extra care while encoding the motivic information in the analytic setting.

\subsection{\texorpdfstring{$G$}{G}-atoms of projective manifolds} \label{sec:Gatoms_proj}

In this section we again assume that  $\kay \subset \bbC$, $k$ is of characteristic zero, $G$ is a proreductive group over $k$, $\epsilon_{G} \in G$ is a central element of order two, and $H^{\bullet} \, \colon \, \CA{\kay} \ \to \ \mathsf{Rep}_{k}(G)$ is a $(G,\epsilon_{G})$-symmetric Weil cohomology theory, as in Definition~\ref{def:MTsymmetry}. Fix an algebraic closure $\kbar \supset k$ and consider  the algebraically closed non-archimedean field  $\bbk = \kbar\dbp{\fy^{\bbQ}} \supset k$.

\subsubsection{Germs of analytic Lie groups} \label{sss:germsG}

We introduce the notion of local Lie groups and germs of Lie groups in the non-archimedean setting, similar to the classical notions in differential geometry.
A \strongemph{$\bbk$-analytic local Lie group} consists of data $(U,\imath,\mathbf{e},V,\mathsf{m})$, where $U$ is a $\bbk$-analytic space over $\bbk$, $\imath \colon U \to U$ is an analytic involution, $\mathbf{e} \in U$ is a point fixed by $\imath$, $V \subset U\times U$ is an open
subset and $\mathsf{m} \colon V \to U$ is an analytic map, so that
\begin{enumerate}
\item $U\times \{\mathbf{e}\} \, \cup \, \{\mathbf{e}\}\times U \, \subset \, V$, and $\mathsf{m}(g,\mathbf{e}) = \mathsf{m}(\mathbf{e},g) = g$ for all $g \in V$.
\item $\mathsf{Graph}(\imath) \subset V$, and $\mathsf{m}(\imath{g},g) = \mathsf{m}(g,\imath(g)) = \mathbf{e}$.
\item If $(g_{1},g_{2},g_{3}) \in U^{\times 3}$ are such that the pairs  $(g_{1},g_{2})$, $(g_{2},g_{3})$, $(g_{1},\mathsf{m}(g_{2},g_{3}))$, and $(\mathsf{m}(g_{1},g_{2}),g_{3})$ are all in $V$, then $\mathsf{m}(g_{1},\mathsf{m}(g_{2},g_{3})) = \mathsf{m}(\mathsf{m}(g_{1},g_{2}),g_{3})$.
\end{enumerate}
Local $\bbk$-analytic Lie groups form a category, where a morphism between 
$(U_{1},\imath_{1},\mathbf{e}_{1},V_{1},\mathsf{m}_{1})$ and $(U_{2},\imath_{2},\mathbf{e}_{2},V_{2},\mathsf{m}_{2})$ is given by
an analytic map $\varphi \colon U_{1}\to U_{2}$ which sends $\mathbf{e}_{1}$ to $\mathbf{e}_{2}$, intertwines $\imath_{1}$
with $\imath_{2}$, and is such that $\varphi\times \varphi$ maps $V_{1}$ to $V_{2}$ and intertwines $\mathsf{m}_{1}$ and $\mathsf{m}_{2}$.

A \strongemph{germ} of a $\bbk$-analytic Lie group consists of a $\bbk$-analytic local Lie group $(U,\imath,\mathbf{e},V,\mathsf{m})$ and a (not necessarily analytic)  subgroup $H \subset U$. We will denote such a germ by $(U,H)$ when the rest of the data is understood from the context. A strict morphism between two germs $(U_{1},H_{1})$ and $(U_{2},H_{2})$ is a morphism $\varphi\colon U_{1}\to U_{2}$ of $\bbk$-analytic local Lie groups with $\varphi(H_{1})\subset H_{2}$. 

The category of germs $\mathsf{LieGerms}_\bbk$ of $\bbk$-analytic Lie groups is  the localization of the category of germs of $\bbk$-analytic Lie groups with strict morphisms along the class of morphisms  \linebreak $\varphi\colon (U_{1},H_{1})\to (U_{2},H_{2})$, which induce an isomorphism of $\bbk$-analytic local Lie groups between $U_{1}$ and an open neighborhood of $H_{2}$.

\begin{example} \label{ex:germG}
Suppose $\mathbb{G}$ is a non-archimedean analytic Lie group over $\bbk$. Let $\mathbb{H} \subset \mathbb{G}$ be a subgroup.   Then 
$(\mathbb{G},\mathbb{H})$ is an object in $\mathsf{LieGerms}_{\bbk}$ which is called the analytic germ  of $\mathbb{G}$ along $\mathbb{H}$.
\end{example}

Analytic actions of analytic local Lie groups on non-archimedean analytic spaces are defined in the obvious manner. Suppose $U$ is a $\bbk$-analytic local Lie group which acts analytically on a non-archimedean analytic space $B$. Suppose $H \subset U$ is a (not necessarily analytic) subgroup. Suppose $S$ is a closed analytic subvariety in $B$ and suppose $S \subset B^{H}$, i.e.\ $S$ is contained in the fixed point locus for the $H$-action on $B$. Then for any analytic open neighborhood $S \subset N \subset B$ of $S$ in $B$ we can find an analytic local group neighborhood $H \subset U' \subset U$ of $H$ in $U$, so that the action map $U \times B \to B$ restricts to an action map $U'\times N \to N$. In particular, the analytic Lie group germ of $U$ along $H$ will act on the analytic germ of the space $B$ along $S$.

\subsubsection{Local $G$-atoms.} \label{sssec:localGatoms}

Let $\mathfrak{X}$ be a non-empty, smooth projective variety over $\kay$ and consider the maximal, non-archimedean $\mathsf{A}$-model F-bundle associated to $\mathfrak{X}$, $\kay$, $k$, and $H^{\bullet}(-)$, that we described in Definition~\ref{def:nonarchFkay}.
This F-bundle is constructed as an analytic  object over $k\dbp{\fy}$ and we can base change  it to an analytic object  over the algebraic  closure
$\bbk \supset k\dbp{\fy}$. To avoid unnecessary proliferation of notation we will still write  $(\cH,\nabla)/B_{\mathfrak{X}}$ for the resulting non-archimedean analytic F-bundle over $\bbk$.

Let $G^\an$ denote the Berkovich analytification over $k$ with trivial valuation (see \cite{Berkovich_Spectral_theory}).
It contains an analytic subgroup $G^\beth$, which is a compact closed analytic subdomain, consisting of valuations on the function field of $G$ with centers inside $G$ (see \cite[Definition 1.3]{Thuillier_Geometrie_toroidale}).
Taking base extension from $k$ to $\bbk$, we obtain an inclusion of $\bbk$-analytic groups $G_\bbk^\beth \subset G_\bbk^\an$.
Let $\germG$ denote the germ of $G_\bbk^\an$ along $G_\bbk^\beth$.
By definition $\germG$ contains $G(k)$ as a subgroup, and in particular contains $\epsilon_{G}$.

Since by definition $H^{\bullet}(\mathfrak{X}) \in \op{Rep}_{k}(G)$, we have $G_\bbk^\an$ acting on $(H^{\bullet}(\mathfrak{X})\otimes_{k} \bbk)^{\op{an}}$.
Since the components $B_{\mathfrak{X},q}$ and $B_{\mathfrak{X},t}^{\even}$ of the base $B_{\mathfrak{X}}$ are defined by strict inequalities of norms, they are preserved by the germ $\germG$.
Hence $\germG$ acts on the non-archimedean $\mathsf{A}$-model F-bundle $(\cH,\nabla)/B_{\mathfrak{X}}$.

Since $H^{\bullet}(\mathfrak{X})$ is a finite dimensional representation of the reductive group $G$, it follows that under the action of $\germG$ on $B_{X}$ the fixed points of $G(k)$, $G_\bbk^\beth$ and $\germG$ are all the same, i.e.\ 
\begin{equation} \label{eq:commonfixedpoints}
B_{\mathfrak{X}}^{G(k)} \, = \, B_{\mathfrak{X}}^{G_\bbk^\beth} \, = B_{\mathfrak{X}}^{\germG}.
\end{equation}
To simplify notation we will denote the fixed point locus \eqref{eq:commonfixedpoints} by 
$B^G_{\mathfrak{X}}$. 

\medskip

Since $\epsilon_{G}$ acts as $\pm 1$ on the even/odd parts of $H^{\bullet}(\mathfrak{X})$, it follows that the fixed point locus  $B^G_{\mathfrak{X}}$  is a purely even connected closed non-archimedean smooth $\bbk$-analytic subvariety in $B_{\mathfrak{X}}$.  We will use the subvariety $B^G_{\mathfrak{X}} \subset B_{\mathfrak{X}}$ to define the $G$-atoms for $\mathfrak{X}$.
Consider

\begin{description}
    \item $\tB_{\mathfrak{X}} \to B^{G}_{\mathfrak{X}}$ - the ramified covering given by the spectrum of $\bkappa|_{B^{G}_{\mathfrak{X}}} = \Eu|_{B^{G}_{\mathfrak{X}}}\qup(-)$,
    \item $U_{\mathfrak{X}} \subset B^G_{\mathfrak{X}}$ - the open locus where the number of eigenvalues of $\Eu|_{B^{G}_{\mathfrak{X}}}\qup(-)$ is maximal, i.e.\ the locus over which 
    the reduced cover $\tB_{\mathfrak{X},\op{red}} \to  B^{G}_{\mathfrak{X}}$ is unramified.
    \item $\tU_\mathfrak{X} \coloneqq \tB_{\mathfrak{X},\op{red}} \times_{B_{\mathfrak{X}^{G}}} U_{\mathfrak{X}}$ - the restriction of the cover $\tB_{\mathfrak{X},\op{red}}$ to $U_{\mathfrak{X}}$.
\end{description}

\begin{definition} \label{def:localGatoms} 
The \strongemph{set of local $G$--atoms of $\mathfrak{X}$} is the finite set $\pi_0(\widetilde{U}_{\fX})$ of connected components of  $\widetilde{U}_{\fX}$. 
The multiplicity $\mathsf{mult}_{\alpha}(\fX)$ of an element $\alpha \in \pi_{0}(\widetilde{U}_{\fX})$  is defined as the degree of the cover $\widetilde{U}_{\fX,\alpha} \to  U_{\fX}$, where we write $\tU_{\fX,\alpha}$ for the connected component of $\widetilde{U}_{\fX}$ labeled by $\alpha$. 
\end{definition}

The group $\op{Aut}(\mathfrak{X})$ acts on the set $\pi_{0}(\widetilde{U}_{\fX})$, preserving the multiplicity function. The quotient set $\pi_{0}(\widetilde{U}_{\fX})/\op{Aut}(\mathfrak{X})$ depends only on the isomorphism class of $\mathfrak{X}$. Hence we can consider the disjoint union
\[
\mathsf{Atoms}_{G}^{\kay,\op{loc}} \ \coloneqq \ \bigsqcup_{[\mathfrak{X}]} \, 
\pi_{0}(\widetilde{U}_{\fX})/\op{Aut}(\mathfrak{X})
\]
where the union runs
over all isomorphism classes of non-empty smooth projective varieties defined over $\kay$. We will call $\mathsf{Atoms}_{G}^{\kay,\op{loc}}$ \strongemph{the set of local $G$--atoms of smooth projective $\kay$-varieties}. By definition this set depends only on the base field $\kay$.

\medskip

Next we will define some additional and natural 
elementary equivalences on local $G$-atoms.

\subsubsection{Elementary equivalences from disjoint unions.} \label{sssec:eqdisjointGatoms} 
If $\mathfrak{X}_1$ and $\mathfrak{X}_2$ are  two non-empty smooth projective varieties defined over $\kay$. Then we have
\[
\pi_{0}(\widetilde{U}_{\fX_{1}\sqcup \fX_{2}})  = \pi_{0}(\fX_{1})\sqcup \pi_{0}(\fX_{2}).
\]
In this situation we declare each local atom $\alpha \in 
\pi_{0}(\widetilde{U}_{\fX_{1}})/\op{Aut}(\mathfrak{X}_{1}) \ \subset \ \mathsf{Atoms}_{G}^{\kay,\op{loc}}$ to be \strongemph{disjoint union elementary equivalent} to its image $\alpha'$ under the embedding
\[
\pi_{0}(\widetilde{U}_{\mathfrak{X}_1})/\op{Aut}(\mathfrak{X}_{1}) \ \hookrightarrow \ \pi_{0}(\widetilde{U}_{\mathfrak{X}_1\sqcup \mathfrak{X}_2})/\op{Aut}(\mathfrak{X}_{1}\sqcup \mathfrak{X}_{2}).
\]

\subsubsection{Elementary equivalences from blowups.} \label{sssec:eqblowupGatoms}
If $\mathfrak{X}$ is a pure-dimensional non-empty smooth projective $\kay$-variety containing a  non-empty smooth projective $\kay$-subvariety $\mathfrak{Z}\subset \mathfrak{X}$ of codimension $m\ge 2$, then we have the blowup $\widehat{\mathfrak{X}} \coloneqq \op{Bl}_{\mathfrak{Z}} \mathfrak{X}$ which is also defined over $\kay$. Write $\fX'$
for the disjoint union $\fX' \coloneqq \fX \sqcup \fZ \sqcup \cdots \sqcup \fZ$ of $\fX$ and $(m-1)$-copies of $\fZ$. Consider 

\medskip 

\begin{description}
\item[] {$\left.\left(\tensor[^{\widehat{\fX}}]{\cH}{},\tensor[^{\widehat{\fX}}]{\nabla}{}\right)\right/\!B_{\widehat{\fX}}$} - the maximal non-archimedean analytic $\mathsf{A}$-model  F-bundle of $\widehat{\fX}$;
    \item[] {$\left.\left(\tensor[^{\fX}]{\cH}{},\tensor[^{\fX}]{\nabla}{}\right)\right/\!  B_{\fX}$, $\left.\left(\tensor[^{\fZ}]{\cH}{},\tensor[^{\fZ}]{\nabla}{}\right)\right/\!B_{\fZ}$} - the maximal non-archimedean analytic $\mathsf{A}$-model  F-bundles for $\fX$ and $\fZ$ respectively; 
    \item[] {$\left.\left(\tensor[^{\fX'}]{\cH}{},\tensor[^{\fX'}]{\nabla}{}\right)\right/\! B_{\fX'}$}  - the maximal non-archimedean analytic $\mathsf{A}$-model  F-bundles for $\fX'$.
\end{description}

\medskip

Let  
$\tensor[^{\widehat{\fX}}]{\bkappa}{} \, \colon \, \tensor[^{\widehat{\fX}}]{{\cH_{o}}}{} \ \to \ 
\tensor[^{\widehat{\fX}}]{{\cH_{o}}}{}$, \ 
$\tensor[^{\fX}]{\bkappa}{} \, \colon \, \tensor[^{\fX}]{{\cH_{o}}}{} \ \to \ 
\tensor[^{\fX}]{{\cH_{o}}}{}$, \ 
$\tensor[^{\fZ}]{\bkappa}{} \, \colon \, \tensor[^{\fZ}]{{\cH_{o}}}{} \ \to \ 
\tensor[^{\fZ}]{{\cH_{o}}}{}$, \ and  
$\tensor[^{\fX'}]{\bkappa}{} \, \colon \, \tensor[^{\fX'}]{{\cH_{o}}}{} \ \to \ 
\tensor[^{\fX'}]{{\cH_{o}}}{}$
\ denote the associated operators of quantum multiplication by the respective Euler vector fields.

\medskip

By definition we have 
\begin{equation} \label{eq:FbundleX'split}
\left.\left(\tensor[^{\fX'}]{\cH}{},\tensor[^{\fX'}]{\nabla}{}\right)\right/\! B_{\fX'} \ = \ 
\left.\left(\tensor[^{\fX}]{\cH}{},\tensor[^{\fX}]{\nabla}{}\right)\right/\!  B_{\fX} \, \boxplus \, \underbrace{\left.\left(\tensor[^{\fZ}]{\cH}{},\tensor[^{\fZ}]{\nabla}{}\right)\right/\!B_{\fZ} \, \boxplus \, \cdots \, \boxplus \, \left.\left(\tensor[^{\fZ}]{\cH}{},\tensor[^{\fZ}]{\nabla}{}\right)\right/\!B_{\fZ}}_{(m-1)-\text{times}},
\end{equation}
and similarly
\begin{equation} \label{eq:kappaX'split}
\left(\tensor[^{\fX'}]{{\cH_{o}}}{},\tensor[^{\fX'}]{\bkappa}{}\right)  \ = \ 
\left(\tensor[^{\fX}]{{\cH_{o}}}{},\tensor[^{\fX}]{\bkappa}{}\right) \, \boxplus \, \underbrace{\left(\tensor[^{\fZ}]{{\cH_{o}}}{},\tensor[^{\fZ}]{\bkappa}{}\right) \, \boxplus \, \cdots \, \boxplus \, \left(\tensor[^{\fZ}]{{\cH_{o}}}{},\tensor[^{\fZ}]{\bkappa}{}\right)}_{(m-1)-\text{times}}.
\end{equation}
Applying \cref{thm:blowup}, we obtain nonempty connected open subsets   $\widehat{U} \, \subset \, B_{\widehat{\fX}}$ and 
$U'\, \subset \, B_{\fX'}$  and a canonical isomorphism 
\begin{equation} \label{eq:canblowiso}
\left.\left(\tensor[^{\widehat{\fX}}]{\cH}{},\tensor[^{\widehat{\fX}}]{\nabla}{}\right)\right/\!\widehat{U} \ \cong \ 
\left.\left(\tensor[^{\fX'}]{\cH}{},\tensor[^{\fX'}]{\nabla}{}\right)\right/\!U'
\end{equation}
of the restricted F-bundles and hence a canonical isomorphism of bundles with $\bkappa$-operators:
\begin{equation} \label{eq:kappablowup}
\left.\left(\tensor[^{\widehat{\fX}}]{{\cH_{o}}}{},\tensor[^{\widehat{\fX}}]{\bkappa}{}\right)\right|_{\widehat{U}} \, \cong \, \left.\left(\tensor[^{\fX'}]{{\cH_{o}}}{},\tensor[^{\fX'}]{\bkappa}{}\right)\right|_{U'}  \ = \ 
\Big(\left(\tensor[^{\fX}]{{\cH_{o}}}{},\tensor[^{\fX}]{\bkappa}{}\right) \, \boxplus \, \underbrace{\left(\tensor[^{\fZ}]{{\cH_{o}}}{},\tensor[^{\fZ}]{\bkappa}{}\right) \, \boxplus \, \cdots \, \boxplus \, \left(\tensor[^{\fZ}]{{\cH_{o}}}{},\tensor[^{\fZ}]{\bkappa}{}\right)}_{(m-1)-\text{times}}\Big)\Big|_{U'}.
\end{equation}
The isomorphisms are $G$-equivariant by the uniqueness of the blowup decomposition (see \cref{rem:blowup_uniqueness}).

\smallskip

Next, let $\hU^G$ and $U'^G$ be the $G$-fixed loci.
Recall that in the proof of \cref{thm:blowup}, when we apply \cref{lem:convergence}, we get an annulus in the variable $\fq$ and open disks centered at $0$ in the other variables.
Note that the $G$-fixed locus is a linear subspace, and the $\fq$-direction is fixed because it is algebraic, therefore the loci $\hU^G$ and $U'^G$ are nonempty.
Let 
\[
\maxeigbup_{\widehat{\fX}} \, \coloneqq \,  U_{\widehat{\fX}} \cap \widehat{U}^G \qquad \text{and}  
\qquad \maxeigbup_{\fX'} \, \coloneqq \, U_{\fX'} \cap U'^G
\]
be the maximal open subsets over which the reduced spectral covers of 
\[\Big(\tensor[^{\widehat{\fX}}]{{\cH_{o}}}{},\tensor[^{\widehat{\fX}}]{\bkappa}{}\Big)\Big|_{\widehat{U}^G} \qquad\text{and}\qquad\left.\left(\tensor[^{\fX'}]{{\cH_{o}}}{},\tensor[^{\fX'}]{\bkappa}{}\right)\right|_{U'^G}\] are unramified.
Since $\maxeigbup_{\widehat{\fX}}$ is the complement in the connected $\bbk$-analytic space $\widehat{U}^G$ of the branch subspace for the reduced spectral cover, it is itself connected.
By the same token $\maxeigbup_{\fX'}$ is also a connected open in $U'^G$.

Following our notational convention, we denote these unramified covers by 
$\widetilde{\maxeigbup}_{\widehat{\fX}} \ \to \ \maxeigbup_{\widehat{\fX}}$ and 
$\widetilde{\maxeigbup}_{\fX'} \ \to \ \maxeigbup_{\fX'}$ respectively. By construction we have
\[
\widetilde{\maxeigbup}_{\widehat{\fX}} \, = \, \widetilde{U}_{\widehat{\fX}} \times_{U_{\widehat{\fX}}} \maxeigbup_{\widehat{\fX}} \qquad \text{and} \qquad 
\widetilde{\maxeigbup}_{\fX'} \, = \, \widetilde{U}_{\fX'} \times_{U_{\fX'}} 
\maxeigbup_{\fX'}.
\]
Furthermore, the isomorphism \eqref{eq:kappablowup} induces an isomorphism of covers
\[
\xymatrix{\widetilde{\maxeigbup}_{\widehat{\fX}} \ar[d] \ar[r]^-{\sim} & 
\widetilde{\maxeigbup}_{\fX'}  \ar[d] \\
\maxeigbup_{\widehat{\fX}} \ar[r]^-{\sim} & \maxeigbup_{\fX'},
}
\]
and hence a bijection  of sets of connected components
\[
\pi_{0}(\widetilde{\maxeigbup}_{\widehat{\fX}}) \cong  \pi_{0}(\widetilde{\maxeigbup}_{\fX'}).
\]
Since $\maxeigbup_{\widehat{\fX}} \subset U_{\widehat{\fX}}$ and 
$\maxeigbup_{\fX'} \subset U_{\fX'}$ are connected open subsets, we also have that the natural maps $\pi_{0}(\widetilde{\maxeigbup}_{\widehat{\fX}}) \ \to \ 
\pi_{0}(\widetilde{U}_{\widehat{\fX}})$ and 
$\pi_{0}(\widetilde{\maxeigbup}_{\fX'}) \ \to \ 
\pi_{0}(\widetilde{U}_{\fX'})$ are surjective. 
Finally, note that the identification \eqref{eq:kappaX'split} gives an identification of unramified covers of $U_{\fX'} = U_{\fX}\times U_{\fZ}^{\times (m-1)}$:
\[
\widetilde{U}_{\fX'} \ = \ \widetilde{U}_{\fX}\times U_{\fZ}^{\times (m-1)} \, \bigsqcup \, 
U_{\fX} \times \left(\bigsqcup_{a = 1}^{m-1} U_{\fZ}\times \cdots \times \underset{\text{slot } a}{\widetilde{U}_{\fZ}} \times \cdots U_{\fZ}\right),
\]
and thus we get a canonical bijection of sets of connected components
\[
\pi_{0}(\widetilde{U}_{\fX'}) \ = \ \pi_{0}(\widetilde{U}_{\fX})\, \sqcup \, \underbrace{\pi_{0}(\widetilde{U}_{\fZ}) \sqcup \cdots \sqcup \pi_{0}(\widetilde{U}_{\fZ})}_{(m-1)-\text{times}}.
\]
Altogether we get a correspondence of sets of local $G$-atoms:
\begin{equation} \label{eq:pi0correspondence}
\xymatrix@C-2pc@R-1pc{
& & \pi_{0}(\widetilde{\maxeigbup}_{\widehat{\fX}}) \cong  \pi_{0}(\widetilde{\maxeigbup}_{\fX'}) \ar@{->>}[dl] \ar@{->>}[dr] & & \\
& \pi_{0}(\widetilde{U}_{\widehat{\fX}}) \ar@{->>}[dl] & & \pi_{0}(\widetilde{U}_{\fX})\, \sqcup \, \pi_{0}(\widetilde{U}_{\fZ})^{\sqcup (m-1)} \ar@{->>}[dr] & \\
\pi_{0}(\widetilde{U}_{\widehat{\fX}})/\op{Aut}(\widehat{\fX}) & & & & 
\pi_{0}(\widetilde{U}_{\fX'})/\op{Aut}(\fX').
}
\end{equation}
We declare a local atom $\balpha \in \pi_{0}(\widehat{\fX})/\op{Aut}$ to be \strongemph{blowup elementary equivalent} to a local atom $\balpha' \in \pi_{0}(\fX')/\op{Aut}(\fX')$ if $\balpha$ and $\balpha'$ are related by the correspondence 
\eqref{eq:pi0correspondence}.

\subsubsection{Elementary equivalences from  projective bundles.} 
\label{sssec:eqprojbunGatoms}
Suppose $\mathfrak{X}/\kay$ is a non-empty  smooth projective variety 
and $\mathcal{E}$ is a vector bundle of rank $r\ge 2$ on $\mathfrak{X}$. 
By the results of Iritani and Koto \cite{Iritani_Koto}, and 
Hinault, Yu, Zhang, and Zhang \cite[Section~5]{HYZZ_Decomposition} reviewed in Section~\ref{ssec:LH}, 
we can identify   a non-empty connected domain $\maxeigbup_{\bbP(\cE)}$ in $U_{\bbP(\cE)}$ with the product of $r$ copies of a domain $\maxeigbup_{\fX} \subset U_{\fX}$ so that the corresponding F-bundles and operators $\bkappa$ are compatible. Again, this gives a correspondence between the local $G$-atoms for $\bbP(\cE)$ and the local $G$-atoms for 
$\fX^{\sqcup r}$. We declare that a local atom $\balpha \in \pi_{0}(\widetilde{U}_{\bbP(\mathcal{E}))})/\op{Aut}(\bbP(\mathcal{\mathcal{E}}))$
is \strongemph{Leray-Hirsch elementary equivalent} to a local atom $\balpha' \in \tU_{\mathfrak{X}^{\sqcup r}}/\op{Aut}(\fX^{\sqcup r})$ if $\balpha$ and $\balpha'$ are related by the correspondence.

\subsubsection{\texorpdfstring{$G$}{G}-atoms of \texorpdfstring{$\kay$}{K}-varieties.} \label{sssec:Gatomskay}
With all this in place we can now define $G$-atoms of projective varieties.

\begin{definition} \label{def:G.atoms}
The set of \strongemph{$G$-atoms of smooth projective $\kay$-varieties} is the quotient $\mathsf{Atoms}_{G}^{\kay}$ of the set of local atoms  $\mathsf{Atoms}_{G}^{\kay,\op{loc}}$ by the equivalence relation generated by the elementary equivalences defined in Sections~\ref{sssec:eqdisjointGatoms}--\ref{sssec:eqprojbunGatoms}.
\end{definition}

The set $\mathsf{Atoms}_{G}^{\kay}$ of $G$-atoms is naturally filtered by subsets
\[\mathsf{Atoms}_{G,\dim \leq 0}^{\kay} 
\ \subset \ \mathsf{Atoms}_{G,\dim \leq 1}^{\kay} \ \subset \ \cdots,  \qquad 
\mathsf{Atoms}_{G}^{\kay}=\bigcup_{d\geq 0}\mathsf{Atoms}_{G,\dim \leq d}^{\kay},
\]
where an atom $\balpha$ belongs to  $\mathsf{Atoms}_{G,\dim \leq d}^{\kay}$ if there exists 
a smooth projective $\kay$-variety of dimension $\leq d$ whose atomic composition contains $\balpha$. 

Combined with the weak factorization theorem of Abramovich, Karu, Matsuki, and W\l{}odarczyk  \cite{Wlodarczyk_factorization,Abramovich_Torification_and_factorization} this makes $G$-atoms candidates for obstructions to birational equivalence.
Concretely we have the following simple

\begin{proposition} \ {\bfseries (Non-rationality criterion) } \ \label{prop:Gnon-rational}
Let $\fX$ be a smooth projective $\kay$-variety of dimension $d \geq 2$. Suppose we have a $G$-atom $\balpha \in \Atoms_{G}^{\kay}$ which appears in the atomic composition of $\fX$ and is such that $\balpha \notin \Atoms_{G,\dim \leq d-2}^{\kay}$. Then $\fX$ can not be birationally equivalent to $\bbP^{d}_{\kay}$.
\end{proposition}
\begin{proof} Suppose $\fX$ is birational to $\bbP^{d}_{\kay}$. By the weak factorization theorem 
\cite{Wlodarczyk_factorization,Abramovich_Torification_and_factorization} a birational  map between $\bbP^{d}_{\kay}$ and $\fX$ will be given by a backward/forward sequence of blowups and blow-downs with smooth centers. Since the centers must have codimension at least two, and the  atomic composition of $\bbP^{d}_{\kay}$ consists of $d+1$ copies of the $G$-atom of a point, the blowup equivalence relation implies  that the atomic composition of $\fX$ must be comprised only of atoms in $\Atoms_{G,\dim \leq d-2}^{\kay}$. This is a contradiction. 
\end{proof}

\paragraph{Chemical formulas of \texorpdfstring{$\kay$}{K}-varieties.} \label{par:CFkay}
\ Write $\bbZ^{\oplus \, \mathsf{Atoms}^{\kay}_{G}}$  for the free abelian group generated by $\mathsf{Atoms}^{\kay}_{G}$, and  $\bbZ^{\oplus \, \mathsf{Atoms}^{\kay}_{G}}_{\geq 0}$ for the additive monoid of free linear combinations of atoms with non-negative integer coefficients or. Equivalently,  we can think of 
of $\bbZ^{\oplus \, \mathsf{Atoms}^{\kay}_{G}}$ (respectively $\bbZ^{\oplus \, \mathsf{Atoms}^{\kay}_{G}}_{\geq 0}$) as the ring (respectively semiring) of finitely  supported $\bbZ$-valued (respectively $\bbZ_{\geq 0}$-valued) functions on $\mathsf{Atoms}^{\kay}_{G}$.

Splitting the  product of varieties into its atoms gives rise to a 
well-defined map
\[
\mathsf{Atoms}^{\kay}_{G} \, \times \, \mathsf{Atoms}^{\kay}_{G} \  
\longrightarrow  \ \bbZ^{\oplus \, \mathsf{Atoms}^{\kay}_{G}}_{\ge 0}.
\]
This map endows the additive  monoid 
$\bbZ^{\oplus \, \mathsf{Atoms}^{\kay}_{G}}_{\ge 0}$
with the structure of a commutative associative unital semiring, and similarly endows the free abelian group 
$\bbZ^{\oplus \, \mathsf{Atoms}^{\kay}_{G}}$
with the structure of a unital commutative ring.

\medskip

Note that the 
assignment
which sends an isomorphism class of varieties $[\fX]$ to its atomic composition:
\[
\xymatrix@C+0.7pc@M+0.3pc{ 
\CF_{G}(\mathfrak{X}) = {\displaystyle \hspace{-1pc} \sum_{\alpha \in \pi_{0}(\tU_{\fX})/\op{Aut}(\mathfrak{X})}} \hspace{-2pc}  \mathsf{mult}_{\mathfrak{X}}(\alpha)
 \cdot \delta_{\balpha}\, \colon \, \hspace{-2pc} & 
\mathsf{Atoms}_{G}^{\kay} \ar[r] & \bbZ_{\geq 0}},
\]
is naturally a finite  multiset. The multiset $\CF_{G}(\mathfrak{X})$ can be thought of as the \strongemph{chemical formula} of $\fX$, i.e.\ the list of the $G$-atoms in the atomic composition of $\fX$ together with their multiplicities.

\paragraph{Additivity properties of \texorpdfstring{$\CF$}{CF}.} \label{par:CF}
The elementary equivalences we built into the definition of $G$-atoms imply that the chemical formula multiset satisfies a number of natural additvity properties:
\begin{enumerate}[wide]
\item[{\bfseries (1)}] \, For two smooth projective $\kay$-varieties $\mathfrak{X}_1$, $\mathfrak{X}_2$, we have
    \[ \CF_G(\mathfrak{X}_1 \sqcup \mathfrak{X}_2) = \CF_{G}(\mathfrak{X}_1) + \CF_{G}(\mathfrak{X}_2). \]
\item[{\bfseries (2)}] \,  For a smooth projective $\kay$-variety $\mathfrak{X}$, a smooth closed subvariety $\mathfrak{Z} \subset \mathfrak{X}$ of codimension $r$, and the blowup $\widehat{\mathfrak{X}}$ of $\mathfrak{X}$ along $\mathfrak{Z}$, we have
    \[ \CF_{G}\left(\widehat{\mathfrak{X}}\right) = \CF_{G}(\mathfrak{X}) + (r-1)\cdot 
    \CF_{G}(\mathfrak{Z}). \]
\item[{\bfseries (3)}] \, For a smooth projective $\kay$-variety $\mathfrak{X}$, a rank $r$ vector bundle $V$ on $\mathfrak{X}$, and the associated projective bundle $\bbP(V)$, we have
    \[ \CF_{G}(\bbP(V)) = r \cdot \CF_{G}(\mathfrak{X}) .\]
\end{enumerate}    

\medskip

Combined with Bittner's  relations \cite{Bittner}, these properties imply  that the multiset 
assignment $\fX \mapsto \CF_{G}(\mathfrak{X})$
induces a well defined ring 
homomorphism 
\[
\cf_{G} \colon K_0\left(\mathsf{Var}_{\kay}\right) \ \longrightarrow  \ \bbZ^{\oplus \, \mathsf{Atoms}^{\kay}_{G}}
\]
from the Grothendieck ring of varieties $K_0\left(\mathsf{Var}_{\kay}\right)$ to the ring 
$\bbZ^{\oplus \, \mathsf{Atoms}^{\kay}_{G}}$
The homomorphism $\cf_{G}$ can be thought of as the function assigning to each $K$-class of varieties its chemical formula.

\bigskip

Specific chemical formulas are pretty difficult to compute as they require a detailed knowledge of quantum multiplication. Still we have the 
following general

\begin{lemma} \label{lem:GatomsnefK}
Let $X$ be a connected smooth projective variety over an algebraically closed field 
$\kay = \kaybar$ with a numerically effective canonical class.
Then the $G$-atomic composition of $X$ consists of a single atom
$\boldsymbol{\eta}(X)$.
\end{lemma}
\begin{proof}
    Let $(T_i)$ be a homogeneous basis of $H^{\bullet}(X)$, let $(t_i)$ be the associated formal coordinates, and let $(T^i)$ be the dual basis with respect to the Poincar\'{e} pairing.
    Recall that the big quantum product is given by
    \[ T_i \qup T_j = \sum_r \frac{\partial^3 \Phi}{\partial t_i \partial t_j \partial t_r} T^r, \]
    where
    \[
    \frac{\partial^3\Phi}{\partial t_i \partial t_j \partial t_r} = \sum_{n\ge 0,\beta} \frac{q^\beta}{n!} \sum_{i_1,\dots,i_n} \braket{T_i T_j T_r T_{i_1}\cdots T_{i_n}}_{0,n+3}^{\beta} t_{i_1}\cdots t_{i_n} .\]
    Since $K_X$ is nef, it follows from the virtual dimension calculation that if $T^r$ has nonzero coefficient in $T_i \qup T_j$, then
    \[ \deg T^r \ge \deg T_i + \deg T_j. \]
    We conclude that the spectrum of the Euler vector field action on $H^{\bullet}(X)$ contains only one element, coming from $H^0(X)$.
\end{proof}
 
\paragraph{Change of group.} \label{par:changeG}  \ Let $\iota \colon \tG \subset G$ be a reductive subgroup of $G$ defined over $\bbQ$. Assume also that the central order two element 
$\epsilon_{G} \in G$ is contained in $\tG$. Then the fiber functor \linebreak 
$\tH^{\bullet} \colon \CA{\kay} \to \mathsf{Rep}_{k}(\tG)$ which is the composition of $H^{\bullet} \colon \CA{\kay} \to \mathsf{Rep}_{k}(G)$ with the pullback by the inclusion 
$\iota^{*} \colon \mathsf{Rep}_{k}(G) \to \mathsf{Rep}_{k}(\tG)$ is a $(\tG,\epsilon_{G})$-symmetric Weil cohomology theory. 

Suppose $\fX$ is  a smooth projective $\kay$-variety. By definition 
$H^{\bullet}(\fX)$ and $\tH^{\bullet}(\fX)$ coincide as super $k$-vector spaces and only differ as representations: $H^{\bullet}(\fX)$ is a representation of $G$ and $\tH^{\bullet}(\fX)$  is the same vector space viewed as a representation of the subgroup $\tG \subset G$. Therefore, the constructions of the two maximal 
$\mathsf{A}$-model $\bbk$-analytic F-bundles associated with these symmetric Weil cohomology theories are compatible. In concrete terms, the maximal 
$\mathsf{A}$-model $\bbk$-analytic F-bundle associated with the $(\tG,\epsilon_{G})$-symmetric Weil cohomology theory $\tH^{\bullet}$ equals the maximal 
$\mathsf{A}$-model $\bbk$-analytic F-bundle $(\cH,\nabla)/B_{\fX}$ associated with the $(G,\epsilon_{G})$-symmetric Weil cohomology theory $H^{\bullet}$ but instead of considering 
$(\cH,\nabla)/B_{\fX}$ as a $G$-equivariant F-bundle, one must consider it as a $\tG$-equivariant F-bundle. Thus we have an inclusion $B_{\fX}^{G} \subset B_{\fX}^{\tG}$ and the spectral cover of $\bkappa|_{B_{\fX}^{G}}$ is the restriction of the spectral cover of 
$\bkappa|_{B_{\fX}^{\tG}}$ to the submanifold $B_{\fX}^{G}$.

Tracing the way the components of the unramified part of the reduced  spectral cover split under this restriction gives us a natural map of local atoms
\[
\iota^{\diamond,\op{loc}} \colon \mathsf{Atoms}_{\tG}^{\kay,\op{loc}} \ \longrightarrow \ \bbZ^{\oplus \, \mathsf{Atoms}_{G}^{\kay,\op{loc}}}_{\geq 0},
\]
and also a natural map  
\[
\iota^{\diamond} \colon \mathsf{Atoms}_{\tG}^{\kay} \ \longrightarrow \ 
\bbZ^{\oplus \, \mathsf{Atoms}_{G}^{\kay}}_{\geq 0}
\]
of atoms.

Thus we tautologically get a natural compatibility of the corresponding chemical formula multisets, namely
\[
\CF_{G}(\fX)\circ \iota^{\diamond} \, = \, \CF_{\tG}(\fX).
\]

\medskip

\subsubsection{Reduction to $G$-invariants.} \label{ssec:Ginv}
For computational utility it is very useful to note that the local atoms of a projective variety, are captured completely by the eigenvalues for the action of the Euler vector field on the $G$-invariant part of the fiber of the F-bundle. To explain this we will need some preliminary statements. Fix a field $k$ of characteristic zero. Suppose $G$ is an algebraic reductive group over $k$, and $\epsilon_{G} \in G$ is a central element of order $2$.

Suppose $(A,\fp,\mathbf{1}_{A})$ is a finite dimensional commutative unital superalgebra over some algebraic closure $\kbar \supset k$. Suppose the reductive group $G$ acts linearly on $A$, preserving the unital algebra structure, and so that $\epsilon_{\mathbf{G}}$ acts as the parity operator. Because $\epsilon_{\mathbf{G}}$  is central, we will have that 
the $G$-fixed part $A^{G}$ of $A$ is a totally even subalgebra in $A$. Let $a \in A^{G}$ be a $G$-invariant element. Multiplication by $a$ gives a linear operator on the finite dimensional vector space $A$ which automatically will preserve the subalgebras $A^{G} \subset A^{\even} \subset A$. 

Given a linear operator $T \colon V \to V$ on a finite dimensional $\kbar$-vector space $V$ we will 
write $\mathsf{redspec}(T) \, \subset \, \kbar$ for the reduced spectrum of $T$,  i.e.\ for the set of distinct eigenvalues of $T$. 

With this notation we now have the following

\begin{lemma} \label{lem:Ginvariants}  If $a \in A^{G}$, then multiplication by $a$ has the same eigenvalues on all of $A$ and on the 
subalgebra of $G$-invariants $A^{G}$. More precisely we have
\[
\mathsf{redspec}(a\fp\, (-)) \ = \ \mathsf{redspec}\left( (a\fp\, (-))|_{A^{\even}}\right) \ = \ \mathsf{redspec}\left( (a\fp\, (-))|_{A^{G}}\right).
\]
\end{lemma}

\begin{proof}
We have a natural algebra homomorphism $\mathsf{ev}_{a} \, \colon \,\kbar[x] \, \to \, A^{G} \, \subset  A$, given by $\mathsf{ev}_{a}(f(x)) = f(a)$. Since $A$ is finite dimensional over $\kbar$, the homomorphism $\mathsf{ev}_{a}$ has a non-trivial kernel ideal which will be generated by a monic polynomial $p(x)$, i.e.\ $\ker(\mathsf{ev}_{a}) = (p) \, \triangleleft \, \kbar[x]$. 

Since 
$\kbar[x]/(p)  \hookrightarrow A^{\even} \subset A$ is an injective algebra homomorphism, it follows that  the map 
$\Spec(A^{\even}) \, \to \,  \Spec (\kbar[x]/(p))$ is dominant. If we decompose 
$p$ into a product of linear factors $p(x) = \prod_{i=1}^{s} (x-\lambda_{i})^{d_{i}}$ with $\lambda_{i} \neq \lambda_{j}$ if $i \neq j$,
then $\kbar[x]/(p) = \oplus_{i=1}^{m} \, \kbar[x]/((x-\lambda_{i})^{d_{i}})$ and so $\Spec (A^{\even}))$ dominates the closed zero dimensional subscheme $M  = \sqcup_{i = 1}^{m} \, \Spec(\kbar[x]/(x-\lambda_{i})^{d_{i}}) \ \subset \, \mathbb{A}^1_{\kbar}$. Thus $\Spec(A)$ is a disconnected affine superscheme which is a disjoint union $\Spec(A) = \sqcup_{i = 1}^{m} \, S_{i}$, where $S_{i}$ is the preimage of  the $i$-th connected component $\Spec(\kbar[x]/((x-\lambda_{i})^{d_{i}}))$ of $M$ under the map \linebreak 
$\Spec(A) \to \Spec(A^{\even}) \to M$.

Let $A_{i}$ be the algebra of global functions on $S_{i}$. Then $A$ decomposes into a direct sum $A = \oplus_{i=1}^{m} \, A_{i}$ of unital 
commutative super algebras. Since $a$ is $G$-invariant, this decomposition is a decomposition of unital algebras which are $G$-representations, and so we must have  $G\cdot \mathbf{1}_{A_{i}} = \mathbf{1}_{A_{i}}$. Similarly, the projection $a_{i}$ of $a$ into $A_{i}$ must be $G$-invariant. But by construction, the minimal polynomial of 
$(a\fp (-))|_{A_{i}} = a_{i}\fp (-) \colon A_{i} \to A_{i}$ is equal to 
$(x-\lambda_{i})^{e_{i}}$ for some $0< e_{i} \leq d_{i}$, and so 
for every $i = 1, \ldots, m$ we have 
\[
\mathsf{redspec}\left((a\fp \, (-))|_{A_{i}^{G}}\right) = \mathsf{redspec}\left((a\fp \, (-))|_{A_{i}^{\even}}\right) = \mathsf{redspec}\left((a\fp \, (-))|_{A_{i}}\right) = \{\lambda_{i}\}.
\]
This concludes the proof of the lemma. 
\end{proof}

\medskip

Suppose now  $\mathfrak{X}/\kay$ is a smooth projective variety, and 
suppose $(\cH,\nabla)/B_{\mathfrak{X}}$ is the corresponding  $\mathsf{A}$-model non-archimedean $\bbk$-analytic F-bundle associated with 
$\mathfrak{X}$, $k \subset \bbk$, and a $(G,\epsilon_{G})$-symmetric Weil cohomology theory $H^{\bullet}(-)$  on $\kay$-varieties.

Recall that $(\cH,\nabla)/B_{\mathfrak{X}}$ is a $G(k)$-equivariant maximal F-bundle, such that for $\cH_{o} = \cH|_{B_{\mathfrak{X}}\times \{u=0\}}$ the induced quantum product $\qup \colon \cH_{o}^{\otimes 2} \to \cH_{o}$ and the Euler vector field $\Eu \in H^0(B_{\mathfrak{X}},T_{B_{\mathfrak{X}}}) = H^0(B_{\mathfrak{X}},\cH_{o})$ are preserved by $G(k)$. The group $G(k)$ acts on 
the restriction $\cH_{o}|_{B^G_{\mathfrak{X}}}$ by bundle automorphisms and so the quantum multiplication by $\Eu|_{B^G_{\mathfrak{X}}}$ preserves the subbundle $(\cH_{o}|_{B^G_{\mathfrak{X}}})^{G(k)} \subset \cH_{o}|_{B^G_{\mathfrak{X}}} $ of $G(k)$-invariants, which is also equal to the bundle of $\germG$-invariants\footnote{Note that since $G(k)$ acts trivially on the base $B_{\mathfrak{X}}^{G}$ and by linear analytic automorphisms on $\cH_{o}$, it follows that the linear action of $G(k)$ along the fibers of $\cH_{o}|_{B^G_{\mathfrak{X}}}$ extends to the analytic action of $\germG$ corresponding to the linear action of $G_{\bbk}$ on $H^{\bullet}(\mathfrak{X})\otimes_{k} \bbk$. Again since $H^{\bullet}(\mathfrak{X})$ is a finite dimensional representation of the proreductive group  $G(k)$, the fixed points of the $\germG$-action on $(H^{\bullet}(\mathfrak{X})\otimes_{k} \bbk)^{\op{an}}$ will be equal to 
$(H^{\bullet}(\mathfrak{X})^{G(k)}\otimes_{k} \bbk)^{\op{an}}$, i.e.\ to the fixed points of the $G(k)$-action on 
$(H^{\bullet}(\mathfrak{X})\otimes_{k} \bbk)^{\op{an}}$.}. Let $\mathfrak{eu}$ denote  the restriction of the 
operator $\Eu|_{B^G_{\mathfrak{X}}}\qup (-)$ to the subbundle $\left(\cH_{o}|_{B^G_{\mathfrak{X}}}\right)^{G(k)}$.
Let  $\mathfrak{S}_{\mathfrak{X}} \to B^G_{\mathfrak{{X}}}$ be the spectral cover parametrizing the eigenvalues of the operator 
\[
\mathfrak{eu} \  \colon \left(\cH_{o}|_{B^G_{\mathfrak{X}}}\right)^{G(k)} \to \left(\cH_{o}|_{B^G_{\mathfrak{X}}}\right)^{G(k)}.
\]
By construction $\mathfrak{S}_{\mathfrak{X}} \subset \tB_{\mathfrak{X}}$ is a closed analytic subspace in the spectral cover $\tB_{\mathfrak{X}}$ parametrizing the eigenvalues of 
$\Eu|_{B^G_{\mathfrak{X}}} \, \qup (-)$ acting on all of 
$\cH_{o}|_{B^G_{\mathfrak{X}}}$. However, since $\cH_{o}$ is a bundle of unital commutative super-algebras and since $\Eu$ is even, Lemma~\ref{lem:Ginvariants} implies that the inclusion $\mathfrak{S}_{\mathfrak{X}} \subset \tB_{\mathfrak{X}}$ induces an identification
$\mathfrak{S}_{\mathfrak{X},\op{red}}  = \tB_{\mathfrak{X},\op{red}}$ of the associated reduced $\bbk$-analytic spaces. 
Therefore we can also describe  $\tU_{\mathfrak{X}} \to U_{\mathfrak{X}}$ as the unramified part of the reduced spectral cover $\mathfrak{S}_{\mathfrak{X},\op{red}} \to B^G_{\mathfrak{X}}$ for $\Eu\qup (-)$ acting on the bundle of $G(k)$-invariants, or equivalently the bundle of $\germG$-invariants.

\medskip

\subsection{\texorpdfstring{$G$}{G}-atomic F-bundles} \label{ssec:GatomicF}

The abstract definition of atoms in Section~\ref{sssec:Gatomskay} is universal but hard to use in practice.
In this section, we introduce the notion of $G$-atomic F-bundle, which will serve as a simpler version of $G$-atom for the purpose of this paper.

\subsubsection{Definition of an atomic F-bundle} \label{sssec:defatomicF}

Next we will consider actions of germs of analytic Lie groups on F-bundles.
Let $k$ be a field of characteristic zero, and let $\bbk$ be  an algebraically closed non-archimedean field containing $k$,  extending the trivial valuation on $k$.
Fix a pair $(G,\epsilon_{G})$ as in \cref{ssec:MTnc}.
More precisely, we replace $G$ by any finite dimensional quotient.

Let $G$, $G^\beth \subset G^\an$, $G_\bbk^\beth \subset G_\bbk^\an$ and $\germG$ be as in \cref{sssec:localGatoms}.
Let $(\cH,\nabla)/B$ be a $\germG$-equivariant $\bbk$-analytic F-bundle where the $\germG$-action on $\bbD$ is trivial.
Let $b \in B$ be a rigid point fixed by $\germG$ and $(B,b)$ the germ of $B$ at $b$.
Restricting $(\cH,\nabla)/B$ to $(B,b)$, we obtain a \strongemph{$\germG$-equivariant $\bbk$-analytic F-bundle germ}.

\begin{definition} \label{def:GatomicFbundle}
A \strongemph{$(G,\epsilon_{G})$-atomic F-bundle} (or a \strongemph{$G$-atomic F-bundle} for short) is a $\germG$-equivariant $\bbk$-analytic F-bundle germ $(\cA,\nabla)/(B,b)$ such that 
    \begin{enumerate}[wide]
        \item $(\cA,\nabla)/(B,b)$ is maximal, 
        \item the action of $\epsilon_{G}$ on the super vector bundle $\cA$ is the $\bbZ/2$-grading action, 
        \item the $\Eu$-action on $\cA|_{(B^\germG,b), u = 0}$ has a single eigenvalue.
    \end{enumerate}
\end{definition}

\medskip

In the rest of this section we will discuss the $G$-atomic F-bundles coming from Gromov-Witten theory.  Ultimately the invariants we need will be universal $\mathsf{A}$-model $G$-atomic F-bundles, i.e.\ $\mathsf{A}$-model $G$-atomic F-bundles  associated with some Tannakian theory of motives or some particular realization of motives.  Such $G$-atomic F-bundles are dubbed $G$-atoms and we discuss the most useful examples below. 

\subsubsection{Geometric $G$-atomic F-bundles} 

We use the notations from Section~\ref{sssec:localGatoms}.
By the spectral decomposition theorem (Theorem \ref{thm:K-decomposition}), for every rigid point $b \in U_{\mathfrak{X}}$,  the restriction of the F-bundle $(\cH,\nabla)$ to the germ $(U_{\mathfrak{X}}, b)$ splits into a collection of $G$-atomic F-bundles $\cA_{\mathfrak{X},b,\tilde b}$ indexed by the points $\tilde b$ of the fiber $\tU_{\mathfrak{X},b}$.

\begin{definition} \label{def:G-atom} 
  A \strongemph{geometric $G$-atomic F-bundles} of $\kay$-varieties is a $G$-atomic F-bundle $\cA_{\mathfrak{X},b,\tilde b}$ arising from the above construction.  We consider the equivalence relation on geometric $G$-atomic F-bundles generated by the following:
  \begin{enumerate}
    \item Two geometric $G$-atomic F-bundles are equivalent if the $G$-atomic F-bundles are isomorphic.
    \item Two geometric $G$-atomic F-bundles $\cA_{\mathfrak{X},b_1,\tilde b_1}$ and $\cA_{\mathfrak{X},b_2,\tilde b_2}$ arising from the same variety $\mathfrak{X}$ are equivalent if $\tilde b_1$ and $\tilde b_2$ belong to the same connected component of $\tU_{\mathfrak{X}}$.
  \end{enumerate}
  We denote by $\mathsf{Atoms}_{G}^{\kay,\mathsf{F}}$ the set of geometric $G$-atomic F-bundles modulo the above equivalence relation.
  Given a smooth projective $\kay$-variety $\mathfrak{X}$, let $\CF_G^{\mathsf{F}}(\mathfrak{X})$ denote the multiset of geometric $G$-atomic F-bundles associated to $\mathfrak{X}$.
  Formally, $\CF_G^{\mathsf{F}}(\mathfrak{X})$  is a multiplicity function $
  \CF_G^{\mathsf{F}}(\mathfrak{X}) \colon \CF_G^{\kay,\mathsf{F}}  \to \bbZ_{\geq 0}$ with finite support.
\end{definition}

\begin{proposition} \label{prop:atoms_to_atomic_F-bundles}
  We have a map from the set $\Atoms_G^\kay$ of $G$-atoms of varieties over $\kay$ to the set $\Atoms_G^{\kay,\sF}$ of geometric $G$-atomic F-bundles sending each connected component $\tU_{\fX,\alpha}$ to a $G$-atomic F-bundle $\cA_{\mathfrak{X},b,\tilde b}$ for some rigid point $b\in U_\fX$ and $\tilde b\in\tU_{\fX,b}$.
\end{proposition}
\begin{proof}
The equivalences (1) and (2) in Definition~\ref{def:G-atom}, imply that 
there is a well defined map from local $G$-atoms to geometric $G$-atomic F-bundles. Indeed, suppose  $\alpha \in \pi_{0}(\tU_{\fX})$, let $b \in U_{\fX}$ be a rigid point, and let $x \in \tU_{\fX}$ be a rigid point sitting over $b$. Then in the spectral decomposition (see Theorem~\ref{thm:K-decomposition} ) of the F-bundle germ  $(\cH,\nabla)/(B_{\fX},b)$ there is a factor $(\cA_{\fX,b,x},\nabla)/(\tensor[^x]{{B_{\fX}}}{}\!\!,\!\!\tensor[^x]{b}{})$ corresponding to the generalized $\bkappa$-eigenspace with eigenvalue $x$. This factor is by construction a geometric $G$-atomic 
F-bundle whose equivalence class depends only on $\alpha$ but not the points $b$ and $x$.

This gives a map from the set $\Atoms_G^{\kay,\loc}$ of local $G$-atoms of varieties over $\kay$ to the set $\Atoms_G^{\kay,\sF}$. It remains to check that this map factors through the equivalence relations defined in Sections~\ref{sssec:eqdisjointGatoms}--\ref{sssec:eqprojbunGatoms}.
Concretely, we need to check the additivity properties:
  \begin{enumerate}[wide]
    \item For two smooth projective $\kay$-varieties $\mathfrak{X}_1$, $\mathfrak{X}_2$, we have
    \[ \CF_{G}^{\mathbf{F}}(\mathfrak{X}_1 \sqcup \mathfrak{X}_2) = \CF_{G}^{\mathbf{F}}(\mathfrak{X}_1) + \CF_{G}^{\mathbf{F}}(\mathfrak{X}_2). \]
    \item For a smooth projective $\kay$-variety $\mathfrak{X}$, a smooth closed subvariety $\mathfrak{Z} \subset \mathfrak{X}$ of codimension $r$, and the blowup $\widehat{\mathfrak{X}}$ of $\mathfrak{X}$ along $\mathfrak{Z}$, we have
    \[ \CF_{G}^{\mathbf{F}}\left(\widehat{\mathfrak{X}}\right) = \CF_{G}^{\mathbf{F}}(\mathfrak{X}) + (r-1)\cdot \CF_{G}^{\mathbf{F}}(\mathfrak{Z}). \]
    \item For a smooth projective $\kay$-variety $\mathfrak{X}$, a rank $r$ vector bundle $V$ on $\mathfrak{X}$, and the associated projective bundle $\bbP(V)$, we have
    \[ \CF_{G}^{\mathbf{F}}(\bbP(V)) = r \cdot \CF_{G}^{\mathbf{F}}(\mathfrak{X}) .\]
  \end{enumerate}
  These hold almost tautologically. The $\mathsf{A}$-model F-bundle associated to a disjoint union $\mathfrak{X}_1 \sqcup \mathfrak{X}_2$ is isomorphic to the external direct sum of the respective $\mathsf{A}$-model F-bundles.
  This implies (1).
  The remaining statements (2) and (3)  follow again from Theorems~\ref{thm:blowup} and \ref{thm:projective_bundle}.
\end{proof}

\begin{proposition} \label{prop:iso_of_representations}
  For a $G$-atom of $\kay$-varieties represented by a $G$-atomic F-bundle $(\cA,\nabla)$ over a germ $(B, b)$, the isomorphism class of the $G_{\bbk}$-representation $\cA|_{b,u=0}$ is well-defined and independent of the choice of the representative.
\end{proposition}
\begin{proof}
  It suffices to check the independence for the equivalence in Definition~\ref{def:G-atom}(1).
  Using the notations therein, if $\tilde b_1$ and $\tilde b_2$ belong to a same connected component $\mathfrak{U}$ of $\tU_{\mathfrak{X}}$, we obtain a subbundle $\cA \subset \cH|_{U_{\mathfrak{X}}\times \{u=0\}}$ over $U_{\mathfrak{X}}$ given by generalized eigenspaces of the $\Eu$-action with eigenvalues in $V$.
  Since $G$ is a proreductive group, its finite dimensional representations are rigid (see e.g.\ \cite[Proposition~15.15]{milne-AG}).
  
  Therefore, every fiber of $\cA$ gives the same isomorphism class of $G$-representations.
\end{proof}

\medskip

Again, the set $\mathsf{Atoms}_{G}^{\kay}$ of $G$-atoms is naturally filtered by subsets
\[\mathsf{Atoms}_{G,\dim \leq 0}^{\kay,\mathsf{F}} 
\ \subset \ \mathsf{Atoms}_{G,\dim \leq 1}^{\kay,\mathsf{F}} \ \subset \ \cdots,  \qquad 
\mathsf{Atoms}_{G}^{\kay,\mathsf{F}}=\bigcup_{d\geq 0}\mathsf{Atoms}_{G,\dim \leq d}^{\kay,\mathsf{F}},
\]
where an atom $\balpha$ belongs to  $\mathsf{Atoms}_{G,\dim \leq d}^{\kay,\mathsf{F}}$ if there exists 
a smooth projective $\kay$-variety of dimension $\leq d$ whose atomic composition contains $\balpha$. 

\medskip

Splitting the  product of varieties into its atoms gives rise to a 
well-defined map
\[
\Atoms^{\kay,\mathsf{F}}_{G} \, \times \, \Atoms^{\kay,\mathsf{F}}_{G} \ 
\longrightarrow  \ \bbZ^{\oplus \, \Atoms^{\kay,\mathsf{F}}_{G}}_{\ge 0}.
\]
This map endows the additive monoid $\bbZ^{\oplus \, \Atoms^{\kay,\mathsf{F}}_{G}}_{\ge 0}$
with the structure of a commutative associative unital semiring, and similarly endows the abelian group 
$\bbZ^{\oplus \, \Atoms^{\kay,\mathsf{F}}_{G}}$
with the structure of a commutative ring.

\medskip

Note that Bittner's  relations \cite{Bittner} ensure  that the 
assignment which sends an isomorphism class of varieties to its chemical formula
\[
\xymatrix@1@C+1pc@M+0.5pc{[\mathfrak{X}]\ar[r] & \CF_{G}^{\mathsf{F}}(\mathfrak{X})},
\]
gives rise to a well defined ring 
homomorphism 
\[
\cf_{G}^{\mathsf{F}} \colon K_0\left(\mathsf{Var}_{\kay}\right) \longrightarrow  \bbZ^{\oplus \, \Atoms^{\kay,\mathsf{F}}_{G}}
\]
from the Grothendieck ring of varieties $K_0\left(\mathsf{Var}_{\kay}\right)$ to the ring 
$\bbZ^{\oplus \, \Atoms^{\kay,\mathsf{F}}_{G}}$.
Again the homomorphism $\cf_{G}^{\mathsf{F}}$ can be thought of as the function assigning to each $K$-class of varieties its \strongemph{atomic F-bundle valued chemical formula}, i.e.\ the list of its $G$-atomic F-bundles and their multiplicities.

\medskip

\subsection{Hodge atoms} \label{sec:Hodge_atoms}

In this section we discuss a particularly versatile example of $(G,\epsilon_{G})$-atoms which we call Hodge atoms. These are the atoms corresponding to the symmetry pair from Example ~\ref{ex:Z/2Hodge}. In other words, we take $\kay = \bbC$, $k = \mathbb{Q}$ and we consider the Mumford-Tate pair $(\hodge,\epsilon_{\hodge})$  controlling the polarizable pure rational $\bbZ/2$-graded Hodge structures. We will also fix the rational Betti cohomology $H^{\bullet}_{B}(-,\bbQ)$ as our $(\hodge,\epsilon_{\hodge})$-symmetric Weil cohomology theory. Let $\bbk = \overline{\bbQ}\dbp{\fy^{\bbQ}}$ be the algebraically-closed non-archimedean field of Puiseux series in an auxiliary variable $\fy$. For a smooth complex projective variety $X$ we take $(\cH,\nabla)/B_{X}$ to be the maximal non-archimedean $\mathsf{A}$-model F-bundle associated with $X$, $\kay = \bbC$, $\bbk = \overline{\bbQ}\dbp{\fy^{\bbQ}} \supset k = \bbQ$, and $H^{\bullet}(-,\bbQ)$, that we described in Definition~\ref{def:nonarchF}. By construction 
$(\cH,\nabla)/B_{X}$ is $(\hodge(\bbQ),\epsilon_{\hodge})$-equivariant, and so by specializing Definition~\ref{def:G-atom} to this context we arrive at the notion of  $\hodge$-atoms which we will call \strongemph{Hodge atoms}. Since Hodge atoms are our main invariant let us explicate.

Let $B^{\hodge}_{X} \subset B_{X}$ be the locus of $\hodge(\bbQ)$-fixed points in the non-archimedean super $\bbk$-analytic variety $B_{X}$.
Let $\tB \to B^{\hodge}$ be the ramified covering given by the spectrum of the $\Eu$-action, $U_X \subset B^\hodge$ the locus where the number of eigenvalues of the $\Eu$-action is maximal, and \linebreak $\tU_X \coloneqq \tB_{\op{red}} \times_B U_X$.
The set of \strongemph{local Hodge atoms} associated to $X$ is again  $\pi_{0}(\tU_{X})$ and the  multiplicity of an $\alpha \in \pi_{0}(\tU_{X})$ is defined to be the degree of the corresponding connected component 
$\tU_{X,\alpha}$ of $\tU_{X}$ as a covering of $U_{X}$.

The set $\HAtoms$ of all Hodge atoms of smooth complex projective varieties
is the quotient 
\[
\HAtoms \, \coloneqq \, \left.\left( \bigsqcup_{[X]} \, \pi_{0}(\tU_{X})/\op{Aut}(X)\right)\right/\sim, 
\] 
where the union is taken over isomorphism classes of complex smooth projective varieties, and the equivalence relation $\sim$ is generated by the elementary equivalences corresponding to 
disjoint unions, blowups with smooth centers, and projective bundles.

For future reference let us record a few elementary but useful properties of Hodge atoms.

\begin{lemma} \label{lem:nefK}
Let $X$ be a connected complex smooth projective variety with a numerically effective canonical class.
Then $\HAtoms(X)$ consists of a single atom
$\boldsymbol{\eta}(X)$.
\end{lemma}
\begin{proof}
This is a special case of Lemma~\ref{lem:GatomsnefK}.
\end{proof}

\begin{lemma} \label{lem:acindependence} Let $X$ be a smooth complex projective variety.
  For any rigid point $b \in U_{X}$, the $\hodge_{\bbk}(\bbk)$-representation $(\cH_{o})_b = \cH_{(b,0)}$ is induced from a finite dimensional $\overline{\bbQ}$-linear representation of $\hodge(\overline{\bbQ})$ whose isomorphism class is independent of the choice of $b$.
\end{lemma}
\begin{proof}
This is a special case of Proposition~\ref{prop:iso_of_representations}. By definition 
\[
(\cH_{o})_b \cong H^{\bullet}_{B}(X,\bbk) = H^{\bullet}_{B}(X,\bbQ)\otimes_{\bbQ} \bbk
\]
is equipped with the action of $\hodge_{\bbk}(\bbk)$ induced from the action of $\hodge(\bbQ)$ on $H^{\bullet}_{B}(X,\bbQ)$ by extension of scalars. 

Since $\bbQ \subset \overline{\bbQ} \subset \overline{\bbQ}\dbp{\boldsymbol{y}^{\bbQ}} = \bbk$,
we  can also view the $\bbk$-linear action of $\hodge_{\bbk}(\bbk)$ on 
$H^{\bullet}_{B}(X,\bbk) = H^{\bullet}_{B}(X,\overline{\bbQ})\otimes_{\overline{\bbQ}} \bbk$
as being induced by extension of scalars from the natural action of $\hodge_{\overline{\bbQ}}(\overline{\bbQ})$ on $H^{\bullet}_{B}(X,\overline{\bbQ})$.
But $\hodge_{\overline{\bbQ}}$  is a proreductive group over the  algebraically closed field $\overline{\bbQ}$, and the vector space 
$H^{\bullet}_{B}(X,\overline{\bbQ})$ is finite dimensional, so 
the action $\psi_{b}$ of $\hodge_{\overline{\bbQ}}(\overline{\bbQ})$ on $H^{\bullet}_{B}(X,\overline{\bbQ})$ factors through a completely reducible action of a finite dimensional reductive group over $\overline{\bbQ}$. The irreducible representations of this finite dimensional group are classified by their highest weights and so can not vary continuously with $b$.
  Therefore, the isomorphism class of the representation $\psi_{b}$ is independent of the choice of $b$.
\end{proof}

Using this we can define some basic numerical invariants of Hodge atoms. Let $\balpha \in \HAtoms$ be a Hodge atom. Then there is a smooth projective variety $X$ so that  the equivalence class $\balpha$ is represented by a local atom $\alpha \in \HAtoms^{\op{loc}}(X)$, i.e.\ by a choice of a connected component $\mathfrak{U}_{\alpha}$ of  the unramified part $\tU_{X}$ of the reduced spectral cover 
of the operator $\bkappa = \Eu\qup \, (-) \, \colon \, \cH_{o}|_{B^{\hodge}_{X}} \ \to \ \cH_{o}|_{B^{\hodge}_{X}}$
for the $\bbk$-analytic $\mathsf{A}$-model maximal F-bundle $(\cH,\nabla)/B_{X}$. 

The subbundle $\mathcal{E}^{\alpha} \subset \cH_{o}|_{U_{X}}$ of generalized eigenspaces of $\Eu\qup \, (-)$ corresponding to the eigenvalues parametrized by 
$\tU_{X,\alpha}$ is a $\hodge_{\bbk}(\bbk)$-invariant. By Lemma~\ref{lem:acindependence}, for any rigid point $b \in U_{X}$ the representation of $\hodge_{\bbk}(\bbk)$ on $\mathcal{E}^{\alpha}_{b}$ is induced by extension of scalars from some finite dimensional $\overline{\bbQ}$-linear representation $E^{\balpha}$ of $\hodge_{\overline{\bbQ}}$ whose isomorphism class is independent of $b$. As we have indicated by the notation, the isomorphism class of the $\hodge_{\overline{\bbQ}}$-module $E^{\balpha}$ depends only on the atom $\balpha$ and not on the particular local atom $\alpha$ representing the equivalence class $\balpha$. Indeed, this  follows immediately from the fact that the elementary  equivalences used  defining Hodge atoms all respect the action of $\hodge$. 

\medskip

Thus to every Hodge atom $\balpha$  we have assigned  a finite dimensional $\overline{\bbQ}$-linear representation $E^{\balpha}$ of $\hodge_{\bbQ}$ which is well defined up to isomorphism.
With this notation we now have the following

\begin{definition} \label{def:numinv} For every Hodge atom $\balpha \in \HAtoms$ we define
\begin{enumerate}[wide]
\item the \strongemph{dimension} $\rho_{\balpha}$ of the space of Hodge classes in $\balpha$  as 
$\dim_{\bbC} \, (E^{\balpha}\otimes \bbC)^{\hodge_{\bbC}}$, and
\item the \strongemph{Hodge polynomial} $P_{\balpha}(t)$ of $\balpha$ as the polynomial $P_{\balpha}(t) \in \bbZ[t,t^{-1}]$ whose coefficient at $t^{k}$ is equal to 
$\dim_{\bbC} \, (E^{\balpha}\otimes \bbC)^{k}$, where $(E^{\balpha}\otimes \bbC)^{k}$ is the subspace in $E^{\alpha}\otimes \bbC$ on which the subgroup $\bbC^{\times} \, \cong \, \HHgr_{\bbC}(\bbC) \, \subset \,  \hodge_{\bbC}(\bbC)$ acts with weight $k$.
\end{enumerate}
\end{definition}

\begin{remark} \label{rem:weight}
Suppose the $\overline{\bbQ}$-linear representation $E^{\balpha}$ of $\hodge_{\overline{\bbQ}}$ is obtained by extension of scalars from a $\bbQ$-linear representation of $\hodge$.  But a $\bbQ$-linear representation of $\hodge$ is a pure $\bbZ/2$-weighted $\bbQ$-Hodge structure, i.e.\  an ordinary pure  $\bbQ$-Hodge structures whose bigrading has been folded to the single Hochschild  grading.  
Therefore if $E^{\balpha} = V\otimes_{\bbQ} \overline{\bbQ}$ for 
some pure $\bbQ$-Hodge structure, then $(E^{\balpha}\otimes \bbC)^{k} = \oplus_{p,q\, :\, p-q = k} V^{p,q}$. 
In other words, $P_{\balpha}(t)$ is obtained from the usual two variable  Hodge polynomial of $\sum_{p,q} \, \dim_{\bbC} \, V^{p,q} \, u^{p}v^{q}$ of $V$ after the substitution $u = t, \, v = t^{-1}$.
The symmetry of Hodge numbers of $V$ then implies that in this case the Hodge polynomial $P_{\balpha}(t)$ will be reciprocal (= invariant under $t \mapsto t^{-1}$). 

Of course,  for a general $\hodge_{\overline{\bbQ}}$-representation the associated  Hodge polynomial will not have any symmetry, so for a general Hodge atom the Hodge polynomial $P_{\balpha}(t)$ will not be reciprocal.
\end{remark}

\

\begin{proposition} \label{prop:Ginv1}
    For each Hodge atom $\balpha$  the dimension $\rho_{\balpha}$  of the space of Hodge classes in $\balpha$ satisfies  $\rho_{\balpha} \ge 1$.
\end{proposition}
\begin{proof}
    This follows immediately from the fact that fiber of the Hodge atom at a rigid point of $U_{X}$ is of the form $A_{\alpha}\otimes_{\overline{\bbQ}} \bbk$ where 
    $A_{\alpha}$ is a non-zero unital commutative associative superalgebra over $\overline{\bbQ}$ which is 
    endowed with a $\hodge_{\overline{\bbQ}}(\overline{\bbQ})$-action. In particular there is  a non-zero even $\hodge_{\overline{\bbQ}}(\overline{\bbQ})$-invariant vector, namely the unit $\mathbf{1}_{A_{\alpha}}$ for the multiplication  in $A_{\alpha}$. 
\end{proof}

\begin{remark} \label{rem:atomicFhodgefiber}     
Let $\balpha \in \HAtoms$ be a Hodge atom and let $\alpha \in \pi_{0}(\tU_{X})$ be a local atom representative of $\balpha$. Let $x \in \tU_{X,\alpha}$ be a rigid point, let $b_{x} \in U_{X}$ be its image, and let $\left(\tensor[^x]{\cH}{},\tensor[^x]{\nabla}{})\right)$   the associated $\hodge$-atomic F-bundle over the germ $(\tensor[^x]{B}{},b_{x})$. Then $\dim \, (\tensor[^x]{B}{})^{\hodge(\overline{\bbQ})} \, = \, \rho_{\balpha}$ and the fiber $\tensor[^x]{\cH}{}_{(b_{x},0)}$ is isomorphic to $E^{\balpha}\otimes_{\overline{\bbQ}} \bbk$ for a finite dimensional $\overline{\bbQ}$-linear representation  $E^{\balpha}$ of $\hodge_{\overline{\bbQ}}(\overline{\bbQ})$  which has  Hodge polynomial $P_{\balpha}(t)$ and has a space $(E^{\balpha})^{\hodge(\overline{\bbQ})}$ of Hodge classes of dimension $\dim \, (E^{\balpha})^{\hodge(\overline{\bbQ})} \, = \, \rho_{\balpha}$.
\end{remark}

\medskip

Finally, note that as in the case of the general $G$-atoms, the set of Hodge atoms is filtered by the minimal dimension of a smooth projective variety in whose atomic composition a given atom can appear. Again, this filtration together with the weak factorization theorem \cite{Wlodarczyk_factorization,Abramovich_Torification_and_factorization} imply that Hodge atoms can obstruct birational equivalence. Concretely we have the following

\begin{proposition} \ {\bfseries (Non-rationality criterion) } \ \label{prop:non-rational}
Let $X$ be a smooth complex projective variety of dimension $d \geq 2$. Suppose we have a Hodge atom $\balpha$ which appears in the atomic decomposition of $X$ and is such that $\balpha \notin \HAtoms_{\dim \leq d-2}$. Then $X$ can not be a rational variety.
\end{proposition}
\begin{proof} This is a special case of Proposition~\ref{prop:Gnon-rational}.
\end{proof}

\subsection{Atoms from motivated motives} \label{subsec:motM.atoms}

Suppose that we are in the setting of Example~\ref{ex:MotM}, i.e.\  $\kay \subset \bbC$, $k = \bbQ$, and $H^{\bullet}(-)$ is Betti cohomology. Take $\bbk \coloneqq \overline{\bbQ}\dbp{\fy^{\bbQ}}$. As discussed in section~\ref{par:casekay} the corresponding motivic  Galois group $\motM^{\kay}$ of the category $\CA{\kay}/\mathsf{Tate}$ of Andr\'{e}'s pure $\bbZ/2$-graded motivated motives is a proreductive algebraic group over $\bbQ$.

Let $\mathfrak{X}$ be a smooth projective variety over $\kay$. View the Betti cohomology as a \linebreak $(\motM^{\kay},\epsilon_{\motM^{\kay}})$-symmetric fiber functor 
$H^{\bullet}(-) \, \colon \, \CA{\kay} \ \to \ \mathsf{Rep}_{\bbQ}(\motM^{\kay})$. Then we have a $\motM^{\kay}$-equivariant maximal non-archimedean $\bbk$-analytic $\mathsf{A}$-model F-bundle $(\cH,\nabla)/B_{\mathfrak{X}}$   corresponding to $\mathfrak{X}$, $\kay$, $k$, and $H^{\bullet}(\mathfrak{X})$, obtained by base changing the maximal $k\dbp{\fy}$-analytic F-bundle introduced in Definition~\ref{def:nonarchFkay} from $k\dbp{\fy}$ to $\bbk$. The atomic F-bundles arising from such a $(\cH,\nabla)/B_{\mathfrak{X}}$ are by definition 
$\motM^{\kay}$-equivariant F-bundles defined over the locus $B_{\mathfrak{X}}^{\motM^{\kay}}$ corresponding to algebraic cycles on $\mathfrak{X}$ that are defined over $\kay$. The corresponding multiset of atoms $\Atoms_{\motM^{\kay}}(\mathfrak{X})$ is the most sensitive obstruction to $\kay$-rationality our atom formalism provides.  Note also that if $\mathfrak{X}$ is such that the eigenvalues of $\Eu\qup(-)$ happen to be already defined over $\bbQ\dbp{\fy}$, then by section~\ref{subpar:twistF}, the base of the corresponding atomic F-bundle contains a non-split torus which is twisted by a homomorphism 
$\Gal(\kaybar/\kay) \to \Gal(\overline{\bbQ}/\bbQ)$ from the Galois group of the field of definition, to the Galois group of the coefficient field of our cohomology theory.

\medskip

More invariantly, without making a choice of a fiber functor, we can define the \strongemph{intrinsic motivated atoms} of smooth projective varieties over $\kay$. The category of \nc-Andr\'e motives $\CA{\kay}/\mathsf{Tate}$ is a $\bbZ/2$-graded $\bbQ$-linear semisimple tensor category. Let $\triv \in \mathsf{ob} \, (\CA{\kay}/\mathsf{Tate})$
be the zero dimensional Tate motive. As before, for a smooth projective variety $\fX$ we denote the Andr\'{e} motive of $\fX$ by $h(\fX)$. Then the $\bbZ/2$-graded rational vector space $M(\fX) \coloneqq \op{Hom}_{\CA{\kay}/\mathsf{Tate}}(\triv,h(\fX))$ is intrinsically associated with $\fX$, and the standard construction of Gromov-Witten classes gives again a quantum product 
\[
\qup \, \colon \, M(\fX)\otimes M(\fX) \ \longrightarrow \ M(\fX)\, \widehat{\otimes} \, \bbQ\dbb{\NE(\fX,\bbZ)} \, \widehat{\otimes}\, \widehat{\mathsf{Sym}}\, M(\fX)^{\vee}. 
\]
Since $\NE(\fX,\bbZ)$ parametrizes numerically effective classes of curves in $\fX$ that are defined over $\kay$ and $M(\fX)$ parametrizes motivated cycles in $\fX$ defined over $\kay$, it follows that the actions of $\Gal(\kaybar/\kay)$ on $\NE(\fX,\bbZ)$ and $M(\fX)$ are both trivial. The construction from section~\ref{par:casekay} now produces an overmaximal $\bbk$-analytic $\mathsf{A}$-model F-bundle  
$(\mycalb{M},\bnabla)/\boldsymbol{b}_{\fX}$ and a maximal $\bbk$-analytic $\mathsf{A}$-model F-bundle  
$(\mathcal{M},\nabla)/b_{\fX}$. Here $\boldsymbol{b}_{\fX}$ is a product of an $\bbk$-analytic affine line (for the motivated cycles of cohomological  degree zero), a  polydisk (for the motivated cycles of cohomological degree $>0$), and a tube domain in the analytic torus $(\NS(\fX,\bbZ)_{\op{tf}}\otimes \bbG_{m,\bbk})^{\op{an}}$.
Similarly $b_{\fX}$ is a product an $\bbk$-analytic affine line (for the motivated cycles of cohomological  degree zero), a  polydisk (for the motivated cycles of cohomological degree $>2$), and a tube domain in the analytic torus $(\NS(\fX,\bbZ)_{\op{tf}}\otimes \bbG_{m,\bbk})^{\op{an}}$.

Note that if we use the Betti cohomology fiber functor, then these F-bundles are realized by the $\motM^{\kay}$-invariant part of the F-bundles we described in section~\ref{par:casekay}, i.e.\ passing to Betti cohomology identifies  
$(\mycalb{M},\bnabla)/\boldsymbol{b}_{\fX}$ with $(\omH^{\motM^{\kay}},\bnabla)/\omB_{\fX}^{\motM^{\kay}}$ and similarly identifies  $(\mathcal{M},\nabla)/b_{\fX}$  
$(\cH^{\motM^{\kay}},\nabla)/B_{\fX}^{\motM^{\kay}}$. In particular, by section~\ref{ssec:Ginv}  the unramified part $\tU_{\fX}$  of the reduced spectral cover for the operator $\bkappa$ associated with  $(\cH,\nabla)/B_{\fX}^{\motM^{\kay}}$ is isomorphic  to the unramified part $\tilde{\mathfrak{u}}_{\fX}$ of the reduced spectral cover of the operator $\bkappa$ associated with $(\cH^{\motM^{\kay}},\nabla)/B_{\fX}^{\motM^{\kay}} \cong (\mathcal{M},\nabla)/b_{\fX}$. 

Using $(\mathcal{M},\nabla)/b_{\fX}$ as a basic invariant of $\fX$ we can define the set of 
intrinsic motivated local atoms
of $\fX$ as the set $\pi_{0}(\tilde{\mathfrak{u}}_{\fX})$ of connected components of $\tilde{\mathfrak{u}}_{\fX}$
and the set of intrinsic atoms of $\kay$-varieties as the quotient 
\[
\mathsf{Atoms}_{\op{intr}}^{\kay} \, \coloneqq \, \left.\left( \bigsqcup_{[X]} \, \pi_{0}(\tilde{\mathfrak{u}}_{\fX})/\op{Aut}(\fX)\right)\right/\sim, 
\] 
where the union is taken over isomorphism classes of complex smooth projective varieties, and the equivalence relation $\sim$ is generated by the elementary equivalences corresponding to 
disjoint unions, blowups with smooth centers, and projective bundles.

As explained in the previous paragraph, the choice of Betti cohomology as a fiber functor on $\CA{\kay}$, gives a bijection between the  set of intrinsic motivated atoms $\mathsf{Atoms}_{\op{intr}}^{\kay}$ and the set 
$\mathsf{Atoms}_{\motM^{\kay}}^{\kay}$ of $\motM^{\kay}$-atoms. The intrinsic viewpoint however, has the advantage  that 
the set of atoms $\mathsf{Atoms}_{\op{intr}}^{\kay}$ does not depend on any choices and by construction is equipped with an action of the Galois group $\Gal(\bbk/\bbQ)$ which is an extension of $\Gal(\overline{\bbQ}/\bbQ)$ by the group $\widehat{\bbZ}$ of profinite integers.

\section{New obstructions to rationality} \label{sec:obstructions}

In this section we will use Hodge atoms and more general $G$-atoms to prove new non-rationality results. We will use the setup of Sections~\ref{sec:Gatoms_proj}, \ref{ssec:GatomicF}, and \ref{sec:Hodge_atoms}.
Specifically, we consider the pair  $\left(\hodge,\epsilon_{\hodge}\right)$ defined in Example~\ref{ex:Z/2Hodge}, where $\hodge$ is the Galois group of the Tannakian category of $\bbZ/2$-weighted, polarizable, pure $\mathbb{Q}$-Hodge structures, and $\epsilon_{\hodge} \in \hodge$ is the $\bbZ/2$-grading element. The atoms for this symmetry group are the Hodge atoms and we will use them as obstructions to rationality.

Before we can discuss the atomic composition of specific smooth projective varieties, we need to revisit Givental's calculation of the quantum cohomology and quantum differential equation of complete intersections and recast it in the F-bundle language.

\subsection{\texorpdfstring{$F$}{F}-bundles and Givental's theorem} \label{ssec:givental}
 
Let $X$ be a smooth complex complete intersection of $k \geq 1$ hypersurfaces of degrees $d_{1}, \ldots, d_{k}$ in $\bbP^{N-1}$. We will assume that $d_{i} > 1$ for all $i =1, \ldots, k$, and that $d_{\op{tot}} = d_{1} + \cdots + d_{k} \, \leq \, N$, i.e.\ that $X$ is either a Fano or a Calabi-Yau. Under these assumptions we have $1 \leq k < d_{\op{tot}} \leq N$, 
$\dim X = N - 1 - k \geq 1$, and $-K_{X} = \cO_{X}(N - d_{\op{tot}})$. 

The image $\imageH \subset H^{\bullet}(X,\bbQ)$ of the rational cohomology of $\bbP^{N-1}$ in the rational cohomology of $X$ is the $(N-k)$-dimensional subalgebra spanned over $\bbQ$ by the classes $\{\hpl^{i} \}_{i = 0}^{N-k-1}$, where $\hpl = c_{1}(\cO(1)) \in H^2(X,\bbQ)$ is the hyperplane class.

Let $(\cH,\nabla)/B_{X}$ be the maximal non-archimedean, analytic $\mathsf{A}$-model F-bundle associated with  $X$. Consider the subgroup $\bbZ\hpl \subset \NS(X,\bbZ)_{\op{tf}}$. It gives rise to a subtorus 
\[
\smallbase_{X} \coloneqq (\bbZ\hpl \otimes \bbG_{m,\bbk})^{\op{an}} \, \subset \, (\NS(X,\bbZ)_{\op{tf}}\otimes \bbG_{m,\bbk})^{\op{an}} = B_{X,q}.
\]
The torus $\smallbase_{X}$ is one dimensional with natural coordinate $q$ which is dual to $\hpl$.

The restriction of $(\cH,\nabla)|_{\smallbase_{X}\times \bbD}$ is an analytic F-bundle. The comparison \cite{Givental_Equivariant_Gromov-Witten_invariants} of the Gromov-Witten invariants of $X$ and the equivariant Gromov-Witten invariants of $\bbP^{N-1}$ implies that the restricted connection $\nabla|_{\smallbase_{X}\times \bbD}$ preserves the subbundle 
\[
\smallfiber \coloneqq \imageH \otimes \cO_{\smallbase_{X}\times \bbD} \ \subset \ H^{\bullet}(X,\bbQ)\otimes \cO_{\smallbase_{X}\times \bbD} \, = \, \cH|_{\smallbase_{X}\times \bbD}.
\]
This gives  a new smaller analytic F-bundle $(\smallfiber,\nabla)/\smallbase_{X}$ of rank $N - k = \dim X + 1$ over the one dimensional base $\smallbase_{X}$. If we write $q$ for the coordinate on the analytic affine line $(\bbk\cdot \hpl)^{\op{an}}$ that is dual to  
basis vector $\hpl$, then the meromorphic connection $\nabla$ on $\smallfiber$ becomes
\begin{equation} \label{eq:ansmallnabla}
\left| \
\begin{aligned}
\nabla_{\partial_{u}} & = \partial_{u} - u^{-2} \mathbf{K} + u^{-1} \mathbf{G}, \\
\nabla_{\partial_{q}} & = \partial_{q} + u^{-1}q^{-1} \mathbf{A},
\end{aligned}
\right.
\end{equation}
where $\mathbf{G}$ is the grading operator on $\smallfiber$, $\mathbf{A}$ is the operator of quantum multiplication by $\hpl$, and $\mathbf{K} = (N-d_{\op{tot}})\mathbf{A}$ is the operator of quantum multiplication by $c_{1}(T_{X})$. These operators are coefficients of the restricted quantum  $\nabla|_{\smallbase{X}\times \bbD}$ acting on the subbundle $\smallfiber$ and so we have that $\mathbf{K}$ is equal to the action of $\bkappa|_{\smallbase_{X}}$ on the subbundle $\smallfiber|_{u=0}$, $\mathbf{G}$ is the restriction of $\Gr$ to the subbundle $\smallfiber|_{u=0}$, and $\mathbf{A}$ is the restriction quantum multiplication by $\hpl$ to the subbundle $\smallfiber|_{u=0} \, \subset \, \cH|_{\smallbase_{X}\times \{0\}}$.

In \cite[Section~9]{Givental_Equivariant_Gromov-Witten_invariants} Givental described the sheaf of horizontal sections of $(\smallfiber,\nabla)/\smallbase_{X}$ in terms of the fundamental solution of an explicit auxiliary linear ODE of order $N-k$, which was later dubbed the \strongemph{quantum differential equation}. Givental used his   description to give  a recursive procedure for computing the Gromov-Witten invariants of $X$. 
For us however, the main takeaway from Givental's  theorem is that along the one dimensional base $\smallbase_{X}$ the action of the operator $(\Eu\, \fp (-))|_{\smallbase_{X}}  \ \colon \  \smallfiber \to \smallfiber$ can be computed from the quantum differential equation. 

To unpack this, let us recall Givental's calculation from \cite[Section~9]{Givental_Equivariant_Gromov-Witten_invariants}.  Consider the following setup:
\begin{enumerate}[wide]
    \item Let $\dd$ be a non-negative integer and let
$\Mbar_{0,2}(\bbP^{N-1},\dd)$ denote the moduli stack of genus zero degree $\dd$ stable maps 
$\varphi \colon (C,p_{1},p_{2}) \to \bbP^{N-1}$ with two marked points.  
\item Let $\boldsymbol{e}_{\dd}$ be the top Chern class of the vector bundle on $\Mbar_{0,2}(\bbP^{N-1},\dd)$ whose fiber at $(C,p_{1},p_{2},\varphi)$ is $H^0(C,\varphi^{*}(\cO(d_{1})\oplus \cdots \cO(d_{k}))$.  
\item 
Let $c \in H^2(\Mbar_{0,2}(\bbP^{N-1},\dd),\bbQ)$ denote the first Chern class of the universal tangent line at the first marked point.
\item Write $\ev_{1} \colon \Mbar(\bbP^{N-1},\dd) \to \bbP^{N-1}$ for the evaluation at the first marked point. 
\end{enumerate}
In \cite[Section~9]{Givental_Equivariant_Gromov-Witten_invariants} Givental argued that the equivariant version of \cite[Corollary~6.4]{Givental_Equivariant_Gromov-Witten_invariants} implies that the series 
\[
S(t,\hbar) \, \coloneqq \, e^{\hpl t/\hbar} \sum_{\dd = 0}^{\infty} (X\hookrightarrow \bbP^{N-1})^{*} \ev_{1*}\left(\frac{1}{\hbar +c} \boldsymbol{e}_{\dd}\right) \, \in \, \imageH(\!(\hbar)\!)\dbb{e^t}[t],
\]
is a horizontal section of the formal quantum connection for $X$.
Givental's variables $(t,\hbar)$ are related to our variables $(q,u)$ by the change of variables $q = e^{t}$, $u = -\hbar$. Therefore, we have that $S$ as a function of $(q,u)$ belongs to  
$\imageH(\!(u)\!)\dbb{q}[\log q]$, and  viewed as a formal section of $\smallfiber$ is annihilated by the connection $\nabla$ given by \eqref{eq:ansmallnabla}.

\medskip

Furthermore, in \cite[Theorem~9.1 and Corollary~9.2]{Givental_Equivariant_Gromov-Witten_invariants}
Givental proved:
\begin{enumerate}[wide]
    \item[(i)] $S(t,\hbar)$ is given by the explicit formula
\[
S(t,\hbar)  \, = \, e^{t\hpl/\hbar}\sum_{{\dd}=0}^\infty e^{\dd t} \frac{\prod_{i=1}^k \left(\prod_{m_i=1}^{\dd d_i}(d_i \hpl+m_i\hbar)\right)}{\prod_{m=1}^{\dd} (\hpl+m\hbar)^N} \ \in  \ \imageH(\!(\hbar)\!)\,\dbb{e^t}\,[t],
\]
\item[(ii)] $S(t,\hbar)$ satisfies the so called  \strongemph{quantum differential equation} 
\begin{equation} \label{eq:qODE}
\bigl(\hbar \partial_t\bigr)^{N-k} S= e^t \prod_{i=1}^k\Bigl(d_i\prod_{m_i=1}^{d_i-1} \hbar\bigl(d_i\partial_t+m_i\bigr) \Bigr)\,S,
\end{equation}
and the components $S_{i}(t,\hbar) \in \bbQ(\!(\hbar)\!)\dbb{e^t}[t]$ of $S = \sum_{i=0}^{\dim X} \, S_{i}\cdot \hpl^{i}$ form a basis of the space of solutions of the linear ODE \eqref{eq:qODE}.
\end{enumerate}

\begin{remark} \label{rem:qODE}
\ {\bfseries (a)}  \ Givental's description implies that in fact 
$S(t,\hbar)$ belongs to the subring $\imageH[\hbar,\hbar^{-1}]\dbb{e^t}[t]$ and so its specialization $S(t,1)$ as $\hbar\rightsquigarrow 1$ makes sense. Thus we have the specialized series 
\[
S(t,1) = e^{\hpl t}\sum_{{\dd}=0}^\infty e^{\dd t} \frac{\prod_{i=1}^k \left(\prod_{m_i=1}^{\dd d_i}(d_i \hpl +m_i)\right)}{\prod_{m=1}^{\dd} (\hpl +m)^N} \ \in \, \imageH\dbb{e^t}\,[t] 
\]
or, in terms of the variables $(q,u)$, the $u\rightsquigarrow -1$ specialized series 
\[
S(q,-1) = e^{\hpl\log(q)}\sum_{{\dd}=0}^\infty q^{\dd} \frac{\prod_{i=1}^k \left(\prod_{m_i=1}^{{\dd} d_i}(d_i \hpl +m_i)\right)}{\prod_{m=1}^{\dd} (\hpl+m)^N} \ \in \, \imageH\dbb{q}\,[\log q].
\]

\noindent
{\bfseries (b)} \ Rewriting the equation \eqref{eq:qODE} in terms of the variables $(q,u)$ we get that $S$ as a series in $\imageH[u,u^{-1}]\dbb{q}[\log q]$ satisfies 
\begin{equation} \label{eq:qODEqu}
    \bigl(u\, q \partial_q\bigr)^{N-k} S= (-1)^{N-d_{\op{tot}}}
    q \prod_{i=1}^k\Bigl(d_i\prod_{m_i=1}^{d_i-1} u\bigl(d_i\, q\partial_q+m_i\bigr) \Bigr)\,S. 
\end{equation}
Note that the component $S_{0}(q,u)$ solving \eqref{eq:qODEqu} is particularly simple. 
It is given by the hypergeometric series
\[
S_{0} = e^{\hpl\log(q)}\sum_{{\dd}=0}^\infty q^{\dd} \frac{\prod_{i=1}^k \left(\prod_{m_i=1}^{{\dd} d_i}(d_i \hpl+m_i)\right)}{\prod_{m=1}^{\dd} (\hpl+m)^N} \ \in \, \imageH\dbb{q}\,[\log (q)].
\]
\end{remark}

Consider now  the operator $\mathbf{A} \colon \imageH(\!(\hbar)\!)\dbb{e^t}[t] \to \imageH(\!(\hbar)\!)\dbb{e^t}[t]$ of quantum multiplication by $\hpl$. Its  matrix in the standard basis 
$\mathbf{1}, \hpl, \ldots, \hpl^{\dim X}$ has a \strongemph{diagonal-jump-shape} with the step of the jump equal to the index $(N-d_{\op{tot}})$ of the Fano variety $X$  
(in the example $\dim X = 8$, $N-d_{\op{tot}} = 3$)
\[ 
\mathbf{A} \ = \  \begin{pmatrix} 0& 0 & \str\, e^{t} & 0 & 0 & \str\, e^{2t} & 0 & 0 & \str\, e^{3t}\\
\rr & 0& 0 & \str\, e^{t} & 0 & 0 & \str\, e^{2t} & 0 & 0 \\
0 & \rr & 0 & 0 & \str \, e^{t}  & 0 & 0 & \str\, e^{2t} & 0\\
0 & 0 & \rr & 0 & 0 & \str\, e^{t} & 0 & 0 & \str\, e^{2t} \\
0 & 0 & 0 & \rr & 0 & 0 & \str\, e^{t} &  0 & 0 \\
0 & 0 & 0 & 0 & \rr & 0 & 0 & \str\, e^{t} & 0 \\
0 & 0 & 0 & 0 & 0 & \rr & 0 & 0 & \str\, e^{t} \\
0 & 0 & 0 & 0 & 0 & 0  & \rr & 0 & 0 \\
0 & 0 & 0 & 0 & 0 & 0  & 0 & \rr & 0
\end{pmatrix}. 
\]
Here the unknown coefficients ${\color{red}{*}}$ are integers, and $\mathbf{A}$ is symmetric with respect to a flip along the antidiagonal. Equivalently, in the variables $(q,u)$, the matrix of $\mathbf{A} \, \colon \, \imageH(\!(u)\!)\dbb{q}[\log q] \, \to \, 
\imageH(\!(u)\!)\dbb{q}[\log q]$ is given by the same matrix but with the exponentials 
$e^{t}$, $e^{2t}$, $e^{3t}$, \ldots replaced by $q$, $q^2$, $q^3$, \ldots, respectively.

\medskip

Consider now the second horizontality condition $\nabla_{\partial_{q}}\Psi = 0$ for the connection $\nabla$. Explicitly it reads 
$uq\partial_{q}\Psi = - \mathbf{A}\Psi$, or

\begin{equation} \label{eq:secondhorizontality}
u\,q \partial_q\begin{bmatrix}\psi_0 \\ \psi_1 \\ \\ \vdots \\ \\ \vdots \\ \\ \\ \psi_{\dim X}\end{bmatrix}= - 
\begin{pmatrix} 0 & \cdots & 0 &\str q & 0 & \cdots &  0 & \str q^2& 0 & \cdots \\
\rr & 0 & \cdots & 0 & \str q & 0 & \cdots &  0 & \str q^2 & \cdots  \\
0 & \rr & 0 & \cdots &  0 & \str q & 0 & \cdots & 0  & \cdots \\
0 & 0 & \rr & 0 & \cdots & 0 & \str q & 0 & \cdots & \cdots \\
0 & 0 & 0 & \rr & 0 & \cdots & 0 & \str q & 0  & \cdots \\
0 & 0 & 0 & 0 & \rr & 0 & \cdots & 0 & \str q & \cdots \\
0 & 0 & 0 & 0 & 0 & \rr & 0 & \cdots & 0 & \cdots \\ 
0 & 0 & 0 & 0 & 0 & 0 & \rr & 0 & \cdots & \cdots\\ 
\cdots & \cdots & \cdots & \cdots & \cdots & \cdots & \cdots & \cdots & \cdots & \cdots
\end{pmatrix}\cdot 
\begin{bmatrix}\psi_0 \\ \psi_1 \\ \\ \vdots \\ \\ \vdots \\ \\ \\ \psi_{\dim X}\end{bmatrix},
\end{equation}
where the coordinates $\psi_{i}$ of the vector function $\Psi$ are functions of the variable of differentiation $q$,  depending on the parameter $u$. 

The connection between \eqref{eq:secondhorizontality} and Givental's quantum differential equation \eqref{eq:qODEqu} arises as follows. The last coordinate function 
$\psi_{\dim X}$ obeys the quantum differential equation, i.e.\ is a $\bbC$-linear combination of the functions $S_{0}$, $S_{1}$, \ldots, $S_{\dim X -1}$.  Now, starting with a  vector valued solution $\Psi = (\psi_{0}, \ldots, \psi_{\dim X})^{T}$ of \eqref{eq:secondhorizontality}, we can use the diagonal-jump-shape of $\mathbf{A}$ to solve recursively for the coordinate functions 
$\psi_{\dim X -1}$, $\psi_{\dim X - 2}$, \ldots , $\psi_{1}$, $\psi_{0}$ in terms of the 
last coordinate function $\psi_{\dim X}$. Thus the system \eqref{eq:secondhorizontality} becomes  equivalent to $\psi_{\dim X}$ obeying an ODE of order $N - k$ which has to be Givental's  quantum differential equation. Comparing the coefficients of the ODE of order  $N-k$ that we obtain from the recursive procedure with the coefficients of  \eqref{eq:qODEqu} we obtain a system of algebraic equations on the unknown entries $\str$ in the matrix $\mathbf{A}$. This gives an algorithm for computing $\mathbf{A}$ and $\mathbf{K} = (N-d_{\op{tot}}) \mathbf{A}$ from the quantum differential equation.

\begin{example} \label{ex:quadrics_and_cubics}
\ {\bfseries (i)} \ To illustrate the process, let us compute $\mathbf{A}$ in the case when $X$ is a 
$3$-dimensional quadric, i.e.\ $N = 5$, $k = 1$, $d = 2$. In this case the system \eqref{eq:secondhorizontality} reads
\[
uq\partial_{q} \begin{bmatrix} \psi_{0} \\ \psi_{1} \\ \psi_{2} \\ \psi_{3} \end{bmatrix} \ = \ 
- \begin{pmatrix}
0 & 0 & a_{1}q & 0 \\
1 & 0 & 0 & a_{2}q \\
0 & 1 & 0 & 0 \\
0 & 0 & 1 & 0
\end{pmatrix}\cdot \begin{bmatrix} \psi_{0} \\ \psi_{1} \\ \psi_{2} \\ \psi_{3} \end{bmatrix}
\]
where $a_{1}, a_{2}$ are some unknown integers. 
If we simplify notation by setting $\phi = \psi_{3}$ we get 
\[
\begin{aligned}
\psi_{2} & = - uq\partial_{q}\phi, \\
\psi_{1} & = - uq\partial_{q}\psi_{2}, \\
\psi_{0} & = - uq\partial_{q}\psi_{1} - a_{2}q\phi, \\
uq\partial_{q}\psi_{0} & = - a_{1}q\psi_{2}.
\end{aligned}
\]
Solving recursively for all $\psi_{i}$ in terms of $\phi$ we arrive at the $4$-th order linear ODE
\[
(uq\partial_{q})^4\phi = -qu((a_{1}+a_{2})q\partial_{q} + a_{2})\phi.
\]
Comparing this with the quantum differential equation 
\[
(uq\partial_{q})^4\phi = -qu(2(2q\partial_{q} + 1))\phi,
\]
we get that $a_{1}+a_{2} = 4$, $a_{2} =2$. Hence
for this case we get 
\[
\mathbf{A} \ = \  \begin{pmatrix}
0 & 0 & 2q & 0 \\
1 & 0 & 0 & 2q \\
0 & 1 & 0 & 0 \\
0 & 0 & 1 & 0
\end{pmatrix}.
\]

\medskip

Following this strategy, we can compute all coefficients in the connection \eqref{eq:ansmallnabla} of the F-bundle $(\smallfiber,\nabla)/\smallbase_{X}$ for any complete intersection. In particular we have the following

\medskip

\noindent
{\bfseries (ii)} \  Suppose $X$ is a smooth  $3$-dimensional cubic, i.e.\ $N=5$, $k=1$, $d = 3$. Then the coefficients of the connection \eqref{eq:ansmallnabla} defining the F-bundle $(\smallfiber,\nabla)/\smallbase_{X}$ are
{\scriptsize
\[ 
\mathbf{A} \ = \  \begin{pmatrix}
0 & 6q & 0 & 0 \\
1 & 0 & 15q & 0 \\
0 & 1 & 0 & 6q \\
0 & 0 & 1 & 0
\end{pmatrix}, \quad 
\mathbf{K} \ = \  2\begin{pmatrix}
0 & 6q & 0 & 0 \\
1 & 0 & 15q & 0 \\
0 & 1 & 0 & 6q \\
0 & 0 & 1 & 0
\end{pmatrix}, \quad 
\mathbf{G} \ = \  \begin{pmatrix}
-3/2 &  & 0 & 0 \\
0 & -1/2 & 0 & 0 \\
0 & 0 & 1/2 & 0 \\
0 & 0 & 0 & 3/2
\end{pmatrix}.
\]
}

\medskip

\noindent
{\bfseries (iii)} \ Suppose $X$ is a smooth  $4$-dimensional cubic, i.e.\ $N=6$, $k=1$, $d = 3$. 
The coefficients of the connection \eqref{eq:ansmallnabla} defining the F-bundle $(\smallfiber,\nabla)/\smallbase_{X}$ are
{\scriptsize
\[
\mathbf{A} \ = \  \begin{pmatrix}
0 & 0 & 6q & 0 & 0 \\
1 & 0 & 0 & 15q & 0 \\
0 & 1 & 0 & 0 & 6q \\
0 & 0 & 1 & 0 & 0 \\
0 & 0 & 0 & 1 & 0
\end{pmatrix}, \quad 
\mathbf{K} \ = \  3\begin{pmatrix}
0 & 0 & 6q & 0 & 0 \\
1 & 0 & 0 & 15q & 0 \\
0 & 1 & 0 & 0 & 6q \\
0 & 0 & 1 & 0 & 0 \\
0 & 0 & 0 & 1 & 0
\end{pmatrix}, \quad 
\mathbf{G} \ = \  \begin{pmatrix}
-2 & 0 & 0 & 0  & 0 \\
0 & -1 & 0 &  0 & 0 \\
0 & 0 & 0 & 0  & 0 \\
0 & 0 & 0 & 1 & 0 \\
0 & 0 & 0 & 0 & 2
\end{pmatrix}.
\]
}
\end{example}

\begin{remark}
\label{rem-other.techniques} In this section we used Givental's theorem to compute the action of the Euler vector field for the $\mathsf{A}$-model F-bundle of a projective manifold, which can in turn be used to analyze the atomic decomposition of Fano complete intersections. It is important to note that even though Givental's original results  for toric complete intersections do not work directly for other types of Fano varieties, the current state of the art comprises many other methods that can be used to compute the quantum spectrum virtually in any specific example. Three of these techniques deserve special attention. The first is  the quantum Lefschetz/Serre package, which includes the quantum Lefschetz theorems of Lee, Kim, and Coates-Givental \cite{Lee_YP-qLefschetz,Kim-qLefschetz,CoatesGivental}, and the quantum Serre theorems of Coates-Givental and  Iritani-Mann-Mignon \cite{CoatesGivental,IritaniMannMignon}. The second is the toric degeneration  technique which, as demonstrated in the work of Ciocan-Fontanine, Kim, and Sabbah \cite{CFKS}, can be used to derive the exact form of the quantum differential equation for complete intersections in homogeneous spaces. Last but not least, one can use Homological Mirror Symmetry and the explicit forms of 
mirrors as provided by the Hori-Vafa construction \cite{Hori_Vafa_Mirror_symmetry} to compute the quantum spectrum as the critical values of the mirror Landau-Ginzburg potential.
\end{remark}

\subsection{Four-dimensional cubics} \label{ssec:Cubics}

In this section we will apply the theory to some classical rationality problems.

\begin{theorem} \label{thm:cubic4} 
A very general $4$-dimensional cubic hypersurface in $\bbC\bbP^5$ is not rational.
\end{theorem}
\begin{proof}
Let $X$ be a smooth cubic hypersurface in $\bbC\bbP^5$. Assume that $X$ is Noether-Lefschetz general in the sense that the only rational Hodge classes on $X$ are  the powers of the hyperplane class  $\hpl = c_{1}(\cO_{\bbP^5}(1))|_{X}$. Such cubic fourfolds exist e.g.\ by Voisin's proof \cite{Voisin-cubicTorelli} of the Torelli theorem for cubic fourfolds.

The odd degree cohomology of $X$ is zero and the total dimension of the even cohomology is $27$. Following the construction explained in Section~\ref{sssec:AmodelF} we choose a homogeneous basis $\{T_{i}\}_{i = 1}^{27}$ of $H^{\bullet}_{B}(X,\bbQ)$, so that $T_{1} = \mathbf{1} \in H^0(X,\bbQ)$, $T_{2} = \hpl$, and $\deg T_{i} > 2$ for $i > 2$.  We write $q$ for the coordinate on $\NS(X,\bbQ) = \bbQ\cdot \hpl$ to $\hpl$ and we write $t_{i}$ for the coordinates dual to $T_{i}$ for $i \neq 2$. 
Let $(\cH,\nabla)/B_{X}$ be the maximal analytic ${\mathsf A}$-model F-bundle for $X$ defined in Section~\ref{sssec:AmodelF} with connection $\nabla$ given by 
\begin{equation} \label{eq:cubicnabla}
\left| \,
\begin{aligned}
\nabla_{\partial_{u}} & = \partial_{u} - u^{-2} \left(\Eu \, \qup\, (-)\right) + u^{-1} \frac{\Deg - 4\cdot \op{id}}{2}, \\
\nabla_{\partial_{q}} & = \partial_{q} + u^{-1}q^{-1} \left( \hpl \, \qup \, (-)\right), \\
\nabla_{\partial_{t_{i}}} & = \partial_{t_{i}} + u^{-1} \left( T_{i} \, \qup \, (-)\right), \ \text{for} \ i \neq 2.
\end{aligned}
\right.
\end{equation}

Let $b \in B_{X}$ be the rigid point with coordinates $q = 1$, and $t_{i} = 0$, for all $i \neq 2$. Then $b \in \smallbase_{X} \ \subset \ B_{X}^{\hodge} \ \subset B_{X}$ and 
by Lemma~\ref{lem:Ginvariants} we have that the reduced spectrum of the operator 
$\left(\Eu \, \qup\, (-)\right)_{b} \, \colon \, \cH_{(b,0)} \ \to \ \cH_{(b,0)}$ is the same as the reduced spectrum of the action of $\left(\Eu \, \qup\, (-)\right)_{b}$ on the subspace 
of Hodge classes $\cH^{\hodge}_{(b,0)}$. By our genericity assumption 
\[
\cH^{\hodge}_{(b,0)} \ = \ \oplus_{a = 0}^4 \, \bbk\cdot \hpl^{a} \ = \ \smallfiber_{(b,0)}.
\]
Thus 
\[
\left(\Eu \, \qup\, (-)\right)_{b}|_{\smallfiber_{(b,0)}} \ = \ \mathbf{K}|_{q=1},
\]
where $\mathbf{K}$ is the operator introduced in the previous section. 

As we saw in Section~\ref{ssec:givental}, the operator  $\mathbf{K} \colon \smallfiber|_{\smallbase_{X}\times \{0\}}  \ \to \ \smallfiber|_{\smallbase_{X}\times \{0\}}$ can be computed from Givental's theorem 
and from Example~\ref{ex:quadrics_and_cubics} {\bfseries (iii)} we have 
\[
\mathbf{K} \ = \ 3\begin{pmatrix}
0 & 0 & 6q & 0 & 0 \\
1 & 0 & 0 & 15q & 0 \\
0 & 1 & 0 & 0 & 6q \\
0 & 0 & 1 & 0 & 0 \\
0 & 0 & 0 & 1 & 0
\end{pmatrix} \qquad \text{and so} \qquad 
\mathbf{K}|_{q=1} \ = \ 3\begin{pmatrix}
0 & 0 & 6 & 0 & 0 \\
1 & 0 & 0 & 15 & 0 \\
0 & 1 & 0 & 0 & 6 \\
0 & 0 & 1 & 0 & 0 \\
0 & 0 & 0 & 1 & 0
\end{pmatrix}
\]
The characteristic polynomial of this matrix is $\lambda^5 - 3^6\lambda^2$ and so 
$\mathbf{K}|_{q=1}$ has eigenvalues $0$ (appearing with multiplicity $2$), and  $9$, $9\zeta$, $9\zeta^2$ where $\zeta$ is a primitive third root of unity. Therefore the reduced spectrum of $\bkappa_{b} \coloneqq \Eu_{b}\fp (-)$ acting  on all of $\cH_{(b,0)}$ is also equal to  $\{0,9,9\zeta,9\zeta^2\}$. 

By the spectral decomposition theorem Theorem~\ref{thm:K-decomposition} we then have that
over some admissible open neighborhood $b \in U \subset B_{X}^{\hodge}$ the  $\mathsf{A}$-model F-bundle $(\cH,\nabla)/B_{X}^{\hodge}$  decomposes into an external direct sum of maximal F-bundles:
\[
\left(\cH,\nabla\right)/U \ = \ \bplus_{\substack{\lambda \, \in \, \{0,9,9\zeta,9\zeta^2\}}} \
\left(\cH^\lambda,\nabla^\lambda\right)/U^{\lambda}
\]
compatible with the decomposition of $\cH_{(b,0)}$ into a direct sum of generalized eigenspaces for $\bkappa_{b} = \Eu_{b}\fp (-)$.

Let now $\balpha \in \HAtoms$ be a Hodge atom appearing in the atomic composition of $X$. By definition $\balpha$ is represented by some $\alpha \in  \pi_{0}(\tU_{X})$. Let $\tU_{X,\alpha}$ be the corresponding connected component of $\tU_{X}$. Since $U_{X}$ is open and dense in $B_{X}^{\hodge}$, the intersection 
$U \cap U_{X}$ will be non-trivial 
and so the equivalence class of the atomic F-bundle associated with $\alpha$ must be a generalized eigenbundle for the action of $\Eu\qup(-)$ on just one of the four pieces $(U\times \bbD \to U^{\lambda}\times \bbD)^{*}\left(\cH^{\lambda},\nabla^{\lambda}\right)$ .

This implies that if $E^{\balpha}$ is the finite dimensional $\overline{\bbQ}$-linear $\hodge_{\overline{\bbQ}}$-representation 
corresponding to $\balpha$, then we have
\[
\dim_{\overline{\bbQ}} \, (E^{\balpha})^{\hodge(\overline{\bbQ})} \ \leq \ \min_{\substack{\lambda \, \in \, \{0,9,9\zeta,9\zeta^2\}}} \, \op{rank} \, (\cH^{\lambda}_{u=0})^{\hodge} \ \leq \  \min_{\substack{\lambda \, \in \, \{0,9,9\zeta,9\zeta^2\}}} \,  
\dim_{\bbk} \,  \smallfiber_{(b,0})^{\lambda},
\]
where $\smallfiber_{(b,0})^{\lambda}$ denotes the generalized eigenspace of the operator $\mathbf{K}|_{q=1}$ acting on $\smallfiber_{(b,0)} = \oplus_{i = 0}^4 \, \bbk\cdot \hpl^{i}$.

Therefore we get that for 
every Hodge atom $\balpha$ appearing in the atomic composition of $X$, the dimension $\rho_{\balpha}$ of Hodge classes in $\balpha$ satisfies 
\[
\rho_{\balpha} \ = \ \dim_{\overline{\bbQ}} (E^{\balpha})^{\hodge(\overline{\bbQ})} \ \leq  \ 2.
\]
On the other hand, by our  assumption of Noether-Lefschetz genericity of $X$ the Hodge pieces  in the degree $4$ primitive cohomology of $X$ corresponding to the row 
\[
\left(h^{4,0}_{\op{prim}}(X),h^{3,1}_{\op{prim}}(X), h^{2,2}_{\op{prim}}(X),h^{1,3}_{\op{prim}}(X),h^{0,4}_{\op{prim}}(X)\right) = (0,1, 20, 1,0)
\]
in the primitive Hodge diamond of $X$ all  consist of transcendental cycles. Since by definition the atomic composition of the cohomology of $X$ is compatible with the Hochschild grading, this implies that in 
the atomic composition of $X$ we must have an atom $\balpha$ with $\op{Coeff}_{t^2}(P_{\balpha}(t)) = 1$. 

Thus, according to Proposition~\ref{prop:non-rational}, it only remains to show that a Hodge atom $\balpha$ with $\op{Coeff}_{t^2}(P_{\balpha}(t)) = 1$ and $\rho_{\balpha} \leq 2$ can not appear in the atomic composition of a smooth projective variety of dimension $\leq 2$, i.e.\ can not come from points, curves, or surfaces.
Moreover, for each birational class of surfaces it is sufficient to consider the atomic composition of one representative in the class.

But if $\balpha$ is an atom of a point or a curve, then $\op{Coeff}_{t^2}(P_{\balpha}(t)) = 0$. If, on the other hand, $\balpha$ is an atom appearing in the atomic composition of a smooth projective surface $S$, then $\op{Coeff}_{t^2}(P_{\balpha}(t)) \leq  H^{2,0}(S) = p_{g}(S)$ and by the classification of surfaces  
$S$ must be either an abelian surface, a K3 surface, an elliptic surface with $\bkappa =1$ and $p_{g} =1$, or a surface of general type. But every such surface has a nef $K_{S}$, and so by Lemma~\ref{lem:nefK} the atomic composition of $S$ consists of a single atom $\boldsymbol{\eta}(S)$, and the corresponding representation 
$E^{\boldsymbol{\eta}(S)}$ of $\hodge_{\overline{\bbQ}}(\overline{\bbQ})$ is just equal to 
$H^{\bullet}_{B}(S,\overline{\bbQ})$. Since $S$ is projective it has at least one algebraic cycle of dimension $2$ and two more algebraic cycles of dimension $0$ and $4$ respectively.
Thus $\rho_{\boldsymbol{\eta}(S)} \, \geq \, 3$. This shows that $\balpha$ can not be an atom of a surface and so $X$ can not be rational.
\end{proof}

\begin{remark} \label{rem:atomsS}
In the proof of the previous theorem we used the fact that 
Hodge atoms of algebraic surfaces for which the   $p-q=2$ part of Hodge structure is nontrivial must at least a 3 dimensional space Hodge cycles spanned by the $0$-th, $1$-st, and $2$-nd powers of the hyperplane class. 
In fact the Enriques-Kodaira  classification of surfaces gives a complete list of Hodge atoms for varieties  of dimension $\leq 2$. We have summarized this list in the following table, where again we have simplified the notation by labeling the Hodge atoms by the corresponding subspaces in the cohomology $H^{\bullet}$ of the respective variety.

\begin{table}[ht!] 
    \begin{center}
    \begin{adjustbox}{max width=\textwidth}
        \begin{tabular}{|c|c|c|c|c|c|c|c|c|} 
        \hline
            $\dim$ & 0 & 1 & 1 & 2 & 2 & 2 & 2 & 2   \\ \hline 
            $X$ & pt & $\bbP^1$ & \begin{tabular}{c}
                 curve $C$, \\
                 $g(C) \geq 1$
            \end{tabular} & $\bbP^2$ & \begin{tabular}{c}
                 $\bbP^1 \times C$ \\
                 $g(C) \geq 1$
            \end{tabular}  &  \begin{tabular}{c}
                 abelian \\
                 surface $S$ \\
                 or \\
                 K3 surface
            \end{tabular} & \begin{tabular}{c}
                 Enriques \\
                 or \\
                 elliptic\\
                 or\\
                 bielliptic
            \end{tabular} & \begin{tabular}{c}
                 general type\\
                 with \\
                 ADE singularity\\
                 $\to$ blow up $\widetilde{S}$
            \end{tabular}\\ \hline 
            \begin{tabular}{c}
                Hodge atoms
            \end{tabular} &  \begin{tabular}{c}$H^{\bullet}(pt)$ \\ $=\bbC$ \end{tabular} & $H^{\bullet}(pt \cup pt)$ &  $H^{\bullet}(C)$ & $H^{\bullet}(pt \cup pt \cup pt)$ &   $H^{\bullet}(C) \oplus H^{\bullet}(C)$ & \begin{tabular}{c} $H^{\bullet}(S)$  \end{tabular} &\begin{tabular}{c}
                $|p-q|\leq 1$\\
                in\\
                nc-Hodge\\
                structure
            \end{tabular} & \begin{tabular}{c} $H^{\bullet}(\widetilde{S})$ \end{tabular} \\ \hline
        \end{tabular}
    \end{adjustbox}
    \end{center}
    \caption{Hodge atoms for varieties of dimension $\leq 2$}
\end{table}
\end{remark}

Let $X$ be a smooth cubic fourfold which is Noether-Lefschetz general as in the proof of Theorem~\ref{thm:cubic4}. The coarse analysis of the atomic composition of $X$ that we used in the previous theorem, can be made more precise. First we have the following

\begin{lemma} \label{lem:transc0}
    For any $\phi \in H_\prim^4(X)$, we have $\bkappa_{b}(\phi) = 0$.
\end{lemma}
\begin{proof}
    It suffices to show that for any $\psi \in H^{\bullet}_{B}(X)$, and any curve class $\beta \in H_2(X)$, the Gromov-Witten invariant $\langle \hpl, \phi, \psi\rangle_{0,n,\beta}^X = 0$.
    By the virtual dimension calculation, if $\langle \hpl, \phi, \psi\rangle_{0,n,\beta}^X \neq 0$, then $\codim \psi \ge 4$.
    The only possibility is $\psi \in H^8(X)$.
    Now both $\hpl$ and $\psi$ come from the ambient projective space.
    Since $\phi \in H_\prim^4(X)$, we deduce that $\hpl\smile \phi \smile \psi = 0$, hence $\langle \hpl, \phi, \psi\rangle_{0,n,\beta}^X = 0$, completing the proof.
\end{proof}

\begin{corollary} \label{cor:atomscubic4}
Let $X$ be a very general cubic fourfold over $\bbC$, then $X$ has $4$ Hodge atoms in its atomic composition. Three of those are one dimensional and correspond to the eigenvectors with eigenvalues  $9$, $9\zeta$, and $9\zeta^2$ for the action of $\bkappa_{b}$ on $\cH_{(b,0)}$,
and one is $24$-dimensional and corresponds to the generalized eigenspace of 
$\bkappa_{b}$ for the eigenvalue $0$. 
\end{corollary}
\begin{proof} 
By Lemma~\ref{lem:transc0} the action 
of $\bkappa_{b} = \mathbf{K}|_{q=1}$ on $H^{\bullet}_{B}(X,\bbk)$ will have a single eigenvalue $0$ on the subspace $H^4(X)_{\text{prim}}\otimes \bbk$. Therefore, the piece 
$(\cH^0,\nabla^0)/U^0$ in the spectral decomposition of $(\cH,\nabla)/U$ is a maximal F-bundle of rank $24$.

The Torelli theorem \cite{Voisin-cubicTorelli,Looijenga-cubicTorelli,Charles-cubicTorelli} for cubic $4$-folds also implies that for a very general $X$ the transcendental part of the cohomology of $X$ is an irreducible representation of $\hodge_{\overline{\bbQ}}$. Furthermore, the operator $\mathbf{K}|_{q=1}$ has a size 2 Jordan block corresponding to the eigenvalue $0$, and 
since by Proposition~\ref{prop:Ginv1} a Hodge atom must have dimension of Hodge classes $\geq 1$, it follows that  a Hodge atom $\balpha$ in the atomic composition of $X$
which comes from a generalized eigenspace factor in $(\cH^0,\nabla^0)/U^0$ must  be the whole thing, i.e.\ have $\dim_{\overline{\bbQ}} E^{\balpha} = 24$. This proves the corollary.
\end{proof}

\begin{remark} \label{rem:Kuznetsov} We expect that the $24$ dimensional representation of $\hodge$ associated with the atom 
is the cohomology of the Kuznetsov component \cite{Kuznetsov-cubic4fold} of $X$. 
\end{remark}

\subsection{Birational maps of Calabi-Yau manifolds} \label{ssec:birCY}

Hodge atoms  can also be used effectively to study the motivic properties of birational equivalences between smooth projective varieties. In particular, Hodge atoms can be used to give a new proof of Batyrev's theorem \cite[Corollary~6.29]{Batyrev-StringyHodge} asserting that birationally equivalent Calabi-Yau manifolds have equal Hodge numbers. The traditional proofs of this result \cite{Batyrev-StringyHodge,Wang-Kequivalence,Ito} use either motivic integration, or $p$-adic Hodge theory and C\v{e}botarev density. The proof via atoms has a completely different flavor as it uses Gromov-Witten theory and \nc-Hodge theory. 

To show the equality of Hodge numbers we begin with the following observation

\begin{lemma} \label{lem:birCYatoms}
Let $X_1$ and $X_2$ be two complex birational Calabi-Yau manifolds of dimension $d$. Then the Betti cohomology spaces $H^{\bullet}_{B}(X_1,\overline{\bbQ})$ and  $H^{\bullet}_{B}(X_2,\overline{\bbQ})$ are isomorphic as representations of $\hodge_{\overline{\bbQ}}(\overline{\bbQ})$. In particular the 
$\bbZ/2$-folded Hodge polynomials of $X_{1}$ and $X_{2}$ are equal, or equivalently  
\[
\sum_{p-q = k} \ h^{p,q}(X_{1}) \ = \ \sum_{p-q = k} \ h^{p,q}(X_{2}), \quad \text{for all} \ k = -d, \ldots, d.
\]
\end{lemma}
\begin{proof}
Since $K_{X_i}$  is numerically trivial  for $i = 1,2$, Lemma~\ref{lem:nefK} implies that the atomic compositions of $X_{1}$ and $X_{2}$ each consist of a single atom. Let us denote these respective atoms by $\boldsymbol{\eta}(X_{1})$ and $\boldsymbol{\eta}(X_{2})$.
Since $X_{1}$ and $X_{2}$ are birational, the weak factorization theorem \cite{Abramovich_Torification_and_factorization} implies that $X_{1}$ and $X_{2}$ are related by a sequence of forward/backward blowups with smooth centers. But the blowup decomposition Theorem~\ref{thm:blowup} implies that each such blowup adds to the atomic composition only atoms of smooth projective varieties of dimension $\leq d- 2$.
In particular for every such atom $\balpha$ we must have $\mathsf{Coeff}_{t^{d}}(P_{\balpha}(t)) = 0$. On the other hand  $h^{d,0}(X_i)=1$ for $i = 1,2$  and so 
$\mathsf{Coeff}_{t^{d}}(\boldsymbol{\eta}(X_{1})) \ = \ \mathsf{Coeff}_{t^{d}}(\boldsymbol{\eta}(X_{2})) \ = \ 1$. Thus we must have $\boldsymbol{\eta}(X_{1}) = \boldsymbol{\eta}(X_{2})$ which proves the lemma. 
\end{proof}

Next we will use some special features of the $\mathsf{A}$-model F-bundle in the Calabi-Yau case. Let $X$ be a $d$-dimensional complex projective manifold.
Let $(\cH,\nabla)/B_{X}$ be the $\bbk$-analytic maximal $\mathsf{A}$-model F-bundle of $X$.
Fix any rigid point $b \in B_{X}^{\even} \subset B_{X}$, so that the coordinate of $b$ corresponding to $\mathbf{1} \in H^0_{B}(X,\bbk)$ vanishes. Consider the restriction of
$(\cH,\nabla)|_{\{b\}\times \Spf \bbk\dbb{u}}$. Then we have
\[
\begin{aligned}
\cH|_{\{b\}\times \Spf \bbk\dbb{u}} \ & = \ H^{\bullet}(X)\dbb{u} \\
\nabla|_{\{b\}\times \Spf \bbk\dbb{u}} \ & \text{is given by} \quad \nabla_{\partial_{u}} \ = \ \partial_{u} \, - \, u^{-2}\Eu\qup (-) \, + \, u^{-1}\Gr,
\end{aligned}
\]
where 
\[
\Gr = \frac{\Deg - d\cdot \op{id}}{2}
\]
is the grading operator, and we write  $\Eu\qup (-)$ as a shortcut for 
$\Eu_{b}\qup_{b} (-)$ and $H^{\bullet}(X)$ as a shortcut for 
$H^{\bullet}_{B}(X,\bbk)$. To keep track of the half integer eigenvalues of $\Gr$, let us  also introduce the operator $\mathbf{T}   \colon H^{\bullet}(X) \to H^{\bullet}(X)$ which is the identity on each $H^{a}(X)$ with $a-d $ odd and is zero on each $H^{a}(X)$ with $a-d$ even. Note that by definition the operator $\mathbf{g} = \Gr + \frac{1}{2}\mathbf{T}$ is semisimple with integral eigenvalues. Therefore, we have a  meromorphic gauge transformation   $u^{\mathbf g}$.

With this notation we now have the following

\begin{claim} \label{claim:nefK}
Suppose $K_{X}$ is nef. Then the
connection 
\[
\nabla_{\partial_{u}} \ = \ \partial_{u} \, - \, u^{-2}\Eu\qup (-) \, + \, 
u^{-1}\Gr \ \colon \ H^{\bullet}(X)\dbb{u} \ \longrightarrow \ 
H^{\bullet}(X)\dbp{u}
\]
has a regular singularity at $u=0$. More precisely, the gauge transformed connection 
\[
u^{\mathbf{g}}\circ \left(\nabla + \frac{1}{2}u^{-1}\mathbf{T}du\right)\circ u^{-\mathbf{g}}
\]
has a first order pole and a nilpotent residue at $u=0$.
\end{claim}
\begin{proof} 
Let $\widetilde{\nabla} = \nabla + \frac{1}{2}u^{-1}\mathbf{T}du$, then $\widetilde{\nabla}_{u\partial_{u}} = \nabla_{u\partial_{u}} + \frac{1}{2}\mathbf{T}$, and since $\mathbf{g}$ commutes with $\Gr$ we get 
\[
\begin{aligned}
u^{\mathbf{g}}\circ \widetilde{\nabla}_{u\partial_{u}}\circ u^{-\mathbf{g}} & = u\partial_{u} - \mathbf{g} + \frac{1}{2}\mathbf{T} + u^{-1} u^{\mathbf{g}}\circ (\Eu\qup (-))\circ u^{-\mathbf{g}} + u^{\mathbf{g}}\circ \Gr \circ u^{-\mathbf{g}}  \\
& = u\partial_{u} - u^{-1} u^{\mathbf{g}}\circ (\Eu\qup (-))\circ u^{-\mathbf{g}}.
\end{aligned}
\]
Since $K_{X} \geq 0$, and $b \in B_{X}^{\even}$ has a vanishing $H^0(X)$-coordinate, the operator $\Eu\qup(-)$ will move the eigenspace of $\Gr$ corresponding to an eigenvalue $m \in \frac{1}{2}\bbZ$ to the sum of eigenspaces for eigenvalues $\geq m + 1$.  
Thus if we let $\bkappa_{ij} \colon H^{i}(X,\bbk) \to H^{i + 2j}(X,\bbk)$ denote the respective block matrix component of $\Eu\qup (-)$, then $\bkappa_{ij} = 0$ for $j\leq 0$. The corresponding block matrix component of the gauged transformed operator $u^{-1} u^{\mathbf{g}}\circ (\Eu\qup (-))\circ u^{-\mathbf{g}}$ is given by the composition
\[
\xymatrix@1@C+2pc@M+0.5pc{H^{i}(X) \ar[r]^-{u^{\mathsf{const}-\floor{\frac{i}{2}}}} & H^{i}(X) \ar[r]^-{u^{-1}\bkappa_{ij}} & 
H^{i+2j}(X) \ar[r]^-{u^{\floor{\frac{i}{2}}+j-\mathsf{const}}} & H^{i+2j}(X),
}
\]
and $\mathsf{const} \in \bbZ$. 
In other words the block matrix component  of the transformed operator is $u^{-1+j}\bkappa_{ij}$ and so the gauge  transformed connection operator  
$u^{\mathbf{g}}\circ \widetilde{\nabla}_{u\partial_{u}}\circ u^{-\mathbf{g}}$ has no pole at $u=0$ and the $u = 0$ residue of  $u^{\mathbf{g}}\circ \widetilde{\nabla}_{\partial_{u}}\circ u^{-\mathbf{g}}$ is strictly lower triangular as a block matrix with entries labeled by the degree of cohomology.
\end{proof}

\begin{theorem} \label{thm:two.norms}
Let $X$ be a smooth projective complex Calabi-Yau. Then the Hodge numbers of $X$ can be reconstructed from the atomic F-bundle associated with the Hodge atom $\boldsymbol{\eta}(X)$ of $X$.
In particular, birationally equivalent Calabi-Yau manifolds have equal Hodge numbers.
\end{theorem}
\begin{proof}
Let $(\cH,\nabla)/B_{X}$ be the maximal analytic $\mathsf{A}$-model F-bundle of $X$. 
Let $b \in B_{X}^{\hodge}$ be a general point with vanishing component in $H^0(X)$. Then, since $\boldsymbol{\eta}(X)$ is the unique Hodge atom in the atomic composition of $X$, it follows that the germ $(\cH,\nabla)/B_{X}(b)$ represents the atomic F-bundle associated with 
$\boldsymbol{\eta}(X)$.  The restriction of this atomic F-bundle 
to $\{b\}\times \Spf \bbk\dbb{t}$ gives rise to a $k\dbp{u}$-vector space $H^{\bullet}(X)\dbp{u}$ equipped with a meromorphic connection $\nabla$ which by Claim~\ref{claim:nefK} has a regular singularity at $u=0$. 
Thus the $k\dbp{u}$-vector space $H^{\bullet}(X)\dbp{u}$ comes equipped with two natural lattices, i.e.\ two $\bbk\dbb{u}$-submodules $H^{\op{int}}, \, H^{\op{can}}$, such that 
$H^{\op{int}}\otimes_{\bbk\dbb{u}} \, \bbk\dbp{u} \ = \ H^{\op{can}} \otimes_{\bbk\dbb{u}} \,  \bbk\dbp{u} \ = \ H^{\bullet}(X)\dbp{u}$. Here  $H^{\op{int}}$ is the intrinsic lattice, i.e.\ the module of global sections of the original bundle $\cH|_{\{b\}\times \Spf \bbk\dbb{u}}$, and 
$H^{\op{can}}$ is the module given by the Deligne's canonical extension \cite{Deligne-de,Deligne-canext}
of the regular singular connection $(H^{\bullet}(X)\dbp{u},\nabla)$, i.e.\ the extension uniquely characterized by the property that in this extension the connection has a pole of order $\leq 1$ with a nilpotent residue. 

This characterization of $H^{\op{can}}$ implies that what we showed in 
Claim~\ref{claim:nefK} is  that the automorphism $u^{\mathbf{g}} \colon 
H^{\bullet}(X)\dbp{u} \to H^{\bullet}(X)\dbp{u}$ moves the lattice $H^{\op{int}}$ to
$H^{\op{can}}$. Thus an automorphism of $H^{\bullet}(X)\dbp{u}$ which moves $H^{\op{int}}$ isomorphically to
$H^{\op{can}}$ will induce an additional filtration  labeled by cohomological degree, i.e.\ by 
$p+q$. Combined with the grading by $p-1$ this gives the $(p,q)$ bigrading on $H^{\bullet}(X,\bbk)$. 
\end{proof}

\subsection{Enhanced atoms and further examples} \label{ssec:enhanced}

Hodge atoms can be enhanced in several ways to produce stronger obstructions to rationality.
In particular, we can consider F-bundles equipped with additional algebraic structures, such as pairings, duality automorphisms, and integral structures. For the analytic $\mathsf{A}$-model F-bundle associated with a smooth projective variety $X$,  these structures arise naturally from natural  structures on  the cohomology of $X$ -  the Mukai pairing (which is non-symmetric in general), the Serre functor automorphism, and the $\Gamma$-integral structure.  
This leads to enhanced versions of the Hodge atoms -  Hodge atoms with  pairings, Hodge atoms with duality automorphisms, and Hodge atoms with integral structures. 

\medskip

In more detail, let $\nu \colon B\times \bbD \, \to \, B\times \bbD$ be the involution given by $\nu(b,u) = (b,-u)$. 
A \strongemph{non-degenerate pairing} on an analytic F-bundle $(\cH,\nabla)/B$ is an even  $\cO_{B\times \bbD}$-linear map
\[
\bpsi \colon \, \cH\otimes \nu^{*}\cH \ \longrightarrow \cO_{B\times \bbD}, 
\]
which is symmetric (in the sense that $\bpsi(\xi_{1},\xi_{2}) = (\nu^{*}\bpsi)(\xi_{2}, \xi_{1})$ for all local sections $\xi_{1} \in \cH$, $\xi_{2} \in \nu^{*}\cH$), non-degenerate (in the sense that it induces an isomorphism 
$\cH \cong \nu^{*}\cH^{\vee}$), and preserved by  $\nabla$ (in the sense that 
$d\bpsi(\xi_{1},\xi_{2}) = \bpsi(\nabla\xi_{1},\xi_{2}) + \bpsi(\xi_{1},(\nu^{*}\nabla)\xi_{2})$). Alternatively, a non-degenerate pairing can be specified by an isomorphism $\bPsi \colon \cH \, \widetilde{\to} \, \nu^{*}\cH^{\vee}$ of super vector bundles on $B\times \bbD$, which intertwines the connections on the source and target, and is self-dual in the sense that $\nu^{*}\bPsi^{\vee} = \bPsi$.
Note that the by definition non-degenerate pairings are compatible with the spectral decomposition of F-bundles. This immediately gives induced pairings on $G$-atoms.

A duality automorphism of an F-bundle $(\cH,\nabla)/B$ equipped with a non-degenerate pairing $\bpsi$  is an automorphism $S \colon \, \cH|_{B\times \{u = 1\}} \, \to \, \cH|_{B\times \{u = 1\}}$ of the restricted bundle $\cH|_{B\times \{u = 1\}}$ which satisfies appropriate compatibility conditions with $\nabla$ and $\bpsi$. It takes some work to formulate these compatibilities since they are inherently transcendental and require the Riemann-Hilbert correspondence which is not readily available in the non-archimedean analytic setting.  
We will discuss this in more detail in \cite{KKPY-Gamma} but for now let us give a brief topological discussion of this interpretation.

\medskip

Suppose that $S^1  = \{ \, u \, \colon \, \ |u| = 1 \, \} \ \subset \bbC$ is the unit circle, and let $\nu \colon S^1 \to S^1$ be the negation map $\nu(u) = -u$. Let $\cH \to S^1$ be a local system of finite dimensional complex super vector spaces. As before we define a symmetric non-degenerate pairing on $\cH$ as  an isomorphism
$\bPsi \colon \cH \, \widetilde{\to} \, \nu^{*}\cH^{\vee}$, such that $\nu^{*}\bPsi^{\vee} = \bPsi$.  
Similarly, if  $V$ is a finite dimensional complex super vector space, then we define a (not necessarily symmetric) non-degenerate pairing on $V$ as a linear isomorphism $G \, \colon \, V \, \widetilde{\to} \, V^{\vee}$.

The basic observation here is that  datum $(\cH,\bPsi)$ is equivalent to the datum $(V,G)$. Indeed:

\begin{enumerate}[wide]
    \item  Given a pair 
$(\cH,\bPsi)$ as above, we can assign to it a pair $(V,G)$ as follows. Take  $V$ to be the fiber $V = \cH_{1}$ of $\cH$ at the point $1 \in S^1$, and take $G \, \colon \, V \, \to \, V^{\vee}$ to be the composition $G = \bPsi_{-1}\circ \boldsymbol{\tau}_{\mathsf{arc}^{+}}$. Here $\bPsi_{-1} \, \colon \, \cH_{-1} \, \to \, \cH_{1}^{\vee}$ is the pairing at the point $(-1) \in S^1$ and $\boldsymbol{\tau}_{\mathsf{arc}^{+}}  \colon \cH_{1} \to \cH_{-1}$ is the parallel transport in $\cH$ along the positively oriented 
upper semicircle $\mathsf{arc}^{+} \, \colon  \, [0,\pi] \, \to \, S^1$, $t \to \exp(\sqrt{-1}t)$.  
 \item Given a pair  $(V,G)$ we can assign to it a pair $(\cH,\bPsi)$ as follows. Take $\cH$ to be the local system which on the upper semicircle is constant with fiber $V$, on the lower semicircle is constant with fiber $V^{\vee}$, and such that the fibers at $(-1)$ are glued by the isomorphism $G \, \colon \,  V \, \to \, V^{\vee}$ and the fibers at $1$ are glued by the isomorphism $G^{\vee} \, \colon \, V \, \to \, V^{\vee}$. Schematically we can describe $\cH$ by the gluing diagram
 \[
 \cH \ \colon \ \text{\begin{minipage}[c]{1.5in} $\xymatrix{ V \ar@/_1pc/[d]_-{G} \ar@/^1pc/[d]^-{G^{\vee}} \\ V^{\vee}}$\end{minipage}}.
 \]
 Note that with this definition, if we identify $\cH_{1}$ with $V$, then the monodromy operator $\mathsf{mon}_{1} \colon \cH_{1} \, \to \cH_{1}$ of $\cH$ at the point $1 \in S^1$ (for going once around $S^1$ in the counterclockwise direction), is given by  $\mathsf{mon}_{1} = (G^{\vee})^{-1}\circ G \, \colon \, V \, \to \, V$.

This description immediately yields analogous explicit presentations of $\nu^{*}\cH$ and $\nu^{*}\cH^{\vee}$:
\[
\nu^{*}\cH \ \colon \ \text{\begin{minipage}[c]{1.5in} 
$\xymatrix{ V \ar@/_1pc/[d]_-{(G^{\vee})^{-1}} \ar@/^1pc/[d]^-{G^{-1}} \\ V^{\vee}}$\end{minipage}}
\qquad \text{and} \qquad 
\nu^{*}\cH^{\vee} \ \colon \ \text{\begin{minipage}[c]{1.5in} $\xymatrix{ V \ar@/_1pc/[d]_-{G} \ar@/^1pc/[d]^-{G^{\vee}} \\ V^{\vee}}$\end{minipage}},
\]
and we  define the map $\bPsi \, \colon \, \cH \, \to \, \nu^{*}\cH^{\vee}$ to be 
\[
\xymatrix@C+2pc{ V \ar@/_1pc/[d]_-{G} \ar@/^1pc/[d]^-{G^{\vee}} \ar[r]^-{\op{id}} & 
V \ar@/_1pc/[d]_-{G} \ar@/^1pc/[d]^-{G^{\vee}}  \\ 
V^{\vee} \ar[r]_-{\op{id}} & V^{\vee}}.
\]
\end{enumerate}

The two assignments $(\cH,\bPsi) \mapsto (V,G)$ and $(V,G) \mapsto (\cH,\bPsi)$ are inverse to each other and 
establish the equivalence of pairing data. 

\medskip

Let now $X$ be a complex  smooth projective variety. Assume for simplicity that the quantum product is convergent and so we have a complex analytic version $(\cH,\nabla)/B_{X}$ of the $\mathsf{A}$-model maximal F-bundle  for $X$. In this case we have $\cH_{u} \coloneqq \cH|_{B\times \{ u\}} = H^{\bullet}_{B}(X,\bbC)\otimes\cO_{B}$, and 
we can define $\bpsi_{u} \colon \cH_{u}\otimes_{\cO_{B}} \cH_{-u}  \ \to \ \cO_{B}$ to be the pairing between $\cH|_{B\times\{u\}} = H^{\bullet}_{B}(X,\bbQ)\otimes \cO_{B}$ to be the $\cO_{B}$-linear extension of the Poincar\'{e} pairing on  $H^{\bullet}_{B}(X,\bbQ)$. The pairing $\bpsi$ defined in this way is compatible with geometric data:
\begin{itemize}[wide]
\item[(a)] The pairing  $\psi$ is preserved by the connection $\nabla$. Indeed, in the $u$-direction we have $\nabla_{\partial_{u}} = \partial_{u} - u^{-2} \Eu\qup (-) + u^{-1}\Gr$, the operators 
$\Eu\qup (-)$ and $\mathsf{Gr}$ are independent of $u$, and by definition 
$\Eu\qup (-)$ is selfadjoint with respect to the Poincar\'{e} pairing while $\Gr$ is anti-selfadjoint with respect to the Poincar\'{e} pairing. Similarly, along $B$, the connection is given by covariant derivatives 
$\nabla_{\partial{\tau}} = \partial_{\tau} + u^{-1} \tau\qup(-)$ or $\nabla_{\xi q\partial{q}} = \xi q\partial_{q} + u^{-1} \xi\qup(-)$, and since the operators $\tau\qup(-)$ or $\xi\qup(-)$ are all self adjoint, it follows that $\psi$ is horizontal along $B$ as well.
\item[(b)]  The pairing $\psi$ takes integral values on the $\Gamma$-lattice in $\cH$. Indeed, recall from \cite[Section~3.1]{Katzarkov_Hodge_theoretic_aspects} that if $(\cH,\nabla)/B_{X}$ is of exponential type, one can construct a special $\bbZ$-local subsystem  $\mathcal{E} \subset \left(\cH|_{B_{X}\cdot \mathbf{D}^{\times}}\right)^{\nabla}$ in the local system of $\nabla$-horizontal sections in the restriction of $\cH$ to the complement of the divisor $\{u = 0\}$. The local system $\mathcal{E}$ has the property, that $\mathcal{E}\otimes_{\bbZ} \cO \, \cong \, \cH$ over $B_{X}\times \mathbf{D}^{\times}$, and conjecturally, for every $b \in B_{X}$ its restriction to $\{b\}\times \mathbf{D}^{\times}$ is compatible with the Deligne-Malgrange-Stokes filtration for the wildly ramified  connection $(\cH,\nabla)|_{\{b\}\times \mathbf{D}}$. For any point $(b,u) \in B_{X}\times \mathbf{D}^{\times}$, the fiber $\mathcal{E}_{(b,u)}$ of the local system $\mathcal{E}$ is defined as a particular embedding $\imath_{(\tilde{b},\log u)}  \colon \, K_{\op{top}}(X) \ \to \cH_{(b,u)}$, which depends on a lift $(\tilde{b},\log u)$ of $(b,u)$ in the universal cover of $B_{X}\times \mathbf{D}^{\times}$, and is given by a multiplication of the Chern character map by the Gamma class of $X$. The compatibility of $\psi$ and $\mathcal{E}$ is then expressed in the property that 
for every $(b,u)$, any two elements $s_{1}, s_{2} \in \mathcal{E}_{(b,u})$, and any choice of a branch of the logarithm which is well defined on the upper semicircle $\mathsf{arc}^{+}$,  we have 
that 
\[
\psi_{u}(\imath_{(\tilde{b},\log u)}(s_{1}),\imath_{(\tilde{b},\log (-u))}(\boldsymbol{\tau}_{\mathsf{arc}^{+}}(s_{2})) \ \in \ \bbZ.
\]
Furthermore, the value of this pairing does not depend on the lifts and if $s_{1}$ and $s_{2}$  are the Chern characters of two objects 
$E_{1}$ and $E_{2}$ in the derived category $D^{b}(X)$, then we have
\[\psi_{u}(\imath_{(\tilde{b},\log u)}(s_{1}),\imath_{(\tilde{b},\log (-u))}(\boldsymbol{\tau}_{\mathsf{arc}^{+}}(s_{2})) = \chi(\mathsf{RHom}(E_{1},E_{2})).\]
\item[(c)] Equivalently, the compatibility described in (b) can be reformulated as follows.
Consider the Euler pairing
defined by 
\[
\chi \, \colon \, H^{\bullet}_{B}(X,\bbC)\otimes H^{\bullet}_{B}(X,\bbC) \, \to \, \bbC, \quad \chi(a,b) \coloneqq \left\langle \, \sqrt{td(X)}\smile a,\, \sqrt{td(X)}\smile b\, \right\rangle,
\] 
where 
$\langle \, -,\, - \rangle$ is the \strongemph{Mukai pairing} 
\[
\langle\, v, \, v'\rangle \ = \ \int_{X} \, (v^{\vee} \smile v' )\smile \exp(c_{1}(X)/2),
\]
and the \strongemph{dual} $v^{\vee}$ of a cohomology class $v = \sum_{k} v_{k} \, \in \, \oplus_{k} H^{k}_{B}(X,\bbC)$ is defined by $v^{\vee} = \sum_{k} \, (\sqrt{-1})^{k} v_{k}$. 
As usual, here the Euler pairing is normalized so that if $E$ and $F$ are objects in $D^{b}(X)$, then $\chi(ch(E),ch(F)) = \sum (-1)^{k} \dim_{\bbC} \op{Ext}^{k}(E,F)$.

\qquad Let $V = H^{\bullet}_{B}(X,\bbC)$, and let $G \colon V \to V^{\vee}$ be the map $a \mapsto \chi(a,-)$. Then composing $G$ with the clockwise half circle parallel transport in $(\cH,\nabla)$ we get a non-degenerate symmetric pairing $\cH\otimes\nu^{*}\cH^{\vee} \to \cO$  which we will call the \strongemph{F-bundle Euler pairing}. When $(\cH,\nabla)/B_{X}$ is of exponential type,  the F-bundle Euler pairing is identified with the Poincar\'{e} duality non-degenerate pairing defined above. This follows from the observation that 
for every cohomology class 
$s \in H^{2\,\odd}_{B}(X)$, we have that the operator on $H^{\bullet}_{B}(X)$ of classical multiplication by $\exp(s)$ preserves the Euler pairing $\chi(-,-)$. Since the multiplication by $\sqrt{\mathsf{td}(X)}$ and the multiplication by the Gamma class differ by multiplications by such $\exp(s)$ we conclude that the non-symmetric pairing attached to $\psi$ and the map given by $a \mapsto \chi(a,-)$ will coincide.

\qquad Extra work is required to carry out this comparison for the non-archimedean $\mathsf{A}$-model F-bundles, and this will be discussed in more detail in \cite{KKPY-Gamma}.
\item[(d)] Finally, we have the duality  automorphism $S \colon H^{\bullet}_{B}(X,\bbC) \to  H^{\bullet}_{B}(X,\bbC)$ given by the cup product with $(-1)^{d}\exp(-c_{1}(x))$, which by Serre duality satisfies $\chi(a,b) = \chi(b,S(a))$ for all $a, b$. In particular $S = (G^{\vee})^{-1}\circ G$, i.e.\ $S$ equals the monodromy operator on $\cH_{1}$ for the loop in the $u$ plane going once in the counterclockwise direction around the origin. This gives an analytic automorphism of $\cH_{1}$ which will again denote by $S$. By construction $S$ will be preserved by the connection along $B$. Therefore, if $(\cH,\nabla)/B$ is a complex analytic F-bundle of exponential type with a non-degenerate pairing $\bpsi$ we can define the \strongemph{F-bundle duality automorphism} $S \, \colon \, \cH|_{B\times \{u=1\}} \, \to \, \cH|_{B\times \{u=1\}}$ to be the monodromy in the $u$-direction. Its compatibility with the pairing is expressed by the property $\chi(a,b) =\chi(b,S(a))$, where $\chi$ is the non-symmetric pairing on $\cH|_{B\times \{u=1\}}$ associated with $\bpsi$ via the above procedure.  Again, the definition in the case of non-archimedean analytic F-bundle requires extra work, and will be discussed in \cite{KKPY-Gamma}.
\end{itemize}

\medskip

The theory of Hodge atoms enhanced with Euler pairings and Serre automorphisms is completely straightforward and is just a repeat of the theory of undecorated Hodge atoms. The theory of Hodge atoms enhanced with an integral structure is much more difficult and requires substantial work and new ideas. Already the definition of an integral structure on a $\bbk$-analytic F-bundle is non-trivial due to the limited understanding of the Riemann-Hilbert correspondence  we have in the equal characteristic non-archimedean analytic setting. Additionally one needs to prove the existence of the $\Gamma$-corrected integral structure for analytic $\mathsf{A}$-model F-bundles of projective manifolds, and has to show that the blowup decomposition from Theorem~\ref{thm:blowup} is compatible with $\Gamma$-corrected integral structures. We deal with all these issues in  forthcoming work \cite{KKPY-Gamma}. 

Here we just note that with the enhancements in place, we get a stronger version of the non-rationality criterion Proposition~\ref{prop:non-rational}. In this version we conclude that if a smooth projective variety $X$ has an enhanced atomic composition containing an enhanced atom that does not come from a smooth projective variety of dimension $\leq \dim X - 2$, then $X$ can not be rational. 

\medskip

In the remainder of this section we will discuss some interesting applications of this theory.

\begin{example} \label{ex:cubic+plane}  Let $X$ be a  very general cubic
fourfold containing a plane $\mathsf{P} \subset X$. Then
the linear projection centered at $\mathsf{P}$ gives a quadric
  surface fibration on $Y \to \bbP^2$ on the fourfold $Y =
  \op{Bl}_{\mathsf{P}}X$. The fibration $Y \to \bbP^2$ gives rise to
a double cover $S \to \bbP^2$ - the $K3$ surface
parametrizing the rulings of the quadrics and a $\bbP^1$-bundle $\mathcal{R} \to S$ parametrizing the  lines in the quadrics.

If $\alpha \in H^2(S,\mathcal{O}^{\times})$ is the characteristic
class of the Brauer-Severi variety $\mathcal{R}$, then $\alpha$ defines a $2$-torsion gerbe ${}_{\alpha}S$ on $S$ , whose weight one derived category is the derived category $D^{b}(S,\alpha)$ of $\alpha$-twisted sheaves on $S$. 

In \cite{Kuznetsov-cubic4fold} Kuznetsov proved that for a general cubic with a plane the category $D^{b}(S,\alpha)$ can not be equivalent to the derived category of any smooth projective surface. In fact, he showed 
that the numerical $K$-group $K_{\op{num}}(S,\alpha)$ of  $D^{b}(S,\alpha)$ viewed as a lattice equipped with the $\alpha$-twisted Mukai pairing, can not be isomorphic to the numerical $K$-lattice of any smooth projective surface. On the other hand the spectrum calculation in the proof of Theorem~\ref{thm:cubic4} and formalism of Hodge atoms enhanced with pairings and integral structures immediately gives  that the $\alpha$-twisted Hodge structure on $H^{\bullet}(S)$ together with its integral structure and Mukai pairing gives an atom in the enhanced atomic composition of $X$. Since enhanced atoms give obstructions to rationality this implies that $X$ is not rational.
\end{example}

\begin{example} \label{ex:3quadrics} Let $X = \cap_{t \in \net}
  Q_{t}$ be the base locus of a general net $\net \cong \bbP^{2} \subset
  |\mathcal{O}_{\bbP^7}(2)|$ of quadrics in $\bbP^7$.
  Then $X$ contains a line $\ell \subset X$, and
\begin{enumerate}[wide]
\item $Y = \{ (t,\Pi) \, | \, \ell \subset \Pi \subset Q_{t} \}
  \subset \mathsf{Gr}(3,8)$  is a smooth 
  fourfold $Y$, which is birational to $X$.
\item $Y \to \net$ is a quadric surface fibration
  which gives rise to

\ \qquad - a double cover $S \to \bbP^2$ - a surface
of general type parametrizing the rulings of the quadrics.

\ \qquad - 
a $\bbP^1$-bundle $\mathcal{R} \to S$ parametrizing the
  lines in the fibers of $Y \to \bbP^2$.
\end{enumerate}
If $\alpha \in H^2(S,\mathcal{O}^{\times})$ is the characteristic
class of the Brauer-Severi variety $\mathcal{R}$, then again we get 
a $2$-torsion gerbe ${}_{\alpha}S$ on $S$, and a twisted derived category 
$D^{b}(S,\alpha)$.

Again the $\alpha$-twisted Hodge structure of $S$ together with its Mukai pairing and integral structure gives rise of an enhanced Hodge atom in the atomic composition of $X$.
Similarly to the case of a cubic with a plane, it  can be shown  \cite{DonagiPantev-3quadrics} that the numerical $K$-lattice of $D^{b}(X,\alpha)$
is not isomorphic to the numerical $K$-lattice of any smooth projective surface, and so $D^{b}(X,\alpha)$ can not be equivalent to the derived category of any smooth projective surface. Thus the very general $X$ can not be rational. This argument gives a new proof of a non-rationality theorem of Hassett-Pirutka-Tschinkel \cite{HPT-3quadrics}.  
\end{example}

\begin{example} \label{ex:4dimquartic}
Let $X \subset \bbP^5$ be a very general smooth quartic hypersurface. Its non-rationality is known by  the works of Koll\'{a}r \cite{KOL}, Schreieder \cite{SCH}, and Totaro \cite{TOT}. We outline an atom
theory argument also  proving that very general $X$ is not rational.

Suppose $X$ is Noether-Lefschetz general so that it has  Picard rank one. 
By repeating the argument in the proof of
Theorem~\ref{thm:cubic4} we see that the Hodge atom corresponding to
the eigenvalue zero for the quantum multiplication by the Euler vector
field looks like the Hodge atom of a surface of general type.

At the point $b$ in $B_{X}$ corresponding to the hyperplane class, the atomic composition  of $H^{\bullet}(X)$ looks like
\begin{equation} \label{eq:decompquartic}
H^{\bullet}(X)= \balpha(X)\oplus \balpha_{+}(X)\oplus \balpha_{-}(X).
\end{equation}
Here $\balpha(X)$ is the atom corresponding to eigenvalue zero of $\Eu\qup (-)$ and is 
shaped like an atom of a surface of a general type, while $\balpha_{+}(X)$ and $\balpha_{-}(X)$ are one dimensional atoms corresponding to eigenvalues which are scalings of $\pm 1$. 
The  Witt algebra argument from Remark~\ref{rem:Fbundle.witt.move} shows that  there is no further splitting if we move in a small analytic neighborhood of $b$.

As explained in Claim~\ref{claim:nefK} the quantum connection for a surface of general type has regular singularities and so the conjugacy class of its monodromy is well
defined and is unipotent on the even cohomology of the surface. 

However, from the point of view of the decomposition \eqref{eq:decompquartic}, in the cohomological grading the Serre automorphism  $S$ associated with the $\balpha(X)$ has a graded minimal polynomial
$S^3=[4]$.
This contradicts the unipotency of the monodromy on the even part of an atom of a surface of general type. Thus $\balpha(X)$ cannot be an enhanced atom associated with surface of general type which proves the non rationality of  very general $X$. 
\end{example}

\begin{example} Consider  a smooth three dimensional cubic
$X \subset \bbP^4$.  Again for the point $b \in B_{X}$ corresponding to the hyperplane class, the  atomic composition of $X$ consists of three atoms - two one dimensional atoms corresponding to non-zero eigenvalues of $\Eu\qup (-)$, and one atom $\balpha(X)$ corresponding to the eigenvalue $0$. The Witt algebra argument from Remark~\ref{rem:Fbundle.witt.move} shows that there is no further  splitting in a neighborhood of $b$.

By Claim~\ref{claim:nefK}  the quantum connection for a higher genus curve will have regular singularities, and a well defined conjugacy class of the  monodromy which will be unipotent on the even cohomology and quasi-unipotent with eigenvalue $-1$ on the odd cohomology.

The atom $\balpha(X)$  looks like the atom of a smooth projective curve of genus $5$  but  when viewed as a Hodge atom enhanced with a Serre automorphism, the corresponding Serre automorphism $S$ in the cohomological $\bbZ$-grading satisfies the graded minimal polynomial  
$S^5 = [3]$. Thus $S$ can not have an eigenvalue $-1$ and so $S$  cannot be  a Serre automorphism for a smooth genus $5$ curve. So we get that all  smooth  three dimensional cubics  are not rational. Similar approach applies to other three dimensional Fanos.
\end{example}

\bibliographystyle{plain}
\bibliography{dahema}

@article {Iritani_Quantum_D-modules,
    AUTHOR = {Iritani, Hiroshi},
     TITLE = {Quantum {$D$}-modules and generalized mirror transformations},
   JOURNAL = {Topology},
  FJOURNAL = {Topology. An International Journal of Mathematics},
    VOLUME = {47},
      YEAR = {2008},
    NUMBER = {4},
     PAGES = {225--276},
      ISSN = {0040-9383},
   MRCLASS = {53D45 (14N35 53D40)},
  MRNUMBER = {2416770},
MRREVIEWER = {Hsian-Hua Tseng},
       DOI = {10.1016/j.top.2007.07.001},
       URL = {https://doi.org/10.1016/j.top.2007.07.001},
}

@article{HYZZ_Decomposition,
  title={Decomposition and framing of {F}-bundles and applications to quantum cohomology},
  author={Hinault, Thorgal and Yu, Tony Yue and Zhang, Chi and Zhang, Shaowu},
  journal={arXiv preprint arXiv:2411.02266},
  year={2024}
}

@incollection {Dubrovin_Geometry_of_2D,
    AUTHOR = {Dubrovin, Boris},
     TITLE = {Geometry of {$2$}{D} topological field theories},
 BOOKTITLE = {Integrable systems and quantum groups ({M}ontecatini {T}erme,
              1993)},
    SERIES = {Lecture Notes in Math.},
    VOLUME = {1620},
     PAGES = {120--348},
 PUBLISHER = {Springer, Berlin},
      YEAR = {1996},
   MRCLASS = {58D29 (14N10 20F55 32G34 58F07 81T40)},
  MRNUMBER = {1397274},
MRREVIEWER = {N. J. Hitchin},
       DOI = {10.1007/BFb0094793},
       URL = {https://doi.org/10.1007/BFb0094793},
}

@article {Andre_Pour_une_theorie_inconditionnelle_des_motifs,
    AUTHOR = {Andr\'{e}, Yves},
     TITLE = {Pour une th\'{e}orie inconditionnelle des motifs},
   JOURNAL = {Inst. Hautes \'{E}tudes Sci. Publ. Math.},
  FJOURNAL = {Institut des Hautes \'{E}tudes Scientifiques. Publications
              Math\'{e}matiques},
    VOLUME = {83},
      YEAR = {1996},
     PAGES = {5--49},
      ISSN = {0073-8301},
   MRCLASS = {14F99 (14C25)},
  MRNUMBER = {1423019},
       URL = {http://www.numdam.org/item?id=PMIHES_1996__83__5_0},
}

@book {Manin_Frobenius_manifolds,
	AUTHOR = {Manin, Yuri I.},
	TITLE = {Frobenius manifolds, quantum cohomology, and moduli spaces},
	SERIES = {American Mathematical Society Colloquium Publications},
	VOLUME = {47},
	PUBLISHER = {American Mathematical Society, Providence, RI},
	YEAR = {1999},
	PAGES = {xiv+303},
	ISBN = {0-8218-1917-8},
	MRCLASS = {53D45 (14H10 14N35 18D50 32G34)},
	MRNUMBER = {1702284},
	MRREVIEWER = {Alexandre I. Kabanov},
	DOI = {10.1090/coll/047},
	URL = {https://doi.org/10.1090/coll/047},
}

@incollection {Katzarkov_Hodge_theoretic_aspects,
	AUTHOR = {Katzarkov, Ludmil and Kontsevich, Maxim and Pantev, Tony},
	TITLE = {Hodge theoretic aspects of mirror symmetry},
	BOOKTITLE = {From {H}odge theory to integrability and {TQFT} tt*-geometry},
	SERIES = {Proc. Sympos. Pure Math.},
	VOLUME = {78},
	PAGES = {87--174},
	PUBLISHER = {Amer. Math. Soc., Providence, RI},
	YEAR = {2008},
	MRCLASS = {14J32 (14A22 14C30 34M40)},
	MRNUMBER = {2483750},
	MRREVIEWER = {Mark Gross},
	DOI = {10.1090/pspum/078/2483750},
	URL = {https://doi.org/10.1090/pspum/078/2483750},
}

@inproceedings {Givental_Homological_geometry_and_mirror_symmetry,
	AUTHOR = {Givental, Alexander B.},
	TITLE = {Homological geometry and mirror symmetry},
	BOOKTITLE = {Proceedings of the {I}nternational {C}ongress of
	{M}athematicians, {V}ol. 1, 2 ({Z}\"{u}rich, 1994)},
	PAGES = {472--480},
	PUBLISHER = {Birkh\"{a}user, Basel},
	YEAR = {1995},
	MRCLASS = {58D10 (14J40 14N10 32G20)},
	MRNUMBER = {1403947},
	MRREVIEWER = {Bruce Hunt},
}

@article{Hori_Vafa_Mirror_symmetry,
	title={Mirror symmetry},
	author={Hori, Kentaro and Vafa, Cumrun},
	journal={arXiv preprint hep-th/0002222},
	year={2000}
}

@article {Abramovich_Torification_and_factorization,
	AUTHOR = {Abramovich, Dan and Karu, Kalle and Matsuki, Kenji and
	W\l{}odarczyk, Jaros\l{}aw},
	TITLE = {Torification and factorization of birational maps},
	JOURNAL = {J. Amer. Math. Soc.},
	FJOURNAL = {Journal of the American Mathematical Society},
	VOLUME = {15},
	YEAR = {2002},
	NUMBER = {3},
	PAGES = {531--572},
	ISSN = {0894-0347},
	MRCLASS = {14E05 (14E15 14L24)},
	MRNUMBER = {1896232},
	MRREVIEWER = {Alexandr V. Pukhlikov},
	DOI = {10.1090/S0894-0347-02-00396-X},
	URL = {https://doi.org/10.1090/S0894-0347-02-00396-X},
}

@preamble{
   "\def\cprime{$'$} "
}

@article {Givental_Equivariant_Gromov-Witten_invariants,
    AUTHOR = {Givental, Alexander B.},
     TITLE = {Equivariant {G}romov-{W}itten invariants},
   JOURNAL = {Internat. Math. Res. Notices},
  FJOURNAL = {International Mathematics Research Notices},
      YEAR = {1996},
    VOLUME = {13},
     PAGES = {613--663},
      ISSN = {1073-7928},
   MRCLASS = {14D07 (14D05 14J32 14N10 32G20)},
  MRNUMBER = {1408320 (97e:14015)},
MRREVIEWER = {Claire Voisin},
       DOI = {10.1155/S1073792896000414},
       URL = {http://dx.doi.org/10.1155/S1073792896000414},
}

@article {Kontsevich_Gromov-Witten_classes,
    AUTHOR = {Kontsevich, Maxim and Manin, Yuri},
     TITLE = {Gromov-{W}itten classes, quantum cohomology, and enumerative
              geometry},
   JOURNAL = {Comm. Math. Phys.},
  FJOURNAL = {Communications in Mathematical Physics},
    VOLUME = {164},
      YEAR = {1994},
    NUMBER = {3},
     PAGES = {525--562},
      ISSN = {0010-3616},
     CODEN = {CMPHAY},
   MRCLASS = {14N10 (53C15 58D10 58F05)},
  MRNUMBER = {1291244 (95i:14049)},
MRREVIEWER = {Dietmar A. Salamon},
       URL = {http://projecteuclid.org/euclid.cmp/1104270948},
}

@article {Behrend_Intrinsic_normal_cone,
    AUTHOR = {Behrend, Kai and Fantechi, Barbara},
     TITLE = {The intrinsic normal cone},
   JOURNAL = {Invent. Math.},
    VOLUME = {128},
      YEAR = {1997},
    NUMBER = {1},
     PAGES = {45--88},
      ISSN = {0020-9910},
     CODEN = {INVMBH},
   MRCLASS = {14F99 (14C15 14D20)},
  MRNUMBER = {1437495 (98e:14022)},
MRREVIEWER = {Tohru Nakashima},
       DOI = {10.1007/s002220050136},
       URL = {http://dx.doi.org/10.1007/s002220050136},
}

@article {Thuillier_Geometrie_toroidale,
    AUTHOR = {Thuillier, Amaury},
     TITLE = {G\'eom\'etrie toro\"\i dale et g\'eom\'etrie analytique non
              archim\'edienne. {A}pplication au type d'homotopie de certains
              sch\'emas formels},
   JOURNAL = {Manuscripta Math.},
  FJOURNAL = {Manuscripta Mathematica},
    VOLUME = {123},
      YEAR = {2007},
    NUMBER = {4},
     PAGES = {381--451},
      ISSN = {0025-2611},
     CODEN = {MSMHB2},
   MRCLASS = {14G22 (14E15 14M25)},
  MRNUMBER = {2320738 (2008g:14038)},
MRREVIEWER = {Alessandra Bertapelle},
       DOI = {10.1007/s00229-007-0094-2},
       URL = {http://dx.doi.org/10.1007/s00229-007-0094-2},
}

@article{Behrend_Gromov-Witten_invariants,
	Author = {Behrend, Kai},
	Coden = {INVMBH},
	Issn = {0020-9910},
	Journal = {Invent. Math.},
	Mrclass = {14D20 (14C25 14D22)},
	Mrnumber = {1431140 (98i:14015)},
	Mrreviewer = {Barbara Fantechi},
	Number = {3},
	Pages = {601--617},
	Title = {Gromov-{W}itten invariants in algebraic geometry},
	Volume = {127},
	Year = {1997}}

@book{Berkovich_Spectral_theory,
	Author = {Berkovich, Vladimir G.},
	Isbn = {0-8218-1534-2},
	Mrclass = {32P05 (32C15 32C37 46S10 47S10)},
	Mrnumber = {1070709 (91k:32038)},
	Mrreviewer = {W. Bartenwerfer},
	Pages = {x+169},
	Publisher = {American Mathematical Society, Providence, RI},
	Series = {Mathematical Surveys and Monographs},
	Title = {Spectral theory and analytic geometry over non-{A}rchimedean fields},
	Volume = {33},
	Year = {1990}}

@article {hertling_tt*_geometry,
    AUTHOR = {Hertling, Claus},
     TITLE = {{$tt^*$} geometry, {F}robenius manifolds, their connections,
              and the construction for singularities},
   JOURNAL = {J. Reine Angew. Math.},
  FJOURNAL = {Journal f\"{u}r die Reine und Angewandte Mathematik. [Crelle's
              Journal]},
    VOLUME = {555},
      YEAR = {2003},
     PAGES = {77--161},
      ISSN = {0075-4102},
   MRCLASS = {32S30 (32G20 32J25 53C28 53D45)},
  MRNUMBER = {1956595},
MRREVIEWER = {Christian Sevenheck},
       DOI = {10.1515/crll.2003.015},
       URL = {https://doi.org/10.1515/crll.2003.015},
}

@book {hertling-frobenius,
    AUTHOR = {Hertling, C.},
     TITLE = {Frobenius manifolds and moduli spaces for singularities},
    SERIES = {Cambridge Tracts in Mathematics},
    VOLUME = {151},
 PUBLISHER = {Cambridge University Press, Cambridge},
      YEAR = {2002},
     PAGES = {x+270},
      ISBN = {0-521-81296-8},
   MRCLASS = {32S40 (14B05 34M35 53D45)},
  MRNUMBER = {1924259},
MRREVIEWER = {Ezra Getzler},
       DOI = {10.1017/CBO9780511543104},
       URL = {https://doi.org/10.1017/CBO9780511543104},
}

@article{Voisin-cubicTorelli,
    author = {Voisin, Claire},
    title = {Th\'{e}or\`{e}me de {T}orelli pour les cubiques de $\mathbb{P}^{5}$},
    journal = {Invent. Math.},
    volume = 86,
    number = 3,
    year = 1986,
    pages = {577-601},
}

@article{Looijenga-cubicTorelli,
    author = {Looijenga, Eduard},
    title = {The period map for cubic fourfolds.},
    journal = {Invent. Math.},
    year = 2009,
    volume = 177,
    number = 1,
    pages = {213-233},
}

@misc{Charles-cubicTorelli,
      title={A remark on the {T}orelli theorem for cubic fourfolds}, 
      author={Charles, Fran\c{c}ois},
      year={2012},
      eprint={1209.4509},
      archivePrefix={arXiv},
      primaryClass={math.AG}
}

@article{Wlodarczyk_factorization,
    author = {W{\l}odarczyk, Jaroslaw},
    title = {Birational cobordisms and factorization of birational maps},
    journal = {J. Algebraic Geom.},
    volume = {9},
    number = {3},
    year = {2000},
    pages = {425--449}
}

@book {hms-periods.nori,
    AUTHOR = {Huber, Annette and M\"{u}ller-Stach, Stefan},
     TITLE = {Periods and {N}ori motives},
    SERIES = {Ergebnisse der Mathematik und ihrer Grenzgebiete. 3. Folge. A
              Series of Modern Surveys in Mathematics [Results in
              Mathematics and Related Areas. 3rd Series. A Series of Modern
              Surveys in Mathematics]},
    VOLUME = {65},
      NOTE = {With contributions by Benjamin Friedrich and Jonas von
              Wangenheim},
 PUBLISHER = {Springer, Cham},
      YEAR = {2017},
     PAGES = {xxiii+372},
      ISBN = {978-3-319-50925-9; 978-3-319-50926-6},
   MRCLASS = {14F42 (11G05 14C15 14C30 19E15 32G20)},
  MRNUMBER = {3618276},
MRREVIEWER = {Alberto\ Collino},
}

@misc{nori-tifr,
Author =  {Nori, Madhav},
Title  =  {{TIFR} notes on motives},
Year   =  {1998},
Pages  =  {32pp.}
}

@misc{stacksproject,
author  =  {Authors, The Stacks Project},
title   =  {Stacks Project},
url     =  {https://stacks.math.columbia.edu/},
year    = {2025}
}

@book {murre,
    AUTHOR = {Murre, Jacob P. and Nagel, Jan and Peters, Chris A. M.},
     TITLE = {Lectures on the theory of pure motives},
    SERIES = {University Lecture Series},
    VOLUME = {61},
 PUBLISHER = {American Mathematical Society, Providence, RI},
      YEAR = {2013},
     PAGES = {x+149},
      ISBN = {978-0-8218-9434-7},
   MRCLASS = {14-02 (14C15 14C25 19E15)},
  MRNUMBER = {3052734},
MRREVIEWER = {Florence\ Lecomte},
       DOI = {10.1090/ulect/061},
       URL = {https://doi.org/10.1090/ulect/061},
}

@book {Andre-book,
    AUTHOR = {Andr\'{e}, Yves},
     TITLE = {Une introduction aux motifs (motifs purs, motifs mixtes,
              p\'{e}riodes)},
    SERIES = {Panoramas et Synth\`eses [Panoramas and Syntheses]},
    VOLUME = {17},
 PUBLISHER = {Soci\'{e}t\'{e} Math\'{e}matique de France, Paris},
      YEAR = {2004},
     PAGES = {xii+261},
      ISBN = {2-85629-164-3},
   MRCLASS = {14F42 (11J91 14C25 19E15)},
  MRNUMBER = {2115000},
MRREVIEWER = {Luca\ Barbieri Viale},
}

@misc{Iritani_blowup,
      title={Quantum cohomology of blowups}, 
      author={Iritani, Hiroshi},
      year={2023},
      eprint={2307.13555},
      archivePrefix={arXiv},
      primaryClass={math.AG},
      pages = {60pp.}
}

@misc{Iritani_Koto,
      title={Quantum cohomology of projective bundles}, 
      author={Iritani, Hiroshi and Koto, Yuki},
      year={2024},
      eprint={2307.03696},
      archivePrefix={arXiv},
      primaryClass={math.AG},
      pages = {38pp.},
}

@article {Bittner,
    AUTHOR = {Bittner, Franziska},
     TITLE = {The universal {E}uler characteristic for varieties of
              characteristic zero},
   JOURNAL = {Compos. Math.},
  FJOURNAL = {Compositio Mathematica},
    VOLUME = {140},
      YEAR = {2004},
    NUMBER = {4},
     PAGES = {1011--1032},
      ISSN = {0010-437X,1570-5846},
   MRCLASS = {14F42 (14C35 14E05)},
  MRNUMBER = {2059227},
MRREVIEWER = {Thomas\ Geisser},
       DOI = {10.1112/S0010437X03000617},
       URL = {https://doi.org/10.1112/S0010437X03000617},
}

@article {Beauville-quadrics,
    AUTHOR = {Beauville, Arnaud},
     TITLE = {Vari\'et\'es de {P}rym et jacobiennes interm\'ediaires},
   JOURNAL = {Ann. Sci. \'Ecole Norm. Sup. (4)},
  FJOURNAL = {Annales Scientifiques de l'\'Ecole Normale Sup\'erieure.
              Quatri\`eme S\'erie},
    VOLUME = {10},
      YEAR = {1977},
    NUMBER = {3},
     PAGES = {309--391},
      ISSN = {0012-9593},
   MRCLASS = {14K30 (14C30 14J99 14M20)},
  MRNUMBER = {472843},
MRREVIEWER = {T.\ Oda},
       URL = {http://www.numdam.org/item?id=ASENS_1977_4_10_3_309_0},
}

@article {Beauville-Donagi,
    AUTHOR = {Beauville, Arnaud and Donagi, Ron},
     TITLE = {La vari\'et\'e{} des droites d'une hypersurface cubique de
              dimension {$4$}},
   JOURNAL = {C. R. Acad. Sci. Paris S\'er. I Math.},
  FJOURNAL = {Comptes Rendus des S\'eances de l'Acad\'emie des Sciences.
              S\'erie I. Math\'ematique},
    VOLUME = {301},
      YEAR = {1985},
    NUMBER = {14},
     PAGES = {703--706},
      ISSN = {0249-6291},
   MRCLASS = {14J35 (14C30)},
  MRNUMBER = {818549},
MRREVIEWER = {Fabio\ Bardelli},
}

@article {debarre-GM,
    AUTHOR = {Debarre, Olivier and Mongardi, Giovanni},
     TITLE = {Gushel-{M}ukai varieties with many symmetries and an explicit
              irrational {G}ushel-{M}ukai threefold},
   JOURNAL = {Boll. Unione Mat. Ital.},
  FJOURNAL = {Bollettino dell'Unione Matematica Italiana},
    VOLUME = {15},
      YEAR = {2022},
    NUMBER = {1-2},
     PAGES = {133--161},
      ISSN = {1972-6724,2198-2759},
   MRCLASS = {14J45 (14E08 14J30 14J42 14J50 14J70 14K22 14K30)},
  MRNUMBER = {4390546},
MRREVIEWER = {Constantin\ Shramov},
       DOI = {10.1007/s40574-021-00293-6},
       URL = {https://doi.org/10.1007/s40574-021-00293-6},
}

@article {DK-GM,
    AUTHOR = {Debarre, Olivier and Kuznetsov, Alexander},
     TITLE = {Gushel-{M}ukai varieties: classification and birationalities},
   JOURNAL = {Algebr. Geom.},
  FJOURNAL = {Algebraic Geometry},
    VOLUME = {5},
      YEAR = {2018},
    NUMBER = {1},
     PAGES = {15--76},
      ISSN = {2313-1691,2214-2584},
   MRCLASS = {14J10 (14E07 14J45 14J60)},
  MRNUMBER = {3734109},
MRREVIEWER = {James\ McKernan},
       DOI = {10.14231/ag-2018-002},
       URL = {https://doi.org/10.14231/ag-2018-002},
}

@article {DK-GMIJ,
    AUTHOR = {Debarre, Olivier and Kuznetsov, Alexander},
     TITLE = {Gushel-{M}ukai varieties: intermediate {J}acobians},
   JOURNAL = {\'Epijournal G\'eom. Alg\'ebrique},
  FJOURNAL = {\'Epijournal de G\'eom\'etrie Alg\'ebrique. EPIGA},
    VOLUME = {4},
      YEAR = {2020},
     PAGES = {Art. 19, 45},
      ISSN = {2491-6765},
   MRCLASS = {14J45 (14J30 14J40 14K30 14M15)},
  MRNUMBER = {4191422},
MRREVIEWER = {James\ McKernan},
       DOI = {10.46298/epiga.2020.volume4.6475},
       URL = {https://doi.org/10.46298/epiga.2020.volume4.6475},
}

@article {Hassett-cubic4fold,
    AUTHOR = {Hassett, Brendan},
     TITLE = {Special cubic fourfolds},
   JOURNAL = {Compositio Math.},
  FJOURNAL = {Compositio Mathematica},
    VOLUME = {120},
      YEAR = {2000},
    NUMBER = {1},
     PAGES = {1--23},
      ISSN = {0010-437X,1570-5846},
   MRCLASS = {14J35 (14J10 14J28 14J45)},
  MRNUMBER = {1738215},
MRREVIEWER = {Elham\ Izadi},
       DOI = {10.1023/A:1001706324425},
       URL = {https://doi.org/10.1023/A:1001706324425},
}

@incollection {Hassett-survey,
    AUTHOR = {Hassett, Brendan},
     TITLE = {Cubic fourfolds, {K}3 surfaces, and rationality questions},
 BOOKTITLE = {Rationality problems in algebraic geometry},
    SERIES = {Lecture Notes in Math.},
    VOLUME = {2172},
     PAGES = {29--66},
 PUBLISHER = {Springer, Cham},
      YEAR = {2016},
      ISBN = {978-3-319-46208-0; 978-3-319-46209-7},
   MRCLASS = {14J35 (14E08 14J28 14M20)},
  MRNUMBER = {3618665},
MRREVIEWER = {Zhiyu\ Tian},
}

@article {HPT-qb,
    AUTHOR = {Hassett, Brendan and Pirutka, Alena and Tschinkel, Yuri},
     TITLE = {Stable rationality of quadric surface bundles over surfaces},
   JOURNAL = {Acta Math.},
  FJOURNAL = {Acta Mathematica},
    VOLUME = {220},
      YEAR = {2018},
    NUMBER = {2},
     PAGES = {341--365},
      ISSN = {0001-5962,1871-2509},
   MRCLASS = {14E08 (14D06)},
  MRNUMBER = {3849287},
MRREVIEWER = {Andreas\ H\"oring},
       DOI = {10.4310/ACTA.2018.v220.n2.a4},
       URL = {https://doi.org/10.4310/ACTA.2018.v220.n2.a4},
}

@article {HPT-ds,
    AUTHOR = {Hassett, Brendan and Pirutka, Alena and Tschinkel, Yuri},
     TITLE = {A very general quartic double fourfold is not stably rational},
   JOURNAL = {Algebr. Geom.},
  FJOURNAL = {Algebraic Geometry},
    VOLUME = {6},
      YEAR = {2019},
    NUMBER = {1},
     PAGES = {64--75},
      ISSN = {2313-1691,2214-2584},
   MRCLASS = {14E08},
  MRNUMBER = {3904799},
MRREVIEWER = {Wenhao\ Ou},
       DOI = {10.14231/ag-2019-004},
       URL = {https://doi.org/10.14231/ag-2019-004},
}

@article {kulikov-cubic4fold,
    AUTHOR = {Kulikov, Viktor S.},
     TITLE = {A remark on the nonrationality of a generic cubic fourfold},
   JOURNAL = {Mat. Zametki},
  FJOURNAL = {Matematicheskie Zametki},
    VOLUME = {83},
      YEAR = {2008},
    NUMBER = {1},
     PAGES = {61--68},
      ISSN = {0025-567X,2305-2880},
   MRCLASS = {14J35 (14C30 14E08)},
  MRNUMBER = {2399998},
MRREVIEWER = {Adrian\ Langer},
       DOI = {10.1134/S0001434608010070},
       URL = {https://doi.org/10.1134/S0001434608010070},
}

@article {KP-Fano3nc,
    AUTHOR = {Kuznetsov, Alexander and Prokhorov, Yuri},
     TITLE = {Rationality of {F}ano threefolds over non-closed fields},
   JOURNAL = {Amer. J. Math.},
  FJOURNAL = {American Journal of Mathematics},
    VOLUME = {145},
      YEAR = {2023},
    NUMBER = {2},
     PAGES = {335--411},
      ISSN = {0002-9327,1080-6377},
   MRCLASS = {14E08 (14J45)},
  MRNUMBER = {4570985},
MRREVIEWER = {Igor\ Krylov},
       DOI = {10.1353/ajm.2023.0008},
       URL = {https://doi.org/10.1353/ajm.2023.0008},
}

@incollection {KP-Mukai,
    AUTHOR = {Kuznetsov, Alexander and Prokhorov, Yuri},
     TITLE = {Rationality of {M}ukai varieties over non-closed fields},
 BOOKTITLE = {Rationality of varieties},
    SERIES = {Progr. Math.},
    VOLUME = {342},
     PAGES = {249--290},
 PUBLISHER = {Birkh\"auser/Springer, Cham},
      YEAR = {[2021] \copyright2021},
      ISBN = {978-3-030-75420-4; 978-3-030-75421-1},
   MRCLASS = {14E08 (14E05 14E30 14J45)},
  MRNUMBER = {4383701},
MRREVIEWER = {Andrea\ Bruno},
       DOI = {10.1007/978-3-030-75421-1\_10},
       URL = {https://doi.org/10.1007/978-3-030-75421-1_10},
}

@incollection {Kuznetsov-derived,
    AUTHOR = {Kuznetsov, Alexander},
     TITLE = {Derived categories view on rationality problems},
 BOOKTITLE = {Rationality problems in algebraic geometry},
    SERIES = {Lecture Notes in Math.},
    VOLUME = {2172},
     PAGES = {67--104},
 PUBLISHER = {Springer, Cham},
      YEAR = {2016},
      ISBN = {978-3-319-46208-0; 978-3-319-46209-7},
   MRCLASS = {14F05 (14E07 14E08 14J45)},
  MRNUMBER = {3618666},
MRREVIEWER = {Sofia\ Tirabassi},
}

@incollection {Kuznetsov-cubic4fold,
    AUTHOR = {Kuznetsov, Alexander},
     TITLE = {Derived categories of cubic fourfolds},
 BOOKTITLE = {Cohomological and geometric approaches to rationality
              problems},
    SERIES = {Progr. Math.},
    VOLUME = {282},
     PAGES = {219--243},
 PUBLISHER = {Birkh\"auser Boston, Boston, MA},
      YEAR = {2010},
      ISBN = {978-0-8176-4933-3},
   MRCLASS = {14F05 (14E05)},
  MRNUMBER = {2605171},
MRREVIEWER = {Paolo\ Stellari},
       DOI = {10.1007/978-0-8176-4934-0\_9},
       URL = {https://doi.org/10.1007/978-0-8176-4934-0_9},
}

@article {ABP-qb,
    AUTHOR = {Auel, Asher and B\"ohning, Christian and Pirutka, Alena},
     TITLE = {Stable rationality of quadric and cubic surface bundle
              fourfolds},
   JOURNAL = {Eur. J. Math.},
  FJOURNAL = {European Journal of Mathematics},
    VOLUME = {4},
      YEAR = {2018},
    NUMBER = {3},
     PAGES = {732--760},
      ISSN = {2199-675X,2199-6768},
   MRCLASS = {14C25 (14D06 14E05 14E08 14F22 14J20 14J26)},
  MRNUMBER = {3851115},
MRREVIEWER = {Claudio\ Pedrini},
       DOI = {10.1007/s40879-018-0233-1},
       URL = {https://doi.org/10.1007/s40879-018-0233-1},
}

@incollection {Pirutka-0cycles,
    AUTHOR = {Pirutka, Alena},
     TITLE = {Varieties that are not stably rational, zero-cycles and
              unramified cohomology},
 BOOKTITLE = {Algebraic geometry: {S}alt {L}ake {C}ity 2015},
    SERIES = {Proc. Sympos. Pure Math.},
    VOLUME = {97.2},
     PAGES = {459--483},
 PUBLISHER = {Amer. Math. Soc., Providence, RI},
      YEAR = {2018},
      ISBN = {978-1-4704-3578-3},
   MRCLASS = {14E08 (14C25 14F22)},
  MRNUMBER = {3821181},
MRREVIEWER = {Carla\ Novelli},
       DOI = {10.1090/pspum/097.2/16},
       URL = {https://doi.org/10.1090/pspum/097.2/16},
}

@incollection {HPT-threeq,
    AUTHOR = {Hassett, Brendan and Pirutka, Alena and Tschinkel, Yuri},
     TITLE = {Intersections of three quadrics in {$\mathbb P^7$}},
 BOOKTITLE = {Surveys in differential geometry 2017. {C}elebrating the 50th
              anniversary of the {J}ournal of {D}ifferential {G}eometry},
    SERIES = {Surv. Differ. Geom.},
    VOLUME = {22},
     PAGES = {259--274},
 PUBLISHER = {Int. Press, Somerville, MA},
      YEAR = {2018},
      ISBN = {978-1-57146-361-6},
   MRCLASS = {14E08 (14M10)},
  MRNUMBER = {3838120},
MRREVIEWER = {Andrea\ Fanelli},
}

@article {Hyubrechts-K3cubic,
    AUTHOR = {Huybrechts, Daniel},
     TITLE = {The {K}3 category of a cubic fourfold},
   JOURNAL = {Compos. Math.},
  FJOURNAL = {Compositio Mathematica},
    VOLUME = {153},
      YEAR = {2017},
    NUMBER = {3},
     PAGES = {586--620},
      ISSN = {0010-437X,1570-5846},
   MRCLASS = {14F05 (14J28 14J35)},
  MRNUMBER = {3705236},
MRREVIEWER = {Alan\ Matthew\ Thompson},
       DOI = {10.1112/S0010437X16008137},
       URL = {https://doi.org/10.1112/S0010437X16008137},
}

@book {Huybrechts-survey,
    AUTHOR = {Huybrechts, Daniel},
     TITLE = {The geometry of cubic hypersurfaces},
    SERIES = {Cambridge Studies in Advanced Mathematics},
    VOLUME = {206},
 PUBLISHER = {Cambridge University Press, Cambridge},
      YEAR = {2023},
     PAGES = {xvii+441},
      ISBN = {978-1-009-28000-6; [9781009280020]},
   MRCLASS = {14J70 (14C30 14J30 14J35)},
  MRNUMBER = {4589520},
MRREVIEWER = {Fei\ Hu},
}

@article {AT,
    AUTHOR = {Addington, Nicolas and Thomas, Richard},
     TITLE = {Hodge theory and derived categories of cubic fourfolds},
   JOURNAL = {Duke Math. J.},
  FJOURNAL = {Duke Mathematical Journal},
    VOLUME = {163},
      YEAR = {2014},
    NUMBER = {10},
     PAGES = {1885--1927},
      ISSN = {0012-7094,1547-7398},
   MRCLASS = {14F05 (14J10 14J28 14J35)},
  MRNUMBER = {3229044},
MRREVIEWER = {Pawel\ Sosna},
       DOI = {10.1215/00127094-2738639},
       URL = {https://doi.org/10.1215/00127094-2738639},
}

@article {Voisin-unirational,
    AUTHOR = {Voisin, Claire},
     TITLE = {Unirational threefolds with no universal codimension {$2$}
              cycle},
   JOURNAL = {Invent. Math.},
  FJOURNAL = {Inventiones Mathematicae},
    VOLUME = {201},
      YEAR = {2015},
    NUMBER = {1},
     PAGES = {207--237},
      ISSN = {0020-9910,1432-1297},
   MRCLASS = {14M20 (14E08 14F45 14J30 14K30)},
  MRNUMBER = {3359052},
MRREVIEWER = {M.\ Kh.\ Gizatullin},
       DOI = {10.1007/s00222-014-0551-y},
       URL = {https://doi.org/10.1007/s00222-014-0551-y},
}

@incollection {Voisin-decomposition,
    AUTHOR = {Voisin, Claire},
     TITLE = {Birational invariants and decomposition of the diagonal},
 BOOKTITLE = {Birational geometry of hypersurfaces},
    SERIES = {Lect. Notes Unione Mat. Ital.},
    VOLUME = {26},
     PAGES = {3--71},
 PUBLISHER = {Springer, Cham},
      YEAR = {[2019] \copyright2019},
      ISBN = {978-3-030-18637-1; 978-3-030-18638-8},
   MRCLASS = {14C15 (14E08)},
  MRNUMBER = {4401008},
       DOI = {10.1007/978-3-030-18638-8\_1},
       URL = {https://doi.org/10.1007/978-3-030-18638-8_1},
}

@incollection {Voisin-stable,
    AUTHOR = {Voisin, Claire},
     TITLE = {Stable birational invariants and the {L}\"uroth problem},
 BOOKTITLE = {Surveys in differential geometry 2016. {A}dvances in geometry
              and mathematical physics},
    SERIES = {Surv. Differ. Geom.},
    VOLUME = {21},
     PAGES = {313--342},
 PUBLISHER = {Int. Press, Somerville, MA},
      YEAR = {2016},
      ISBN = {978-1-57146-322-7; 1-57146-322-4},
   MRCLASS = {14E08 (14C15)},
  MRNUMBER = {3525100},
MRREVIEWER = {Kieran\ G.\ O'Grady},
}

@incollection {Voisin-integral,
    AUTHOR = {Voisin, Claire},
     TITLE = {Integral {H}odge classes, decompositions of the diagonal, and
              rationality questions},
 BOOKTITLE = {Trends in contemporary mathematics},
    SERIES = {Springer INdAM Ser.},
    VOLUME = {8},
     PAGES = {137--149},
 PUBLISHER = {Springer, Cham},
      YEAR = {2014},
      ISBN = {978-3-319-05253-3; 978-3-319-05254-0},
   MRCLASS = {14E08 (14C15)},
  MRNUMBER = {3586396},
MRREVIEWER = {C.\ A. M. Peters},
}

@misc{Chen-ET,
      title={On the exponential type conjecture}, 
      author={Chen, Zihong},
      year={2024},
      eprint={2409.03922},
      archivePrefix={arXiv},
      primaryClass={math.SG},
      url={https://arxiv.org/abs/2409.03922}, 
}

@misc{PS-ET,
      title={The quantum connection, {F}ourier-{L}aplace transform, and families of {A}-infinity-categories}, 
      author={Pomerleano, Daniel and Seidel, Paul},
      year={2024},
      eprint={2308.13567},
      archivePrefix={arXiv},
      primaryClass={math.SG},
      url={https://arxiv.org/abs/2308.13567}, 
}

@incollection {Batyrev-StringyHodge,
    AUTHOR = {Batyrev, Victor V.},
     TITLE = {Stringy {H}odge numbers of varieties with {G}orenstein
              canonical singularities},
 BOOKTITLE = {Integrable systems and algebraic geometry ({K}obe/{K}yoto,
              1997)},
     PAGES = {1--32},
 PUBLISHER = {World Sci. Publ., River Edge, NJ},
      YEAR = {1998},
      ISBN = {981-02-3266-7},
   MRCLASS = {14J32 (14B05 14M25)},
  MRNUMBER = {1672108},
}

@incollection {Wang-Kequivalence,
    AUTHOR = {Wang, Chin-Lung},
     TITLE = {{$K$}-equivalence in birational geometry},
 BOOKTITLE = {Second {I}nternational {C}ongress of {C}hinese
              {M}athematicians},
    SERIES = {New Stud. Adv. Math.},
    VOLUME = {4},
     PAGES = {199--216},
 PUBLISHER = {Int. Press, Somerville, MA},
      YEAR = {2004},
      ISBN = {1-57146-155-8},
   MRCLASS = {14E05},
  MRNUMBER = {2497984},
MRREVIEWER = {Michael\ A.\ van Opstall},
}

@article {Ito,
    AUTHOR = {Ito, Tetsushi},
     TITLE = {Stringy {H}odge numbers and {$p$}-adic {H}odge theory},
   JOURNAL = {Compos. Math.},
  FJOURNAL = {Compositio Mathematica},
    VOLUME = {140},
      YEAR = {2004},
    NUMBER = {6},
     PAGES = {1499--1517},
      ISSN = {0010-437X,1570-5846},
   MRCLASS = {14F43 (11S80 14E05 14F30)},
  MRNUMBER = {2098399},
MRREVIEWER = {Julien\ Sebag},
       DOI = {10.1112/S0010437X04001095},
       URL = {https://doi.org/10.1112/S0010437X04001095},
}

@incollection {Deligne-canext,
    AUTHOR = {Deligne, Pierre},
     TITLE = {Local behavior of {H}odge structures at infinity},
 BOOKTITLE = {Mirror symmetry, {II}},
    SERIES = {AMS/IP Stud. Adv. Math.},
    VOLUME = {1},
     PAGES = {683--699},
 PUBLISHER = {Amer. Math. Soc., Providence, RI},
      YEAR = {1997},
      ISBN = {0-8218-0634-3},
   MRCLASS = {14D07 (14N10 32J25)},
  MRNUMBER = {1416353},
MRREVIEWER = {Bernd\ Siebert},
       DOI = {10.1090/amsip/001/25},
       URL = {https://doi.org/10.1090/amsip/001/25},
}

@misc{KKPY-Gamma,
    author = {Katzarkov, Ludmil and Kontsevich, Maxim and Pantev, Tony and YU, Tony Yue},
    title = {Enhanced atoms and integrality},
    year  = {2025},         
    note = {in preparation}
}

@misc{CGKK,
    author = {Cavenaghi, Leonardo and Grama, Lino and  Katzarkov, Ludmil and Kontsevich, Maxim},
    title = {Atoms meet symbols},
    year  = {2025},         
    note = {in preparation}
}

@misc{DonagiPantev-3quadrics,
    author = {Donagi, Ron and Pantev, Tony},
    title  = {The intersection of three quadrics in {$\mathbb{P}^{7}$} revisited},
    year  = {2025},
    note = {preprint}
}

@incollection {HPT-3quadrics,
    AUTHOR = {Hassett, Brendan and Pirutka, Alena and Tschinkel, Yuri},
     TITLE = {Intersections of three quadrics in {$\mathbb{P}^7$}},
 BOOKTITLE = {Surveys in differential geometry 2017. {C}elebrating the 50th
              anniversary of the {J}ournal of {D}ifferential {G}eometry},
    SERIES = {Surv. Differ. Geom.},
    VOLUME = {22},
     PAGES = {259--274},
 PUBLISHER = {Int. Press, Somerville, MA},
      YEAR = {2018},
      ISBN = {978-1-57146-361-6},
   MRCLASS = {14E08 (14M10)},
  MRNUMBER = {3838120},
MRREVIEWER = {Andrea\ Fanelli},
}

@article {Lee_YP-qLefschetz,
    AUTHOR = {Lee, Yuan-Pin},
     TITLE = {Quantum {L}efschetz hyperplane theorem},
   JOURNAL = {Invent. Math.},
  FJOURNAL = {Inventiones Mathematicae},
    VOLUME = {145},
      YEAR = {2001},
    NUMBER = {1},
     PAGES = {121--149},
      ISSN = {0020-9910,1432-1297},
   MRCLASS = {14N35},
  MRNUMBER = {1839288},
MRREVIEWER = {Gilberto\ Bini},
       DOI = {10.1007/s002220100145},
       URL = {https://doi.org/10.1007/s002220100145},
}

@article {Kim-qLefschetz,
    AUTHOR = {Kim, Bumsig},
     TITLE = {Quantum hyperplane section theorem for homogeneous spaces},
   JOURNAL = {Acta Math.},
  FJOURNAL = {Acta Mathematica},
    VOLUME = {183},
      YEAR = {1999},
    NUMBER = {1},
     PAGES = {71--99},
      ISSN = {0001-5962,1871-2509},
   MRCLASS = {14N35 (14J32 14M10)},
  MRNUMBER = {1719555},
MRREVIEWER = {Andrew\ Kresch},
       DOI = {10.1007/BF02392947},
       URL = {https://doi.org/10.1007/BF02392947},
}

@article {CFKS,
    AUTHOR = {Ciocan-Fontanine, Ionu\c{t} and Kim, Bumsig and Sabbah,
              Claude},
     TITLE = {The abelian/nonabelian correspondence and {F}robenius
              manifolds},
   JOURNAL = {Invent. Math.},
  FJOURNAL = {Inventiones Mathematicae},
    VOLUME = {171},
      YEAR = {2008},
    NUMBER = {2},
     PAGES = {301--343},
      ISSN = {0020-9910,1432-1297},
   MRCLASS = {14N35 (14F43 53D45)},
  MRNUMBER = {2367022},
MRREVIEWER = {Jake\ Philip\ Solomon},
       DOI = {10.1007/s00222-007-0082-x},
       URL = {https://doi.org/10.1007/s00222-007-0082-x},
}

@article {CoatesGivental,
    AUTHOR = {Coates, Tom and Givental, Alexander},
     TITLE = {Quantum {R}iemann-{R}och, {L}efschetz and {S}erre},
   JOURNAL = {Ann. of Math. (2)},
  FJOURNAL = {Annals of Mathematics. Second Series},
    VOLUME = {165},
      YEAR = {2007},
    NUMBER = {1},
     PAGES = {15--53},
      ISSN = {0003-486X,1939-8980},
   MRCLASS = {14N35 (14C40)},
  MRNUMBER = {2276766},
MRREVIEWER = {Hsian-Hua\ Tseng},
       DOI = {10.4007/annals.2007.165.15},
       URL = {https://doi.org/10.4007/annals.2007.165.15},
}

@book {GrauertRemmert_Stein,
    AUTHOR = {Grauert, Hans and Remmert, Reinhold},
     TITLE = {Theory of {S}tein spaces},
    SERIES = {Classics in Mathematics},
      NOTE = {Translated from the German by Alan Huckleberry,
              Reprint of the 1979 translation},
 PUBLISHER = {Springer-Verlag, Berlin},
      YEAR = {2004},
     PAGES = {xxii+255},
      ISBN = {3-540-00373-8},
   MRCLASS = {32E10 (32-01 32Q28)},
  MRNUMBER = {2029201},
       DOI = {10.1007/978-3-642-18921-0},
       URL = {https://doi.org/10.1007/978-3-642-18921-0},
}

@book {Sabbah-Isomonodromic,
    AUTHOR = {Sabbah, Claude},
     TITLE = {Isomonodromic deformations and {F}robenius manifolds},
    SERIES = {Universitext},
   EDITION = {French},
      NOTE = {An introduction},
 PUBLISHER = {Springer-Verlag London, Ltd., London; EDP Sciences, Les Ulis},
      YEAR = {2007},
     PAGES = {xiv+279},
}

@incollection {KyojiSaito-residue,
    AUTHOR = {Saito, Kyoji},
     TITLE = {The higher residue pairings {$K\sb{F}\sp{(k)}$}\ for a family
              of hypersurface singular points},
 BOOKTITLE = {Singularities, {P}art 2 ({A}rcata, {C}alif., 1981)},
    SERIES = {Proc. Sympos. Pure Math.},
    VOLUME = {40},
     PAGES = {441--463},
 PUBLISHER = {Amer. Math. Soc., Providence, RI},
      YEAR = {1983},
      ISBN = {0-8218-1466-4},
   MRCLASS = {32G11 (32A27 32C37)},
  MRNUMBER = {713270},
MRREVIEWER = {Yue\ Lin L. Tong},
       DOI = {10.1090/pspum/040.2/713270},
       URL = {https://doi.org/10.1090/pspum/040.2/713270},
}

@article {KyojiSaito-primitive,
    AUTHOR = {Saito, Kyoji},
     TITLE = {Period mapping associated to a primitive form},
   JOURNAL = {Publ. Res. Inst. Math. Sci.},
  FJOURNAL = {Kyoto University. Research Institute for Mathematical
              Sciences. Publications},
    VOLUME = {19},
      YEAR = {1983},
    NUMBER = {3},
     PAGES = {1231--1264},
      ISSN = {0034-5318,1663-4926},
   MRCLASS = {32G11 (32B30)},
  MRNUMBER = {723468},
MRREVIEWER = {Helmut\ Hamm},
       DOI = {10.2977/prims/1195182028},
       URL = {https://doi.org/10.2977/prims/1195182028},
}

@article {MorihikoSaito-Brieskorn,
    AUTHOR = {Saito, Morihiko},
     TITLE = {On the structure of {B}rieskorn lattice},
   JOURNAL = {Ann. Inst. Fourier (Grenoble)},
  FJOURNAL = {Universit\'e{} de Grenoble. Annales de l'Institut Fourier},
    VOLUME = {39},
      YEAR = {1989},
    NUMBER = {1},
     PAGES = {27--72},
      ISSN = {0373-0956,1777-5310},
   MRCLASS = {32S40 (32S35)},
  MRNUMBER = {1011977},
MRREVIEWER = {Autorreferat},
       DOI = {10.5802/aif.1157},
       URL = {https://doi.org/10.5802/aif.1157},
}

@book {Deligne-de,
    AUTHOR = {Deligne, Pierre},
     TITLE = {\'Equations diff\'erentielles \`a{} points singuliers
              r\'eguliers},
    SERIES = {Lecture Notes in Mathematics},
    VOLUME = {Vol. 163},
 PUBLISHER = {Springer-Verlag, Berlin-New York},
      YEAR = {1970},
     PAGES = {iii+133},
   MRCLASS = {14D05 (14C30)},
  MRNUMBER = {417174},
MRREVIEWER = {Helmut\ Hamm},
}

@article{KOL,
    author = {Koll\'{a}r, J\'{a}nos},
    title = {Nonrational hypersurfaces},
    journal = {J. Amer. Math. Soc.},
    year = {1995},
volume = {8},
pages = {241–249}
}

@article{SCH,
    author = {Schreieder, Stefan},
    title = {Stably irrational hypersurfaces of small slopes},
    journal = {J. Amer. Math. Soc.},
    year = {2019},
volume = {32},
pages = {1171–1199}
}

@article{TOT,
    author = {Totaro, Burt},
    title = {Hypersurfaces that are not stably rational},
    journal = {J. Amer. Math. Soc.},
    year = {2016},
volume = {29},
pages = {883–891}
}

@book {milne-AG,
    AUTHOR = {Milne, James S.},
     TITLE = {Algebraic groups},
    SERIES = {Cambridge Studies in Advanced Mathematics},
    VOLUME = {170},
      NOTE = {The theory of group schemes of finite type over a field},
 PUBLISHER = {Cambridge University Press, Cambridge},
      YEAR = {2017},
     PAGES = {xvi+644},
      ISBN = {978-1-107-16748-3},
   MRCLASS = {14L15 (14-01 17B45 20-01 20G15)},
  MRNUMBER = {3729270},
MRREVIEWER = {Boris\ \`E.\ Kunyavski\u i},
       DOI = {10.1017/9781316711736},
       URL = {https://doi.org/10.1017/9781316711736},
}

@misc{milne-tannakian,
    author = {Milne, James S.},
    title = {Tannakian categories},
    pages = {v+324},
    year = {2025},
     url ={https://www.jmilne.org/math/Books/tcdraft.pdf}
}

@book {DeligneMilne,
    AUTHOR = {Deligne, Pierre and Milne, James S. and Ogus, Arthur and Shih,
              Kuang-yen},
     TITLE = {Hodge cycles, motives, and {S}himura varieties},
    SERIES = {Lecture Notes in Mathematics},
    VOLUME = {900},
 PUBLISHER = {Springer-Verlag, Berlin-New York},
      YEAR = {1982},
     PAGES = {ii+414},
      ISBN = {3-540-11174-3},
   MRCLASS = {14Kxx (10D25 12A67 14A20 14F30 14K22)},
  MRNUMBER = {654325},
}

@misc{milne-MT,
   author = {Milne, James S.},
    title = {Classification of the {M}umford–{T}ate Groups of Rational
Polarizable {H}odge Structures},
    pages = {11},
    year = {2023},
url = {https://www.jmilne.org/math/articles/CMT.pdf}
}

@article {IritaniMannMignon,
    AUTHOR = {Iritani, Hiroshi and Mann, Etienne and Mignon, Thierry},
     TITLE = {Quantum {S}erre theorem as a duality between quantum
              {$D$}-modules},
   JOURNAL = {Int. Math. Res. Not. IMRN},
  FJOURNAL = {International Mathematics Research Notices. IMRN},
      YEAR = {2016},
    VOLUME = {9},
     PAGES = {2828--2888},
      ISSN = {1073-7928,1687-0247},
   MRCLASS = {14N35 (53D45)},
  MRNUMBER = {3519131},
MRREVIEWER = {Hsian-Hua\ Tseng},
       DOI = {10.1093/imrn/rnv215},
       URL = {https://doi.org/10.1093/imrn/rnv215},
}

\bigskip

Ludmil Katzarkov, IMSA, University of Miami, USA; ICMS, IMI, Bulgaria; 
HSE University, Russia

\emph{Email:} \url{l.katzarkov@math.miami.edu}

\medskip

Maxim Kontsevich, Institut des Hautes Études Scientifiques, France

\emph{Email:} \url{maxim@ihes.fr}

\medskip

Tony Pantev, University of Pennsylvania, USA

\emph{Email:} \url{tpantev@math.upenn.edu}

\medskip

Tony Yue Yu, California Institute of Technology, USA

\emph{Email:} \url{yuyuetony@gmail.com}

\end{document}